\documentclass[12pt]{amsart}

\usepackage{amssymb,latexsym}

\usepackage{graphicx,epsfig}
\usepackage[dvips]{color}

\def\cal{\mathcal}
\def\frak{\mathfrak}
\def\Bbb{\mathbb}

\def\Re{\text{\rm Re\,}}

\def\sgn{\text{\rm sgn\,}}

\def\pad{\phi^a}
\def\tpad{\tilde\phi^a}

\def\tka{\tilde\kappa}

\def\ra{\rangle}
\def\laa{\langle}

\def \supp {\text{\rm supp\,}}

\def \sgn {\text{\rm sgn\,}}

\def\r{{\frak r}}
\def\hl{{h_{\rm \,lin}}}
\def\A{{\cal A}}
\def\D{{\cal D}}

\def\N{{\cal N}}
\def\S{{\cal S}}

\def\bC{{\Bbb C}}

\def\NN{{\Bbb N}}

\def\bR{{\Bbb R}}
\def\RR{{\Bbb R}}
\def\bZ{{\Bbb Z}}

\def\La{{\Lambda}}
\def\vp{{\varphi}}
\def\al{{\alpha}}
\def\be{{\beta}}
\def\ga{{\gamma}}
\def\la{{\lambda}}
\def\om{{\omega}}
\def\eps{{\epsilon}}
\def\x{(x_1,x_2)}
\def\y{(y_1,y_2)}
\def\pa{{\partial}}

\def\ve{{\varepsilon}}
\def\si{{\sigma}}

\def\de{{\delta}}
\def\ze{{\zeta}}

\def\ka{{\kappa}}
\def\tka{{\tilde\kappa}}
\def\th{{\theta}}

\def\pr{\text{\rm pr\,}}

\def\bpm{\begin{pmatrix}}
\def\epm{\end{pmatrix}}

\def\noi{\noindent}
\def\bee{\begin{enumerate}}
\def\ee{\end{enumerate}}

\newtheorem{thm}{Theorem}[section]
\newtheorem{prop}[thm]{Proposition}
\newtheorem{proposition}[thm]{Proposition}
\newtheorem{cor}[thm]{Corollary}
\newtheorem{lemma}[thm]{Lemma}
\newtheorem{remark}[thm]{Remark}

\newtheorem{assumption}[thm]{Assumption}

\begin{document}

\title[A sharp restriction theorem]{
$L^p$-$L^2$  Fourier restriction  for  hypersurfaces in  $\bR^3:$
Part II\\
VERSION 13.10.2014}

\author[I. A. Ikromov]{Isroil A. Ikromov}
\address{Department of Mathematics, Samarkand State University,
University Boulevard 15, 140104, Samarkand, Uzbekistan}
\email{{\tt ikromov1@rambler.ru}}

\author[D. M\"uller]{Detlef M\"uller}
\address{Mathematisches Seminar, C.A.-Universit\"at Kiel,
Ludewig-Meyn-Stra\ss{}e 4, D-24098 Kiel, Germany}
\email{{\tt mueller@math.uni-kiel.de}}
\urladdr{{http://analysis.math.uni-kiel.de/mueller/}}

\thanks{We gratefully acknowledge the support for this work by the Deutsche Forschungsgemeinschaft.}

\begin{abstract}
This is the second  in a serious of two  articles,  in which we prove a sharp  $L^p$-$L^2$ Fourier
restriction theorem for  a large class of smooth, finite type hypersurfaces in
$\RR^3,$ which includes in particular all  real-analytic hypersurfaces.
\end{abstract}

\maketitle

\maketitle


\tableofcontents

\thispagestyle{empty}

\setcounter{equation}{0}
\section{Introduction}\label{intro}
This is the second part of a pair of articles whose main goal is to prove our main result, Theorem 1.7, in \cite{IM-rest1} on $L^p$-$L^2$ Fourier restriction estimates for smooth hypersurfaces of finite type in $\RR^3.$  For the relevant statements, definitions and bibliographical references we  therefore refer the reader to the introduction to that article. Under the assumption that our hypersurface $S$ is given as the graph of a smooth function $\phi$ defined near the origin and satisfying the conditions $\phi(0,0)=0$ and $\nabla\phi(0,0)=0,$  we had covered in \cite{IM-rest1} all situations with the exception of the cases where  the so-called linear height $\hl(\phi)$  satisfies $2\le \hl(\phi)<5.$ For this case, substantially more refined  methods  than the ones used  in \cite{IM-rest1} are needed, since the use of Drury's restriction estimate for non-degenerate curves turns out to be insufficient. In fact, the  method that we shall develop in this second  part will work whenever $ \hl(\phi)\ge 2.$
\medskip

Throughout this article, we shall make the following general 
\begin{assumption}\label{assumption}
There is no linear coordinate system which is adapted to $\phi.$
\end{assumption}
Moreover, we may and shall assume that we are in linearly adapted coordinates, so that $d=\hl(\phi)\ge 2.$ Recall also from  \cite{IM-rest1} that this assumption implies that  the principal face $\pi(\phi)$ of the Newton polyhedron of $\phi$ is a compact edge  which is intersected by the bi-sectrix 
$$
\Delta:=\{(t_1,t_2)\in\bR^2:t_1=t_2\},
$$
in an interior point, given by $(d,d),$ and that $\pi(\phi)$ is contained in the {\it principal line}
$$
L:=\{(t_1,t_2)\in\bR^2:\ka_1t_1+\ka_2t_2=1\},
$$
and thus determines a weight $\ka:=(\ka_1,\ka_2),$ so that also  $m=\ka_2/\ka_1\ge 2.$

\medskip

{\bf Conventions:}  As in \cite{IM-rest1}, we shall use the ``variable constant'' notation in this article, i.e.,  many constants appearing in the paper, often  denoted by  $C,$  will typically have different values  at different lines.  Moreover, we shall use  symbols such as  $\sim, \lesssim$ or $\ll$ in order to avoid  writing down constants. By $A\sim B$ we mean that there are constants $0<C_1\le C_2$ such that $C_1 A\le B\le C_2 A,$ and  these constants will not depend on the relevant parameters arising in the context in which the quantities $A$ and $B$  appear. Similarly, by $A\lesssim B$ we mean that there  is a  (possibly large)  constant  $C_1>0$ such that  $A\le C_1B,$  and by $A\ll B$ we mean that there  is a sufficiently  small  constant  $c_1>0$ such that  $A\le c_1B,$  and  again these constants do not depend on the relevant parameters.

By $\chi_0 $ and $\chi_1$ we shall  always denote  smooth cut-off functions with compact support on $\RR^n,$ where $\chi_0 $ will   be supported in a neighborhood of the origin, whereas $\chi_1=\chi_1(x)$  will be support away from the origin in each of its coordinates $x_j,$ i.e., $|x_j|\sim 1$ for every $j=1,\dots,n.$ These cut-off functions may also vary from line to line, and may  in some instances, where several of  such functions of different variables appear within the same formula, even   designate different functions.

Also, if we speak of the {\it slope}  of a line such as a supporting line to a Newton polyhedron, then we shall actually  mean the modulus of the slope.

\setcounter{equation}{0}
\section{The case when $\hl(\phi)\ge 2$: Reminder of the open cases}\label{remind}

Recall from \cite{{IM-rest1}}, Section 9,  the following two {\bf Cases:}
\begin{enumerate}
 \item[(a)] The principal face $\pi(\pad)$ of the Newton polyhedron $\N(\pad)$  of $\pad$ is a compact edge, which lies on a line $L^a,$ which we call the {\it principal line} of $\N(\pad)$ 

\item[(b)]  $\pi(\pad)$ is the vertex $(h,h).$  
 \end{enumerate}

What had remained open in \cite{{IM-rest1}}   was the study of  the piece of the surface $S$ corresponding to the domain $D_\pr$  containing the principal root jet $\psi,$ in the cases (a) and (b),  i.e.,
\begin{equation}\label{restdomain2}
D_\pr:=\begin{cases}  
\{ \x: 0<x_1<\ve, \,|x_2-\psi(x_1)|\le N x_1^a\} &   \mbox{ in Case (a)},   \\
      \{ \x: 0<x_1<\ve, \, |x_2-\psi(x_1)|\le \ve x_1^{a}\}, &   \mbox{ in Case (b)},
\end{cases}
\end{equation}
when $2\le \hl(\phi)<5.$  Indeed, we shall here develop an approach which will work whenever $\hl(\phi)\ge 2.$ Our goal will thus be to prove the following extension of Proposition 12.1 in \cite{{IM-rest1}} to the case where $\hl(\phi)\ge 2:$

\begin{proposition}\label{rdomain}
Assume that $\hl(\phi)\ge 2,$ and that we are in Case (a) or (b). When $\ve>0$ is sufficiently small, and $N$ is sufficiently large in Case (a),   then
$$
\Big(\int_{D_{\pr}} |\widehat f|^2\,d\mu\Big)^{1/2}\le C_{p} \|f\|_{L^p(\RR^3)},\qquad  f\in\S(\RR^3),
$$
whenever $p'\ge p'_c .$
\end{proposition}

In order to prove this proposition, we follow the domain decomposition algorithm for the domain $D_\pr$ developed in  Section 12 of  \cite{IM-rest1}. In Case (a), that algorithm led  to a finite family of subdomains   $E_{(l)}$ (so-called transition domains) and domains $D'_{(l)}, l\ge 1,$ of the form
$$
D'_{(l)}:=\{ \x: 0<x_1<\de, \,  |x_2-\psi^{(l+1)}(x_1)|\le \ve  x_1^{a_{(l)}}\},
$$
where the functions $\psi^{(l)}$  are of the form
$$
\psi^{(l)}(x_1)=\psi(x_1)+\sum_{j=1}^{l-1}c_{j-1} x_1^{a_{(j)}},
$$
with real coefficients $c_j$, and where the exponents $a_{(j)}$ form a strictly increasing sequence
$$
a=a_{(1)}<a_{(2)}<\cdots a_{(l)}<a_{(l+1)}<\cdots
$$
of rational numbers.    Moreover, in the modified adapted coordinates given by 
$$
y_1:= x_1, \quad y_2:=x_2-\psi^{(l+1)}(x_1),
$$
the function $\phi$ is given by
$$
\phi^{(l+1)}\y:=\phi(y_1, y_2+\psi^{(l+1)}(y_1)).
$$
Notice that  we can define these notions  also for $l=0,$  and then have $\psi^{(0)}=\psi$ and $\phi^{(0)}=\pad.$

\medskip
Moreover, the domain  $D'_{(l)}$  is  associated to an ``edge''  $\ga'_{(l)}=[(A'_{(l-1)},B'_{(l-1)}), (A'_{(l)},B'_{(l)})]$ (which is indeed an edge, or can degenerate to a single point) of the Newton polyhedron of $\phi^{(l+1)}$ in the following way:

\medskip
The edge with index $l$ will lie on a line 
$$
L_{(l)}:=\{(t_1,t_2)\in\bR^2:\ka^{(l)}_1t_1+\ka^{(l)}_2t_2=1\},
$$
of slope $1/ a_{(l)}$ (here, we shall always mean the modulus of the slope), 
where $ a_{(l)}=\ka^{(l)}_2/\ka^{(l)}_1.$  Introduce corresponding ``$\ka^{(l)}$-dilations'' $\de_r=\de_r^{\ka^{(l)}}$ by putting 
$\de_r\y=(r^{\ka_1^{(l)}}y_1,r^{\ka_2^{(l)}}y_2), \ r>0.$  Then the domain
$$
D'^a_{(l)}:=\{ \y: 0<y_1<\ve, \,  |y_2|\le \ve y_1^{a_{(l)}}\},
$$
which represents the domain $D'_{(l)}$  in the coordinates $\y,$ is invariant under these dilations, and the Newton diagram of the $\ka^{(l)}$-principal part  $\phi^{(l+1)}_{\ka^{(l)}}$ of $\phi^{(l+1)}$ agrees with the edge $\ga'_{(l)}.$

Recall also that the first   edge $\ga'_{(1)}$ agrees with the principal face  $\pi(\phi^{(2)})$ of \ $\phi^{(2)}$ and lies on the principal line $L^a$ of the Newton polyhedron of $\pad,$ and it  intersects the bi-sectrix $\Delta,$
whereas for $l\ge 2$ the edge $\ga'_{(l)}$  will lie in the closed half-space below the bi-sectrix.

Moreover, the Newton polyhedra of $\pad$ and of $\phi^{(l)}$ do agree in the closed  half-space above the bi-sectrix.

\medskip
Now, in  Section 8 of \cite{{IM-rest1}}, setting  $v=(1,0),$ we had distinguished  between the cases where $\pa_2\phi^{(l+1)}_{\ka^{(l)}}(v) \neq 0$ (Case 1),  $\pa_2\phi^{(l+1)}_{\ka^{(l)}}(v) =0$ and $\pa_1\phi^{(l+1)}_{\ka^{(l)}}(v) \ne 0$ (Case 2),  and the case where $\nabla \phi^{(l+1)}_{\ka^{(l)}}(v) =0$ (Case 3), and studied restriction estimates for the pieces of the surface $S$ corresponding to the domain $D'_{(l)}.$ 

Only in Case 2  we had made use of the assumption $\hl(\phi)\ge 5,$ so we can concentrate in the sequel on Case 2.

Notice also that our decomposition algorithm worked as well in Case (b), only that we had to skip the first step of the algorithm. We shall therefore first study the domain $D'_{(1)}$ in  Case (a), and in the last section describe the minor modifications needed to treat also the domains $D'_{(l)}$ for $l\ge 2,$ which will then also cover Case (b) at the same time.

 We can localize to the domain $D'_{(l)}$ by means of a cut-off function 
$$
\rho_{(l)}\x:=\chi_0\Big(\frac {x_2-\psi^{(l+1)}(x_1)}{\ve x_1^{a_{(l)}}}\Big),
$$
where $\chi_0\in\D(\RR).$ 
Let us again fix a suitable smooth cut-off function $\chi\ge 0$ on $\RR^2$ supported in an annulus $\A:=\{x\in \RR^2:1/2\le |y|\le R\}$ such that the functions $\chi^a_k:=\chi\circ \de_{2^k}$ form a partition of unity. Here, $\de_r=\de_r^{\ka^{(l)}}$ denote the dilations associated to the weight $\ka^{(l)}.$  In the original coordinates $x,$ these correspond to the functions $\chi_k(x):=\chi^a_k(x_1,x_2-\psi^{(l+1)}(x_1)).$ We then decompose the measure $\mu^{\rho_{(c_0)}}$ dyadically as 
\begin{eqnarray}\label{2.2nII}
\mu^{\rho_{(l)}}=\sum_{k\ge k_0}\mu_k,
\end{eqnarray}
where 
$$\mu_k:=\mu^{(l)}_k :=\mu^{\chi_k\rho_{(l)}}.
$$
  Notice that by choosing the support of $\eta$ sufficiently small, we can choose $k_0\in \NN$ as large as we need. It is also important to observe that  this decomposition can essentially we achieved by means of a dyadic decomposition with respect to the variable $x_1,$  which again allows to apply Littlewood-Paley theory (see \cite{IM-rest1}).  
  
  Moreover, changing to modified adapted coordinates in the integral defining $\mu_k$ and scaling by $\de_{2^{-k}}$ we find that
\begin{eqnarray}\nonumber
\laa\mu_k,\, f\ra &=&2^{-k|\ka^{(l)}|}\int f(2^{-\ka_1^{(l)} k}x_1,\, 2^{-\ka_2^{(l)} k}x_2+2^{-m\ka_1^{(l)} k}x_1^m\om(2^{-\ka_1^{(l)}k} x_1),\, 2^{-k}\phi_k(x))\\
&&\hskip7cm  \eta(x) \, dx,\label{2.3nII}
\end{eqnarray}
where $\om=\om^{(l)}$ is given by 
$$
\psi^{(l+1)}(x_1)=x_1^m\om(x_1),
$$
so that $\om(0)\ne 0,$ 
$\eta=\eta_{k}^{((l)}$ is a smooth function supported  where $x_1\sim1,\,|x_2|<\ve $ (for some  small $\ve>0$), whose derivatives are uniformly bounded  $k,$  and where
\begin{equation}\label{7.2}
\phi_k(x)=\phi_k^{(l+1)}(x):=2^k\phi^{(l+1)}(\de_{2^{-k}}x)=\phi^{(l+1)}_{\ka^{(l)}} (x)+ \mbox{error terms of order } O(2^{-\delta k})
\end{equation}
with respect to the $C^\infty$ topology (and  $\delta>0).$

\medskip
In order to prove Proposition \ref{rdomain}, we then still need to prove 

\begin{proposition}\label{Dl}
Assume that $\hl\ge 2,$   that we are in Case 2, i.e., $\pa_2\phi^{(l+1)}_{\ka^{(l)}}(1,0) =0$ and $\pa_1\phi^{(l+1)}_{\ka^{(l)}}(1,0) \ne 0,$ and  recall that $p'_c=2h^r+2 .$ When $\ve>0$ is sufficiently small and $k_0\in\NN$ is sufficiently large, then for every $l\ge 1,$ 
\begin{equation}\label{estk}
\Big(\int |\widehat f|^2\,d\mu_k\Big)^{1/2}\le C_{p_c} \|f\|_{L^p(\RR^3)},\qquad  f\in\S(\RR^3), \ k\ge k_0,
\end{equation}

 where the constant $C_p$ is independent of $k.$
\end{proposition}

\setcounter{equation}{0}
\section{Restriction estimates for the domain $D'_{(1)}$ }\label{dom1}

Let us assume that we are in Case (a), where  the principal face $\pi(\pad)$ is a compact edge. In the enumeration of edges $\ga_l$ of the Newton polyhedron associated to $\pad$ in Section 7 of \cite{IM-rest1}, this edge corresponds to the index $l=l_\pr,$ i.e.,
\begin{equation}\label{lpr}
\pi(\pad)=\ga_{l_\pr}.
\end{equation}
 The weight $\ka^{(1)}$ is  here the principal weight $\ka^{l_\pr}$  from \cite{IM-rest1}, and the line $L_{(1)}$ is the principal line $L^a=L_{l_\pr}$ of the Newton polyhedron of $\pad.$ 
We then put 
$$
\tilde\ka:=\ka^{(1)}, \quad\mbox{so that } \quad a=\frac {\tilde \kappa_2}{\tilde \kappa_1}, \ \pad_{\tilde\kappa}=\pad_\pr.
$$
In particular, $h_{l_\pr}+1$ is the second coordinate of the point of intersection of the line 
$$
\Delta^{(m)}:=\{(t,t+m+1):t\in\RR\}
$$
with the line $L^a,$ and according to \cite{IM-rest1}, display (1.11), is given by 
\begin{equation}\label{hlpr}
 h_{l_\pr}+1=\frac {1+(m+1)\tilde\ka_1}{|\tilde\ka|} .
\end{equation}

The domain $D'_{(1)}$ that we have to study is then of the form
$$
D'_{(1)}=\{\x:0<x_1<\ve, \,  |x_2-\psi(x_1)-c_0 x_1^{a}|\le \ve x_1^{a}\},
$$
where $\psi(x_1)+c_0 x_1^{a}=\psi^{(2)}(x_1).$  Moreover,
\begin{equation}\label{phi2}
\phi^{(2)}\x=\pad(x_1,x_2-c_0 x_1^{a})=:\tilde\pad\x,
\end{equation}
so that $\tpad$ represents $\phi$ in the modified adapted coordinates 
\begin{equation}\label{modada}
y_1:=x_1,\quad y_2:= x_1-\psi(x_1)-c_0 x_1^{a},
\end{equation}
compared to the adapted coordinates $y_1:=x_1,\quad y_2:= x_1-\psi(x_1),$ in which $\phi$ is represented by $\pad.$ 

Notice that  the exponent $a$ may be non-integer (but rational), so that $\psi^{(2)}$ is in general  only fractionally smooth, i.e., a  smooth function of $x_2$ and some fractional power of $x_1$ only.  The same applies to every $\psi^{(l)}$ with $l\ge 2,$ whereas $\pad$ is still smooth, i.e., when we express $\phi$ in our adapted coordinates, we still get a smooth function,  whereas when we pass to modified adapted coordinates, we may only get fractionally  smooth functions.

\medskip

We shall write $D^a$  for the domain $D'^a_{(1)},$ i.e., 
$$
D^a:=\{\y: 0<y_1<\ve, |y_2|<\varepsilon y_1^a \},
$$
so that $D^a$ represents our domain $D'_{(1)}$ in our modified adapted coordinates, in which $\phi$ is represented by $\tilde\pad.$ 

\medskip
We assume that we are in Case 2, so that 
 $\pa_2\tilde\pad_{\tilde\ka}(1,0) =0$ and $\pa_1\tilde\pad_{\tilde\ka}(1,0) \ne 0.$
\medskip

We choose $B\ge 2$ minimal so that $\pa_2^B\tpad_{\tilde \kappa}(1,0) \ne 0.$ Since $\tpad_{\tilde \kappa}$ is $\tilde\ka$-homogeneous,  the principal part of $\tpad$ is then of the form (cf. (9.6) in \cite{{IM-rest1}})
\begin{equation}\label{2.2II}
\tpad_{\tilde \kappa}\y=y_2^BQ\y+c_1y_1^n,\quad c_1\neq 0, \ Q(1,0)\ne 0,
\end{equation}
where $Q$ is a $\tilde\kappa$-homogeneous smooth function.  Note that $n$ is rational, but not necessarily integer, since we are in modified adapted coordinates.

Observe also that this implies that we may write
\begin{equation}\label{2.3II}
\tpad\y=y_2^B \,b_B\y+y_1^n\al(y_1)+\sum_{j=1}^{B-1}y_2^j \, b_j(y_1),
\end{equation}
with smooth functions  $b_B,\al$ such that $\al(0)\ne 0$  and 
$$
b_B\y=Q\y+ \mbox{ terms of $\tilde\ka$-degree  strictly bigger than that of $Q,$}
$$
and smooth functions  $b_1,\dots, b_{B-1}$ of $y_1,$ which are  either flat, or of finite type $b_j(y_1)= y_1^{n_j} \al_j(y_1),$ with smooth functions $\al_j$ such that  $\al_j(0)\ne 0.$ 

 For convenience, we shall also write $b_j(y_1)= y_1^{n_j} \al_j(y_1)$ when $b_j$ is flat, keeping in mind that in this case we may choose $n_j\in\NN$ as large as we please (but  $\al_j(0)=0$). 

Notice that then, for $j=1,\dots,B-1,$ 
$y_2^jb_j(y_1)$ consists of terms of $\tilde\ka$-degree strictly bigger than $1.$

 \medskip
 
 Recall that the Newton diagram of  the $\tilde\ka$-principal part $\tpad_{\tilde\ka}$ is the line segment $\ga'_{(1)}=[(A'_{(0)},B'_{(0)}), (A'_{(1)},B'_{(1)})],$ which must then contain a point with second coordinate given by $B.$ It then follows easily   that the following relations hold true:
$$
\tilde\ka_2<1,\quad 2\le m<\frac1{\tilde\ka_1}=n,\quad B\le B'_{(0)}.
$$
Actually, since $\tilde\ka_2 B'_{(0)}\le1,$ we even have 
\begin{equation}\label{2.5II}
\tilde\ka_2 B\le1,\quad m<a=\frac{\tilde\ka_2}{\tilde\ka_1}\le
\frac{n}{B}.
\end{equation}

As in  \cite{IM-rest1}, we define  normalized 
measures $\nu_k$  corresponding to the $\mu_k$  by
$$
\laa\nu_k,f\ra:=\int f\Big(x_1,\,
2^{(m\tilde\ka_1-\tilde\ka_2)k}x_2+x_1^m\omega(2^{-k\tilde\ka_1}x_1),\,\phi_k(x)\Big)\, \eta(x)\,dx,
$$
where again $\eta$ is a smooth function with $\supp \eta\subset\{x_1\sim1,\,|x_2|<\ve  \}$ (for some  small $\ve>0$)  and $\phi_k\x:=2^k\tpad(2^{-\tilde\ka_1 k} x_1,2^{-\tilde\ka_2 k} x_2)$  is given by
\begin{eqnarray}\nonumber
\phi_k\x&:=&x_2^B(Q(x)+O(2^{-\ve' k}))+x_1^n\alpha(2^{-\tilde\ka_1 k}x_1)
 +\sum_{j=1}^{B-1} x_2^j \,2^{(1-j\tilde\ka_2)k} \,b_j(2^{-\tilde\ka_1 k} x_1)
\label{2.6II}
\end{eqnarray}
for some $\ve'>0.$ Observe that 
$$
2^{(1-j\tilde\ka_2)k} \,b_j(2^{-\tilde\ka_1 k} x_1)= 
x_1^{n_j}2^{-(j\tilde\ka_2+n_j\tilde\ka_1-1)k}\al_j(2^{-\tilde\ka_1 k}x_1),
$$
where $(j\tilde\ka_2+n_j\tilde\ka_1-1)>0.$

We write $\nu_k$ as $\nu_\de,$ by putting 
\begin{equation}\label{nud}
\laa\nu_\de,f\ra:=\int f\Big(x_1,\,\de_0x_2+x_1^m\omega(\de_1x_1),\,\phi_\de(x)\Big)\,\eta(x)\,dx,
\end{equation}

where $\phi_\de$ is of the form
\begin{equation}\label{phid}
\phi_\de(x):=x_2^B\, b(x_1,x_2,\de)+x_1^n\alpha(\de_1x_1)+ r(x_1,x_2,\de),
\end{equation}
with
\begin{equation}\label{defr}
r(x_1,x_2,\de):=\sum_{j=1}^{B-1}\de_{j+2}\,x_2^jx_1^{n_j}\al_j(\de_1x_1),
\end{equation}
 
and $\de=(\de_0,\,\de_1,\,\de_2,\,\de_3,\dots,\de_{B+1})$ is given by 

\begin{equation}\label{defde}
\de:=(2^{-k(\tilde\ka_2-m\tilde\ka_1)},\,2^{-k\tilde\ka_1},\,2^{-k\tilde\ka_2},\,2^{-(n_1\tilde\ka_1+\tilde\ka_2-1)k},\dots,2^{-(n_{B-1}\tilde\ka_1+(B-1)\tilde\ka_2-1)k)}).
\end{equation}

Recall   that  $\alpha(0)\neq 0,$ and that either $\al_j(0)\ne 0,$ and then $n_j$  is fixed (the type of the finite type function $b_j$), or $\al_j(0)=0,$ and then we may assume that $n_j$ is as large as we please.

Observe that $\de\to 0$ as $k\to \infty,$ that every $\de_j$ is a   power of $\de_0,$ 
$$
\de_j=\de_0^{q_j}, \quad j=1,\dots, B+1,
$$
with positive exponents $q_j>0$ which are  fixed rational numbers, except for those $j\ge 2$ for which $\al_{j-2}(0)=0,$ for which we may choose the exponents $q_j$ as large as we please.

Moreover, $b(x_1,x_2,\de)$ is a smooth function of all three arguments, and 
\begin{equation}\label{bQ}
b(x_1,x_2,0)=Q(x_1,x_2).
\end{equation}
For $\de$ sufficiently small, this implies in particular that $b(x_1,0,\de)\ne 0$ when $x_1\sim 1$ and $|x_2|<\ve.$

Assume we can prove that 
\begin{equation}\label{2.7II}
 \left(\int|\hat f|^2d\nu_\de \right)^\frac12
\le C\|f\|_{L^{p_c}},
\end{equation}
with $C$ independent of $\de.$ Then straight-forward re-scaling by means of the  $\tilde\ka-$ dilations  leads to the estimate
\begin{equation}\label{rescale}
\left(\int|\hat f|^2d\mu_k \right)^\frac12 \le C
2^{-k\frac{|\tilde\ka|}2\left(1-\frac{2(h_{l_\pr}+1)}{p_c'}\right)}
\|f\|_{L^{p_c}},
\end{equation}
where $p_c'\ge 2(h_{l_{pr}}+1)$ (cf. 11.5) in \cite{IM-rest1}).   So, our goal is to verify \eqref{2.7II}.

Observe also that the $\tilde\ka$-principal parts of $\tpad$ and $\phi_\de$ do agree.
\bigskip

Recall that $B\ge 2,$ and $d\ge 2.$ We shall often use the   interpolation parameter $\th_c:= 2/p'_c= 1/(h^r+1).$ Since, by definition, $h^r\ge d,$ the second assumption  implies
\begin{equation}\label{ontc}
\th_c\le \frac 13.
\end{equation}

\medskip

We first derive some  useful estimates from below for $p'_c=2(h^r+1).$ We put $H:=1/\tilde\ka_2,$ so that 
\begin{equation}\label{Hdef}
n=1/\tilde\ka_1,\qquad H=1/\tilde\ka_2.
\end{equation}
Note that $H$ is rational, but not necessarily entire. We next define
$$
\tilde h^r:=\frac {mH}{m+1}, \quad \tilde p'_c:=2(\tilde h^r+1), \quad \tilde \th_c:=\frac 2{\tilde p'_c}=\frac {m+1}{mH+m+1}\le \tilde\th_B:=\frac {m+1}{mB+m+1}.
$$

Let us also put $\tilde p'_B:=2/\tilde\th_B\le \tilde p'_c,$ and 

$$
p'_{H}:=\frac {12H}{3+H}, \quad \th_H:=\frac 2 {p'_H}=\frac 1{2H}+\frac 1 6,
$$
and define $p'_B, \th_B$ accordingly, with $H$ replaced by $B.$

\begin{lemma}\label{pc-est1} \begin{itemize}
\item[(a)] We have  $p'_c>\tilde p'_c,$ unless $h^r=\tilde h^r=d$  and  $h^r+1\ge  H.$  In the latter case,  $p'_c= \tilde p'_c=2(d+1).$  

\item[(b)]  If $m\ge 3$ and $H\ge 2,$ or $m=2$ and $H\ge 3,$ then 
$$
 \tilde p'_c\ge p'_H\ge p'_B,
$$
where the inequality $\tilde p'_c\ge p'_H$ is even strict unless $m=2$ and $H=3.$ 
 \end{itemize}
\end{lemma}

\vspace{2cm}
\begin{figure}[!h]
\centering
\bigskip
\includegraphics[width=0.7\textwidth]{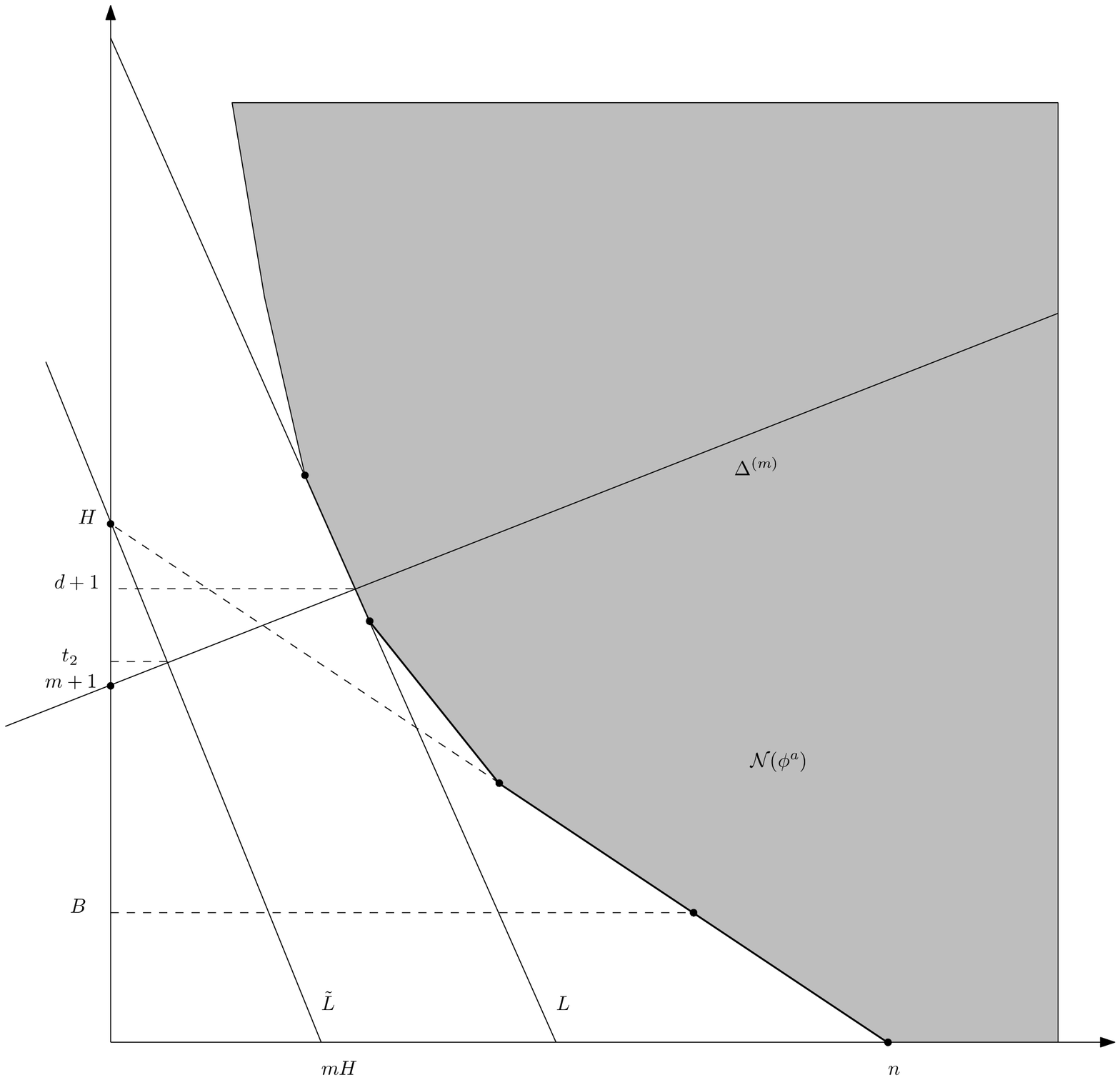}\label{fig1}
\caption{}
\end{figure}

\begin{proof} 
(a)  The Newton polyhedron $\N(\tpad)$ of $\tpad$ is contained in the closed half-space bounded from below by the 
principal line of $\tpad,$ which passes through the points $(0,H)$  and $(n,0).$ Moreover, it  known that the principal line $L$ of $\phi$ is a supporting line  to $\N(\tpad)$ (this follows from Varchenko's algorithm), and it has slope $1/m.$ It is therefore parallel to the line $\tilde L$ passing through the points $(0,H)$ and $(mH,0)$ and lies ``above'' $\tilde L$  (see Figure 1). Thus the second coordinate $d+1$ of the point of intersection of $L$ with $\Delta^{(m)}$ is greater or equal to the second coordinate $t_2$ of the point of intersection $(t_1,t_2)$ of $\tilde L$ with 
$\Delta^{(m)},$ so that $h^r\ge d\ge t_2-1.$ 

But, the point  $(t_1,t_2)$ is determined by the equations $t_2=m+1+ t_1$ and $t_2=H-t_1/m,$ so that $t_2=(mH+m+1)/(m+1).$ This shows that $h^r\ge \tilde h^r, $ hence $p'_c\ge \tilde p'_c.$ 

Notice also that $d+1>t_2,$ hence $p'_c> \tilde p'_c,$ unless $L=\tilde L.$ 

So, assume that  $L=\tilde L.$  
Then  $d=1/(1/H+1/mH)=\tilde h^r, $ and  the principal face $\pi(\tpad)$ of $\tpad$ must be the edge $[(0,H),(n,0)]$ (see Figure 1). Thus,  if  $h^r+1\ge H,$ then  clearly $h^r=d=\tilde h^r, $  and $p'_c= \tilde p'_c.$ And, if  $h^r+1< H,$ then we see that $h^r+1$ is the second coordinate of the point of intersection of $\Delta^{(m)}$ with $\pi(\tpad),$  and thus 
$$
\tilde h^r+1<h^r+1=h_{l_\pr}+1=\frac {1+(m+1)\tilde\ka_1}{|\tilde\ka|}
$$
(cf.  \eqref{hlpr}).

\medskip

(b) The inequality $\tilde p'_c\ge p'_H$ is equivalent to 
$$
mH^2-(2m+5)H +3m+3\ge 0,
$$
so that the remaining statements are elementary to check.
\end{proof}
The following corollary is a straight-forward  consequence of  the definition of $\th_H$ and Lemma \ref{pc-est1}.
\begin{cor}\label{compth}
 \begin{itemize}
\item[(a)] If $m\ge 3$ and $H\ge 2,$ or $m=2$ and $H\ge 3,$ then   $\th_c<\th_B,$ unless $m=2$ and $H=B=3$ (where $\th_c=\th_B=1/3$).
\item[(b)] If $h^r+1\le B,$ then $\th_c<\tilde\th_c,$ unless $B=H=h^r+1=d+1,$ where $\th_c=\tilde\th_c.$ 
\item[(c)] If $H\ge 3,$ then $\th_c< 1/3,$ unless $H=3$ and $m=2.$
 \end{itemize}
\end{cor}


\medskip




Recall next that the complete phase corresponding to $\phi_\de$  has the form
$$
\Phi(x,\de,\xi):=\xi_1x_1+\xi_2(\de_0x_2+x_1^m\omega(\de_1x_1))+\xi_3\phi_\de\x,
$$
where
$$
\phi_\de(x)=x_2^B\, b(x_1,x_2,\de)+x_1^n\alpha(\de_1x_1)+r(x_1,x_2,\de)$$
so that 
\begin{eqnarray*}
\Phi(x,\de,\xi)&=&\xi_1x_1+\xi_2 x_1^m\omega(\de_1x_1)+\xi_3 x_1^n\alpha(\de_1x_1)\\
&+& \xi_2\de_0x_2+\xi_3\Big( x_2^B\, b(x_1,x_2,\de)+r(x_1,x_2,\de)\Big).
\end{eqnarray*}

\setcounter{equation}{0}
\section{Spectral localization to frequency boxes where $|\xi_i|\sim \la_i:$ \\The case where not all $\la_i$'s are comparable}\label{specloc}

Denote by $T_\de$ the operator of convolution with $\widehat{\nu_\de},$  where we recall that 
$$
\widehat{\nu_\de}(\xi)=\int e^{-i\Phi(x,\de,\xi)}\, \eta(x) \, dx.
$$
In a next step, as in Section 5 of \cite{IM-rest1} (where $2^{-j}$ plays the same role is $\de_0$ here),  we decompose 
$$\nu_\de=\nu_k=\sum_{\la}\nu_\de^\la,
$$  
 where the sum is taken over all triples $\la=(\la_1,\la_2,\la_3)$ of dyadic  numbers $\la_i\ge 1,$  and where $\nu_k^\la$ is  localized to frequencies $\xi$ such that $|\xi_i|\sim \la_i,$  if $\la_i> 1,$ and $|\xi_i|\lesssim 1,$  if $\la_i=1.$ The cases where $\la_i=1$ for at least one $\la_i$ can be dealt with in the same way as the corresponding cases where $\la_i=2,$  and therefore we shall always assume in the sequel that
$$\la_i>1. \qquad i=1,2.3.
$$

The spectrally localized measure $\nu_\de^\la(x)$ is then given by 
$$
\widehat{\nu^\la_\de}(\xi):=\chi_1\Big(\frac{\xi_1}{\la_1}\Big)\chi_1\Big(\frac{\xi_2}{\la_2}\Big)\chi_1\Big(\frac{\xi_3}{\la_3}\Big)\; \widehat{\nu_\de}(\xi),
$$
i.e., 
\begin{eqnarray}\nonumber
\nu^{\la}_\de(x)=\la_1\la_2\la_3&\int \check\chi_1\Big({\la_1}(x_1-y_1)\Big) \, \check\chi_1\Big({\la_2}(x_2-\de_0 y_2-y_1^m\omega(\de_1y_1)\Big)\\
&\check\chi_1\Big({\la_3}(x_3-\phi_\de(y)\Big) \,\eta(y)\,dy, \label{2.8II}
\end{eqnarray}
where $\check\chi$ is the inverse Fourier transform. Recall also that 
\begin{equation}\label{2.9II}
\supp \eta\subset\{y_1\sim1,\,|y_2|<\ve  \}, \quad (\ve \ll 1).
\end{equation}

Arguing as in \cite{IM-rest1},  by making use of the localizations given by the first and the third factor of the integrand, the integration in $y_1$ and $y_2$ (here we can apply  van der Corput type estimates) yield 
$$
\|\nu_\de^\la\|_\infty \lesssim \la_1\la_2\la_3
\la_1^{-1}\la_3^{-\frac1{B}}.
$$
Similarly, the localizations given by the first and second factor imply 
$$
\|\nu_\de^\la\|_\infty \lesssim \la_1\la_2\la_3
\la_1^{-1}(\la_2\de_0)^{-1},
$$
 and consequently
\begin{equation}\label{2.10II}
 \|\nu_\de^\la\|_\infty \lesssim \min\{\la_2\la_3^\frac{B-1}{B},\,
 \la_3\de_0^{-1}\}.
\end{equation}


 We have to  distinguish various cases. Notice first that it is  easy to see that the phase function $\Phi$ has no critical point with respect to $x_1$  if one of the components $\la_i$ of $\la$ is much bigger than the two others, so that integrations by parts yield that 
$$
\|\widehat{\nu_\de^\la}\|_\infty\lesssim |\la|^{-N}\quad \mbox{for every }\ N\in \NN,
$$
which easily implies estimates for  the operator  $ T_\de^\la:\vp\mapsto\vp *\widehat{\nu^\la_\de}$ which are better than needed.
We may therefore concentrate on the following, remaining cases.
\medskip

Recall  the  interpolation parameter $\th_c=2/ {p'_c}\le \frac 13.$ 

\medskip
{\bf 1. Case: $\la_1\sim \la_3,\, \la_2\ll\la_1$.} 
Applying first the method of  stationary phase  in  
 $x_1,$ and then van der Corput's  lemma in $x_2,$ we find that 
 \begin{equation}\label{2.11II}
 \|\widehat{\nu_\de^\la}\|_\infty \lesssim \la_1^{-\frac12-\frac1{B}}.
\end{equation}
By interpolation, using this estimate and the first one in  \eqref{2.10II}, we obtain
$$
\|T_\de^\la\|_{p\to p'}\lesssim
\la_1^{-\frac1{B}-\frac12+\frac32\th}\la_2^\th,
$$
where $\th=2/p'.$ 
Summation over all dyadic $\la_2$ with $\la_2\ll\la_1$ yields
$$
\sum_{\la_2\ll\la_1}\|T_\de^\la\|_{p\to p'}\lesssim
\la_1^{-\frac1{B}-\frac12+\frac52\th}.
$$
Notice  that for $\th:=\th_B$ we have 
$$
-\frac1{B}-\frac12+\frac52\th= \frac 14 (\frac 1B-\frac 13)\le 0, \qquad \mbox{if $B\ge 3,$}
$$
and that strict inequality holds when $\th<\th_B.$ But,  Corollary \ref{compth} (a)  shows that if $H\ge 3,$ then indeed $\th_c<\th_B,$ unless $m=2$ and $H=B=3.$ 
Consequently,  for $\th:=\th_c$  we can sum over all dyadic $\la_1$ (unless  $H=B=3$ and $m=2$) and obtain
\begin{equation}\label{TI}
\|T_\de^I\|_{p_c\to p'_c}\lesssim \sum_{\la_1\sim \la_3,\, \la_2\ll\la_1}\|T_\de^\la\|_{p_c\to p'_c}\lesssim 1,
\end{equation}
where  
$
T_\de^{I}:=\sum_{\la_1\sim \la_3,\, \la_2\ll\la_1}T_\de^{\la}
$
denotes the contribution by the operators $T_\de^\la$ which arise in this case. The constant in this estimate does not depend on $\de.$ 
\medskip

If  $H=B=3$ and $m=2,$ then  we only get a uniform estimate 
$$\sum_{\la_2\ll\la_1}\|T_\de^\la\|_{p_c\to p'_c}\lesssim 1.
$$ 

Finally, assume that $B=2.$  Since we assume that $\th_c\le 1/3,$ we than again find that 
$$
-\frac1{B}-\frac12+\frac52\th_c\le  -1+\frac52\,\frac 13<0,$$
so that \eqref{TI} remains valid.

\medskip
Let us return to the case where  $H=B=3$ and $m=2,$  hence $\th_c=1/3,$ which will require more refined methods. 

 In a first step, we shall take the sum of the $\nu_\de^\la$ over all dyadic $\la_2\ll \la_1.$  Moreover, since $\la_1\sim \la_3,$ we may reduce to the case where $\la_3=2^M\la_1,$ where $M\in\NN$ is fixed and not too large. For the sake of simplicity of notation, we then assume that $M=0.$ All this then amounts to considering the functions $\si_\de^{\la_1}$ given by 
$$
\widehat{\si_\de^{\la_1}}(\xi)=\chi_1\Big(\frac {\xi_1}{\la_1}\Big)\chi_0\Big(\frac {\xi_2}{\la_1}\Big)\chi_1\Big(\frac {\xi_3}{\la_1}\Big)\, \widehat{\nu_\de}(\xi),
$$
where now $\chi_0$ is smooth, compactly supported in an interval $[-\ve,\ve],$ where $\ve >0$ is sufficiently small, and $\chi_0\equiv 1$ in the interval $[-\ve/2,\ve/2].$  In particular, $\si_\de^{\la_1}(x)$ is given again by the expression \eqref{2.8II}, only with the second  factor $\check\chi_1\Big({\la_2}(x_2-\de_0 y_2-y_1^m\omega(\de_1y_1)\Big)$ in the integrand replaced by $\check\chi_0\Big({\la_1}(x_2-\de_0 y_2-y_1^m\omega(\de_1y_1)\Big)$ and $\la_2$ replaced by $\la_1.$  Thus we obtain the same type of estimates as in  \eqref{2.10II}, i.e., 
\begin{equation}\label{case1est}
\|\widehat{\si_\de^{\la_1}}\|_\infty \lesssim \la_1^{-\frac56}, \qquad  \|\si_\de^{\la_1}\|_\infty \lesssim\la_1^{\frac 23} \min\{\la_1, \de_0^{-1} \la_1^{\frac 13}\}=\la_1\min\{\la_1^{\frac 23}, \de_0^{-1} \}.
\end{equation}
By $T_\de^{\la_1}$  we shall denote the operator of convolution with $\widehat{\si_\de^{\la_1}}.$

\medskip 
In view of \eqref{case1est}, we shall  distinguish between two subcases:

\medskip

{\bf 1.1. The subcase where $\la_1\le \de_0^{-3/2}.$} In this case, by \eqref{case1est}   we have 
\begin{equation}\label{case1.1}
 \|\widehat{\si_\de^{\la_1}}\|_\infty \lesssim \la_1^{-\frac56}, \qquad \|\si_\de^{\la_1}\|_\infty \lesssim  \la_1^{\frac 53},
\end{equation}
so that  
$$
\|T_\de^{\la_1}\|_{p_c\to p'_c}\lesssim
\la_1^{-\frac13}\la_1^{\frac 13}=1,
$$ 
and summing these estimates does not lead to  the desired uniform estimate.  Let us denote by  
$$
T^{I_1}_\de:=\sum_{ \la_1\le \de_0^{-3/2}}T_\de^{\la_1}
$$
the contribution by the operators $T_\de^\la$ which arise in this subcase. In order to  prove the desired estimate 
\begin{equation}\label{TI_1}
\|T_\de^{I_1}\|_{p_c\to p'_c}\lesssim  1,
\end{equation}
we shall therefore have to apply an interpolation argument (see Subsection \ref{realint1}).

\medskip

{\bf 1.2. The subcase where $\la_1> \de_0^{-3/2}.$}  In this case we have 
$
 \|\si_\de^{\la_1}\|_\infty \lesssim \de_0^{-1} \la_1,
$ 
and interpolation yields
$$
\|T_\de^{\la_1}\|_{p_c\to p'_c}\lesssim
\de_0^{-\frac 13} \la_1^{-\frac29}.
$$
If  we denote by 
$
T_\de^{I_2}:=\sum_{ \la_1> \de_0^{-3/2}}T_\de^{\la_1}
$
the contribution by the operators $T_\de^\la$ which arise in this subcase, we thus obtain
\begin{equation}\label{TI_2}
\|T_\de^{I_2}\|_{p_c\to p'_c}\lesssim \sum_{ \la_1> \de_0^{-3/2}} \de_0^{-\frac 13} \la_1^{-\frac29}\lesssim 1.
\end{equation}

\medskip
{\bf 2. Case: $\la_2\sim\la_3$ and $\la_1\ll\la_2$.} Here, we can estimate $\widehat{\nu^{\la}_{\de}}$ in the same way as in the previous
case and obtain $\|\widehat{\nu^{\la}_{\de}}\|_\infty\lesssim \la_2^{-1/2} \la_3^{-1/B}\sim \la_2^{-1/2-1/B}.$ Moreover, by \eqref{2.10II}, 
we have $\|\nu^{\la}_{\de}\|_\infty\lesssim \la_2\min\{\la_2^{(B-1)/B}, \de_0^{-1}\}.$ Both these estimates are independent of $\la_1.$  Assuming  here  without loss of generality that $\la_2=\la_3,$ we therefore  consider the sum over all $\nu^\la_\de$ such that $\la_1\ll 1,$  by putting
 $
\si_\de^{\la_2}:=\sum_{\la_1\ll \la_2}\nu^{(\la_1,\la_2,\la_2)}_\de.
$
This means that 
$$
\widehat{\si_\de^{\la_2}}(\xi)=\chi_0\Big(\frac {\xi_1}{\la_2}\Big)\chi_1\Big(\frac {\xi_2}{\la_2}\Big)\chi_1\Big(\frac {\xi_3}{\la_2}\Big)\; \widehat{\nu_\de}(\xi),
$$
where now $\chi_0$ is smooth and compactly supported in an interval $[-\ve,\ve],$ where $\ve>0$ is sufficiently small. In particular, $\si_\de^{\la_2}(x)$ is given again by the expression \eqref{2.8II}, only with the first factor $\check\chi_1\Big({\la_1}(x_1-y_1)\Big)$ in the integrand replaced by $\check\chi_0\Big({\la_2}(x_1-y_1)\Big)$ and $\la_1$ replaced by $\la_2.$  Thus we obtain the same type of estimates 
\begin{equation}\label{est2II}
\|\widehat{\si_\de^{\la_2}}\|_\infty\lesssim \la_2^{-\frac 12-\frac 1B}, \quad \|\si_\de^{\la_2}\|_\infty\lesssim \la_2\min\{\la_2^{\frac {B-1} B}, \de_0^{-1}\}.
\end{equation}
Denote by $T_\de^{\la_2}$ the operator of convolution with $\widehat{\si_\de^{\la_2}}.$

\medskip
Interpolating between the  first estimate in \eqref{est2II} and the the estimate  $\|\si_\de^{\la_2}\|_\infty\lesssim \la_2 \la_2^{(B-1)/B},$  we get
$$
\|T_\de^{\la_2}\|_{p\to p'}\lesssim
\la_2^{-\frac1{B}-\frac12+\frac 52\th}.
$$

Arguing as in the previous case, we see that this still suffices to sum over all dyadic $\la_2$ for $\th=\th_c=2/p'_c$, to obtain  the desired estimate
\begin{equation}\label{TII}
\|T_\de^{II}\|_{p_c\to p'_c}\lesssim \sum_{\la_2}\|T_\de^{\la_2}\|_{p_c\to p'_c}\lesssim 1,
\end{equation}
unless $H=B=3$ and $m=2.$ Here  $T_\de^{II}$ denotes the contribution by the operators $T_\de^\la$ which arise in this case.

So, assume  that $H=B=3$ and $m=2,$  so that $\th_c=1/3.$  Then we distinguish two subcases:
\smallskip

{\bf 2.1. The subcase where $\la_2\le \de_0^{-B/(B-1)}=\de_0^{-3/2}.$ }  In this case, \eqref{est2II} reads
\begin{equation}\label{case2.1}
\|\widehat{\si_\de^{\la_2}}\|_\infty\lesssim \la_2^{-\frac 56}, \quad \|\si_\de^{\la_2}\|_\infty
\lesssim \la_2^{\frac 53}, 
\end{equation}  which implies  our previous estimate
$$
\|T_\de^{\la_2}\|_{p_c\to p'_c}\lesssim 1.
$$
  Let us denote by  
$$
T_\de^{II_1}:=\sum_{ \la_2\le \de_0^{-3/2}}T_\de^{\la_2}
$$
the contribution by the operators $T_\de^\la$ which arise in this subcase. In order to  prove the desired estimate 
\begin{equation}\label{TII_1}
\|T_\de^{II_1}\|_{p_c\to p'_c}\lesssim  1,
\end{equation}
we shall thus  have to apply an interpolation argument once more (see Subsection \ref{realint1}).

 \smallskip
 {\bf 2.2. The subcase where $\la_2> \de_0^{-B/(B-1)}=\de_0^{-3/2}.$  } Then \eqref{est2II} implies  that  $\|\si_\de^{\la_2}\|_\infty\lesssim \la_2 \de_0^{-1},$  hence 
$$
\|T_\de^{\la_2}\|_{p_c\to p'_c}\lesssim \de_0^{-\th_c}\, \la_2^{(\frac 32+\frac 1B)\th_c-\frac 12-\frac 1B}=
\de_0^{-\frac 13}\la_2^{-\frac 29}.
$$
As in Subcase 1.2, this implies the desired estimate 
\begin{equation}\label{TII_2}
\| T_{\de}^{II_2}\|_{p_c\to p_c'}\lesssim 1
\end{equation}
for the  contributions  $T_{\de}^{II_2}$ of the operators $T^{\la}_\de$ with $\la$ satisfying the assumptions of this subcase  to $T_{\de}.$

\medskip
{\bf 3. Case: $\la_1\sim\la_2$ and $\la_3\ll\la_1$}. If $\la_3\ll  \la_2\de_0,$ then the phase function has no critical point in $x_2,$ and so an 
integrations by part in $x_2$ and  the stationary phase method
in $x_1$ yield
$$
\|\widehat{\nu_\de^\la}\|_\infty \lesssim \la_1^{-\frac 12}\, (\la_2\de_0)^{-N}\lesssim \la_2^{-\frac12}\, \la_3^{-N}
$$
for every $N\in\NN,$ and the second estimate in  \eqref{2.10II} implies that 
$$
\|\nu_\de^\la\|_\infty \lesssim  \la_3\de_0^{-1}.
$$
Interpolating these  estimates we  obtain
$$
\|T_\de^\la\|_{p\to p'}\lesssim \la_3^{-N'}\la_2^{-\frac12(1-\th)}\de_0^{-\th},
$$
where $N'$ can be chosen arbitrarily large if $\th<1.$  But, if $\th=\th_c,$ then $\th\le 1/3,$ and since $\la_2\de_0\ge 1$ if $\la_3\ll  \la_2\de_0,$ we see that 
$$
\sum_{\la_1\sim \la_2,\ \la_3\ll  \la_2\de_0}\|T_\de^\la\|_{p_c\to p_c'}\lesssim
\de_0^{\frac{1-3\th}2}\lesssim 1.
$$

\medskip

Let us therefore assume from now on that in addition   $\la_3\gtrsim \la_2\de_0.$   Then we can first apply the method of stationary phase to the integration in $x_1$ and subsequently van der Corput's estimate to the $x_2$-integration and obtain  the estimate
\begin{equation}\label{2.14II}
\|\widehat{\nu_\de^\la}\|_\infty \lesssim \la_2^{-\frac12}\la_3^{-\frac1B}.
\end{equation}

In view of \eqref{2.10II},  we distinguish  two subcases.
\smallskip

{\bf 3.1 The subcase where $\la_3^{1/B}> \la_2\de_0$.}
Then interpolation  of the first estimate in \eqref{2.10II} with  \eqref{2.14II} yields 
$$
\|T_\de^\la\|_{p\to p'}\lesssim \la_2^\frac{3\th-1}2\la_3^{\th-\frac 1B}.
$$
Since $\th_c\le 1/3,$ we have $3\th_c-1\le 0$ (even with strict inequality, unless $B=2,$ or $H=B=3$ and  $m=2,$  because of Corollary \ref{compth} (c)). If  $3\th_c-1<0,$  we can sum over all dyadic $\la_2\gg \la_3$ and obtain 
$$
\sum_{\{\la_2:\la_3\ll \la_2 <\de_0^{-1} \la_3^{\frac 1 B}\}}\|T_\de^\la\|_{p_c\to p_c'}\lesssim
\la_3^{-\frac1{B}-\frac12+\frac52\th}.
$$
If 
$$T_\de^{III_1}:= \sum_{\la_1\sim \la_2,\ \la_3 \ll \de_0^{-\frac B{B-1}}, \,\la_3\ll \la_2 <\de_0^{-1} \la_3^{\frac 1 B}}T_\de^\la
$$
 denotes the contribution by the operators $T_\de^\la$ which arise in this subcase,  this implies  in a similar way  as before that  
\begin{equation}\label{TIII_1}
\|T_\de^{III_1}\|_{p_c\to p_c'}\lesssim 1.
\end{equation}

If $H=B=3$ and  $m=2,$ then we only get a uniform estimate
 $$
\|T_\de^\la\|_{p_c\to p'_c}\lesssim 1,
$$
and in order to establish \eqref{TIII_1}, we shall apply an interpolation argument in Subsection \ref{complexint1}.
\medskip
 
  If $B=2$  and $\th_c=1/3$ (hence $m=2$), then we can first sum over the dyadic $\la_3$ with $\la_3>(\la_2\de_0)^2,$  provided $\la_2\de_0\gtrsim 1,$ because $\th_c-1/B=-1/6,$ and  then  we sum over the $\la_2$ for which $\la_2\de_0\gtrsim 1.$ So, in  order  to \eqref{TIII_1} in this case, we are left with the estimation of  the operator
  $$T_\de^{III_0}:= \sum_{\la_1\sim \la_2,\ \la_3 \ll \de_0^{-1}, \,\la_3\ll \la_2 \ll \de_0^{-1} }T_\de^\la.
$$
This will also be done by means  of complex interpolation in Subsection \ref{complexint1}.

 \medskip
 
 {\bf 3.2 The subcase where $\la_3^{1/B}\le  \la_2\de_0.$}  Assuming without loss of generality (in a similar way as before)  that $\la_1=\la_2,$ then   the second estimate in  \eqref{2.10II}  implies that 
$
\|\nu_\de^\la\|_\infty \lesssim \la_3\de_0^{-1},
$
and in combination with \eqref{2.14II} we obtain
\begin{equation}\label{2.16II}
\|T_\de^\la\|_{p_c\to p'_c}\lesssim \la_2^{-\frac{1-\th}2}\la_3^\frac{(B+1)\th-1}B\de_0^{-\th},
\end{equation}
where   $\th:=\th_c.$
Notice that in this subcase
$$
\la_3\le \min\{(\la_2\de_0)^B, \la_2\},\quad \mbox{and} \quad  \la_2\de_0\lesssim \la_3.
$$
In view of this , we shall   distinguish two cases:
\medskip

 {\bf (a) $\la_2\ge \de_0^{-\frac B{B-1}}.$}  Assume first that $B\ge 3.$ Then
 $$
 (B+1)\th_B-1=\frac{B^2-2B+3}{6B}>0
 $$
 for $B\ge 2,$ and since $\th=\th_c\le\th_B=1/(2B)+1/6$  by Corollary \ref{compth}, we see that we can sum the estimates in \eqref{2.16II} over all dyadic $\la_3\ll \la_2$ and get 
 $$
 \sum_{\la_3\ll \la_2} \|T_\de^\la\|_{p_c\to p'_c}\lesssim \la_2^{\frac{(3B+2)\th-B-2}{2B}}\,\de_0^{-\th},
 $$
 Using again that $\th\le\th_B,$ we find that  for $B\ge2,$ 
 $$
(3B+2)\th-B-2\le (3B+2)(\frac 16+\frac 1{2B})-B-2=\frac{-3B^2-B+6}{6B}<0,
$$
so we can sum also in $\la_2\ge \de_0^{-\frac B{B-1}}$ and find that
$
\|T_\de^{III_{2a}}\|_{p_c\to p_c'} \lesssim \de_0^{-\frac{5B\th-B-2}{2(B-1)}},
$
where $T_\de^{III_{2a}}:=\sum_{\la_1\sim\la_2\ge \de_0^{-\frac B{B-1}}, \,\la_3\ll \la_2}T_\de^\la.$
But,
\begin{equation}\label{3.10II}
5B\th-B-2\le 5B\th_B-B-2 =5B(\frac1{2B}+\frac 16)-B-2 =\frac{3-B}{6}\le 0,
\end{equation}
if $B\ge 3,$ and thus for $B\ge 3$ we get 
 \begin{equation}\label{TIII_2a}
\|T_\de^{III_{2a}}\|_{p_c\to p_c'}
\lesssim 1.
\end{equation}
The remains the case $B=2.$ Here, for $\th=\th_c,$ we have $(B+1)\th-1=3\th-1\le 0,$  and 
$$
\|T_\de^\la\|_{p_c\to p'_c}\lesssim \la_2^{-\frac{1-\th}2}\la_3^\frac{3\th-1}2\de_0^{-\th}
 $$
 Assume first that $\th_c<1/3.$ Then we can first sum in $\la_3\ge \la_2\de_0$ (notice that $\la_2\de_0> 1$) and obtain 
 $$
\sum_{\la_3\ge \la_2\de_0}\|T_\de^\la\|_{p_c\to p'_c}\lesssim \la_2^{2\th -1}\de_0^{\frac{\th-1}{2}}.
 $$
 Then we sum over $\la_2\ge \de_0^{-1}$ and get an estimate by $C\de_0^{\frac{1-3\th}{2}}\lesssim1,$ so that \eqref{TIII_2a} remains true also in this case.
 
 Assume finally that $\th_c=1/3.$ Then
 $
\|T_\de^\la\|_{p_c\to p'_c}\lesssim \la_2^{-\frac 13}\de_0^{-\frac 13}.
 $
 Summing first in $\la_3\ll\la_2,$ we get an estimate by $C(\log\la_2) \,\la_2^{-\frac 13}\de_0^{-\frac 13},$  and summation over all $\la_2\ge \de_0^{- B/(B-1)}=\de_0^{-2}$ leads to an estimate  of order $\log (1/\de_0) \de_0^{1/3}\lesssim 1.$ Thus, again \eqref{TIII_2a} holds true.

\medskip

 {\bf (b) $\la_2< \de_0^{-\frac B{B-1}}.$} Then we use $\la_3\le (\la_2\de_0)^B,$ i.e., $\la_2\ge\la_3^{1/B}\de_0^{-1},$ and summation of the estimate \eqref{2.16II} over these $\la_2$ yields
 $$
\sum_{\{\la_2: \la_2\ge\la_3^{1/B}\de_0^{-1}\}} \|T_\de^\la\|_{p_c\to p'_c}\lesssim 
\la_3^{\frac{(3+2B)\th-3}{2B}}\de_0^{\frac{1-3\th}{2}}.
 $$
 If the exponent of $\la_3$ on the right-hand side of this estimate is strictly negative, then we see that 
  \begin{equation}\label{TIII_2b}
\|T_\de^{III_{2b}}\|_{p_c\to p_c'}
\lesssim 1,
\end{equation}
where 
$$
T_\de^{III_{2b}}:=\sum_{\la_1\sim\la_2< \de_0^{-\frac B{B-1}}, \,\la_3\ll (\la_2\de_0)^B} T_\de^\la.
$$
So, assume that the exponent is non-negative,   and notice that our assumptions in this case imply that $\la_3\le \de_0^{-B/(B-1)}.$ Summation over all these dyadic $\la_3$ then leads to 
\begin{eqnarray*}
\sum_{\la_2\ge\la_3^{1/B}\de_0^{-1}, \la_3\le \de_0^{-B/(B-1)}} \|T_\de^\la\|_{p_c\to p'_c}&\lesssim& 
(\log \frac 1{\de_0})( \de_0^{-B/(B-1)})^{\frac{(3+2B)\th-3}{2B}}\de_0^{\frac{1-3\th}{2}}\\
&=&(\log \frac 1{\de_0})\de_0^{-\frac{5B\th-B-2}{2(B-1)}}.
\end{eqnarray*}
However, if we assume without loss of generality that  $\th=\th_B,$ then we have seen in \eqref{3.10II} that 
$5B\th-B-2\le 0,$ if $B\ge 3.$ Moreover, by Corollary \ref{compth} (a) we know that $\th_c<\th_B,$ unless $B=H=3$ and $m=2.$ Thus, using $\th=\th_c$ in place of $\th_B,$ when  $B\ge 3$ we have $5B\th_c-B-2\le 0,$ and we again obtain \eqref{TIII_2b}, unless $B=H=3$ and $m=2.$ Note that in the latter case $\th_c=1/3,$ so that 
$$
\sum_{\{\la_2:\la_2\ge\la_3^{1/3}\de_0^{-1}\}} \|T_\de^\la\|_{p_c\to p'_c}\lesssim 1.
$$
The proof of \eqref{TIII_2b} in this particular case, where 
$$
T_\de^{III_{2b}}=\sum_{\de_0^{-1}\la_3^{1/3}\ll \la_1\sim \la_2<\de_0^{-3/2}} T_\de^\la,
$$
will therefore again require the use of some interpolation argument, in order to control the summation in $\la_3.$ 
We remark that in this case, estimate \eqref{2.16II} reads as 
$
\|T_\de^\la\|_{p_c\to p'_c}\lesssim \la_1^{-1/3 }\la_3^{1/9}\de_0^{-1/3},
$
and 
$$
\sum_{1\le \la_3\lesssim \la_1\de_0}\sum_{ \la_1<\de_0^{-2/3}}\la_1^{-1/3 }\la_3^{1/9}\de_0^{-1/3}\lesssim 1,
$$
since $\la_1\de_0\ge 1$ in this double sum.

This shows that if $B=H=3$ and $m=2$  (hence $\th_c=1/3,$ ), what remains to be estimated is the operator
$$
T_\de^{III_2}:=\sum_{\de_0^{-1}\ll \la_1 <\de_0^{-3/2}} \sum_{ \la_1\de_0\ll \la_3\ll (\la_1\de_0)^3} T_\de^\la,
$$
This will be done in Subsection \ref{complexint2}.

\medskip

Assume  finally that $B=2.$ The case where $\th_c<1/3$ can be treated as in the previous case (a), so assume that $\th_c=1/3.$ Then
 $
\|T_\de^\la\|_{p_c\to p'_c}\lesssim (\la_2\de_0)^{-\frac 13}.
 $
 Summing first over all $\la_2$ such that $\la_2\de_0\ge \la_3^{1/2}$ leads to an estimate  of order $\la_3^{-1/6},$ which then allows to sum also in $\la_3.$ Thus, again \eqref{TIII_2b} holds true.

\setcounter{equation}{0}
\section{Interpolation arguments for the open cases  where $m=2$ and $B=3$ or $B=2$}\label{interpolation1}

Let us assume in  this section that $m=2$ and $B=3$  of $B=2.$ Our goal will be to establish the estimates \eqref{TI_1}, \eqref{TII_1},
\eqref{TIII_1} and \eqref{TIII_2b}  for the operators $T_\de^{I_1}, T_\de^{II_1}, T_\de^{III_1}$ and $T_\de^{III_0},$  as well as  $T_\de^{III_{2b}},$ also at the endpoint $p_c$ corresponding to $\th_c=1/3$ (hence $p'_c=2/\th_c=6$). These cases had been left open in the previous section.

The estimates for the operators $T_\de^{III_1},T_\de^{III_0}$ and $T_\de^{III_{2b}}$ will be established by means of complex interpolation, roughly in analogy with the proof of Proposition 5.1 (a) and (b) in \cite{IM-rest1}. 

Also the operators $T_\de^{I_1}$ and $T_\de^{II_1}$ could be handled by means of complex interpolation. However, a shorter proof is possible by means of the  real interpolation approach developed  by Bak and Seeger  in \cite{bs} (which, however, requires the validity of the expected estimates of the operators $T_\de^{I_1},T_\de^{II_1},$ as well as of  several further operators). 

\subsection{Estimation of  $T_\de^{I_1}$ and $T_\de^{II_1}$: Real interpolation}\label{realint1}
The estimation of the operators  $T_\de^{I_1}$ and $T_\de^{II_1}$ will follow the same scheme, so let us consider  $T_\de^{I_1}$  only,  which is the operator of convolution with $\widehat{\si_\de^{I_1}},$ where 
$\si_\de^{I_1}$ denotes the measure 
$$
\si_\de^{I_1}:=\sum_{ 0\le j\le \de_0^{-3/2}}\si_\de^{2^j}.
$$
In the following discussion, if $\mu$ is any bounded, complex Borel measure on $\RR^d,$ we  shall often denote 
by $T_\mu$ the convolution operator
$$
T_\mu:\vp\mapsto \vp*\hat \mu.
$$

Regretfully,  we cannot apply Theorem 1.1 in \cite{bs} directly to the measure $\mu=\si_\de^{I_1},$ because an essential assumption in \cite{bs} is that $\mu$ is a   bounded positive   measure, and $\si_\de^{I_1}$  will not be positive. However, the method of proof in \cite{bs} easily yields the following variant of Theorem 1.1, which can be applied in our situation:

\begin{proposition}\label{bsint}
Let $\mu$ be a  positive Borel  measure on $\RR^d$ of total mass $\|\mu\|_1\le 1,$ and  let 
$p_0\in [1,2[.$ Assume that $\mu$ can be decomposed into a finite sum 
$$
\mu=\mu^b +\sum_{i\in I}\mu^i
$$
of  bounded complex Borel measures $\mu^b$ and $\mu^i, i\in I, $ such that the following hold true: There is    a constant $A\ge 0$ such that:
 \begin{itemize}
\item[(a)] The operator $T_{\mu^b}$ is bounded from $L^{p_0}(\RR^d)$ to $L^{p'_0}(\RR^d),$ with 
 \begin{equation}\label{bs1}
\|T_{\mu^b}\|_{p_0\to p'_0}\le A.
\end{equation}
\item[(b)] Each of the measures $\mu^i$ decomposes as 
$$
\mu^i=\sum_{j=1}^{K_i} \mu^i_j=\sum_{j=1}^{K_i} \mu*\phi^i_j,
$$
where $K_i\in\NN\cup\{\infty\},$ and where the $\phi^i_j$ are integrable functions such that
\begin{equation}\label{bs2}
\|\phi^i_j\|_1\le 1 .
\end{equation}
  Moreover, there are constants $a_i>0,b_i>0$ 
   such that for all $i$ and $j,$ 
\begin{eqnarray} \label{bs3}
\|\mu^i_j\|_\infty&\le& A 2^{ja_i};\\  \label{bs4}
\|\widehat{\mu^i_j}\|_\infty&\le& A 2^{-j b_i};\\  \label{bs5}
p_0&=&2\frac{a_i+b_i}{2a_i+b_i} \ (\mbox{ i.e., if } \th a_i-(1-\th)b_i=0,\mbox{ then } \frac 1{p_0}=\frac {\th}{2} +(1-\th)).
\end{eqnarray}
 \end{itemize}
 Then there is a constant $C$ which depends only on $d$ and any compact interval in $]0,\infty[$ containing the $a_i$ and $b_i$  such that for every $i,$ 
 \begin{equation}\label{bs6}
\|T_{\mu^i} f\|_{L^{p'_0}}\le C A \|f\|_{L^{p_0}},
\end{equation}
and consequently 
 \begin{equation}\label{bs7}
\int|\hat f|^2\, d\mu\le C A\|f\|^2_{L^{p_0}(\RR^d)},
\end{equation}
\end{proposition}
  
  \begin{proof}

By essentially following the proof of Proposition 2.1 in \cite{bs}, we define the interpolation parameter $\th:=2/p'_0.$ Observe that by \eqref{bs5} we have $\th=b_i/(a_i+b_i),$ hence  $(1-\th)(-b_i)+\th a_i=0$ for every $i.$  Thus, the two inequalities \eqref{bs3} and \eqref{bs4} allow to apply an interpolation trick due to Bourgain \cite{bourgain} and to conclude that each of the  operators $T_{\mu^i_j}$ is of restricted weak-type $(p_0,p'_0),$ with operator norm $\le C A,$  and if $J$ is any compact subinterval of $]0,\infty[,$ then for $a_i,b_i\in J$ we may chose the constant $C$ so that it depends only on $J.$  In combination with \eqref{bs1}, this implies that also $T_{\mu}$ is of restricted weak-type $(p_0,p'_0),$ with operator norm $\le C A,$ where $C$ may be different from the  previous constant, but with similar properties. By applying Tomas'  $R^*R$- argument for the restriction operator $R,$   we get
$$
\int|\hat f|^2\, d\mu\le C A\|f\|^2_{L^{p_0,1}}.
$$
In combination with Plancherel's theorem and  \eqref{bs3} and   \eqref{bs2} we can next use this estimate as in \cite{bs} to control 
$$
\|T_{\mu^i_j} f\|_2=\|f*\widehat{\mu^i_j} \|_2\le A^{\frac 12} 2^{j\frac {a_i}2}\|\phi^i_j\|_1^{\frac 12}  A^{\frac 12}\|f\|^2_{L^{p_0,1}}\le A 2^{j\frac {a_i}2}\|f\|^2_{L^{p_0,1}}.
$$
It is here where the positivity of the measure $\mu$ is used in an essential way.  The remaining part of the argument in \cite{bs} does not require positivity of the underlying measure, so that it applies to each of the complex measures $\mu^i$  as well, and we may conclude that for any $s\in ]0, \infty],$
$$
\|T_{\mu^i} f\|_{L^{p'_0,s}}\le C A \|f\|_{L^{p_0,s}}
$$
(compare Proposition 2.1, (2.2),  in \cite{bs}), where $L^{p,s}$ denotes the Lorentz space of type $p,s.$
Choosing $s=2,$ so that $p_0\le 2\le p'_0,$ by the nesting properties of the scale of Lorentz spaces this implies in particular that
$$
\|T_{\mu^i} f\|_{L^{p'_0,p'_0}}\le C A \|f\|_{L^{p_0,p_0}},
$$
hence \eqref{bs6}. The same type of estimate then holds for the operator $T_\mu,$ and   Tomas'  argument then leads to \eqref{bs7}.

\end{proof}

Let us return to  our measure $\nu_\de.$ This is a positive measure, so we can choose $\mu:=\nu_\de$ in Proposition \ref{bsint}. The  spectral decompositions of the measure $\nu_\de$ in Section \ref{specloc} as well as Section \ref{equalla}  amounts to a decomposition of the measure $\nu_\de$ into a finite sum of complex measures

\begin{equation}\label{nudecomp}
\nu_\de=\sum_{j\in J} \nu^j +\sum_{i\in I} \mu^i,
\end{equation}
in such a way,  that the  convolution operators $T_{\nu^j}$ corresponding to the measure $\nu^j  \, (j\in J)$ from the first  class will be bounded from $L^{p_c}$ to $L^{p'_c},$  whereas the  measure $\mu^i, \, (i\in I)$ from the second class will satisfy the   conditions required  on the measures $\mu^i$  of Proposition \ref{bsint}.  For instance, the operators $T_\de^{I_2},  T^{II_2}_\de$ and $T_\de^{III_{2a}}$   from Section \ref{specloc} belong to the first class (compare the estimates \eqref{TI_2}, \eqref{TII_2}, \eqref{TIII_2a}),   but also $T_\de^{III_{1}}$ and  $T_\de^{III_{2}}$  (the corresponding estimates \eqref{TIII_1}  and \eqref{TIII_2b} will be established by means of complex interpolation in the next subsection), whereas the operators $T_\de^{I_1}$ and $T_\de^{II_1}$ will belong to the second class. 

We may then put $\mu^b:=\sum_{j\in J} \nu^j$ in Proposition \ref{bsint}. Let us show for instance that the  measure
$\mu^i:=\si_\de^{I_1}$ corresponding to  the operator $T_\de^{I_1}$ satisfies the assumptions of this proposition: 

Recall to this end that $\mu^i_j:=\si^{2^j}_\de=\nu_\de*\phi_j,$ where the Fourier transform of $ \phi_j$ is given by 
$\widehat{\phi_j}(\xi)=\chi_1\Big(\frac {\xi_1}{\la_1}\Big)\chi_0\Big(\frac {\xi_2}{\la_1}\Big)\chi_1\Big(\frac {\xi_3}{\la_1}\Big).$
  This implies a uniform estimate of the $L^1$-norms of the  $\phi_j$  of the form \eqref{bs2}  (possibly not with constant $1,$ but a fixed constant, which does not matter).  Moreover, the estimates \eqref{bs3} and  \eqref{bs4} are satisfied  because of \eqref{case1.1}, with exponents $a_i:=5/3$ and $b_i:=5/6,$  so that $p_0=6/5=p_c,$ 
  
  Similar arguments apply also to the measure $\mu^i:=\si_\de^{II_1}$ corresponding to  the operator $T_\de^{II_1},$  where the exponents $a_i$ and $b_i$ will be the same (compare \eqref{case2.1}), as well as to the other measures of the first class which will appear later.

  \subsection{Estimation of  $T_\de^{III_{1}}$: Complex interpolation}\label{complexint1}
  
  The discussion of this operator will  somewhat resemble the one of the operator $T_{\de,j}^{V}$ in Subsection 8.1 of \cite{IM-rest1}, which arose from the same Subcase 3.1 of Subsection 5.3 in  \cite{IM-rest1}, with $2^{-j}$ playing the role of $\de_0$ here, and where we have had $B=2,$ in place of $B=3$ here. We shall make use of a more refined structural result for the phase 
  function $\phi_\de$  which will be derived in a more general context in Section \ref{m2b34} (compare \eqref{7.2II}). In view of this result,  in combination with Corollary \ref{rem7.1II}, we may and shall assume that 
   
\begin{equation}\label{phideb3}
\phi_\de(x):= x_2^3  \,b(x_1,x_2,\de)+x_1^n\al(\de_1 x_1)+ \de_3x_2  x_1^{n_1}\al_1(\de_1x_1),\qquad (x_1\sim 1, |x_2|<\ve),
\end{equation}
where  we may now  assume (compare \eqref{bde}) that 
\begin{equation}\label{bfine}
b(x_1,x_2,\de)=b_3(\de_1 x_1, \de_2x_2), \quad b_3(0,0)=1.
\end{equation}

Moreover,  $\al(0)\ne 0,$ and   either $\al_1(0)\ne 0,$ and then $n_1$  is fixed, or $\al_1(0)=0,$ and then we may assume that $n_1$ is as large as we please (notice also that by  \eqref{defde2}, $\de_3=2^{-k(n_1\tilde\ka_1+\tilde\ka_2-1)}$ is coupled with $n_1$), so that in particular in this case  $\de_3\ll \de_0.$ Observe also that $\de_2\ll \de_0.$ 
\medskip

Useful tools will also be the Lemmas  7.2 and 8.1 from \cite{IM-rest1} on oscillatory sums and double sums. For the convenience of the reader, let us at least recall the first lemma (in the sharper version of Remark 7.3 of \cite{IM-rest1}):

\begin{lemma}\label{simplesum}
Let $Q=\prod_{j=1}^nÊ[-R_k,R_k]\subset \RR^n$ be a compact cuboid, with $R_k>0, k=1,\dots, n,$ and let $H$ be a $C^1$-function on an open neighborhood of $Q.$ Moreover, let $\al, \be^1, \dots, \be^n\in\RR^\times$  be given. For any given real numbers $a_1,\dots,a_n\in \RR^\times$ and $M\in \NN$ we then put 
\begin{equation}\label{Ft1}
F(t):= \sum_{l=0}^{M} 2^{i \al lt}
(H\chi_Q)\Big(2^{\be^1l}a_1,\dots, 2^{\be^n l}a_n\Big).
\end{equation}
Assume that there are constants $\epsilon\in]0,1]$ and $C_k, \, k=1,\dots, n,$ such that
\begin{equation}\label{simsumb}
\int_0^1\Big|\frac{\pa H}{\pa u_k}(su)\Big|\, ds \le C_k |u_k|^{\epsilon-1}, \qquad \mbox{ for all } u\in Q.
\end{equation}
Then there is a constant $C$ depending on $Q,$ the numbers $\al$ and $\be^k$ and $\epsilon,$ but not on  $H,$ the $a_k,$ $M$ and $t,$   such that 
\begin{equation}\label{simsum}
|F(t)|\le C\frac {|H(0)|+\sum_k C_k} {|2^{i\al t}-1|}, \qquad \mbox{ for all } t\in\RR, a_1,\dots a_2\in \RR^\times \mboxÊ{ and  } \ M\in \NN.
\end{equation}
In particular, we have 
$$
|F(t)|\le C\frac {\|H\|_{C^1(Q)}} {|2^{i\al t}-1|}, \qquad \mbox{ for all } t\in\RR, a_1,\dots a_2\in \RR^\times \mboxÊ{ and  } \ M\in \NN.
$$
\end{lemma}

\medskip
 
Coming back to  the operator $T_\de^{III_{1}}$, observe that  $\la_1\sim \la_2$ in the definition of $T_\de^{III_{1}}$,  so that we may and shall assume without  loss of generality that $\la_1=\la_2.$ In order to verify estimate \eqref{TIII_1}, we then have to prove 
 
\begin{prop}\label{analyint3}
Let $m=2$ and $B=3,$ and consider  the measure 
$$
\nu_\de^{III_{1}}:=\sum_{2^{M}\le\la_3\le 2^{-M}\de_0^{- 3/2}} \sum_{2^M\la_3\le\la_1\le \de_0^{-1}\la_3^{1/3}} \nu^{(\la_1,\la_1,\la_3)}_\de,
$$
where summation is  taken over all sufficiently large dyadic $\la_i\ge 2^M$ in the given range. If we denote by $T_\de^{III_{1}}$  the operator of convolution with $\widehat{\nu_\de^{III_{1}}},$ then, if $M\in \NN$ is sufficiently large (and $\ve$ sufficiently    small), 
\begin{equation}\label{TIII1a}
\|T_\de^{III_{1}}\|_{6/5\to 6}\le C,
\end{equation}
with a constant $C$ not depending on $\de,$ for $\de$ sufficiently small.
\end{prop}
 
 \begin{proof} 
 Recall that, by \eqref{2.14II},  $\|\widehat{\nu^{\la}_{\de}}\|_\infty\lesssim \la_1^{-1/2} \la_3^{-1/3}.$ We  therefore define here for $\ze$ in the strip $\Sigma=\{\zeta\in \bC: 0\le \Re \zeta\le 1\}$  an analytic family of measures by 
$$
\mu_\ze(x):=\gamma(\ze) \sum_{2^{M}\le 2^{k_3}\le 2^{-M}\de_0^{-\frac 32}} \sum_{2^{M+k_3}\le 2^{k_1}\le \de_0^{-1}2^{\frac {k_3}3}} 2^\frac{(1-3\zeta)k_1}2 2^\frac{(1-3\zeta)k_3}3\nu^{(2^{k_1},2^{k_1},2^{k_3})}_\de,
$$
where $\ga(\ze)$ is an entire function which  will serve a similar role as the function $\ga(z)$ in the proof of Proposition 5.2(a) in  \cite{IM-rest1}. We shall choose  $\ga(\ze)=\ga_1(\ze)\ga_2(\ze)\ga_3(\ze)$ as the product of three factors $\ga_j(\ze),$ whose  precise definition will be given in the course of the proof. It will  be uniformly bounded on $\Sigma,$  and such that  $\ga(\th_c)=\ga(1/3)=1.$ 
\medskip

By $T_\ze$  we denote the operator of convolution with $\widehat{\mu_\ze}.$ 
Observe that  for $\ze=\theta_c=1/3,$ we have $\mu_{\th_c}=\nu_{\de}^{III_1},$ hence 
$
T_{\theta_c}=T_{\de}^{III_1},
$
 so that, again  by Stein's interpolation theorem for analytic families of operators,  \eqref{TIII1a}  will follow if we can prove the following estimates on the boundaries of the  strip $\Sigma:$ 
\begin{eqnarray*}
\|\widehat{\mu_{it}}\|_\infty &\le& C \qquad \forall t\in\RR,\\ 
\|{\mu_{1+it}}\|_\infty &\le& C \qquad \forall t\in\RR. 
\end{eqnarray*}

The  first estimate  is an immediate consequence of  our estimate for $\widehat{\nu^{\la}_{\de}},$  since these functions have essentially disjoint supports,  so let us concentrate on the second estimate, i.e., assume that $\zeta=1+it,$ with $t\in \RR.$ We then have to  prove that there is constant $C$ such that 
\begin{equation}\label{5.10}
|\mu_{1+it}(x)|\le C,
\end{equation}
where $C$ is independent of $t,x$ and $\de.$  

Let us  introduce the  measures $ \mu_{\la_1,\la_3}$ given by
$$
\mu_{\la_1,\la_3}(x):=\la_1^{-1}\la_3^{-\frac23} \nu_\de^{(\la_1,\la_1,\la_3)}(x), 
$$
which allow to re-write 
\begin{equation}\label{5.11}
\mu_{1+it}(x)=\gamma(1+it)  \sum_{2^{M}\le\la_3\le 2^{-M}\de_0^{-\frac 32}} \sum_{2^M\la_3\le\la_1\le \de_0^{-1}\la_3^{\frac 13}}
\la_1^{-\frac 32 it}\la_3^{-it} \mu_{\la_1,\la_3}(x).
\end{equation}

Recall also from   \eqref{2.8II} (in combination with \eqref{phideb3}) that 
\begin{eqnarray}\nonumber
\mu_{\la_1,\la_3}(x)&=&\la_1\la_3^\frac 13 \int \check\chi_1\Big({\la_1}(x_1-y_1)\Big) \, \check\chi_1\Big({\la_1}(x_2-\de_0y_2-y_1^2\om(\de_1y_1))\Big)\\
&&\check\chi_1\Big({\la_3}\Big(x_3-y_2^3  \,b(y_1,y_2,\de)-y_1^n\al(\de_1 y_1)- \de_3y_2  y_1^{n_1}\al_1(\de_1y_1)\Big)\Big)\label{mula1la3}
\,\eta(y) \,dy,
\end{eqnarray}
where $\eta$ is supported where $y_1\sim 1$ and $|y_2|<\ve.$
Assume first that $|x_1|\gg 1,$  or $|x_1|\ll 1.$ Since $\check\chi_1$ is rapidly decreasing, and $\la_3\ll \la_1,$ we  easily see that 
$|\mu_{\la_1,\la_3}(x)|\le C_N\la_1^{-N} \la_3^{-N}$ for every $N\in\NN,$ which immediately implies \eqref{5.10}. A similar argument applies if $|x_2|\gg 1.$ However, if  $|x_1|+|x_2|\lesssim 1$ and  $|x_3|\gg 1,$ we can only conclude (after scaling by $1/\la_1$ in $y_1$) that $|\mu_{\la_1,\la_3}(x)|\le C_N \la_3^{-N},$ which allows to sum in $\la_3,$ but the summation in $\la_1$ remains a problem. 

\medskip 
Let us thus assume from now on that $|x_1|\sim 1$ and $|x_2|\lesssim 1.$ 
  
 By means of the change of variables $y_1\mapsto x_1-y_1/\la_1, \  y_2\mapsto y_2/\la_3^{1/3}$ and Taylor expansion around $x_1$ we  may re-write 
\begin{equation}\label{mulalaa}
\mu_{\la_1,\la_3}(x)=\iint  \check\chi_1(y_1)F_\de(\la_1,\la_3,x,y_1,y_2)\, dy_1dy_2,
\end{equation}

where
\begin{eqnarray*}
&&F_\de(\la_1,\la_3,x,y_1,y_2):=\eta(x_1-\la_1^{-1}y_1,\,  \la_3^{-\frac13} y_2) \,\check\chi_1(D-Ey_2+r_1(y_1))\\
&\times& \check\chi_1\Big(A- B y_2-y_2^3 b(x_1-\la_3^{-1}(\la_3\la_1^{-1}y_1),\la_3^{-\frac13}y_2,\de)+\la_3\la_1^{-1}\Big(r_2(y_1)+ (\la_3^{-\frac 13}y_2)\de_3r_3(y_1)\Big)\Big).
\end{eqnarray*}
Here,  the quantities $A$ to $ E$ are given by 
\begin{eqnarray}\nonumber
&& A=A(x,\la_3,\de):=\la_3Q_A(x) , \qquad B:=B(x,\la_3,\de):=\la_3^{\frac 23}Q_B(x),\\
&& D=D(x,\la_1,\de):=\la_1Q_D(x),  \qquad E=E(\la_1,\la_3,\de):=\la_1\la_3^{-\frac 13}\de_0\le 1,\label{fd2}
\end{eqnarray}
with
$$
Q_A(x):=x_3-x_1^n\al(\de_1 x_1),\quad Q_B(x):=\de_3 x_1^{n_1}\al_1(\de_1x_1), \quad Q_D(x):=x_2-x_1^2\om(\de_1 x_1),
$$
and do not depend on $y_1,y_2.$ 
Moreover, the functions $r_i(y_1)=r_i(y_1; \la_1^{-1}, x_1,\de),\,  i=1,2,3, $ are smooth functions of $y_1$  (and $\la_1^{-1}$ and $x_1$)  satisfying estimates of the form
\begin{equation}\label{5.14}
|r_i(y_1)|\le C|y_1|, \ 
\quad \left|\left(\frac{\pa}{\pa (\la_1^{-1})}\right)^lr_i(y_1; \la_1^{-1}, x_1,\de)\right|\le C_l |y_1|^{l+1}\quad \mbox{for every}\quad l\ge1.
 \end{equation}

Notice that we may here assume that $|y_1|\lesssim \la_1,$ because of our assumption  $|x_1|\lesssim1$ and the support properties of $\eta.$ It will also be important to observe that 
$E=\de_0\la_1\la_3^{-\frac13}\le 1$ for the index set of $\la_1,\la_3$ over which we sum in \eqref{5.11}. Notice also that 
$|\la_3^{-1/3}y_2|\le \ve.$

\medskip
Let us choose $c>0$ so that $|r_i(y_1)|\le c(1+|y_1|), i=1,2,3.$

\medskip
In order to verify \eqref{5.10}, given $x,$ we shall split the sum in \eqref{5.11} into four parts, which will be treated subsequently in different ways (compare the analogous discussion in Subsection 8.1 of \cite{IM-rest1}).

\medskip
\noi {\bf 1. The part where $\max\{|A|,|B|\}\ge 1$ and $|D|\ge 4c.$ } 
 Denote  by $\mu_{1+it}^1(x)$ the contribution to $\mu_{1+it}(x)$ by the terms for which 
 $\max\{|A(x,\la_3,\de)|,|B(x,\la_3,\de)|\}\ge 1$ and $|D(x,\la_1,\de)|\ge 4c.$
 
\smallskip
We claim  that here
\begin{equation}\label{5.15}
|\mu_{\la_1,\la_3}(x)|\lesssim  |D|^{-\frac14}\max\{|A|^{\frac 13}, |B|^{\frac 12}\}^{-\frac 14}.
\end{equation}
In view of \eqref{fd2}, this estimate will allow us to sum   over all dyadic $\la_1,\la_3$ for which the corresponding  quantities $A,B$  and $D$ satisfy the conditions of this subcase, and we obtain the right kind of  estimate $|\mu_{1+it}^1(x)|\le C,$ 
in agreement with \eqref{5.10}.

\medskip

In order to prove \eqref{5.15}, let us first consider the contribution $\mu^1_{\la_1,\la_3}(x)$  to the integral defining $ \mu_{\la_1,\la_3}(x)$ by the   region where $|y_1|> |D|/4c.$ Here we may estimate $|\check{\chi_1}(y_1)|\lesssim |D|^{-N}$ for every $N\in\NN.$ Moreover, if $|x_3|\gg 1,$ then $|A|\gg \la_3\gg |B|^{3/2},$ and $A$ becomes the dominant term in the argument of the last factor of $F_\de.$  Therefore we may estimate 
$$
|\mu_{\la_1,\la_2}(x)|\le C_N |D|^{-N} \la_3^{\frac 13} |A|^{-N}
$$
for every $N\in\NN,$ which is stronger than \eqref{5.15}.

And, if $|x_3|\lesssim 1,$ then we may apply Lemma \ref{as1} (with $T:=\la_3^{1/3}, \epsilon:=0$) and obtain the  estimate
$$
|\mu^1_{\la_1,\la_3}(x)|\lesssim |D|^{-N}|\max\{|A|^{\frac 13}, |B|^{\frac 12}\}^{-\frac 12},
$$
which is still  stronger than required in \eqref{5.15}.

\smallskip
Denote next by $ \mu^2_{\la_1,\la_3}(x)$ the contribution by the region where $|y_1|\le |D|/4c.$ Then $|r_i(y_1)|\le |D|/2,$ and thus if in addition $|E y_2|\le |D|/4,$ or $|E y_2|> 2|D|,$ then 
$\Big|D-E y_2+r_1(y_1)\Big|\ge |D|/4.$ We may then estimate 
 the  second factor of $F_\de$ by $C_N|D|^{-N},$ which allows to argue as before.
 So, let us assume that $|r_1(y_1)|\le |D|/2$ and $|D|/4\le |E y_2|\le 2|D|.$ Then $|y_2|\sim |D|/|E|\ge |D|\gtrsim 1.$
In case that $|D|\lesssim\max\{|A|^{\frac 13}, |B|^{\frac 12}\},$ we may apply Lemma \ref{as1} (assuming again that $|x_3|\lesssim 1;$ the other case is again easier) and find that 
\begin{eqnarray*}
|\mu^2_{\la_1,\la_3}(x))|\lesssim \max\{|A|^{\frac 13}, |B|^{\frac 12}\}^{-\frac 12}\le |D|^{-\frac14}\max\{|A|^{\frac 13}, |B|^{\frac 12}\}^{-\frac 14},
\end{eqnarray*}
as desired.
So assume that $|D|\gg\max\{|A|^{\frac 13}, |B|^{\frac 12}\}.$ Then one easily sees that 
$$
\Big|A- B y_2-y_2^3 b(x_1-\la_1^{-1}y_1,\la_3^{-\frac13}y_2,\de)+\la_3\la_1^{-1}\Big(r_2(y_1)
+ (\la_3^{-\frac 13}y_2)\de_3r_3(y_1)\Big)\Big|\gtrsim |y_2|D^2,
  $$ 
so that 
$$
|\mu^2_{\la_1,\la_3}(x))|\lesssim \iint (1+|y_1|)^{-N}(1+|y_2|D^2)^{-N}\, dy_1dy_2\lesssim D^{-2}\lesssim |D|^{-1}
\max\{|A|^{\frac 13}, |B|^{\frac 12}\}^{-1} ,
$$
which is again stronger then required in \eqref{5.15}.

\medskip
\noi {\bf 2. The part where $\max\{|A|,|B|\}\ge 1$ and $|D|< 4c.$ } 
 Denote  by $\mu_{1+it}^2(x)$ the contribution to $\mu_{1+it}(x)$ by the terms for which 
 $\max\{|A(x,\la_3,\de)|,|B(x,\la_3,\de)|\}\ge 1$ and $|D(x,\la_1,\de)|< 4c.$ 
 
Let us fix $\la_3$ satisfying $2^{M}\le\la_3\le 2^{-M}\de_0^{-\frac 32}$  and $\max\{|A(x,\la_3,\de)|,|B(x,\la_3,\de)|\}\ge 1$ in the first sum in \eqref{5.11}.  In order to compute 
$\mu_{1+it}^2(x),$ we then have to study the following sum in $\la_1=2^{k_1}:$ 
$$
\si^2(\la_3, t,x):=\sum_{\{\la_1: 2^M\la_3\le\la_1\le \de_0^{-1}\la_3^{\frac 13}, \la_1|Q_D(x)|< 4c\}}
\la_1^{-\frac 32 it} \mu_{\la_1,\la_3}(x).
$$
Indeed, we have  $\mu_{1+it}^2(x)=\sum_{\la_3}\la_3^{-it} \si^2(\la_3, t,x),$ where summation is over all dyadic $\la_3$ in the range described before.

The oscillatory sum defining $\si^2(\la_3, t,x)$  can  essentially be written in the form \eqref{Ft1}, with $\al:=-3/2, l=k_1$  and
$$
u_1=2^{\be^1l}a_1:= \la_1^{-1}\la_3,\  u_2= 2^{\be^2l}a_2:=\la_1 Q_D(x),\  u_3=2^{\be^3l}a_3:=\la_1(\la_3^{-\frac 13} \de_0)
$$
and where the function $H=H_{\la_3,x,\de}$  of $u:=(u_1,u_2, u_3)$ is given by
\begin{eqnarray*}
&&H(u):=\iint  \check\chi_1(y_1)\eta(x_1-u_1\la_3^{-1}y_1,\,  \la_3^{-\frac13} y_2) \,\check\chi_1(u_2-u_3y_2
+r_1(y_1;\la_3^{-1}u_1,x_1,\de))\\
&&\hskip1cm\times \check\chi_1\Big(A- B y_2-y_2^3 b(x_1-\la_3^{-1} u_1y_1,\la_3^{-\frac13}y_2,\de)
+u_1r_2(y_1;\la_3^{-1}u_1,x_1,\de)\\
&&\hskip3cm+ (\la_3^{-\frac 13}y_2)\de_3r_3(y_1;\la_3^{-1}u_1,x_1,\de)\Big)\, dy_1dy_2.
\end{eqnarray*}
Moreover, the cuboid $Q$ in Lemma \ref{simplesum} is defined by the conditions
$$
|u_1|\le 2^{-M}, \quad |u_2|<4c,\quad |u_3|\le 1.
$$

Let us estimate the $C^1$- norm of $H$ on $Q.$  If $|x_3|\gg 1,$ then $|Q_A(x)|\gg 1,$ whereas $|Q_B(x)|\lesssim 1,$ so that $|A|\gg |B|,$ hence $\max\{|A|,|B|\}=|A|.$  We even have that $|By_2|\ll \la_3|Q_B(x)|\ll |A|,$ as well as 
$|y_2^3 b(x_1-\la_3^{-1} u_1y_1,\la_3^{-\frac13}y_2,\de)|\ll |A|.$  Making use also of the rapid  decay of 
$\check \chi_1(y_1),$ this easily implies that 
$$
|H(u)|\lesssim |A|^{-N} 
$$
for every $N\in \NN,$ uniformly on $Q.$ 

 On the other hand, if $|x_3|\lesssim 1,$ then $|A|\lesssim  \la_3,$ so that the assumptions of Lemma \ref{as1} are satisfied (if we essentially put $T:=\la_3^{1/3}),$ and we may conclude that
\begin{equation}\label{5.18}
|H(u)|\lesssim \max\{|A|^{\frac 13}, |B|^{\frac 12}\}^{-\frac 12}, \qquad \mbox{for all} \ u\in Q.
\end{equation}
We have seen that this estimate holds true no matter which size $|x_3|$ may have.

\medskip
We next consider partial derivatives of $H.$  From our integral formula for $H(u),$ it is obvious that the partial derivative of $H$ with respect to $u_1$ will essentially only produce additional factors of the form $\la_3^{-1} y_1, 
\la_3^{-1} y_1 y_2^3, \la_3^{-1/3} y_2 \la_3^{-1} y_1$ under the double integral. However, powers of $y_1$ can be absorbed by the rapidly decaying factor $\check\chi_1(y_1), $ and $|\la_3^{-1} y_2^3|\le \ve \ll 1,$  so that $|\pa_{u_1}H(u)|$ will satisfy an estimate of the form \eqref{5.18} as well. It is also easy to see that $|\pa_{u_2}H(u)|$  satisfies such an estimate too.

\smallskip
More of a problem is the partial derivative of $H$ with respect to $u_3.$ This will  essentially produce an additional factor $y_2$  under the  double integral. More precisely, let us put
$$
g(y_1,y_2; u):= -\check\chi_1(y_1) \, y_2\,\check\chi'_1(u_2-u_3y_2+r_1(y_1;\la_3^{-1}u_1,x_1,\de)),
$$
so that 
\begin{eqnarray*}
&&\pa_{u_3}H(u):=\iint  \eta(x_1-u_1\la_3^{-1}y_1,\,  \la_3^{-\frac13} y_2) \,g(y_1,y_2; u)\\
&&\hskip1cm\times \check\chi_1\Big(A- B y_2-y_2^3 b(x_1-\la_3^{-1} u_1y_1,\la_3^{-\frac13}y_2,\de)
+u_1r_2(y_1;\la_3^{-1}u_1,x_1,\de)\\
&&\hskip3cm+ (\la_3^{-\frac 13}y_2)\de_3r_3(y_1;\la_3^{-1}u_1,x_1,\de)\Big)\, dy_1dy_2.
\end{eqnarray*}
We claim that for every $\eps\in ]0,1]$  and $s\in [0,1]$ we have 
\begin{equation}\label{5.20}
|g(y_1,y_2; su)|\le C_N |u_3|^{\eps-1}|y_2|^{\eps} \,s^{\eps-1} (1+|y_1|)^{-N}\qquad  \mbox{for every } N\in \NN
\end{equation}
in the integrand.  Indeed, if $|su_3y_2|\gg (1+|y_1|),$ then the third factor in $g(y_1,y_2; u)$ can be estimated  by 
$C|su_3y_2|^{\eps-1},$  because of the rapid decay of $\check\chi'_1,$  and \eqref{5.20} follows, and  if $|su_3y_2|\lesssim (1+|y_1|),$ then $|y_2|\lesssim |y_2|^\eps(|su_3|^{-1}(1+|y_1|))^{1-\eps},$ and \eqref{5.20} follows again.

\smallskip
By means of \eqref{5.20}, we may now estimate
\begin{eqnarray*}
&&|\pa_{u_3}H(su)|\lesssim |u_3|^{\eps-1}\,s^{\eps-1}\iint  (1+|y_1|)^{-N}
\Big(1+|A- B y_2-y_2^3 b(x_1-\la_3^{-1} su_1y_1,\la_3^{-\frac13}y_2,\de)\\
&&\hskip1cm +su_1r_2(y_1;\la_3^{-1}su_1,x_1,\de)
 + (\la_3^{-\frac 13}y_2)\de_3r_3(y_1;\la_3^{-1}su_1,x_1,\de)|\Big)^{-N}\, |y_2|^{\eps} \, dy_1dy_2,
\end{eqnarray*}
and arguing from here on as before (distinguishing between the cases where $|x_3|\gg 1$ and where $|x_3|\lesssim 1),$ we obtain by means of Lemma \ref{as1} that 
$$
|\pa_{u_3}H(su)|\lesssim |u_3|^{\eps-1}\,s^{\eps-1}\max\{|A|^{\frac 13}, |B|^{\frac 12}\}^{\eps-\frac 12}, \qquad \mbox{for all} \ u\in Q.
$$
This implies for every sufficiently small $\eps>0$  that for all $u\in Q,$ 
$$
\int_0^1\Big|\frac{\pa H}{\pa u_k}(su)\Big|\, ds \le C  |u_k|^{\epsilon-1}\max\{|A|^{\frac 13}, |B|^{\frac 12}\}^{\eps-\frac 12}. 
$$
By means of Lemma \ref{simplesum}, we may thus conclude that 
$$
|\si^2(\la_3,t,x)|\lesssim \frac 1{|2^{-\frac 32i t}-1|}\max\Big\{(\la_3|Q_A(x)|)^{\frac 13}, (\la_3^{\frac 23}|Q_B(x)|)^{\frac 12}\Big\}^{\eps-\frac 12}.
$$
Finally, this estimate allows to sum also  in $\la_3,$ and we conclude that  $|\mu_{1+it}^2(x)|\le C,$ provided we choose 
the second factor in the definition of $\ga(\ze)$ as
 $\ga_2(\ze):=2^{\frac 32 (1-\ze)}-1.$

\medskip
\noi {\bf 3. The part where $\max\{|A|,|B|\}< 1$ and $|D|\ge 4c.$ } 
 Denote  by $\mu_{1+it}^3(x)$ the contribution to $\mu_{1+it}(x)$ by the terms for which 
 $\max\{|A(x,\la_3,\de)|,|B(x,\la_3,\de)|\}< 1$ and $|D(x,\la_1,\de)|\ge 4c.$ 

This case can be treated  again by means of Lemma \ref{simplesum}, only with the roles of $\la_1$ and $\la_3$ interchanged. So, let us  here fix $\la_1$ satisfying $2^{2M}\le\la_1\le 2^{-M/3}\de_0^{-\frac 32}$ and $\la_1|Q_D(x)|\ge 4c,$  and consider the remaining sum in $\la_3$ in  \eqref{5.11},  i.e., 
$$
\si^3(\la_1, t,x):=\sum_{\{\la_3: \de_0^3\la_1^3\le \la_3\le 2^{-M} \la_1,\, \la_3|Q_A(x)|< 1, \la_3^{\frac 23}|Q_B(x)|< 1\}}
\la_3^{-it} \mu_{\la_1,\la_3}(x).
$$
Notice that then  $\mu_{1+it}^3(x)=\sum_{\la_1}\la_1^{-3it/2}\si^3(\la_1, t,x),$ where summation is over all dyadic $\la_1$ in the range described before.

Also $\si^3(\la_1, t,x)$  can  essentially be written in the form \eqref{Ft1}, with $\al:=-1$ and $l=k_3$ (if $\la_3=2^{k_3})$,  and
\begin{eqnarray*}
&& u_1=2^{\be^1l}a_1:= \la_3\, \la_1^{-1},\  u_2= 2^{\be^2l}a_2:= \la_3^{-\frac 13},\ u_3=2^{\be^3l}a_3:=\la_3^{-\frac 13}\,(\la_1 \de_0)\\
 && u_4= 2^{\be^4l}a_4:=\la_3 Q_A(x),\ u_5= 2^{\be^5l}a_5:=\la_3^{\frac 23} Q_B(x),\  
\end{eqnarray*}
and where the function $H=H_{\la_1,x,\de}$  of $u:=(u_1,\dots, u_5)$ is given by
\begin{eqnarray*}
&&H(u):=\iint  \check\chi_1(y_1)\eta(x_1-\la_1^{-1}y_1,\,  u_2 y_2) \,\check\chi_1(D-u_3y_2
+r_1(y_1;\la_1^{-1},x_1,\de))\\
&&\hskip1cm\times \check\chi_1\Big(u_4- u_5 y_2-y_2^3 b(x_1-\la_1^{-1}y_1,u_2y_2,\de)
+u_1r_2(y_1;\la_1^{-1},x_1,\de)\\
&&\hskip3cm+ (u_2 y_2)\de_3r_3(y_1;\la_1^{-1},x_1,\de)\Big)\, dy_1dy_2.
\end{eqnarray*}
Moreover, the cuboid $Q$ in Lemma \ref{simplesum} is defined by the conditions
$$
|u_1|\le 2^{-M}, \quad |u_2|\le 2^{-\frac M3},\quad |u_3|\le 1, \quad |u_4|\le 1, \quad |u_5|\le 1.
$$
In order to estimate the $C^1$-norm of $H$ on $Q,$ observe first that here we may estimate 
\begin{eqnarray}\nonumber
&&\Big|\check\chi_1(y_1)\,\check\chi_1\Big(u_4- u_5 y_2-y_2^3 b(x_1-\la_1^{-1}y_1,u_2y_2,\de)
+u_1r_2(y_1;\la_1^{-1},x_1,\de)\\
&&\hskip3cm+ (u_2 y_2)\de_3r_3(y_1;\la_1^{-1},x_1,\de)\Big)\Big|\\ \label{e0}
  &&\quad \le C_N(1+|y_1|)^{-N}(1+|y_2|)^{-N},\nonumber
\end{eqnarray}
for every $N\in\NN$ (just distinguish the cases where $|y_1|\ll |y_2|^3,$ and $|y_1|\gtrsim |y_2|^3$). Notice that this estimate allows in particular to absorb any powers of $y_1$ of $y_2$ in the upcoming estimations. Moreover, we find that 
\begin{equation}\label{e1}
|H(u)|\le C_N \iint (1+|y_1|)^{-N}(1+|y_2|)^{-N}\,|\check\chi_1(D-u_3y_2
+r_1(y_1;\la_1^{-1},x_1,\de))|\, dy_1 dy_2.
\end{equation}
It is easy to see that this allows to estimate 
$$
|H(u)|\le C_N |D|^{-N/2}, \qquad u\in Q.
$$
Indeed, when $1+|y_1|\gtrsim |D|,$ then we can gain a factor $|D|^{-N/2}$ from the first factor in the integral in \eqref{e1}, and when $1+|y_1|\ll |D|$  and $|u_3y_2|\le |D|/2,$ then the last factor in the integral can be estimated by $C'_N|D|^{-N}.$ Finally, when $1+|y_1|\ll |D|$  and $|u_3y_2|\ge |D|/2,$ then $|y_2|\ge |D|/2,$ because $|u_3|\le 1,$ and we can gain a factor $|D|^{-N/2}$ from the second factor in the integral in \eqref{e1}.

Similar estimates hold true also for partial derivatives of $H,$ since these essentially produce only further factors of the order  $|y_2|, |r_2(y_1)|\lesssim (1+|y_1|)$ and $|y_2 r_3(y_1)|\lesssim |y_2|(1+|y_1|)$  under the integral defining $H(u),$ and as we have observed before, such factors can easily be absorbed.

We thus find that $\|H\|_{C^1(Q)}\lesssim |D|^{-1},$ so that, by  Lemma \ref{simplesum},
$$
|\si^3(\la_1,t,x)|\lesssim \frac 1{|2^{-i t}-1|} \Big(\la_1|Q_D(x)|\Big)^{-1}.
$$
This estimate allows to sum   in $\la_1,$  since we are assuming that $\la_1|Q_D(x)|\ge 4c$ in the definition of 
$\mu_{1+it}^3(x),$ and we conclude that also $|\mu_{1+it}^2(x)|\le C,$ provided we choose 
the second factor in the definition of $\ga(\ze)$ as $\ga_2(\ze):=(2^{1-\zeta}-1)/(2^{2/3}-1).$

 \medskip
\noi {\bf 4. The part where $\max\{|A|,|B|\}< 1$ and $|D|< 4c.$ } 
 Denote  by $\mu_{1+it}^4(x)$ the contribution to $\mu_{1+it}(x)$ by the terms for which 
 $\max\{|A(x,\la_3,\de)|,|B(x,\la_3,\de)|\}< 1$ and $|D(x,\la_1,\de)|< 4c.$ 
 
 Under the assumptions of this case,  it is easily seen from  formula \eqref{mulalaa}, in combination with an estimate analogous to \eqref{e0}, that 
  $$
\mu_{\la_1,\la_3}(x)= J(A,B, D, E,\,\la_1^{-1}, \la_3^{-\frac13}, \la_3\la_1^{-1}),
$$
where $J$ is a smooth function of all its (bounded) variables. We may thus invoke Lemma 8.1 from \cite{IM-rest1} on oscillatory double sums in order to conclude that also $|\mu_{1+it}^4(x)|\le C,$ provided we choose the third factor 
$\ga_3(\ze)$ of $\ga(\ze)$ according to Remark  8.2  in \cite{IM-rest1}. 

Since the details are very similar to the discussion the corresponding case in  the  last part of the proof of Proposition 5.2 
in  \cite{IM-rest1}, we shall skip the details.

\medskip
Estimate \eqref{5.10} is a consequence of our estimates on the $\mu_{1+it}^j(x), j=1,\dots,4,$  which completes the proof of Proposition \ref{analyint3}. 

 \end{proof} 
 \bigskip

  \subsection{Estimation of   $T_\de^{III_{0}}$: Complex interpolation}\label{complexint0}
  The discussion of this operator is easier  than the one in the preceding subsection.  Observe that in place of \eqref{phideb3}, we here have 
  
  \begin{equation}\label{phideb0}
\phi_\de(x):= x_2^2  \,b(x_1,x_2,\de)+x_1^n\al(\de_1 x_1)+ \de_3x_2  x_1^{n_1}\al_1(\de_1x_1),\qquad (x_1\sim 1, |x_2|<\ve),
\end{equation}
since $B=2.$ 
  
  Assuming again  without  loss of generality that $\la_1=\la_2,$ we see that we have to prove 
 
\begin{prop}\label{analyint0}
Let $m=2$ and $B=2,$  and consider  the measure 
$$
\nu_\de^{III_{0}}:=\sum_{2^M \le \la_3\le 2^{-M}\de_0^{-1}}
\sum_{2^{M}\la_3\le \la_1\le 2^{-M}\de_0^{-1}}  \nu^{(\la_1,\la_1,\la_3)}_\de,
$$
where summation is  taken over all sufficiently large dyadic $\la_i$ in the given range. If we denote by $T_\de^{III_{0}}$  the operator of convolution with $\widehat{\nu_\de^{III_{0}}},$ then, if $M\in \NN$ is sufficiently large (and $\ve$ sufficiently    small), 
\begin{equation}\label{TIII0}
\|T_\de^{III_{0}}\|_{6/5\to 6}\le C,
\end{equation}
with a constant $C$ not depending on $\de,$ for $\de$ sufficiently small.
\end{prop}

 \begin{proof} 
 For fixed $\la_3$ satisfying $2^M \le \la_3\le 2^{-M}\de_0^{-1}$ we put
 $$
 \si^{\la_3}:=\sum_{\{\la_1:2^{M}\la_3\le \la_1\le 2^{-M}\de_0^{-1}\}}  \nu^{(\la_1,\la_1,\la_3)}_\de,
 $$
 so that
 $$
 \nu_\de^{III_{0}}=\sum_{2^M \le \la_3\le 2^{-M}\de_0^{-1}}\si^{\la_3}.
 $$
 We embed $\si^{\la_3}$ into  an analytic family of measures
 $$
\si^{\la_3}_\ze(x):=\gamma(\ze) 
\sum_{\{\la_1:2^{M}\la_3\le \la_1\le 2^{-M}\de_0^{-1}\}} \la_1^{\frac{1-3\ze}2} \nu^{(\la_1,\la_1,\la_3)}_\de,\quad \ze\in\Sigma,
$$
where  $\ga(\ze):=2^{3(1-\ze)/2}-1,$ so that $\si^{\la_3}_{1/3}=\si^{\la_3}.$ From \eqref{2.14II} we obtain that
$$
\|\widehat{\si^{\la_3}_{it}}\|_\infty \le C\la_3^{-\frac 12} \qquad \forall t\in\RR.
$$
We shall also prove that for every sufficiently small  $\eps>0$ there is a constant $C_\eps$ such that 
$$
\|\si^{\la_3}_{1+it}\|_\infty \le C_\eps\la_3^{\frac {1+\eps}2} \qquad \forall t\in\RR.
$$
By Stein's interpolation theorem for analytic families of operators,  these estimates easily imply that 
$$
\|T^{\la_3}\|_{p_c\to p'_c}\lesssim (\la_3^{-\frac 12})^{\frac 23}(\la_3^{\frac {1+\eps}2})^{\frac 13}=\la_3^{\frac {\eps-1}6},
$$
where $T^{\la_3}$ denotes the operator of convolution with $\widehat{\si^{\la_3}}.$ 
Thus, if we choose $\eps$ sufficiently small, we can also sum in $\la_3$ and obtain \eqref{TIII0}. 
\medskip

Our goal will thus be to show that for $\eps$ sufficiently small, we have
 \begin{equation}\label{III01}
|\si^{\la_3}_{1+it}(x)|  \le C_\eps\la_3^{\frac {1+\eps}2},
\end{equation}
where $C_\eps$ is independent of $t,x,\de$ and $\la_3.$  

To this end, observe that 	
$$
\si^{\la_3}_{1+it}(x):=\gamma(1+it) 
\sum_{\{\la_1:2^{M}\la_3\le \la_1\le 2^{-M}\de_0^{-1}\}} \la_1^{-\frac 32 it} \, \la_1^{-1}\nu_\de^{(\la_1,\la_1,\la_3)}(x),
$$
and, by \eqref{2.8II}, \eqref{2.9II},
\begin{eqnarray}\nonumber
\la_1^{-1}\nu_\de^{(\la_1,\la_1,\la_3)}(x)&=&\la_1\la_3 \int \check\chi_1\Big({\la_1}(x_1-y_1)\Big) \, \check\chi_1\Big({\la_1}(x_2-\de_0y_2-y_1^2\om(\de_1y_1))\Big)\\
&&\check\chi_1\Big({\la_3}\Big(x_3-y_2^2  \,b(y_1,y_2,\de)-y_1^n\al(\de_1 y_1)- \de_3y_2  y_1^{n_1}\al_1(\de_1y_1)\Big)\Big)\label{mula1la0}
\,\eta(y) \,dy,
\end{eqnarray}
where $\eta$ is supported where $y_1\sim 1$ and $|y_2|<\ve.$
Now, if  $|x_1|\gg 1,$  or $|x_2|\gg 1,$  then similar arguments as in the preceding subsection show that 
$|\la_1^{-1}\nu_\de^{(\la_1,\la_1,\la_3)}(x)|\lesssim \la_1^{-N}\la_3\le \la_3^{-1}\la_1^{2-N}$ for every $N\ge 2,$ which implies \eqref{III01}.

\medskip 
We shall therefore  assume from now on that  $|x_1|+|x_2|\lesssim 1.$ 
The  change of variables $y_1\mapsto x_1-y_1/\la_1, \  y_2\mapsto y_2/\la_3^{1/2}$ then  leads  in a similar way as in the previous subsection to $\la_1^{-1}\nu_\de^{(\la_1,\la_1,\la_3)}(x)=\la_3^{1/2} \mu_{\la_1,\la_3}(x),$ where 
$$
\mu_{\la_1,\la_3}(x):=\iint  \check\chi_1(y_1)F_\de(\la_1,\la_3,x,y_1,y_2)\, dy_1dy_2,
$$
with
\begin{eqnarray*}
&&F_\de(\la_1,\la_3,x,y_1,y_2):=\eta(x_1-\la_1^{-1}y_1,\,  \la_3^{-\frac12} y_2) \,\check\chi_1(D-Ey_2+r_1(y_1))\\
&\times& \check\chi_1\Big(A- B y_2-y_2^2 b(x_1-\la_3^{-1}(\la_3\la_1^{-1}y_1),\la_3^{-\frac12}y_2,\de)+\la_3\la_1^{-1}\Big(r_2(y_1)+ (\la_3^{-\frac 12}y_2)\de_3r_3(y_1)\Big)\Big).
\end{eqnarray*}
Here,  the quantities $A$ to $ E$ are given by 
\begin{eqnarray*}
&& A:=\la_3Q_A(x) , \qquad B:=\la_3^{\frac 12}Q_B(x),\\
&& D:=\la_1Q_D(x),  \qquad E:=(\de_0\la_1)\la_3^{-\frac 12},
\end{eqnarray*}
with $ Q_A(x), Q_B(x) $ and  $Q_D(x)$ as in \eqref{fd2}.  The functions $r_i(y_1)$ have properties as before, and we  choose again $c>0$ so that $|r_i(y_1)|\le c(1+|y_1|), i=1,2,3.$ Observe also that in this integral, $|y_1|\lesssim \la_1$ and $|y_2|\ll \la_3^{1/2},$ and that only $D$ and $E$ depend on the summation variable $\la_1.$ 

\medskip 
It will be useful to observe that a simple van der Corput estimate allows to  show that 
\begin{equation}\label{5.30n}
\int \Big|\chi_1\Big(A- B y_2-y_2^2 b(x_1\dots,\de)+\la_3\la_1^{-1}\Big(r_2(y_1)+ (\la_3^{-\frac 12}y_2)\de_3r_3(y_1)\Big)\Big)\Big|\, dy_2\le C,
\end{equation}
with a constant $C$ which does not depend on $A,B,x,y_1,$  the $\la_j$ and $\de.$ 

\medskip
\noi {\bf 1. The part where $|D|\ge 4c.$ } 
 Denote  by $\si_{1+it,1}^{\la_3}(x)$ the contribution to $\si_{1+it}^{\la_3}(x)$ by the terms for which 
  $|D(x,\la_1,\de)|\ge 4c.$
We claim  that, for every $N\in\NN,$ 
\begin{equation}\label{5.31n}
|\mu_{\la_1,\la_3}(x)|\lesssim  |D|^{-N}.
\end{equation}
Clearly, this estimate will allow us to sum   in  $\la_1$ and obtain the right kind of  estimate $|\si_{1+it,1}^{\la_3}(x)|\le C\la_3^{1/2},$ 
in agreement with \eqref{III01}.

\smallskip

In order to prove \eqref{5.31n}, let us first consider the contribution $\mu^1_{\la_1,\la_3}(x)$  to the integral defining $ \mu_{\la_1,\la_3}(x)$ by the   region where $|y_1|> |D|/4c.$ Here we may estimate $|\check{\chi_1}(y_1)|\lesssim |D|^{-N}(1+|y_1|)^{-N}$ for every $N\in\NN,$ and combining this with \eqref{5.30n} clearly yields \eqref{5.31n}.

\smallskip
Denote next by $ \mu^2_{\la_1,\la_3}(x)$ the contribution by the region where $|y_1|\le |D|/4c.$ Then $|r_i(y_1)|\le |D|/2.$  But notice  also  that
$$
|Ey_2|\ll |E|\la_3^{\frac 12}\le \la_1\de_0\ll 1,
$$
which shows that $|D-Ey_2+r_1(y_1)|\ge |D|/4.$   Thus, the second factor in $F_\de$ can be estimated by $C_N|D|^{-N},$ and again we arrive at \eqref{5.31n}.

\medskip
\noi {\bf 2. The part where $|D|< 4c.$ } 
 Denote  by $\si_{1+it,2}^{\la_3}(x)$ the contribution to $\si_{1+it}^{\la_3}(x)$ by the terms for which 
  $|D(x,\la_1,\de)|< 4c.$ 
  
  As in the discussion in the previous subsection (part 2) we see that the oscillatory sum defining $\la_3^{-1/2}\si_{1+it,2}^{\la_3}(x)$ can  essentially be written in the form \eqref{Ft1}, with $\al:=-3/2, l=k_1$  and
$$
u_1=2^{\be^1l}a_1:= \la_1^{-1}\la_3,\  u_2= 2^{\be^2l}a_2:=\la_1 Q_D(x),\  u_3=2^{\be^3l}a_3:=\la_1(\la_3^{-\frac 12} \de_0)
$$
and where the function $H=H_{\la_3,x,\de}$  of $u:=(u_1,u_2, u_3)$ is now given by
\begin{eqnarray*}
&&H(u):=\iint  \check\chi_1(y_1)\eta(x_1-u_1\la_3^{-1}y_1,\,  \la_3^{-\frac12} y_2) \,\check\chi_1(u_2-u_3y_2
+r_1(y_1;\la_3^{-1}u_1,x_1,\de))\\
&&\hskip1cm\times \check\chi_1\Big(A- B y_2-y_2^2 b(x_1-\la_3^{-1} u_1y_1,\la_3^{-\frac12}y_2,\de)
+u_1r_2(y_1;\la_3^{-1}u_1,x_1,\de)\\
&&\hskip3cm+ (\la_3^{-\frac 12}y_2)\de_3r_3(y_1;\la_3^{-1}u_1,x_1,\de)\Big)\, dy_1dy_2.
\end{eqnarray*}
Moreover, the cuboid $Q$ in Lemma \ref{simplesum} is defined by the conditions
$$
|u_1|\le 2^{-M}, \quad |u_2|<4c,\quad |u_3|\le 2^{-\frac 32 M}.
$$

Let us estimate the $C^1$- norm of $H$ on $Q.$ Because of \eqref{5.30n}, we clearly have $\|H\|_{C(Q)}\lesssim 1.$ 
We next consider partial derivatives of $H.$  From our integral formula for $H(u),$ it is obvious that the partial derivative of $H$ with respect to $u_1$ will essentially only produce additional factors of the form $\la_3^{-1} y_1, 
\la_3^{-1} y_1 y_2^2, \la_3^{-1/2} y_2 \la_3^{-1} y_1$ under the double integral. However, powers of $y_1$ can be absorbed by the rapidly decaying factor $\check\chi_1(y_1), $ and $|\la_3^{-1} y_2^2|\le \ve \ll 1,$  so that $|\pa_{u_1}H(u)|\lesssim 1$ too, and the same applies to $|\pa_{u_2}H(u)|.$ The main problem is again caused by the partial derivative with respect to $u_3,$ which produces an additional  factor $y_2.$ 

\medskip
However, arguing as in the preceding subsection (compare \eqref{5.10}), we find that for $\eps\in ]0,1]$ and $s\in [0,1]$
\begin{eqnarray*}
&&|\pa_{u_3}H(su)|\lesssim |u_3|^{\eps-1}\,s^{\eps-1}\iint\limits_{|y_2|\le \la_3^{\frac 12}}  (1+|y_1|)^{-N}
\Big|\check\chi_1(A- B y_2-y_2^2 b(x_1-\la_3^{-1} su_1y_1,\la_3^{-\frac12}y_2,\de)\\
&&\hskip1cm +su_1r_2(y_1;\la_3^{-1}su_1,x_1,\de)
 + (\la_3^{-\frac 12}y_2)\de_3r_3(y_1;\la_3^{-1}su_1,x_1,\de)|\Big|\, |y_2|^{\eps} \, dy_1dy_2.
\end{eqnarray*}
 Estimating  $|y_2|^\eps$ in a trivial way by $|y_2|^\eps\le \la_3^{\eps/2},$  we see by means of \eqref{5.30n} that 
\begin{equation}\label{5.32n}
|\pa_{u_3}H(su)|\lesssim |u_3|^{\eps-1}\,s^{\eps-1}\la_3^{\frac \eps 2}, \qquad \mbox{for all} \ u\in Q.
\end{equation}
By means of Lemma \ref{simplesum} (and our choice of $\ga(\ze)$), this implies  that 
$$
|\la_3^{-1/2}\si_{1+it,2}^{\la_3}(x)|\lesssim\la_3^{\frac \eps 2},
$$
which  completes  the proof of \eqref{III01}, and hence also of Proposition \ref{analyint0}.

 \end{proof}

\bigskip

 \subsection{Estimation of   $T_\de^{III_{2}}$: Complex interpolation}\label{complexint2}
  The discussion of this operator will  somewhat resemble the one of the operator $T_{\de,j}^{VI}$ in Subsection 8.2 of \cite{IM-rest1}, which arose from the same Subcase 3.2 (b)  of Subsection 5.3 in  \cite{IM-rest1}, with $2^{-j}$ playing again the role of $\de_0$ here, and where we have had $B=2,$ in place of $B=3$ here.
  
  Assuming again  without  loss of generality that $\la_1=\la_2,$ we see that here  we have to prove 
 
\begin{prop}\label{analyint4}
Let $m=2$ and $B=3,$ and consider  the measure 
$$
\nu_\de^{III_{2}}:=\sum_{2^{M}\de_0^{-1}\le\la_1<\de_0^{- 3/2}} \sum_{2^M \la_1\de_0\le \la_3\le 2^{-M}
(\la_1\de_0)^3} \nu^{(\la_1,\la_1,\la_3)}_\de,
$$
where summation is  taken over all sufficiently large dyadic $\la_i$ in the given range. If we denote by $T_\de^{III_{2}}$  the operator of convolution with $\widehat{\nu_\de^{III_{2}}},$ then, if $M\in \NN$ is sufficiently large (and $\ve$ sufficiently    small), 
\begin{equation}\label{TIII2}
\|T_\de^{III_{2}}\|_{6/5\to 6}\le C,
\end{equation}
with a constant $C$ not depending on $\de,$ for $\de$ sufficiently small.
\end{prop}

\begin{proof}
 
Recall that, by \eqref{2.14II},  $\|\widehat{\nu^{\la}_{\de}}\|_\infty\lesssim \la_1^{-1/2} \la_3^{-1/3}.$ In analogy to the proof of Proposition \ref{analyint3}, we  therefore define here for $\ze$ in the strip $\Sigma=\{\zeta\in \bC: 0\le \Re \zeta\le 1\}$  an analytic family of measures by 
$$
\mu_\ze(x):=\gamma(\ze) \sum_{2^{M}\de_0^{-1}\le 2^{k_1}<\de_0^{-3/2}} \sum_{2^{M+k_1}\de_0\le 2^{k_3}\le 2^{-M} 2^{3k_1}\de_0^3} 2^\frac{(1-3\zeta)k_1}2 2^\frac{(1-3\zeta)k_3}3\nu^{(2^{k_1},2^{k_1},2^{k_3})}_\de,
$$
where we may here put 
$$\ga(\ze):=\frac {1-2^{\frac 92(1-z)}}{1-2^3}.
$$
By $T_\ze$  we denote again  the operator of convolution with $\widehat{\mu_\ze}.$ 
Observe that  for $\ze=\theta_c=1/3,$ we have $\mu_{\th_c}=\nu_{\de}^{III_2},$ hence 
$
T_{\theta_c}=T_{\de}^{III_2},
$
 so that,  arguing exactly as in the preceding subsection, by means of Stein's interpolation theorem  \eqref{TIII2}  will follow if we can prove that there is a constant $C$ such that  
 \begin{equation}\label{5.25}
|\mu_{1+it}(x)|\le C,
\end{equation}
where $C$ is independent of $t,x$ and $\de.$  

Setting 
$$
\mu_{\la_1,\la_3}(x):=\la_1^{-1}\la_3^{-\frac23} \nu_\de^{(\la_1,\la_1,\la_3)}(x), 
$$
we may  re-write 
\begin{equation}\label{5.26}
\mu_{1+it}(x)=\gamma(1+it) \sum_{2^{M}\de_0^{-1}\le\la_1<\de_0^{- 3/2}} \sum_{2^M \la_1\de_0\le \la_3\le 2^{-M}
(\la_1\de_0)^3}\la_1^{-\frac 32 it}\la_3^{-it} \mu_{\la_1,\la_3}(x).
\end{equation}

Arguing as in the preceding subsection, by means  the identity \eqref{mula1la3} we see again  that we may assume in the sequel that $|x_1|\sim 1$ and $|x_2|\lesssim 1$  (notice that also here we have $\la_3\ll \la_1$).

By means of the change of variables $y_1\mapsto x_1-y_1/\la_1, \  y_2\mapsto y_2/(\la_1\de_0)$ and Taylor expansion around $x_1$ we  may re-write $\mu_{\la_1,\la_3}(x)=\la_3^{1/3}/(\la_1\de_0)\, \tilde\mu_{\la_1,\la_3}(x),$ where
\begin{equation}\label{mulala}
\tilde\mu_{\la_1,\la_3}(x)=\iint  \check\chi_1(y_1)\tilde F_\de(\la_1,\la_3,x,y_1,y_2)\, dy_1dy_2,
\end{equation}
with
\begin{eqnarray*}
&&\tilde F_\de(\la_1,\la_3,x,y_1,y_2):=\eta(x_1-\la_1^{-1}y_1,\,  \de_0^{-1}\la_1^{-1} y_2) \,\check\chi_1(D-y_2+r_1(y_1))\\
&\times& \check\chi_1\Big(A- B y_2-E\,y_2^3 b(x_1-\la_1^{-1}y_1),\de_0^{-1}\la_1^{-1} y_2,\de)+\la_3\la_1^{-1}\Big(r_2(y_1)+ (\de_0^{-1}\la_1^{-1}y_2)\de_3r_3(y_1)\Big)\Big).
\end{eqnarray*}
Here,  the quantities $A$ to $ E$ are given by 
\begin{eqnarray}\nonumber
&& A=A(x,\la_3,\de):=\la_3Q_A(x) , \qquad B:=B(x,\la_1,\la_3,\de):=\frac{\la_3}{\la_1}Q_B(x),\\
&& D=D(x,\la_1,\de):=\la_1Q_D(x),  \qquad E:=\frac{\la_3}{(\de_0\la_1)^3}\le 2^{-M},\label{fd3}
\end{eqnarray}
with
$$
Q_A(x):=x_3-x_1^n\al(\de_1 x_1),\quad Q_B(x):=\frac{\de_3}{\de_0} x_1^{n_1}\al_1(\de_1x_1), \quad Q_D(x):=x_2-x_1^2\om(\de_1 x_1).
$$
Again, the functions $r_i(y_1)=r_i(y_1; \la_1^{-1}, x_1,\de),\,  i=1,2,3, $ are smooth functions of $y_1$  (and $\la_1^{-1}$ and $x_1$)  satisfying estimates of the form \eqref{5.14}.
Moreover, we may assume that $|y_1|\lesssim \la_1$ and $|\de_0^{-1}\la_1^{-1} y_2|\le \ve,$ because of our assumption  $|x_1|\sim1$ and the support properties of $\eta.$ 

The  factor $\la_3^{1/3}/(\la_1\de_0)$ by which $\mu_{\la_1,\la_3}(x)$ and  $\tilde\mu_{\la_1,\la_3}(x)$ differ suggests to decompose  the summation over $k_3$ into three arithmetic progressions $k_3=i+3k_4,\, i=0,1,2$  (cf. a similar discussion in \cite{IM-rest1}).  Restricting  ourselves to anyone of them, let us  assume for simplicity that $i=0,$ so that  $k_3=3k_4,$ with $k_4\in \NN.$ Let us also assume that $\de_0$ is a dyadic number (otherwise, replace $\de_0$ by the biggest dyadic number smaller or equal to $\de_0$).  It is then  convenient to introduce new summation variables $(k_0,k_4)$ in place of $(k_1,k_3)$ by requiring that  $k_1=k_0+k_3/3-\log_2(\de_0)=k_0+k_4-\log_2(\de_0).$ In terms of their exponentials  
$\la_0:=2^{k_0}$ and $\la_4:=2^{k_4},$ this means that 
$$
\la_1=\frac{\la_0\la_4}{\de_0},\quad \la_3=\la_4^3,
$$
and we can re-write the  conditions  on the index sets for  $\la_1$ and $\la_3$ over which we sum in \eqref{5.26} as 
\begin{equation}\label{5.29}
\la_0\ge 2^{\frac M 3},\quad 2^{\frac M 2}\la_0^{\frac 12}\le \la_4\le \de_0^{-\frac 12}\la_0^{-1},
\end{equation}
and correspondingly we shall re-write \eqref{5.26} as 
$$
\mu_{1+it}(x)=\gamma(1+it) \,\de_0^{\frac 32 it}\sum_{\la_0\ge 2^{\frac M 3}}\sum_{2^{\frac M 2}\la_0^{\frac 12}\le \la_4\le  \de_0^{-\frac 12}\la_0^{-1}}\la_0^{-(1+\frac 32 it)}\la_4^{-\frac 92 it}  \tilde\mu_{\frac{\la_0\la_4}{\de_0},\la_4^3}(x).
$$
For $\la_0$ and $x$ fixed, let us put 
\begin{eqnarray*}
f_{\la_0,x}(\la_4)&:=& \tilde\mu_{\frac{\la_0\la_4}{\de_0},\la_4^3}(x),\\
\rho_{t,\la_0}(x)&:=&\ga(1+it)
\sum_{\{\la_4:\, 2^{\frac M 2}\la_0^{\frac 12}\le \la_4\le  \de_0^{-\frac 12}\la_0^{-1}\}}\la_4^{-\frac 92 it}f_{\la_0,x}(\la_4).
\end{eqnarray*}

The previous formula for $\mu_{1+it}(x)$  shows that in order to verify \eqref{5.25},  it will suffice to prove the following uniform estimate: there exist constants $C>0$ and $\epsilon\ge 0$ with $\epsilon <1,$ so that for all $x$ such that $|x_1|+|x_2|\lesssim 1$ and $\de$ sufficiently small we have 
\begin{equation}\label{5.30}
 |\rho_{t,\la_0}(x)|\le C\la_0^\epsilon \qquad\mbox{for} \quad  \la_0\ge2^{\frac M3}.
\end{equation}

In order to prove this, observe that by \eqref{mulala}
\begin{equation}\label{fla0}
f_{\la_0,x}(\la_4)=\iint  \check\chi_1(y_1)F_\de(\la_0,\la_4,x,y_1,y_2)\, dy_1dy_2,
\end{equation}
where
\begin{eqnarray*}
&&F_\de(\la_0,\la_4,x,y_1,y_2):=\eta(x_1-\de_0(\la_0\la_4)^{-1}y_1,\,  (\la_0\la_4)^{-1}y_2,\de) \,\check\chi_1(D-y_2+r_1(y_1))\\
&&\hskip1cm \times\check\chi_1\Big(A-By_2-E\,y_2^3  b(x_1-\de_0(\la_0\la_4)^{-1}y_1,(\la_0\la_4)^{-1}y_2,\de)+\de_3\de_0\la_4\la_0^{-2}\, y_2\, r_3(y_1)\\
&&\hskip4 cm +\de_0\la_4^2\la_0^{-1}r_2(y_1)\Big)
\end{eqnarray*}
and
  \begin{eqnarray}\nonumber
&& A=A(x,\la_4,\de):=\la_4^3Q_A(x) , \qquad B:=B(x,\la_0,\la_4,\de):=\frac{\la_4^2}{\la_0}Q_B(x),\\
&& D=D(x,\la_0,\la_4,\de):=\frac{\la_0\la_4}{\de_0}Q_D(x),  \qquad E:=\la_0^{-3}\le 2^{-M},\label{fd4}
\end{eqnarray}
where 
$$
Q_A(x):=x_3-x_1^n\al(\de_1 x_1),\quad Q_B(x):=\de_3 x_1^{n_1}\al_1(\de_1x_1), \quad Q_D(x):=x_2-x_1^2\om(\de_1 x_1).
$$
  
The functions $r_i(y_1)=r_i(y_1; \la_0^{-1}, \la_4^{-1}, x_1,\de),\,  i=1,2,3, $ are smooth functions of $y_1$  (and $\la_1^{-1}, \la_4^{-1}$ and $x_1$),  satisfying estimates of the form
\begin{equation}\label{5.33}
|r_i(y_1)|\le C|y_1|, \ 
\quad \left|\left(\frac{\pa}{\pa (\la_4^{-1})}\right)r_i(y_1; \la_0^{-1}, \la_4^{-1}, x_1,\de)\right|\le C_l |y_1|^{2}
\end{equation}
(compare \eqref{5.14}).

  \medskip
    Given $x$ and $\la_0,$ we shall split the summation in $\la_4$ into sub-intervals, according to the (relative ) sizes of the quantities $A,B$ and $D,$ which are considered as functions of $\la_4.$ 

\medskip
\noi {\bf 1. The part where $|D|\gg 1$. } Denote by $\rho_{t,\la_0}^1(x)$ the contribution to $\rho_{t,\la_0}(x)$ by the terms for which  $|D|\gg 1.$ 

We first consider the contribution to $f_{\la_0,x}(\la_4)$ given by integrating in \eqref{fla0} over the region where $|y_1|\gtrsim |D|^\ve$ (where $\ve>0$ is assumed to be sufficiently small).  Here, the rapidly decaying  first factor $\check\chi_1(y_1)$  leads to an improved estimate of this contribution of the order $|D|^{-N}$ for every $N\in \NN,$ which allows to sum over the dyadic $\la_4$ for which $|D|\gg 1,$ and the contribution to  $\rho_{t,\la_0}^1(x)$ is of order $O(1),$ which is stronger than what is needed in \eqref{5.30}.

We may therefore restrict ourselves in the sequel to the region where $|y_1|\ll |D|^\ve.$ Observe that, because of \eqref{5.33},   this implies in particular  that $|r_i(y_1)|\ll |D|^\ve, \, i=1,2,3.$  By looking at the second factor in $F_\de,$ we  again see that the contribution by the regions where in addition $|y_2|<|D|/2,$ or  $|y_2|>3|D|/2,$ is again of the  order $|D|^{-N}$ for every $N\in \NN,$  and their contributions to 
$\rho^1_{t,\la_0}(x)$ are again admissible.

  \medskip
What remains is the region where  $|y_1|\ll |D|^\ve$ and $|D|/2\le |y_2|\le 3|D|/2.$ In addition, we may assume that $y_2$ and $D$ have the same sign, since otherwise we can estimate as before. Let us therefore assume, e.g.,  that  $D>0,$ and that $D/2\le y_2\le 3D/2.$ 

\medskip
The change of variables $y_2\mapsto Dy_2$ then allows to re-write the corresponding contribution to 
$f_{\la_0,x}(\la_4)$ as 
\begin{equation}\label{mulala3}
\tilde f_{\la_0,x}(\la_4):=D\int_{|y_1|\ll |D|^\ve} \int_{1/2\le y_2\le 3/2} \check\chi_1(y_1)\tilde F_\de(\la_0,\la_4,x,y_1,y_2)\, dy_2dy_1,
\end{equation}
where here 
\begin{eqnarray*}
&&\tilde F_\de(\la_0,\la_4,x,y_1,y_2):=\eta(x_1-\de_0(\la_0\la_4)^{-1}y_1,\,  (\la_0\la_4)^{-1}D y_2,\de) \,\check\chi_1(D-Dy_2+r_1(y_1))\,\chi_1(y_2)\\
&& \check\chi_1\Big(A-BDy_2-ED^3\,y_2^3  b(x_1-\de_0(\la_0\la_4)^{-1}y_1,(\la_0\la_4)^{-1}D y_2,\de)+\de_3\de_0\la_4\la_0^{-2} D\, r_3(y_1)\, y_2\\
&&\hskip4 cm +\de_0\la_4^2\la_0^{-1}r_2(y_1)\Big),
\end{eqnarray*}
and where $\chi_1$ is supported where $y_2\sim 1.$ In the subsequent discussion, we may and shall assume that  $f_{\la_0,x}(\la_4)$ is replaced by $\tilde f_{\la_0,x}(\la_4).$

Recall also from \eqref{bfine} that $b(x_1,x_2,\de)=b_3(\de_1x_1, \de_2x_2),$  and that $\de_2\ll \de_0.$  The last estimate implies  that 
$$
|\de_2(\la_0\la_4)^{-1}D|=\frac {\de_2}{\de_0}|Q_D(x)|\ll 1,
$$
which shows that the second derivative of the argument of the last factor of  our function $\tilde F_\de(\la_0,\la_4,x,y_1,y_2)$ with respect to $y_2$ is comparable to  $|ED^3|.$ We may therefore apply a  classical van der Corput estimate for the integration in $y_2$  (see \cite{vdC};  also   case (i) in Lemma 2.2 (b) in \cite{IM-rest1}) and obtain that 
$$
|\tilde f_{\la_0,x}(\la_4)|\lesssim |D||ED^3|^{-\frac 12}=\la_0^{\frac 32} |D|^{-\frac 12}.
$$
 Interpolation with  the trivial estimate $|\tilde f_{\la_0,x}(\la_4)|\lesssim 1$ then leads to 
$
|\tilde f_{\la_0,x}(\la_4)|\lesssim  \la_0^{\frac 12}|D|^{-\frac 16}.
$
The second factor allows to sum in $\la_4, $ since we are assuming that $|D|\gg 1 ,$ and we obtain $|\rho_{t,\la_0}^{1}(x)|\le C\la_0^{1/2},$  in agreement with   \eqref{5.30}.

\bigskip
{\bf We may thus in the sequel assume that $|D|\lesssim 1$.}  
Here  we go back to \eqref{fla0} and observe that $\check\chi_1(y_1) \check \chi_1(D-y_2+r_1(y_1))$ can be estimated by $C_N\,(1+|y_1|)^{-N}(1+|y_2|)^{-N}.$ This shows in  particular that any power of $y_1$ or $y_2$ can be ``absorbed'' by these two factors.

\medskip
We shall still have to distinguish between the cases where $|B|\ge 1,$ and where $|B|< 1.$

\medskip
\noi {\bf 2. The part where $|D|\lesssim 1$ and $|B|\ge 1.$ }  Denote by $\rho_{t,\la_0}^{2}(x)$ the contribution to $\rho_{t,\la_0}(x)$ by the terms for which $|D|\lesssim 1$ and $|B|\ge 1.$ 

\medskip
If $|y_2|\gtrsim (|B|/|E|)^{1/2},$ then we see that we can estimate the contribution to  $ f_{\la_0,x}(\la_4)$ by a constant times $(|B|/|E|)^{-1/2}=\la_0^{-3/2}|B|^{-1/2}.$ Summing over all $\la_4$ such that $|B|\ge 1$  then leads to a  uniform estimate  for the contributions of these regions to  $\rho_{t,\la_0}^{2}(x).$

\medskip
So, assume that $|y_2|\ll (|B|/|E|)^{1/2}=:H.$ Applying the change of variables $y_2\mapsto H y_2,$ we then see that we may replace $f_{\la_0,x}(\la_4)$ by 
\begin{equation}\label{flaa0}
\tilde f_{\la_0,x}(\la_4)= H\, \iint \check\chi_1(y_1)\tilde F_\de(\la_0,\la_4,x,y_1,y_2)\, dy_1dy_2,
\end{equation}
where
\begin{eqnarray*}
&&\tilde F_\de(\la_0,\la_4,x,y_1,y_2):=\eta(x_1-\de_0(\la_0\la_4)^{-1}y_1,\,  (\la_0\la_4)^{-1}Hy_2,\de) \,\check\chi_1(D-Hy_2+r_1(y_1))\\
&&\hskip1cm \times\chi_0(y_2)\,  \check\chi_1\Big(A-HB \Big( y_2+  y_2^3 \,\sgn(B)\,  b(x_1-\de_0(\la_0\la_4)^{-1}y_1,(\la_0\la_4)^{-1}Hy_2,\de\Big)\\
&&\hskip3cm+\de_3\de_0\la_4\la_0^{-2} H\, y_2\, r_3(y_1) +\de_0\la_4^2\la_0^{-1}r_2(y_1)\Big)
\end{eqnarray*}
We claim that 
\begin{equation}\label{5.37}
|\tilde f_{\la_0,x}(\la_4)|\le C H|HB|^{-\frac 12}=\la_0^{\frac 34}|B|^{-\frac 14}.
\end{equation}
Since we are here assuming that $|B|\ge 1,$ this estimate would imply the estimate $|\rho_{t,\la_0}^{2}(x)|\le C\la_0^{3/4},$ again   in agreement with   \eqref{5.30}.

In order to prove \eqref{5.37}, observe first that the contribution to $\tilde f_{\la_0,x}(\la_4)$ by  the region where $|y_1|> |HB|$   clearly can be estimated by the right-hand side of \eqref{5.37}, because of the rapidly decaying factor $\check\chi_1(y_1)$ in the integrand.  And, on the remaining region where $|y_1|\le |HB|,$ we have 
\begin{eqnarray*}
|\de_3\de_0\la_4\la_0^{-2} H\, r_3(y_1)|&\le& \de_3\de_0\la_4\la_0^{-2} |H|\, |HB|=  \de_3\frac{\de_0\la_4^2}{\la_0}|Q_B(x)|^{\frac 12}|\, |HB| \\
&\lesssim& \la_0^{-3}\de_3^{\frac 32} \, |HB|\ll  |HB|.
\end{eqnarray*}
Observe also that  $(\la_0\la_4)^{-1}H=|Q_B(x)|^{\frac 12}\lesssim \de_3^{\frac 12}\ll 1.$ 
This shows that if $\ga$  denotes the  argument of the last factor of $\tilde F_\de(\la_0,\la_4,x,y_1,y_2),$ then 
there are constants $0<C_1<C_2,$ such that
$$
C_1|HB|\le \Big|\frac{\pa}{\pa{y_2}} \ga(y_2)\Big|+\Big|\Big(\frac{\pa}{\pa{y_2}}\Big)^2\ga (y_2)\Big|\le C_2|HB|,\quad |y_2|\lesssim 1,$$
uniformly in $x, y_1$ and $\de.$  We may thus apply a van der Corput type estimate (see case (ii) in Lemma 2.2 (b) of \cite{IM-rest1}) to the integration in $y_2$  and again arrive at  an estimate by the right-hand side of \eqref{5.37}, also for the contribution by the region where $|y_1|\le |HB|.$ 

\medskip
\noi {\bf 3. The part where $|D|\lesssim 1, |B|< 1$  and $|A|\gg 1.$}  Denote by $\rho_{t,\la_0}^{3}(x)$ the contribution to $\rho_{t,\la_0}(x)$ by the terms for which  these conditions are satisfied. We claim that  here we get 
\begin{equation}\label{5.38}
|f_{\la_0,x}(\la_4)|\le C|A|^{-N},
\end{equation}
for every $N\in\NN.$  This estimate will  imply the estimate $|\rho_{t,\la_0}^{3}(x)|\le C\la_0^{3/4},$ again   in agreement with   \eqref{5.30}.

In order to prove \eqref{5.38}, observe that  the contributions to $f_{\la_0,x}(\la_4)$  by the regions where  $|y_1|\gtrsim |A|^{1/3},$ or  $|y_2|\gtrsim |A|^{1/3},$ can be estimated by a constant times $(|A|^{1/3})^{-N},$ because  of the rapid decay in $y_1$ and $y_2$ of $F_\de.$  So, assume that $|y_1|+|y_2|\ll|A|^{1/3}.$ 
Then we see that 
$$
|By_2|\ll |B||A|^{1/3}\ll |A| \quad \mbox{and}\quad  |E y_2^3|\ll |EA|\ll |A|,
$$
as well as 
$$
|\de_3\de_0\la_4\la_0^{-2}\, y_2\, r_3(y_1)+\de_0\la_4^2\la_0^{-1}r_2(y_1)|\ll |A|^{\frac 23}\ll |A|,
$$
and thus the last factor  of $F_\de$  is of order $|A|^{-N}.$ Consequently, also the contribution to $f_{\la_0,x}(\la_4)$  by the region where   $|y_1|+|y_2|\ll|A|^{1/3}$ can be estimated as in \eqref{5.38}.

\medskip
\noi {\bf 4. The part where $\max\{|A|, |B|,|D|\}\lesssim 1.$ }  Denote by $\rho_{t,\la_0}^{4}(x)$ the contribution to $\rho_{t,\la_0}(x)$ by the terms for which $\max\{|A|, |B|,|D|\}\lesssim 1.$  Then we can easily  estimate $\rho_{t,\la_0}^{4}(x)$ by means of Lemma \ref{simplesum} in a very similar way as we did in the last part of Section 8 in \cite{IM-rest1}, and obtain that $|\rho_{t,\la_0}^{4}(x)|\le C. $ 

\medskip
This completes the proof of estimate \eqref{5.30} (with $\eps:=3/4$), and hence also of the proof of Proposition \ref{analyint4}.

  \end{proof}

\setcounter{equation}{0}
\section{The case where $\la_1\sim\la_2\sim \la_3$}\label{equalla}

We shall  assume for the sake of simplicity that 
$$
 \la_1=\la_2=\la_3\gg 1.
$$  
The more general case where  $\la_1\sim\la_2\sim \la_3\gg 1$ can be treated in a very similar way. By changing notation slightly, we shall denote in this section by $\la$ the common value of the $\la_j.$

We change coordinates from $\xi=(\xi_1,\xi_2,\xi_3)$ to $s_1, s_2$ and $s_3:=\xi_3/\la,$ i.e., 
$$
\xi_1=s_1\xi_3=\la s_1s_3,\quad \xi_2=\la s_2\xi_3=\la s_2s_3,\quad \xi_3=\la s_3,
$$
and write in the sequel 
$$
s:=(s_1,s_2,s_3) \quad s':=(s_1,s_2).
$$
Then we may re-write 
$$
\Phi(x,\de,\xi)=\la s_3\tilde \Phi(x,s'),
$$
where
\begin{eqnarray} \nonumber
\tilde \Phi(x,s')&:=&s_1x_1+s_2 x_1^m\omega(\de_1x_1)+ x_1^n\alpha(\de_1x_1)\\ 
&+& s_2\de_0x_2+\Big( x_2^B\, b(x_1,x_2,\de)+r(x_1,x_2,\de)\Big),\label{4.1II}
\end{eqnarray}
where $\om(0)\ne 0,\al(0)\ne 0, $ and $b(x_1,0,\de)\ne 0,$ if $x_1\sim 1,$ and where $\de$ and $r(x_1,x_2,\de)$ are given by \eqref{defde} and \eqref{defr}, respectively.

 Now, the first part of $\tilde\Phi$ has at worst an Airy type singularity with respect to $x_1,$ and the derivate of order $B$ with respect to $x_2$ does not vanish, so that we obtain

\begin{equation}\label{hatnuest}
\|\widehat{\nu_\de^\la}\|_\infty \lesssim  \la^{-\frac 13-\frac 1B}
\end{equation}
 (indeed, by localizing near a given point $x^0$ and looking at the corresponding Newton polyhedron of $\tilde \Phi$  at this point, this follows more precisely from the main result in \cite{IM-uniform}).
 On the other hand, standard van der Corput type arguments  (compare Lemma 2.2 in \cite{IM-rest1}) show that here 
\begin{equation}\label{nuest}
\|\nu_\de^\la\|_\infty \lesssim  \min\{  \la^3\la^{-1}\la^{-\frac 1B}, \, \la^3\la^{-1}(\la\de_0)^{-1}\}
=\la\,\min\{\la^{\frac{B-1}B}, \, \de_0^{-1}\}.
\end{equation}
 We therefore distinguish the following cases:
 \medskip
 
  {\bf Case A: $\la\le \de_0^{-B/(B-1)}.$} Then $\|\nu_\de^\la\|_\infty \lesssim \la^{2-1/B},$ and by interpolation we get
 \begin{equation}\label{4.2II}
\|T_\de^\la\|_{p_c\to p'_c}\lesssim \la^{-\frac 13-\frac 1B+\frac 73 \th_c},
\end{equation}
again with $\th_c=2/p'_c.$

\medskip
 
 {\bf  Case B: $\la> \de_0^{-B/(B-1)}.$} Then $\|\nu_\de^\la\|_\infty \lesssim \la\de_0^{-1},$ and by interpolation we get
 \begin{equation}\label{4.3II}
\|T_\de^\la\|_{p_c\to p'_c}\lesssim \la^{-\frac 13-\frac 1B+(\frac 43+\frac 1B) \th_c}\, \de_0^{-\th_c}.
\end{equation}

Observe that in both cases, the exponents of $\la$  in these estimates become strictly smaller if we replace $\th$ by a strictly smaller numbers, and the one of  $\de$ increases.

\medskip

\subsection{The case where $h^r+1>B$}  \label{thlB}
We observe that then $p'_c>2B,$ and thus 
$$\th_c<1/B.$$
This shows that 
$$
-\frac 13-\frac 1B+\frac 73 \th_c<-\frac 13-\frac 1B+\frac 73 \frac 1B=-\frac{B-4}{3B}.
$$
Thus, if $B\ge 4,$ then for $\th=\th_c,$ the exponent  of $\la$ in \eqref{4.2II} is strictly negative, so that in  Case A  we can sum the estimates for $p=p_c$ over  all $\la\le \de_0^{-B/(B-1)}$ and obtain, for $B\ge 4,$ 
\begin{equation}\label{4.4II}
\sum_{\la\le \de_0^{-B/(B-1)}}\|T_\de^\la\|_{p_c\to p'_c}\lesssim 1.
\end{equation}

Similarly, since 
$$
-\frac 13-\frac 1B+(\frac 43+\frac 1B) \th_c<-\frac 13-\frac 1B+(\frac 43+\frac 1B) \frac 1B=-\frac{B^2-B-3}{3B^2},
$$
where $B^2-B-3>0$ if $B\ge 3,$ we see that \eqref{4.3II} implies  in  Case B that
$$
\sum_{\la> \de_0^{-B/(B-1)}} \|T_\de^\la\|_{p_c\to  p'_c}\lesssim \de_0^{\frac{B+3-7B\th_c}{3(B-1)}}
<\de_0^{\frac{B-4}{3(B-1)}},
$$
hence, for $B\ge 4,$  
\begin{equation}\label{4.5II}
\sum_{\la> \de_0^{-B/(B-1)}} \|T_\de^\la\|_{p_c\to p'_c}\lesssim 1,\qquad (B\ge 4).
\end{equation}

\medskip

The case where $B=3$ requires more refined estimates, whereas the case  $B=2$ is rather easy to handle, 
given our assumption  \eqref{ontc}.

\medskip

\noindent {\bf Assume first that $B=2.$}Ê Then, by \eqref{4.2II}, if $\la\le \de_0^{-2},$
$$
\|T_\de^\la\|_{p_c\to p'_c}\lesssim \la^{-\frac 13-\frac 12+\frac 73 \th_c}= \la^{\frac {14\th_c-5}{6}},
$$
and since by our assumption \eqref{ontc} $\th_c\le1/3,$  the exponent of $\la$ in this estimate is strictly 
negative, so that we can sum over $\la$ and again obtain \eqref{4.4II}.

Similarly, by \eqref{4.3II},  if $\la>\de_0^{-2},$
$$
\|T_\de^\la\|_{p_c\to p'_c}\lesssim \la^{-\frac 13-\frac 12+(\frac 43+\frac 12) \th_c}\, \de_0^{-\th_c}
=\la^{\frac{11\th_c-5}{6}}\, \de_0^{-\th_c}.
$$
But, $11\th_c-5\le 11/3-5<0,$ and so we get 
$$
\sum_{\la> \de_0^{-2}} \|T_\de^\la\|_{p_c\to p'_c}\lesssim \de_0^{\frac{5-14\th_c}{3}}\le 1,
$$
so that \eqref{4.5II} holds true also in this case.

\bigskip
\noindent {\bf Assume next  that $B=3.$}Ê  Then in Case A, where  $\la\le \de_0^{-3/2},$  we have by \eqref{4.2II}
$$\|T_\de^\la\|_{p_c\to p'_c}\lesssim \la^{\frac{7\th_c-2}{3}},
$$
and thus, if $7\th_c-2<0,$ then  we can sum these estimates in $\la\le \de_0^{-3/2}$ and obtain \eqref{4.5II}.

Let us therefore assume henceforth that $\th_c\ge 2/7.$  Observe that 
by Lemma \ref{pc-est1}  we have $\th_c< \tilde\th_c,$ 
unless $\tilde h^r=d$ and $h^r+1\ge H,$  in which case we have $\th_c=\tilde\th_c$ and $\tilde p'_c=p'_c.$ Thus 
$$\|T_\de^\la\|_{p_c\to p'_c}\lesssim \la^{\frac{7\tilde \th_c-2}{3}},
$$
with $\tilde\th_c>2/7,$ unless $\th_c=\tilde\th_c=2/7, \tilde h^r=d$ and $h^r+1\ge H.$ 
Note that in the latter case, $H=B=3,$   and since $\tilde\th_c=2/7,$ we find that $m=5$ and $d=5.2.$ 
 
 \medskip
In this particular case, we  only get a uniform estimates for 
$\|T_\de^\la\|_{p_c\to p'_c}\lesssim 1.$ However, here we have  $\|\nu^\la_\de\|_\infty \lesssim \la^{5/3},$  since we are in Case A, and $\|\widehat{\nu^\la_\de}\|_\infty \lesssim \la^{2/3},$ whereas $\th_c=2/7,$ and thus $-(1-\th_c) 2/3+\th_c 5/3=0.$  
Moreover, $\nu^\la_\de=\nu_\de*\phi_\la,$ where the Fourier transform of $ \phi_\la$ is given by 
$\widehat{\phi_\la}(\xi)=\chi_1\Big(\frac {\xi_1}{\la}\Big)\chi_1\Big(\frac {\xi_2}{\la}\Big)\chi_1\Big(\frac {\xi_3}{\la}\Big).$
  This implies a uniform estimate of the $L^1$-norms of the  $\phi_\la$  for all dyadic $\la.$  We may thus  estimate the operator $T^{IV_1}$ of convolution with the   Fourier transform of the complex measure 
  $$
  \nu_\de^{IV_1}:=\sum_{\la\le \de_0^{-3/2}}\nu_\de^\la
  $$
  by means of  the real-interpolation Proposition \ref{bsint} in the same way as we estimated the operators $T^{I_1}$ and $T^{II_1}$ in Subsection \ref{realint1}, by adding the measure $\nu_\de^{IV_1}$ to the family of measures $\mu^i, i\in I,$ from the second class in \eqref{nudecomp}.
  
\medskip

So, assume that $\tilde\th_c>2/7.$ Then we find that 
$$
\sum_{\la\le \de_0^{-3/2}}\|T_\de^\la\|_{p_c\to p'_c}\lesssim \de_0^{\frac{2-7\tilde\th_c}{2}}.
$$

Let us next turn to Case B, where $\la>\de_0^{-3/2}.$ Then, by \eqref{4.3II},
$$
\|T_\de^\la\|_{p_c\to p'_c}\lesssim \la^{\frac{5\th_c-2}{3}}\, \de_0^{-\th_c}.
$$
Since $\th_c\le 1/3,$ we can sum in $\la$ and obtain
$$
\sum_{\la>\de_0^{-3/2}}\|T_\de^\la\|_{p_c\to p'_c}\lesssim  \de_0^{\frac{2-7\th_c}{2}}\le \de_0^{\frac{2-7\tilde\th_c}{2}}.
$$
Combining these estimates, we obtain
\begin{equation}\label{tdest}
\sum_{\la\gg 1}\|T_\de^\la\|_{p_c\to p'_c}\lesssim \de_0^{\frac{2-7\tilde\th_c}{2}}.
\end{equation}

Observe next that 
$$
\frac{2-7\tilde\th_c}2=1-\frac 7{\tilde p'_c}=\frac {2\tilde h^r-5}{\tilde p'_c},
$$
and recall that $\de_0=2^{-(\tilde\ka_2-m\tilde\ka_1)}.$
In combination with the re-scaling estimate \eqref{rescale} this leads to 
\begin{eqnarray}
\left(\int|\hat f|^2d\mu_{1,k} \right)^\frac12 \nonumber
&\lesssim &
2^{-k\Big(\frac{|\tilde\ka|}2-\frac{\tilde\ka_1(1+m)+1}{p_c'}+(\tilde\ka_2-m\tilde\ka_1)\frac {2\tilde h^r-5}{2\tilde p'_c}\Big)}\|f\|_{L^{p_c}}\\
&\le  &C
2^{-k\Big(\frac{|\tilde\ka|}2-\frac{\tilde\ka_1(1+m)+1}{\tilde p_c'}+(\tilde\ka_2-m\tilde\ka_1)\frac {2\tilde h^r-5}{2\tilde p'_c}\Big)}\|f\|_{L^{p_c}}.\label{rse}
\end{eqnarray}
where $\mu_{1,k}$ denotes the measure corresponding to the frequency  domains that we are here considering, i.e., $\mu_{1,k}$ corresponds to the re-scaled measure
$$
\nu_{1,\de}:=\sum_{\la\gg 1}\nu_\de^{(\la,\la,\la)}.
$$

But, 
\begin{eqnarray*}
E&:=&2\tilde p_c'\Big(\frac{|\tilde\ka|}2-\frac{\tilde\ka_1(1+m)+1}{\tilde p_c'}+(\tilde\ka_2-m\tilde\ka_1)\frac {2\tilde h^r-5}{2\tilde p'_c}\Big)\\
&=&|\tilde\ka|(2\tilde h^r+2)-2(\tilde\ka_1(1+m)+1)+(\tilde\ka_2-m\tilde\ka_1)(2\tilde h^r-5)\\
&=&\tilde\ka_2(4\tilde h^r-3)+\tilde\ka_1(3m-2\tilde h^r(m-1))-2,
\end{eqnarray*}
where
\begin{eqnarray*}
\tilde\ka_2(4\tilde h^r-3)&=&\frac {4m}{m+1}-3\tilde\ka_2,\\
\tilde\ka_1(3m-2\tilde h^r(m-1))&=&m\frac{\tilde\ka_1}{\tilde\ka_2}\Big(\frac 3H-2\frac{m-1}{m+1}\Big).
\end{eqnarray*}
Since $H\ge B=3,$ we see that $3/H-2(m-1)/(m+1)\le (3-m)/(m+1)\le 0$ if $m\ge 3.$ Thus, if $m\ge 3,$ then since $\tilde\ka_1/\tilde\ka_2=1/a<1/m$ we see that
$$
\tilde\ka_1(3m-2\tilde h^r(m-1))\ge  3\tilde\ka_2-2\frac{m-1}{m+1}
$$
 and altogether we find that $E\ge 0$ (even with strict inequality, if $H>3$). We thus have proved
\begin{equation}\label{4.6II}
\left(\int|\hat f|^2d\mu_{1,k} \right)^\frac12 \le C \|f\|_{L^{p_c}},
\end{equation}
with a constant $C$ not depending on $k,$  provided that $m\ge 3.$

\bigskip 
{\bf Assume finally that $B=3$ and  $m=2.$ } Recall also that we are still   assuming that $h^r+1>B=3$ and $\th_c\ge 2/7,$  so that $3<h^r+1\le 7/2.$

\medskip
We shall prove that the Newton-polyhedron of $\tpad$ respectively $\phi$ will have a particular structure.
 Indeed, if $\phi$ is analytic, then one can show that $\phi$ is  of type $Z,E,J,Q $ etc.,  in the sense
of Arnol'd's classification of singularities (compare  \cite{agv}). 

  We shall, however, content ourselves with a little less information, which will nevertheless  be sufficient for our purposes. 

Recall from \cite{IM-rest1} the notion of   {\it augmented Newton polyhedron} $\N^r(\tpad)$ of $\tpad.$ If $L$  denotes the principal line of $\N(\phi),$ then it is a supporting line to $\N(\tpad)$ too, and if $(A^+,B^+)$  denotes  the right endpoint of the line segment $L\cap\N(\tpad),$ then let $L^+$ be  the half-line $L^+\subset L$  contained in the principal line of $\N(\phi)$ with right endpoint $(A^+,B^+).$  Then $\N^r(\tpad)$  is the convex hull of the union of $\N(\tpad)$ with  the half-line $L^+.$  Recall also that  $\N^r(\tpad)$ and $\N^r(\pad)$  do agree in the closed half-space above the bi-sectrix $\Delta,$ so that  $h^r+1$ is the second coordinate of the  point at which the line $\Delta^{(m)}$ intersects the boundary of $\N^r(\tpad).$

\begin{proposition}\label{Z}
If $B=3,$  $m=2$ and $3<h^r+1\le 3.5,$ then $(A^+,B^+)=(1,3),$ and $\N^r(\tpad)$ has exactly two edges, $L^+$ and the line segment $[(1,3),(0,n)],$ which  is contained in the principal line $L^a$ of $\N(\tpad).$ 

In particular,   
\begin{equation}\label{typeZ}
\ka=\Big(\frac 17,\frac 27\Big),\quad h^r=d=\frac {7}{3},\quad\mbox{and } \ \tilde \ka=\Big(\frac 1n,\frac{n-1}{3n}\Big),
\end{equation}
where $n>7.$

\end{proposition}

\begin{proof} 

Denote by  $(A',B'):=(A'_{(0)},B'_{(0)})\in L^a$ the left endpoint of the principal face $\pi(\tilde \pad)$ of the Newton polyhedron of $\tilde\pad.$ Then $B'\ge B=3.$  In a first step, we prove that $B'=3.$ 

Assume, to the contrary, that we had $B'\ge 4$ (observe that $B'$ is an integer). Since the line $L^a$ has slope strictly less $1/m=1/2,$ then it easily seen that the line  $L^a$ would intersect the line $\Delta^{(2)}$ at some point with second coordinate $z_2$  strictly bigger than $3.5,$  so that $h^r+1\ge z_2>3.5,$ which would contradict our assumption (Figure 2).
\medskip

Thus, $B'=B=3.$ In a second step, we show that $A'=1.$ To this end, let us here work with $\pad$ in place of $\tilde\pad.$   Note the point $(A',B')$ is  also the left endpoint of the principal face of $\N(\pad),$ and that the principal faces of the Newton polyhedra of $\pad$ and $\tilde \pad$ both lie on the same line $L^a,$ since the last step in the change to modified adapted coordinates \eqref{modada} preserves the homogeneity $\tilde\ka.$ This shows that also $A'\in\NN.$  Moreover,  $A'\ge 1,$ for otherwise we  had $A'=0$ and thus  $h^r+1=3.$ 

Assume that $A'\ge 2.$ We have to distinguish two cases.
\smallskip

a) If the line $L,$ which has slope $1/2,$  contains the point $(A',B'),$  then the assumption $A'\ge 2$ would imply that $h^r+1>3.5$ (see Figure 3).
\smallskip

b) If not, then $\pi(\pad)$  will have an edge   $\ga=[(A'',B''),(A',B')]$ with  right endpoint $(A',B'),$ and $L$ must  touch $\N(\pad)$  in a point contained in an edge strictly left to $\ga.$ But then  the line $L''$ containing $\ga$ must have slope strictly less than  the slope $1/2$ of $L,$ and  necessarily $B''>B'=3,$ hence $B''\ge 4.$ It is the again easily seen that the line  $L''$ would intersect the line $\Delta^{(2)}$ at some point with second coordinate $z_2$  strictly bigger than $3.5,$  so that  again $h^r+1\ge z_2>3.5,$ which would contradict our assumption (Figure 3).

\medskip
We have thus found that $(A',B')=(1,3).$ Assume finally that $\N(\phi)$ had a vertex $(A'',B'')$ to  the left of 
$(1,3).$ Then necessarily $A''=0$ and $B''\ge 4,$ so that the line passing through $(A'',B'')$ and $(1,3)$ had slope at least $1,$  a contradiction. We have seen that $\N(\phi)$ is contained in the half-plane where $t_1\ge 1,$ and thus the line $L$ must pass through the point $(1,3),$ and the  claim on the structure of $\N^r(\tilde\pad)$ is  now obvious.
 
But then clearly $\Delta^{(2)}$ will intersect the boundary of $\N^r(\tilde\pad)$ in a point of $L^+,$ so that $h^r=d.$ 
The remaining statements in \eqref{typeZ} are now easily verified.
\end{proof}

\vspace{1cm}
\begin{figure}[!ht]
\centering
\bigskip
\includegraphics[width=0.7\textwidth]{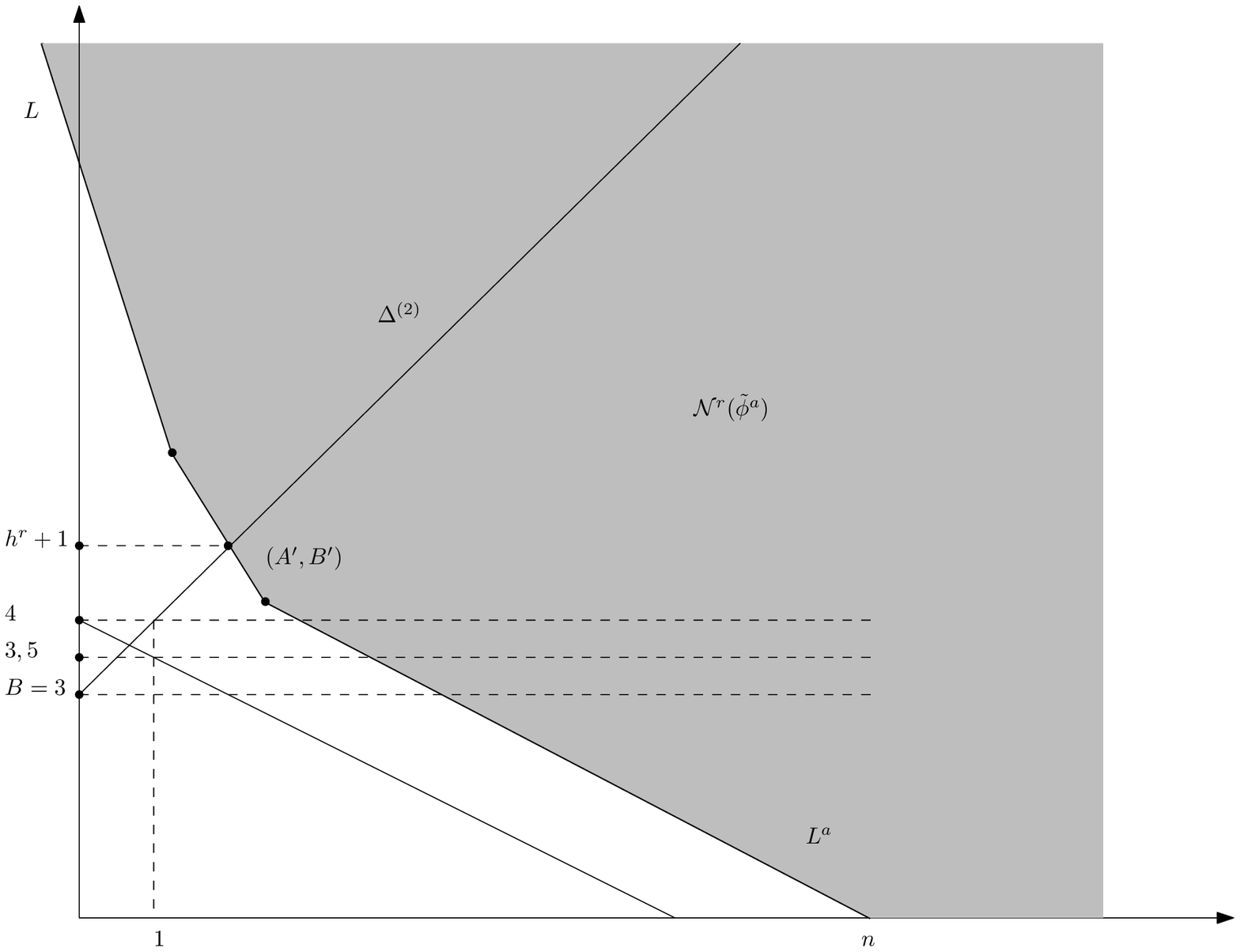}\label{fig2}
\caption{}
\end{figure}

With this structural result, we can now conclude the discussion of this case. Indeed,  by Proposition \ref{Z} we have $\th_c=3/10>2/7,$ and, arguing as before, only with $\th_c$ in place of $\tilde\th_c,$ we obtain 

\begin{equation}\label{tdest2}
\sum_{\la\gg 1}\|T_\de^\la\|_{p_c\to p'_c}\lesssim \de_0^{\frac{2-7\th_c}{2}}.
\end{equation}
 Following the previous discussion, we   find that  the exponent $E$  in the estimate \eqref{rse} is here given  by 
\begin{eqnarray*}
E&:=&2 p_c'\Big(\frac{|\tilde\ka|}2-\frac{\tilde\ka_1(1+2)+1}{p_c'}+(\tilde\ka_2-2\tilde\ka_1)\frac {2h^r-5}{2 p'_c}
\Big)\\
&=&\tilde\ka_2(4 h^r-3)+\tilde\ka_1(6-2 h^r)-2.
\end{eqnarray*}

By means of  \eqref{typeZ} one then computes that 
$$
E=\frac {19}{9}\frac{n-1}n+\frac 43\frac 1n -2=\frac{n-7}{9n}>0,
$$
so that the uniform  estimate \eqref{4.6II} remains valid also in this case.

\subsection{The case where $h^r+1\le B$ }  \label{thgB}
In this case,  since $d<h\le h^r+1,$ we have $d<B,$ and since we are assuming that $d>5/2,$ we see that we may assume that $B\ge 3.$ 

Moreover,  it is obvious from the structure of the Newton polyhedron of $\pad$ that necessarily $m+1\le B,$
and 
\begin{equation}\label{hrform}
h^r+1=h_{l_\pr}+1=\frac{1+(m+1)\tilde\ka_1}{|\tilde\ka|}.
\end{equation}
Indeed, this follows from the geometric interpretation of the notion of $r$-height given directly after Remarks 1.3 in \cite{IM-rest1}, since  the line $\Delta^{(m)}$ intersects the principal face $\pi(\pad)$ (see Figure 3). 

\vspace{2cm}
\begin{figure}[!h]
\centering
\bigskip
\vspace{1.3cm}
\includegraphics[width=0.7\textwidth]{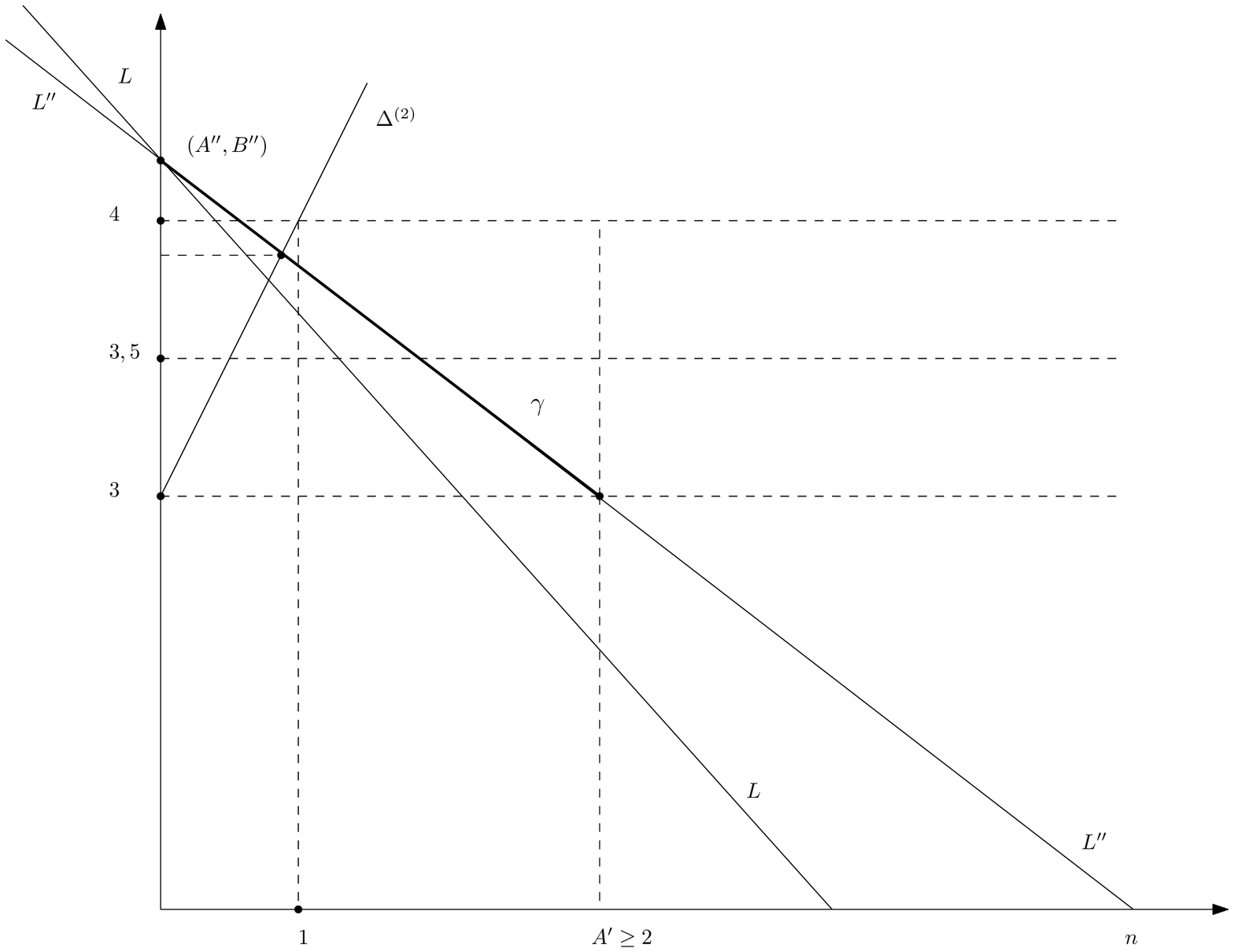}\label{fig3}
\caption{}
\end{figure}

\bigskip

Therefore,  in passing from the measure $\nu_\de$ to $\mu_k,$ no further gain is possible in this situation  in \eqref{rescale}.


\medskip
{\bf Consider first Case A.} 
 Corollary \ref{compth} (b) implies that $\th_c\le \tilde \th_B.$ Thus, by \eqref{4.2II} we have 
$$
\|T_\de^\la\|_{p_c\to p'_c}\lesssim \la^{-\frac 13-\frac 1B+\frac 73\frac{m+1}{mB+m+1}}=\la^{-\frac{M(m,B)}{3B(mB+m+1)}},
$$
where 
$$
M(m,B):= mB^2-(3m+6)B+3(m+1)
$$
is increasing in $B$ if $B\ge3$ and $m\ge 2.$ Since $B\ge m+1,$ we thus have 
$$
M(m,B)\ge m^3-m^2-5m-3,
$$
and the right-hand side of this inequality is increasing in $m$ if $m\ge 2$ and assumes the value $0$ if $m=3.$ 
Therefore, $M(m,B)\ge 0$ if $m\ge 3,$  even with strict inequality if $B>m+1.$ 

We thus see that the estimates of  $\|T_\de^\la\|_{p_c\to p'_c}$ sum for $m\ge 3$ in  $\la\le \de_0^{-B/(B-1)}$ when  $B>m+1,$    and are at least uniform, if $B=m+1.$ 
\medskip

Finally, when  $m=2,$  then $M(2,B)=2B^2-12B+9=2[(B-3)^2-9/2]>0,$ iff $B\ge 6.$
We thus find that 
\begin{equation}\label{4.8II}
\sum_{\la\le \de_0^{-B/(B-1)}}\|T_\de^\la\|_{p_c\to p'_c}\lesssim 1,\qquad \mbox{except possibly when}\ m=2, B=3,4,5, 
\ \mbox{or}\   m=3,B= 4. 
\end{equation}

Moreover, we have 
$$
\|T_\de^\la\|_{p_c\to p'_c}\lesssim 1\qquad \mbox{if} \ m=3,B=4.
$$

But, recall  that $\th_c<\tilde\th_c\le \tilde\th_B,$ unless $4=B=H=h^r+1=d+1$ here, so that $\th_c=1/4,$  so that in the case where we can only obtain the previous uniform estimate for the $T_\de^\la,$  we will have $\th_c=1/4.$ Moreover, by 
\eqref{hatnuest} and \eqref{nuest} we have $\|\nu^\la_\de\|_\infty \lesssim \la^{7/4},$  since we are in Case A, and $\|\widehat{\nu^\la_\de}\|_\infty \lesssim \la^{-7/12},$  where $-(1-\th_c) 7/12+\th_c 7/4=0.$ We may thus  again estimate the operator  of convolution with the   Fourier transform of the complex measure 
  $
  \sum_{\la\le \de_0^{-4/3}}\nu_\de^\la
  $
  by means of  the real-interpolation Proposition \ref{bsint}, in the same way as we did in the corresponding case where $m=2$ and $B=3.$

\bigskip

We are thus left with the cases where $m=2, h^r+1\le B$  and $B=3,4,5.$  So assume in the sequel that that $m=2$ and 
$h^r+1\le B$
\medskip

If $m=2$ and $B=3,$  then  $B=m+1,$ and since we are assuming $h^r+1\le B,$ a look at the Newton polyhedron shows that necessarily $H=B=3=h^r+1.$ 

\begin{lemma}\label{HgB}
Assume that  $m=2$ and  $B=4,5.$ Then we have 
$$
-\frac 13-\frac 1B+\frac 73 \tilde\th_c<0,
$$
provided 
\begin{equation}\label{HBB}
H>H(B):=\begin{cases} 
   \frac{9}{2}, &  \  \mbox{ if } B=4,\\
    \frac{81}{16}, &  \  \mbox{ if } B=5.
\end{cases}
\end{equation}

\end{lemma}

\begin{proof}For $m=2$ we have 
$$
-\frac 13-\frac 1B+\frac 73 \tilde\th_c=-\frac 13-\frac 1B+\frac 7{2H+3}<0
$$
if and only if 
\begin{equation}\label{eHgB}
H>\frac {21}2\frac B{B+3}-\frac 32,
\end{equation}
and it is easily checked that this holds true if and only if  when $H\ge H(B) $  when $B=4,5.$ 
\end{proof}

Since by Corollary \ref{compth} (b) $\th_c\le \tilde\th_c,$ the previous  lemma shows that for $m=2,$ \eqref{4.8II} can be sharpened as follows:
\begin{equation}\label{4.10II} 
\sum_{\la\le \de_0^{-B/(B-1)}}\|T_\de^\la\|_{p_c\to p'_c}\lesssim 1,
\end{equation}
 with the possible exceptions  when $B=H=3=h^r+1,$ or $B=4,5, h^r+1\le B$ and $H\le H(B).$

\bigskip 
{\bf Consider finally  Case B.} Then, since $\th_c\le \tilde\th_B,$ 
$$
\|T_\de^\la\|_{p_c\to p'_c}\lesssim \la^{-\frac 13-\frac 1B+(\frac 43+\frac 1B) \frac{m+1}{mB+m+1}}\, \de_0^{-\tilde\th_B}
=\la^{-\frac{B(mB-3)}{3B(mB+m+1)}}\, \de_0^{-\tilde\th_B}.
$$
The exponent of $\la$ is negative, so we can sum these estimates in $\la$ and obtain
$$
\sum_{\la>\de_0^{-B/(B-1)}}\|T_\de^\la\|_{p_c\to p'_c}\lesssim \de_0^{[\frac 13+\frac 1B-(\frac 43+\frac 1B)\tilde\th_B]\frac B{B-1}-\tilde\th_B}=\de_0^{\frac{B+3-7B\tilde\th_B}{3(B-1)}}=\de_0^{\frac{M(m,B)}{3(B-1)(mB+m+1)}}.
$$
But, our previous discussion of $M(m,B)$ shows that $M(m,B)\ge 0,$ unless $m=2$ and $B=3,4,5$ (notice that the case $m=3,B=4$ still works here ).

In the latter cases, we can improve our estimates again by using $\tilde\th_c$ in place of $\tilde\th_B.$ Indeed, notice that the condition $B+3-7B\tilde\th_c> 0$ is equivalent to \eqref{eHgB}, which by Lemma \ref{HgB} does hold true if $H\le H(B).$

We thus find that 

\begin{equation}\label{4.12II}
\sum_{\la> \de_0^{-B/(B-1)}}\|T_\de^\la\|_{p_c\to p'_c}\lesssim 1,
\end{equation}
unless  $m=2$ and  $B=H=3=h^r+1,$ or $B=4,5$ $h^r+1\le B$ and $H\le H(B).$

\medskip
Finally, we observe that the following consequence of  Lemma \ref{HgB}:
\begin{cor}\label{B45}
Assume that  $m=2,$ and  that $B=H=3,$ or $B=4,5,$ $h^r+1\le B$ and $H\le H(B).$ Then the left endpoint of the  principal face of the Newton polyhedron of $\tpad$ is of the form $(A,B),$ where $A\in\{0,1,\dots,B-3\},$  and we have 
$$
\tilde\ka=\Big(\frac 1n, \frac{n-A}{Bn}\Big)\quad \mbox{and } h^r+1=\frac{n+3}{n+B-A}B,
$$
where $n$ must satisfy $n>2B$ and  $ (Bn)/(n-A)=H\le H(B).$ Moreover, $\tpad$ is smooth here, so that in particular $a$ and $n$ are integers. 
\end{cor}
\begin{proof} 
Adopting the notation for the proof of Proposition \ref{Z}, let  $(A',B')$ denote the left endpoint of $\pi(\tpad),$ which, as we recall, is also the left endpoint of $\pi(\pad).$   If $B=4,5,$ then  $B\le B'\le H\le H(B)<B+1,$ so that $B'=B.$ For $B=3=H,$ the same conclusion applies. Next, since we assume that $h^r+1\le B,$ and since the line $\Delta^{(2)}$ intersects the line where $t_2=B$ in the point $(B-3,B),$ we must have $0\le A'\le B-3.$   This implies the first statement of the corollary, because $A'$ is integer.

 The  remaining statements, except for the last one, follow easily (for the identity for $h^r+1,$ recall \eqref{hrform}).
 
Finally, in order to see that in these cases we must have $a,n\in\NN,$ observe that $\tpad_{\tka}$ must be of the form
$$
\tpad_{\tilde\kappa}(x)=x_1^A x_2^B +c_1x_1^n,\qquad c_1\ne 0.
$$
But then $\pad_{\tka}$ will be given by 
$$
\pad_{\tilde\kappa}(x)=x_1^A( x_2 +c_0 x_1^a)^B+c_1x_1^n
$$
(compare \eqref{phi2}).  This must be a polynomial in $\x,$ since $\pad$ is smooth. Expanding $( x_2 +c_0 x_1^a)^B,$ it is clear that $a$ must be an integer. But then, necessarily also $\tpad$ is smooth, which implies that   $n$ must be an integer too.

\end{proof}

\subsection{The case where $B=5$}

The case where  $B=5$ can now be treated quite easily. Indeed,  going back to the estimations in Subsection \ref{thgB}, recall from \eqref{4.2II} that we have the estimate
$$
\|T_\de^\la\|_{p_c\to p_c'}\lesssim \la^{-\frac13-\frac15+\frac{7\th_c}3}.$$
This estimate is valid in  Case A as well as in Case B (in the latter case, an even stronger estimate is valid, but we won't need it). 
And, according to Corollary \ref{B45}, if $B=5$ then  we have the precise formula 
$$\th_c= \frac{n+5-A}{5(n+3)}
$$
for $\th_c,$ 
where $A\in\{0,1,2\}$ and $n>2B=10.$ Since $n$ is here an integer, we find that $n\ge 11,$ and thus 
$\th_c\le (11+5)/70=8/35,$ with strict inequality, unless $A=0$ and $n=11.$ This show that the exponent of $\la$ in the estimate for $\|T_\de^\la\|_{p_c\to p_c'}$ is strictly negative, so that we can sum these estimates over all dyadic $\la\ge 1,$ unless $A=0$ and $n=11,$ where we get a uniform estimate
\begin{equation}\label{B5A0}
\|T_\de^\la\|_{p_c\to p_c'}\lesssim 1.
\end{equation}
However, if $A=0, B=5$ and $n=11,$ then $\th_c=8/35,$ and moreover, by 
\eqref{hatnuest} and \eqref{nuest}, we have $\|\nu^\la_\de\|_\infty \lesssim \la^{9/5}$  and $\|\widehat{\nu^\la_\de}\|_\infty \lesssim \la^{-8/15},$  where $-(1-\th_c) 8/15+\th_c 9/5=0.$ We may thus  again estimate the operator  of convolution with the   Fourier transform of the complex measure 
  $
  \sum_{\la\gg 1}\nu_\de^\la
  $
  by means of  the real-interpolation Proposition \ref{bsint}, by adding this measure to the list of measures $\mu^i, i\in I,$ of the second class in \eqref{nudecomp}.
  
  \bigskip
  
  \subsection{The remaining cases}
  
  The case where $m=2$ and $B=3$  has essentially already been addressed in the previous subsection. Its treatment will  require  a refined Airy type analysis, following ideas from  Section 6 of \cite{IM-rest1}, but the discussion will be even more involved than in  \cite{IM-rest1}, where we have had   $m=2$ and $B=2.$ Also in the case $m=2$ and $B=4,$ we shall need some Airy type analysis, but of simpler form. 
  
   \medskip
Combining all the previous estimates and applying \eqref{rescale}, we see that we have proved  

\begin{cor}\label{l=1}
 Assume that $m$ and $H,B$ are not such that $m=2$ and  $B=H=3=h^r+1,$ or $m=2, B=4,$ 
 $h^r+1\le 4$ and $H\le H(4).$  Then the  estimates in Proposition \ref{Dl} hold true for $l=1.$ 
 \end{cor}

\medskip
 More precisely, in view of Corollary \ref{B45}, open are the following situations:
\begin{equation}\label{5.17II}
m=2,\, B=3,4,\, h^r+1\le B, \quad \mbox{ and}\   \la_1\sim\la_2\simÊ\la_3,
 \end{equation}
 and the left endpoint of the  principal face of the Newton polyhedron of $\tpad$ is of the form $(A,B),$ where $A\in\{0,\dots,B-3\}.$  
 
 In these situations, we have 
\begin{equation}\label{5.18II}
\tilde\ka=\Big(\frac 1n, \frac{n-A}{Bn}\Big)\quad \mbox{and }\  h^r+1=\frac{n+3}{n+B-A}B,
\end{equation}
where $n$ must be an integer satisfying $n>2B,$ and $ (Bn)/(n-A)=H\le H(B),$  if $B=4.$

\bigskip
The discussion of these cases will occupy the major part of remainder of this article. Before we come to  this, let us first study also the contributions by the remaining domains $D'_{(l)}, \ l\ge 2,$ which will turn out to be simpler.

\setcounter{equation}{0}
\section{Restriction estimates for the domains $D'_{(l)}, \ l\ge 2$ }\label{dl}
For the domans $D'_{(l)}, \ l\ge 2,$ we can essentially argue as in the preceding section, by putting here 
\begin{equation}\label{6.1II}
\tpad:=\phi^{(l+1)},\quad \tilde\ka:=\ka^{(l)},Ê\quad D^a:= D^a_{(l)}, \quad L^a:= L_{(l)}, \quad \mbox {etc.}.
\end{equation}
$H$ and $n$ are defined correspondingly. 

We then have the following analogue of Lemma \ref{pc-est1}.
\begin{lemma}\label{pc-est2} \begin{itemize}  
\item[(a)] If $l\ge 2,$ then  $p'_c>\tilde p'_c.$ 

\item[(b)]  If $m\ge 3$ and $H\ge 2,$ or $m=2$ and $H\ge 3,$ then 
$$
 \tilde p'_c\ge p'_H\ge p'_B.
$$
 \end{itemize}
\end{lemma}

\begin{proof} 
(a) We can follow the proof of Lemma \ref{pc-est1} (a) and see again that $p'_c>\tilde p'_c,$ unless the principal face $\pi(\tpad)$ of $\tpad$ is the edge $[(0,H),(n,0)].$  However, for $l\ge 2,$ we know  from Section \ref{remind} that 
$\pi(\tpad)$ is the edge $\ga'_{(l)},$ which lies below the bi-sectrix $\Delta,$ so that we cannot have $p'_c=\tilde p'_c.$  

(b) follows as  before.
\end{proof} 
This implies the following, stronger analogue of Corollary \ref{compth}.
\begin{cor}\label{compthl}
 \begin{itemize} Assume that $l\ge 2.$ 
\item[(a)] If $m\ge 3$ and $H\ge 2,$ or $m=2$ and $H\ge 3,$ then   $\th_c<\th_B.$ 
\item[(b)] If $h^r+1\le B,$ then $\th_c<\tilde\th_c,$ unless $B=H=h^r+1=d+1,$ where $\th_c=\tilde\th_c.$ 
\item[(c)] If $H\ge 3,$ then $\th_c< 1/3,$ unless $H=3$ and $m=2.$
 \end{itemize}
\end{cor}
Finally, in place of Proposition \ref{Z}, we have 
\begin{proposition}\label{Z2}
Assume that $l\ge 2,$ and that $B=3,$  $m\ge 2.$ Then $h^r+1>3.5.$
\end{proposition}

\begin{proof} 
Assume  we had $h^r+1\le 3.5,$ and denote by  $(A',B'):=(A'_{(l-1)},B'_{(l-1)})\in L^a$ the left endpoint of the principal face $\pi(\tilde \pad)=\ga'_{(l)}$ of the Newton polyhedron of $\tilde\pad.$ Then $B'\ge B=3.$  Arguing in the same way is in the  first step of the proof of Proposition \ref{compth}, we see that $B'=B=3.$ 

\medskip
Then  the preceding edge $\ga'_{(l-1)}$ will have  a left endpoint $(A'',B'')$ with $B''\ge 4.$ Moreover, since the line $L^{(1)}$ has slope $1/a<1/m\le 1/2,$ the line $L^{(l-1)}$ containing $\ga'_{(l-1)}$ has slope strictly less than $1/2.$ But then it will intersect the line $\Delta^{(2)}$ at some point with second coordinate $z_2$ strictly bigger than $3.5,$  so that  we would again arrive at  $h^r+1\ge z_2 >3.5,$  which contradicts our assumption. 
\end{proof}

These results allow us to proceed exactly as in Sections \ref{specloc},  \ref{equalla}, even with some 
simplifications. Indeed, a careful inspection of our arguments in these sections  reveals  that here all the series which  appear do sum, and no further interpolation arguments are required.
 
This is because of the stronger estimate $p'_c>\tilde p'_c$ of Lemma \ref{pc-est2} and the stronger statement of Proposition \ref{Z2}, which implies in particular that $\th_c<1/3$ when $B=3.$ 

Moreover, in Subsection \ref{thlB}, the more delicate case where $B=3$ and $\th_c\ge 2/7$ does not appear anymore, because by Proposition \ref{compthl} we have $\th_c<2/7.$ 

Observe finally that if $m=2,\, B=3,4,5,$ $h^r+1\le B$ and $H\le H(B),$  then Corollary \ref{B45}, whose proof applies equally well  when $l\ge 2,$  shows that left endpoint of the  principal face of the Newton polyhedron of $\tpad$ is of the form $(A,B),$ where $A\le B-3.$ However, if $l\ge 2,$ this endpoint must lie on or below the bi-sectrix, which leads to a contraction. These cases therefore cannot arise when $l\ge2.$

We therefore obtain 

\begin{cor}\label{l>1}
The  estimates in Proposition \ref{Dl} hold true for every  $l\ge 2.$ 
 \end{cor}

\setcounter{equation}{0}
\section{The remaining cases where $m=2$ and $B=3$ or $B=4:$ Preliminaries }\label{m2b34}
We finally turn to the discussion of the remaining cases which are described by \eqref{5.17II} and \eqref{5.18II}.
In view of our definition of $B$ (cf. \eqref{2.2II}),  we see that in these cases $Q\y=c_0y_1^A,$ where $c_0\ne 0$ and $A\in\{0,\dots,B-3\}.$  Indeed,  this is obvious if $A=0,$ hence in particular if  $B=3,$ and if $B=4$ and $A=1,$ then  our assumption that $H\le H(4)<5$ implies that the Taylor support of $\tpad_{\tilde\kappa}$ cannot contain a point on the second coordinate axis. Let us assume without loss of generality that $c_0=1,$ so that by \eqref{2.2II}
\begin{equation}\label{7.1II}
\tpad_{\tilde\kappa}(x)=x_1^A x_2^B +c_1x_1^n,\qquad c_1\ne 0.
\end{equation}Then, by \eqref{phid}, \eqref{defr} and \eqref{bQ}, we may write
\begin{equation}\label{7.2II}
\phi_\de(x):= x_2^B  \,b(x_1,x_2,\de)+x_1^n\alpha(\de_1 x_1)+ r(x_1,x_2,\de),\quad (x_1\sim 1, |x_2|<\ve),
\end{equation}

where $ b(x_1,x_2,0)=x_1^A\sim 1, \al(0)\ne 0$ and 
$$
r(x_1,x_2,\de)=\sum_{j=1}^{B-1}x_2^jx_1^{n_j}\de_{j+2}\al_j(\de_1x_1).
$$
 Moreover, either $\al_j(0)\ne 0,$ and then $n_j$  is fixed (the type of the finite type function $b_j$), or $\al_j(0)=0,$ and then we may assume that $n_j$ is as large as we please.
 
 \medskip Observe also that if $A=0$ (this is necessarily so if $B=3$ ), then we have $Q(x)\equiv 1$ in \eqref{2.2II}, so that $\tilde\ka_2=1/B,$ and consequently 
\begin{equation}\label{bde}
b(x_1,x_2,\de)=b_B(\de_1 x_1, \de_2 x_2)
\end{equation}
 in \eqref{phid}.
 
 \medskip
 Recall also from Corollary \ref{B45} that  here $\tpad$ is  even smooth, not only fractionally smooth, and that  $a$ and $n$ are integers.
 
 \subsection{The fine structure of the phase $\phi_\de$}
 We shall need to derive more specific information on the phase $\phi_\de,$ and begin with the ``error term'' 
 $r(x_1,x_2,\de).$

 \medskip

 \begin{cor}\label{rem7.1II}
 By  some slight change of coordinates, we may even assume that the term with index $j=B-1$ vanishes, i.e., that 
 \begin{equation}\label{7.3II}
r(x_1,x_2,\de)=\sum_{j=1}^{B-2}x_2^jx_1^{n_j}\de_{j+2}\al_j(\de_1x_1).
\end{equation}
In particular, in view of \eqref{defde}, we may assume that 
$\de=(\de_0,\,\de_1,\,\de_2,\,\de_3,\dots,\de_{B})$ is given by 
\begin{equation}\label{defde2}
\de:=(2^{-k(\tilde\ka_2-2\tilde\ka_1)},\,2^{-k\tilde\ka_1},\,2^{-k\tilde\ka_2},\,2^{-k(n_1\tilde\ka_1+\tilde\ka_2-1)},\dots,2^{-k(n_{B-2}\tilde\ka_1+(B-2)\tilde\ka_2-1)}).
\end{equation}
Moreover, the complete phase corresponding to $\phi_\de$ is given by 
$$
\Phi(x,\,\de,\,\xi):=\xi_1x_1+\xi_2(\de_0x_2+x_1^m\omega(\de_1x_1))+\xi_3\phi_\de\x.
$$
\end{cor}
 
 \begin{proof} 
Indeed,  going back to the domain $D^a=\{\y: 0<y_1<\ve, |y_2|<\varepsilon y_1^a \},$ 
 observe that  \eqref{7.1II} implies that $\pa_1^A\pa_2^{B-1}\tpad_{\tilde \kappa}(y_1,0) \equiv 0,$  whereas $\pa_1^A\pa_2^{B}\tpad_{\tilde \kappa}(y_1,0)\equiv A!B!\ne 0$ for $|y_1|<\ve.$ 
 Moreover, since $(A,B)$ is a vertex of $\N(\tpad),$ we also have $\pa_1^A\pa_2^{B}\tpad(0,0)=A!B!\ne 0,$ whereas $\pa_1^A\pa_2^{B-1}\tpad(0,0) =0.$
 
 Then the implicit function theorem  implies that, that for $\ve$ sufficiently small,
 there is a  smooth function $\rho(y_1),\  -\ve<y_1<\ve,$ such   $$
 \pa_1^A\pa_2^{B-1}\tpad(y_1,\rho(y_1)) \equiv 0.
 $$
 and comparing $\tilde\kappa$-principal parts, it is easy to see that the $\tilde\ka$-principal part of $\rho$ has $\tilde\ka$-degree strictly bigger than the degree of $y_1^a.$ Thus, if  we  perform  the  further change of coordinates $(z_1,z_2):=(y_1,y_2-\rho(y_1),$ in which $\phi$ is represented, say, by $\tilde\phi,$ it is easily seen that the Newton polyhedra of $\tpad$ and $\tilde\phi$ as well as  their $\tilde\kappa$-principal parts are the same    (cf. similar arguments in \cite{IM-ada}).  Replacing $\psi^{(2)}(y_1)$ by  $\psi^{(2)}(y_1)+\rho(y_1)$ and modifying $\om(y_1)$ accordingly,  we then find that in the corresponding  modified adapted coordinates $
 z_1=x_1,z_2=x_2-(\psi^{(2)}(x_1)+\rho(x_1)),$ the function $\tilde\phi$ satisfies
\begin{equation}\label{7.4II}
\pa_1^A\pa_2^{B-1}\tilde\phi(z_1,0) \equiv 0.
\end{equation}

 On the other hand,  like $\tpad,$ it must be of the form
 $$
 \tilde\phi(z_1,z_2)=z_2^B \,b_B(z_1,z_2)+z_1^n\al(z_1)+\sum_{j=1}^{B-1}z_2^j \, b_j(z_1),
 $$
 where $b_B(z_1,z_2)=z_1^A,$ except for terms of $\tka$-degree strictly bigger than $\tka_1 A,$
 and thus \eqref{7.4II} implies that $b_{B-1}^{(A)}(z_1)\equiv 0,$ i.e., 
 $b_{B-1}(z_1)=\sum_{k=0}^{A-1} a_k z_1^k.$
 
The corresponding terms $z_2^{B-1}a_k z_1^k$ in $z_2^{B-1}b_{B-1}(z_1)$ have $\tka$-degree 
$k\tka_1+(B-1)\tka_2\le(A-1)\tka_1+(B-1)\tka_2<A\tka_1+B\tka_2=1,$  and thus  must all vanish, since they should also have $\tka$-degree strictly bigger than $1,$ as required in \eqref{2.3II}. We thus find that
 $$
 \tilde\phi(z_1,z_2)=z_2^B \,b_B(z_1,z_2)+z_1^n\al(z_1)+\sum_{j=1}^{B-2}z_2^j \, b_j(z_1),
 $$
 which, after re-scaling, implies \eqref{7.3II}.
 
 \end{proof}

Recall also that we are interested in the frequency domains where $|\xi_j|\sim\la_j, \ j=1,2,3,$  assuming that  $\la_1\sim\la_2\sim\la_3$.

We shall assume for the sake of simplicity that
$$
\la_1=\la_2=\la_3\gg1.
$$
With some slight change of notation, compared to Section \ref{specloc}, we shall here put $\la:=\la_1=\la_2=\la_3,$ and write accordingly
$$
\widehat{\nu^\la_\de}(\xi):=\chi_1\Big(\frac{\xi_1}{\la}\Big)\chi_1\Big(\frac{\xi_2}{\la}\Big)\chi_1\Big(\frac{\xi_3}{\la}\Big)\; \int e^{-i\Phi(y,\de,\xi)}\, \eta(y) \, dy,
$$
i.e., 
$$
\nu^{\la}_\de(x)=\la^3\int \check\chi_1\Big({\la}(x_1-y_1)\Big) \, \check\chi_1\Big({\la}(x_2-\de_0 y_2-y_1^2\omega(\de_1y_1)\Big)
\check\chi_1\Big({\la_3}(x_3-\phi_\de(y)\Big) \,\eta(y)\,dy.
$$

 Recall also that 
$\supp \eta\subset\{x_1\sim1,\,|x_2|<\ve  \}, \quad (\ve \ll 1).$

As before, we change coordinates from $\xi=(\xi_1,\,\xi_2,\,\xi_3)$ to $s_1,\,s_2$ and $s_3:=\xi_3/\la,$ by writing 
$\xi=\xi(s,\la),$ with 
$$
\xi_1=s_1\xi_3=\la s_1s_3,\quad \xi_2=s_2\xi_3=\la s_2s_3,\quad \xi_3=\la s_3.
$$
Accordingly, we write 
\begin{equation}\label{7.6II}
\Phi(y,\,\de,\,\xi)=\la s_3\Big(\Phi_1(y_1,\de_1,s)+s_2\de_0 y_2+y_2^B  \,b(y_1,y_2,\de)+ r(y_1,y_2,\de)\Big),
\end{equation}
where
$$
\Phi_1(y_1,\de_1,s):=s_1y_1+s_2y_1^2\omega(\de_1y_1)+y_1^n\alpha(\de_1 y_1).
$$
Notice that here $|s_j|\sim 1, \, j=1,2,3,$  and we have put $s:=(s_1,s_2).$ By passing from $s_j$ to $-s_j,$ if necessary, we may and shall in the sequel always assumes that 
$$s_j\sim 1,\qquad j=1,2,3
$$
(notice that these changes of signs may cause a change of sign of the $x_j$ and $\om,$ respectively of $\Phi$).

\medskip
Let us fix $s^0$ in this domain, and consider the phase function $\Phi_1(x_1,0,s^0)$  for $\de_1=0.$ In case that this phase has at worst non-degenerate critical points $x_1^c\sim 1,$  then the same is true for sufficiently small $\de_1,$  and the estimate for $\widehat{\nu_\de^\la}$ in Section \ref{equalla} can be improved to 
$$
\|\widehat{\nu_\de^\la}\|_\infty \lesssim  \la^{-\frac 12-\frac 1B},
$$
and thus in Case A we obtain the better estimate
$$
\|T_\de^\la\|_{p_c\to p'_c}\lesssim \la^{-\frac 12-\frac 1B+\frac 52 \th_c},
$$
compared to \eqref{4.2II} (of course, $T_\de^\la$ stands here for $T_\de^{(\la,\la,\la)}$). Moreover, by means of \eqref{5.17II}, \eqref{5.18II}, one checks easily that the exponent of $\la$ in this estimate is strictly negative, if $B=4,$ and zero, if $B=3.$ Thus we can sum these estimate over all dyadic $\la\gg 1$ if $B=4,$ and obtain at least a uniform estimate when $B=3.$ It is easy to see that this case can then still be treated by means of the real interpolation Proposition \ref{bsint}, since the relevant frequencies will here be  restricted essentially to cuboids in the $\xi$-space.

\medskip

In Case B,  where $\la> \de_0^{-B/(B-1)},$ we obtain the better estimate
$$
\|T_\de^\la\|_{p_c\to p'_c}\lesssim \la^{-\frac 12-\frac 1B+(\frac 32+\frac 1B) \th_c}\, \de_0^{-\th_c},
$$ 
compared to \eqref{4.3II}. The exponent of $\la$ is  strictly negative (compare the discussion leading to \eqref{4.12II}), so summing over all  $\la> \de_0^{-B/(B-1)},$ we find that 
$$
\sum_{\la> \de_0^{-B/(B-1)}}\|T_\de^\la\|_{p_c\to p'_c}\lesssim \de_0^{\frac{B+2-5B\th_c}{2B-2}}.
$$
And, since $\th_c\le \tilde\th_B,$ on easily checks that the exponent of $\de_0$ in this estimate is non-negative if $B\ge 3.$ 

\medskip

We  are thus left with the more subtle  situation where  the phase function $\Phi_1(x_1,0,s^0)$ has a degenerate critical point $x_1^c$  of Airy type. Denoting by $\Phi_1', \Phi_1^{''}$ etc. derivatives with respect to $x_1,$  and arguing as in Section 6 of \cite{IM-rest1}, we see by the implicit function theorem that for $s$ sufficiently close to $s^0$ and $\de$ sufficiently small, there is a  unique, non-degenerate critical point   $x_1^c=x_1^c(\de_1,s_2)\sim 1$  of 
$\Phi_1',$ i.e., 
$$
\Phi_1^{''}(x_1^c(\de_1,s_2), \de_1,s)=0, \quad |s-s^0|<\ve, |\de|<\ve,
$$
if $\ve$ is sufficiently small.  We then shift this critical point to the origin,  by putting
\begin{equation}\label{7.7II}
\Phi^\sharp(x,\de,\xi):=\frac 1{s_3\la}\Phi(x_1^c(\de_1,s_2)+x_1,x_2,\de,\xi), \quad |\x|\ll 1
\end{equation}
(notice that we may indeed assume that $|\x|\ll 1,$ since away from $x_1^c,$ we have at worst  non-degenerate 
 critical points, and the previous argument applies).
 
From Lemma 6.2 in \cite{IM-rest1} (with $\be:=\al, \si=1, \de_3=0$ and $b_0=0$), we then immediately get, after scaling in $x_1$ so that we may assume that
\begin{equation}\label{normalization1}
-\frac{2\om(0)}{n(n-1)\al(0)}=1,
\end{equation}
the following result:

\begin{lemma}\label{s}
$\Phi^\sharp$ is of the form
\begin{equation}\label{Phis}
\Phi^\sharp(x,\de,\xi)=x_1^3\, B_3(s_2,\de_1,x_1) -x_1\, B_1(s,\de_1) +  B_0(s,\de_1)+\phi^\sharp(x,\de, s_2),
\end{equation}

with
\begin{eqnarray}\nonumber
\phi^\sharp(x,\de,\,s_2)&:=&x_2^B\, b(x,\de,s_2)+\sum_{j=2}^{B-2}\de_{j+2}x_2^j\,\tilde\alpha_j(x_1,\de,\,s_2)\\
&+&x_2\, \Big(s_2\de_0 +\de_3(x_1^c(\de_1,s_2)+x_1)^{n_1} \al_1(\de_1(x_1^c(\de_1,s_2)+x_1))\Big) ,
\label{phases}
\end{eqnarray}
and where the following hold true:

\medskip
The functions $b$ and $\tilde\al_j$ are smooth, 
$b(x,\de,s_2)\sim 1,$ and also 
 $|\tilde\alpha_j|\sim1,$ unless $\alpha_j$  is a flat function. Moreover,
$B_0,B_1$ and $B_3$ are smooth functions, and $$
B_3(s_2,\de_1,0)=s_2^{\frac{n-3}{n-2}}G_4(\de_1s_2^{\frac 1{n-2}}),
$$
where 
$$
G_4(0)=\tfrac{n(n-1)(n-2)}{6}\al(0).
$$
Furthermore, we may write
\begin{equation}\label{ai2}
\begin{cases} \quad  x_1^c(\de_1,s_2)&=s_2^{\frac 1{n-2}}G_1(\de_1s_2^{\frac 1{n-2}}),\\
\quad B_0(s,\de_1)&=s_1s_2^{\frac 1{n-2}}G_1(\de_1s_2^{\frac 1{n-2}}) - s_2^{\frac n{n-2}} G_2(\de_1s_2^{\frac 1{n-2}}),\\
 \quad B_1(s,\de_1)&= -s_1+s_2^{\frac {n-1}{n-2}}G_3(\de_1s_2^{\frac 1{n-2}}),
\end{cases}
\end{equation}
where
\begin{equation}\label{ai3}
\begin{cases}  
\quad G_1(0)&=1,\\
\quad G_2(0)&= \frac{n^2-n-2}2 \al(0),\\
\quad G_3(0)&= n(n-2)\al(0).
\end{cases}
\end{equation}
Notice that all the numbers in \eqref{ai3} are non-zero, since we assume $n>2B>5.$  Finally,  if we also write $G_5:=G_1G_3-G_2,$ then we have 
\begin{equation}\label{gnz}
G_3(0)\ne 0, \quad  G_5(0)\ne 0.
\end{equation}
\end{lemma}

Observe  that we obtain here a more specific dependency of $x_1^c, B_0, B_1$ and $B_3$ on $\de_1$ and $s_2$ than in \cite{IM-rest1}, due to the fact that the equation for the critical point depends only on the parameter $\de_1^{n-2}s_2$ in the coordinate $y_1:=\de_1x_1.$ 

\medskip
Nevertheless, with a slight abuse of notation, we shall frequently also use the short-hand notation 
$G_j(s_2,\de)$ in place of 
$G_j(\de_1s_2^{\frac 1{n-2}}), \ j=1,\dots,4$
\medskip

We also remark  that the part of the measure $\nu_\de$ corresponding to the region described in \eqref{7.7II}, which 
 we shall simply again denote by $\nu_\de,$   is given by an  expression for its Fourier transform of the form
 \begin{equation}\label{7.12II}
\widehat{\nu_\de}(\xi)= \int e^{-is_3\la\Phi^\sharp(x,\de,\xi)} a(x,\de,s)\, dx \qquad (\mbox{with }\ \xi=\xi(s,\la)),
\end{equation}
where $a$ is a smooth function with compact support in $x$ such that $|x|\le \ve$ on $\supp a.$

\begin{remark}\label{fracpower}
It will be important in the sequel to observe that every single $\de_j$ is a fractional power of $2^{-k},$ hence also of $\de_0$, i.e., there is some positive rational number  $\r>0$ such that $\de_j=\de_0^{q_j\r},$  with positive integers $q_j$  (cf. \eqref{defde2}).
\end{remark}

In the sequel, we shall need more precise information on the structure of the last term of $\phi^\sharp$ in \eqref{phases}:

\begin{lemma}\label{lastterm}
Let 
$$
a_1(x_1,\de,s_2):=s_2\de_0 +\de_3(x_1^c(\de_1,s_2)+x_1)^{n_1} \al_1(\de_1(x_1^c(\de_1,s_2)+x_1))
$$
be the coefficient of $x_2$ in the last term of $\phi^\sharp$ in \eqref{phases}. Then $a_1$ can be re-written in the form
\begin{equation}\label{7.21II}
a_1(x_1,\de,s_2)=\de_{3,0}\,\tilde\al_1(x_1,\de_0^{\r},s_2)+\de_0\,x_1\al_{1,1}(x_1,\de_0^{\r},s_2),
\end{equation}
with smooth functions $\tilde\al_1$ and $\al_{1,1},$  and where $\de_{3,0}$ is of the form $\de_{3,0}=\de_0^{q_{3,0}\r},$ with some positive integer $q_{3,0}.$  Moreover, two possible cases may arrive:
\medskip

{ \bf Case ND:  The non-degenerate case:}   $\al_{1,1}\equiv 0,\ |\tilde\al_1|\sim 1$ and  $\de_{3,0}=\max\{\de_0,\de_3\}\ge \de_0,$ 

or

{ \bf Case D: The degenerate case:}  $|\al_{1,1}|\sim 1,\, \tilde\al_1=\tilde\al_1(\de_0^{\r}, s_2)$ is independent of $x_1,$  $\de_0=\de_3$  and  $\de_{3,0}\ll \de_0$. Moreover,  either $|\tilde\al_1|\sim 1, $ or we  can choose  $q_{3,0}\in \NN$ as large as we wish. \medskip

 In particular, we may write 
 \begin{eqnarray}\nonumber
\phi^\sharp(x,\de,\,s)&=&x_2^B \,b(x,\de_0^{\r},s_2)+\sum_{j=2}^{B-2}\de_{j+2}x_2^j\, \tilde\alpha_j(x_1,\de_0^{\r},\,s_2)\\
&+&\de_{3,0} \,x_2 \,\tilde\al_1(x_1,\de_0^{\r},s_2)+\de_0\, x_1x_2\,\al_{1,1}(x_1,\de_0^{\r},s_2),\label{7.13II}
\end{eqnarray}
with smooth functions $\tilde a_j$  and $b,$ where $|b|\sim 1$ and $\tilde\al_1$ and $\al_{1,1}$ are as in the Cases D, respectively ND.  Moreover, in both cases we have
$$
\max\{\de_0,\de_3\}=\max\{\de_0,\de_{3,0}\}
$$
\end{lemma}

\begin{proof} 
Recall that $\de_0= 2^{-k(\tilde\ka_2-2\tilde\ka_1)}$ and $\de_3=2^{-k(n_1\tilde\ka_1+\tilde\ka_2-1)}.$  We therefore distinguish two cases:
\medskip

{\bf 1. Case: $n_1\ne n-2.$}  Then $\tilde\ka_2-2\tilde\ka_1\ne n_1\tilde\ka_1+\tilde\ka_2-1,$  since $\tilde\ka_1=1/n,$ and thus either $\de_0\gg \de_3,$ or $\de_0\ll\de_3,$ for $k$ sufficiently large. Notice that we may assume this to be true in particular if the function $ \al_1$ is flat, since we  may then choose $n_1$ as large as we want. By putting 
$$
\de_{3,0}:=\max\{\de_0,\de_3\},
$$
we then  clearly may write $a_1$ as in  Case ND.

\medskip

{\bf 2. Case: $n_1= n-2.$} Then $\de_0=\de_3,$   and  we may assume that  $|\al_1|\sim 1.$  Thus, expanding around $x_1=0$ and applying  \eqref{ai2}, we see that  we may write
\begin{eqnarray*}
a_1(x_1,\de,s_2)&=&\de_0 \Big(s_2 +x_1^c(\de_1,s_2)^{n_1} \al_1(\de_1(x_1^c(\de_1,s_2)) +x_1\al_{1,1}(x_1,\de_0^{\r},s_2)\Big)\\
&=& \de_0 s_2\Big(1+G_1(\de_1 s_2^{\frac 1{n-2}})\, \al_1(\de_1 s_2^{\frac 1{n-2}}G_1(\de_1 s_2^{\frac 1{n-2}}))\Big)+\de_0 x_1\al_{1,1}(x_1,\de_0^{\r},s_2)\\
&=& \de_0 s_2\Big(1+g(\de_1 s_2^{\frac 1{n-2}})\Big)+\de_0 x_1\al_{1,1}(x_1,\de_0^{\r},s_2),
\end{eqnarray*}
with smooth functions $1+g$ and $\al_{1,1},$ where $|\al_{1,1}|\sim 1.$ By means of a Taylor expansion of $g$ around the origin we thus find that 
\begin{eqnarray*}
a_1(x_1,\de,s_2)&=&\de_0 s_2(\de_1 s_2^{\frac 1{n-2}})^{N} g_N(\de_1 s_2^{\frac 1{n-2}})+\de_0 x_1\al_{1,1}(x_1,\de_0^{\r},s_2)\\
&=&(\de_0\de_1^N)\tilde\al_1(\de_0^{\r},s_2)+\de_0 x_1\al_{1,1}(x_1,\de_0^{\r},s_2),
\end{eqnarray*}
with $N\in\NN$ and $g_N$ smooth. Moreover,  we may either assume that $|\tilde\al_1|\sim 1$ (if $1+g$ is a finite type $N$ at the origin), or that we may choose $N$ as large as we wish (if $1+g$ is flat).  Notice that if $N=0,$ then we can include the second term into the first term and arrive again at Case ND. In all other cases we arrive at  the situation described by Case D, where the second term cannot be included into the first term.
Notice that then $ \de_{3,0}:=\de_0\de_1^N\ll \de_0.$
\end{proof}

Let us next  introduce the quantity
\begin{equation}\label{rho}
\rho:=\begin{cases}   \de_{3,0}^{\frac B{B-1}}+\sum_{j=2}^{B-2}\de_{j+2}^\frac{B}{B-j}&   \mbox{in Case ND},\\
 \de_0^{\frac{3B}{2B-3}}+\de_{3,0}^{\frac B{B-1}}+\sum_{j=2}^{B-2}\de_{j+2}^\frac{B}{B-j} &       \mbox{in Case D},
\end{cases}
\end{equation}
which we shall view as a function $\rho(\tilde\de)$ of the coefficients 
$$
\tilde\de:=\begin{cases}   (\de_{3,0}, \de_4,\dots,\de_B) &   \mbox{in Case ND},\\
 (\de_0,\de_{3,0}, \de_4,\dots,\de_B) &       \mbox{in Case D}.
\end{cases}
$$

\begin{remark}\label{duist}
Observe that if we scale  the complete phase $\Phi^\sharp$ from  \eqref{Phis} in $x_1$ by the factor $r^{-1/3}$ and in $x_2$ by $r^{-1/B}, \ r>0,$ and multiply by $r,$ i.e., if we look at 
$$ 
\Phi^\sharp_r(u_1,u_2,\de,s):=r\Phi^\sharp(r^{-1/3} u_1,r^{-1/B}u_2,\de,s),
$$ 
 then the  effect is essentially that $\tilde\de$ is replaced by 
\begin{equation}\label{deldil}
\tilde \de^r:=\begin{cases}   (r^{(B-1)/B} \de_{3,0}, \dots, r^{(B-j)/B}\de_{j+2},\dots, r^{2/B}\de_B) &   \mbox{in Case ND},\\
 (r^{\frac {2B-3}{3B}}\de_0,r^{(B-1)/B} \de_{3,0}, \dots, r^{(B-j)/B}\de_{j+2},\dots, r^{2/B}\de_B) &       \mbox{in Case D},
\end{cases}
\end{equation}
whereas $B_1(s,\de_1)$ is replaced by  $r^{2/3}B_1(s_2,\de_1)$ and $B_0(s,\de_1)$ by $rB_0(s,\de_1).$ More precisely, if denote by  $\si_r$ the dilations $\si_r(x_1,x_2):= (r^{-\frac 13} x_1,r^{-\frac 1B} x_2),$ then  we have 

\begin{equation}\label{Phir}
\Phi^\sharp_r(u_1,u_2,\de,s)=u_1^3\, B_3(s_2,\de_1,r^{-\frac 13}u_1) -u_1\, r^{\frac 23} B_1(s,\de_1) +  rB_0(s,\de_1)+\phi^\sharp_r(u_1,u_2,\tilde\de^r, s_2),   
\end{equation}
where
\begin{eqnarray}\nonumber
\phi^\sharp_r(u,\tilde\de^r, s)&:=&u_2^B \,b(\si_{r^{-1}}u,\de_0^{\r},s_2)+\sum_{j=2}^{B-2}\tilde\de^r_{j+2}u_2^j\, \tilde\alpha_j(r^{-\frac 13}u_1,\de_0^{\r},\,s_2)\\
&+&\tilde\de^r_{3,0} \,u_2 \,\tilde\al_1(r^{-\frac 13}u_1,\de_0^{\r},s_2)+\tilde\de^r_0\, u_1u_2\,\al_{1,1}(r^{-\frac 13}u_1,\de_0^{\r},s_2),\label{phir}
\end{eqnarray}

And, under these  dilations, $\rho$ is homogeneous of  degree $1,$ i.e., 
 \begin{equation}\label{rhoinv}
\rho(\tilde\de^r)=r  \rho(\tilde\de), \quad r>0.
\end{equation}
  In particular, after scaling $\Phi^\sharp$ in this way by $r:=1/\rho(\tilde\de),$ we see that we have normalized the coefficients of $\phi^\sharp$ in such a way that $\rho(\tilde\de)=1.$ 
   \end{remark}
 This observation, which is based on ideas by  Duistermaat \cite{duistermaat}, will become  important in the sequel.

\medskip

\subsection{The case where $\la\rho(\tilde\de)\lesssim 1$}\label{larhos}

Assume now that $B\in\{3,4\}.$ Following the proof of Proposition 5.2 (c) in  \cite{IM-rest1},  we define the functions $\nu_{\de,\, Ai}^\la$ and $\nu_{\de,\, l}^\la$ by
\begin{equation}\label{aismall}
\widehat{\nu_{\de,\, Ai}^\la}(\xi):=\chi_0(\la^\frac23B_1(s,\,\de_1))\widehat{\nu_\de^\la}(\xi),
\end{equation}
\begin{equation}\label{lsmall}
\widehat{\nu_{\de,\, l}^\la}(\xi):=\chi_1((2^{-l}\la)^\frac23B_1(s,\,\de_1))\widehat{\nu_\de^\la}(\xi),\quad M_0\le 2^l\le \frac{\la}{M_1},
\end{equation}
so that
$$
\nu_\de^\la=\nu_{\de,\, Ai}^\la+\sum_{M_0\le 2^l\le \frac{\la}{M_1}}\, \nu_{\de,\, l}^\la.
$$
Here, $\chi_0, \chi_1\in C^\infty_0(\RR),$  and $\chi_1(t)$ is supported where $2^{-4/3}\le |t|\le 2^{4/3},$ whereas $\chi_0(t)\equiv 1$ for $|t|\le M_0^{2/3}.$  Thus, by choosing $M_0$ sufficiently large,  we may assume that $2^{-l}\le 1/M_0\ll 1.$ Denote by $T_{\de,\, Ai}^\la$ and $T_{\de,\, l}^\la$ the corresponding operators  of  convolution with the Fourier transforms of these functions.

\medskip
Our goal will be to adjust the proofs of the estimates in Lemma 6.5 and Lemma 6.6 of \cite{IM-rest1} to our present situation in order to derive the following estimates, which are analogous to the corresponding ones in \cite{IM-rest1} (formally, we  only have to replace a factor $\la^{-1/2}$ in the estimates in \cite{IM-rest1} by  the factor $\la^{-1/B}$):

\begin{equation}\label{analog11}
\|\widehat{\nu_{\de,\, Ai}^\la}\|_\infty \le C_1 \la^{-\frac1B-\frac13}
\end{equation}
\begin{equation}\label{analog12}
\|\nu_{\de,\, Ai}^\la\|_\infty \le C_2 \la^{\frac53-\frac1{B}},
\end{equation}
as well as 
\begin{equation}\label{analog21}
\|\widehat{\nu_{\de,\, l}^\la}\|_\infty \le C_1 2^{-\frac{l}6}\la^{-\frac1B-\frac13}
\end{equation}
\begin{equation}\label{analog22}
\|\nu_{\de,\, l}^\la\|_\infty \le C_2 2^{\frac l3}\la^{\frac53-\frac1{B}}
\end{equation}

\medskip
\subsubsection{Estimates for $\nu_{\de,\, Ai}^\la$}

Changing coordinates from $x$ to $u$ by putting 
$x=\si_{1/\la} u=(\la^{-1/3}u_1,\la^{-1/B}u_2)$ in the integral   \eqref{7.12II}, and making use of Remark \ref{duist} (with $r:=\la$),  we find that 
\begin{eqnarray}\nonumber
&&\widehat{\nu_{\de,Ai}^\la}(\xi)=\la^{-\frac 1B-\frac 13} \chi_1(s,s_3) \,\chi_0(\la^\frac23 B_1(s,\,\de_1))  \,e^{-is_3 \la B_0(s,\de_1)}\\
&&\times\iint e^{-i s_3\Big(u_1^3\, B_3(s_2,\de_1,\la^{-\frac 13}u_1) -u_1  
\la^{\frac 23}B_1(s,\de_1)+\phi^\sharp_\la(u_1,u_2,\, \tilde\de^\la, s_2) \Big)} 
 a(\si_{\la^{-1}}u, \de,s) \,du_1du_2, \label{naihat}
 \end{eqnarray}
where $\chi_1(s,s_3):= \chi_1(s_1s_3)\chi_1(s_2s_3)\chi_1(s_3)$ localizes to the region where $s_j\sim 1, j=1,2,3.$ Observe that  here we are  integrating over the large domain where $|u_1|\le \ve \la^{1/3}$ and $|u_2|\le \ve \la^{1/B}.$ Recall also that $\phi^\sharp_\la$ is given by \eqref{phir}, and that $\rho(\tilde\de^\la)=\la \rho(\tilde\de)\lesssim 1,$ and so  we have 
$$
|\tilde\de^\la|\lesssim 1\quad \mbox{and}\quad \la^\frac23|B_1(s,\,\de_1)|\lesssim 1.
$$
By means of this  integral formula for $\widehat{\nu_{\de,Ai}^\la}(\xi),$  we easily obtain

\begin{lemma}\label{duist2}
If $\la\rho(\tilde\de)\lesssim 1,$ then we may write
\begin{eqnarray}\label{hatnuai} &&\widehat{\nu_{\de,Ai}^\la}(\xi)=\la^{-\frac 1B-\frac 13}  \chi_1(s,s_3) 
\,\chi_0(\la^\frac23 B_1(s,\,\de_1)) \,e^{-is_3 \la B_0(s,\de_1)} \\
 &&\hskip2cm \times a(\la^\frac23 B_1(s,\,\de_1),\, \tilde\de^\la, s,s_3,\,\de_0^{\r},\la^{-\frac 1{3B}}), \nonumber
 \end{eqnarray}
 where $a$ is  again a smooth function of all its (bounded) variables.
\end{lemma}

\begin{proof} 
We decompose the integral in \eqref{naihat}  by means of suitable smooth cut-off functions into the integral $I_1$ over the region where $|(u_1,u_2)|\le L,$  the integral $I_2$ over the region where $|(u_1,u_2)|> L$  and $ |u_2|^{B-1}\gg |u_1|,$ and  the integral $I_3$ over the region where $|(u_1,u_2)|> L$  and $ |u_2|^{B-1}\lesssim |u_1|.$  For each of these contributions $I_j,$ we then  show that it is of the form $a_j(\la^\frac23 B_1(s,\,\de_1),\, \tilde\de^\la, s,s_3,\,\de_0^{\r},\la^{-\frac 1{3B}}),$ with a suitable smooth function $a_j,$ provided   $L$ is sufficiently large. 
For $I_1,$ this claim is obvious. 

\smallskip
On the  remaining region where  $|(u_1,u_2)|\ge L,$  we may use iterated integrations by parts with respect to $u_1,$ or  $u_2,$  in order to convert the integral into an absolutely convergent integral, to which we may apply the standard rules  for differentiation with respect to parameters  (such as $s_j, $ etc.).  Denote to this end by $\Phi_c$ the complete phase function appearing in this integral. It is then easily seen that we may estimate
\begin{eqnarray}\label{padu2}
|\pa_{u_2} \Phi_c|&\gtrsim& |u_2|^{B-1}- c|u_1|,\\ 
|\pa_{u_1} \Phi_c|&\gtrsim& u_1^2-c\la^{-\frac 13} |u_2|^B-c |u_2| \label{padu1},
\end{eqnarray}
with a fixed constant $c>0.$ The last terms appear only in the degenerate case D, due to the presence of the term $\tilde\de^\la_0\, u_1u_2\,\al_{1,1}(\la^{-\frac 13}u_1,\de_0^{\r},s_2)$ in $\phi^\sharp_\la.$  

\smallskip
Denote by $I_2$ the contribution by the sub-region on which $ |u_2|^{B-1}\gg |u_1|.$  On this region, we may gain factors of order $|u_2|^{-2N(B-1)}$  (for any $N\in\NN$) in the amplitude, by means of iterated integrations by parts in $u_2.$ And, since $|u_2|^{-2N(B-1)}\lesssim |u_2|^{-N}|u_2|^{-N(B-1)},$ wee that we arrive at an absolutely convergent integral. 

\smallskip
Similarly, denote by $I_3$ the contribution by the sub-region on which $ |u_2|^{B-1}\lesssim |u_1|.$ Observe that 
$|u_2|^B\lesssim |u_1|^{B/(B-1)}\ll u_1^2,$ since $B\ge 3,$  and since we may assume that $|u_1|\gg 1.$ This shows that in this region, we have $|\pa_{u_1} \Phi_c|\gtrsim u_1^2,$ and thus we may gain factors of order $|u_1|^{-2N}$  (for any $N\in\NN$) in the amplitude, by means of iterated integrations by parts in $u_1.$ And, since $|u_1|^{-2N}\lesssim |u_1|^{-N}|u_2|^{-N(B-1)},$ we see that we arrive again at an absolutely convergent integral.

\end{proof}

Lemma \ref{duist} implies in particular  estimate  \eqref{analog11}.   As for the more involved estimate \eqref{analog12}, with Lemma \ref{duist2} at hand, we can basically follow the arguments from Subsection 6.1 in \cite{IM-rest1}, only with the factor $\la^{-1/2}$ appearing there replaced by the factor $\la^{-1/B}$  here, and with the amplitude $g(\la^{2/3}B_1(s',\de,\si),\la,\de, \si,s)$ in \cite{IM-rest1} replaced by our amplitude $a(\la^\frac23 B_1(s,\,\de_1),\, \tilde\de^\la, s,s_3,\,\de_0^{\r},\la^{-\frac 1{3B}})$ here (compare with (6.18) in \cite{IM-rest1}). 

 \medskip
 \subsubsection{Estimates for $\nu_{\de,l}^\la$}\label{nudell}
 Changing coordinates from $x$ to $u$ by putting here
$x=\si_{2^l/\la} u$ in the integral   \eqref{7.12II}, and making use of Remark \ref{duist} (with $r:=\la/2^l$),  we find that 
\begin{eqnarray}\nonumber
&&\widehat{\nu_{\de,l}^\la}(\xi)=(2^{-l}\la)^{-\frac 1B-\frac 13} \chi_1(s,s_3) \,\chi_1((2^{-l}\la)^\frac23 B_1(s,\,\de_1))  \,e^{-is_3 \la B_0(s,\de_1)}\\
&&\hskip1cm\times\iint e^{-i s_3 2^l\Phi(u_1,u_2,s, \de,\la, l)} 
a(\si_{2^l\la^{-1}}u, \de,s) \,du_1du_2,   \label{hatnul}
 \end{eqnarray}
 with phase function 
 \begin{eqnarray}\nonumber
&& \Phi(u_1,u_2,s, \de,\la, l):= \\
 &&\quad u_1^3\, B_3(s_2,\de_1, (2^{-l}\la)^{-\frac 13}u_1) -u_1  
(2^{-l}\la)^{\frac 23}B_1(s,\de_1)+\phi^\sharp_{2^{-l}\la}(u_1,u_2,\, \tilde\de^{2^{-l}\la}, s_2).\label{Philla}
 \end{eqnarray}
 Observe that   we are here  integrating over the large domain where $|u_1|\le (\ve 2^{-l}\la)^{1/3}$ and $|u_2|\le \ve (2^{-l}\la)^{1/B}.$ Recall also that $\phi^\sharp_{2^{-l}\la}$ is given by \eqref{phir}, and that $\rho(\tilde\de^{2^{-l}\la})=2^{-l}\la \rho(\tilde\de)\lesssim 2^{-l},$  so that we have 
$$
|\tilde\de^{2^{-l}\la}|\ll 1\quad \mbox{and}\quad (2^{-l}\la)^\frac23|B_1(s,\,\de_1)| \sim 1.
$$
Notice that this implies that 
\begin{eqnarray}\nonumber
&&\phi^\sharp_{2^{-l}\la}(u_1,u_2,\, \tilde\de^{2^{-l}\la}, s_2) =u_2^B \,b(\si_{(2^{-l}\la)^{-1}}u,\de_0^{\r},s_2)
+\tilde\de_0\, u_1u_2\,\al_{1,1}((2^{-l}\la)^{-\frac 13}u_1,\de_0^{\r},s_2)\\
&&\hskip5cm +\mbox{ small error}. \label{philla}
\end{eqnarray}
where we have abbreviated $\tilde\de_0:=\tilde\de^{2^{-l}\la}_0.$ Recall also that the second term appears only in the case D and that we are here assuming that $2^{-l}\la\gg 1.$  

\medskip
Arguing somewhat like in Section 7 of \cite{IM-rest1}, we first decompose 
$$
\nu_{\de,l}^\la=\nu_{l,0}^\la+\nu_{l,\infty}^\la,
$$
where 
\begin{eqnarray*}
&&\widehat{\nu_{l,0}^\la}(\xi):=(2^{-l}\la)^{-\frac 1B-\frac 13} \chi_1(s,s_3) \,\chi_1((2^{-l}\la)^\frac23 B_1(s,\,\de_1))  \,e^{-is_3 \la B_0(s,\de_1)}\\
&&\hskip1cm\times\iint e^{-i s_3 2^l\Phi(u_1,u_2,s, \de,\la, l)} 
a(\si_{2^l\la^{-1}}u, \de,s) \, \chi_0(u) \,du_1du_2,  \\
&&\widehat{\nu_{l,\infty}^\la}(\xi):=(2^{-l}\la)^{-\frac 1B-\frac 13} \chi_1(s,s_3) \,\chi_1((2^{-l}\la)^\frac23 B_1(s,\,\de_1))  \,e^{-is_3 \la B_0(s,\de_1)}\\
&&\hskip1cm\times\iint e^{-i s_3 2^l\Phi(u_1,u_2,s, \de,\la, l)} 
a(\si_{2^l\la^{-1}}u, \de,s) \, (1-\chi_0(u)) \,du_1du_2. 
\end{eqnarray*}
Here, we choose $\chi_0\in C_0^\infty(\RR^2)$ such that  $\chi_0(u)\equiv 1$ for $|u|\le L,$ where $L$ is supposed to be a sufficiently large positive constant.

\medskip
Let us first  consider the contribution given by the $\nu_{l,\infty}^\la:$ Arguing as in the proof of Lemma \ref{duist2}, we can easily see by means of integrations by parts that, given $k\in\NN,$   then for every $N\in\NN,$ we may write
 
 \begin{eqnarray}\label{hatnulinf} &&\widehat{\nu_{l,\infty}^\la}(\xi)=2^{-lN}\, (2^{-l}\la)^{-\frac 1B-\frac 13}  \chi_1(s,s_3) 
\,\chi_1((2^{-l}\la)^\frac23 B_1(s,\,\de_1))\,e^{-is_3 \la B_0(s,\de_1)} \\
 &&\hskip2cm \times\,  a_{N,l}\Big((2^{-l}\la)^\frac23 B_1(s,\,\de_1), s,s_3,\,\tilde\de^{2^{-l}\la}, \de_0^{\r},(2^{-l}\la)^{-\frac13},\la^{-\frac 1{3B}}\big), \nonumber
 \end{eqnarray}
 where $a_{N,l}$ is  a smooth function of all its (bounded) variables such that $\|a_{N,l}\|_{C^k}$ is uniformly bounded in $l.$ 
 In particular, we see that
\begin{equation}\label{8.37n}
\|\widehat{\nu_{l,\,\infty}^\la}\|_\infty \lesssim  2^{-lN}\la^{-\frac 1B-\frac 13}  \qquad \forall N\in\NN.
\end{equation}
 
 Next, applying the Fourier inversion formula and changing coordinates from $s_1$ to 
 \begin{equation}\label{ctoz}
z:= (2^{-l}\la)^{\frac 23}B_1(s,\de_1), \quad \mbox{ i.e., }\quad s_1=s_2^{\frac {n-1}{n-2}}G_3(s_2,\de_1)-(2^{-l}\la)^{-\frac 23}z,
\end{equation}
as in \cite{IM-rest1}, we find that 
\begin{eqnarray}  \nonumber
\nu_{l,\infty}^\la(x)&=&\la^3 2^{-lN}\, (2^{-l}\la)^{-\frac 1B-1}\int e^{-is_3\la\Psi(z,s_2,\de)}  \chi_1\Big(s_2^{\frac {n-1}{n-2}}G_3(s_2,\de_1)-(2^{-l}\la)^{-\frac 23}z,s_2,s_3\Big)   \\
&&\hskip1cm \times \chi_1(z)  \,a_{N,l} \Big(z,  s_2,s_3,\,\de_0^{\r},\tilde\de^{2^{-l}\la},(2^{-l}\la)^{-\frac13}, \la^{-\frac 1{3B}}\Big) \, dz ds_2 ds_3 \label{8.38n},
\end{eqnarray}
where the phase function $\Psi$ is given by 
\begin{eqnarray}  \nonumber
\Psi(z,s_2,\de)&:=& s_2^{\frac {n}{n-2}}G_5(s_2,\de)- x_1 s_2^{\frac {n-1}{n-2}}G_3(s_2,\de)- s_2x_2-x_3\\
 &&+(2^l\la^{-1})^\frac23 z\, \Big(x_1-s_2^\frac1{n-2}G_1(s_2,\de)\Big)\label{8.39n}
\end{eqnarray}
(compare (7.4) in \cite{IM-rest1}).  Applying our van der Corput type lemma (of order $3$) to the integration in $s_2,$ which allows for the gain of a factor $\la^{-1/3},$ this easily implies that 
 \begin{equation}\label{8.40n}
\|\nu_{l,\,\infty}^\la\|_\infty \lesssim 2^{-lN} \la^{\frac 53-\frac 1B}  \qquad \forall N\in\NN.
\end{equation}

\medskip
Let us next  look at  the contribution given by the $\nu_{l,0}^\la:$  Observe first that on a region where $|u_1|$ is sufficiently small, iterated integrations  by parts with respect to $u_1$ allow to gain factors $2^{-lN}$ in the integral defining $\nu^\la_{l,0},$ for every $N\in\NN.$ Freezing afterwords $u_1,$ we can reduce to the integration in $u_2$ alone.  A similar argument applies whenever we are allowed to integrate by parts in $u_1$ (this is also the case when $B_1$ and $B_3$ have opposite signs).  We shall therefore  assume from now on that $B_1$ and $B_3$ have  the same sign. Moreover, let us assume without loss of generality that $u_1>0.$ Then  there is a unique non-degenerate  critical point $u^c_1=u^c_1(2^{-l}\la)^{-1/B} u_2, \tilde\de^{2^{-l}\la},...)$ of the phase $\Phi$ in \eqref{Philla}, of size $|u^c_1|\sim 1,$ and we may restrict the integration in $u_1$ to a small neighborhood of $u_1^c.$ I.e., we may replace the cut-off function $\chi_0(u_1)$  by a cut-off function $\chi_1(u_1)$  supported in a sufficiently small neighborhood of a point $u_1^0$ containing $u_1^c.$  Then  the method of stationary phase  shows that the corresponding integral in $u_1$ will be of order $2^{-l/2},$ so that this term will give the main contribution. 

For the sake of simplicity, we shall therefore restrict ourselves in the sequel  to the discussion of this main term  $\nu_{l,1}^\la,$  given by 
\begin{eqnarray}\nonumber
&&\widehat{\nu_{l,1}^\la}(\xi):=(2^{-l}\la)^{-\frac 1B-\frac 13} \chi_1(s,s_3) \,\chi_1((2^{-l}\la)^\frac23 B_1(s,\,\de_1))  \,e^{-is_3 \la B_0(s,\de_1)}\\
&&\hskip1cm\times\iint e^{-i s_3 2^l\Phi(u_1,u_2,s, \de,\la, l)} 
a(\si_{2^l\la^{-1}}u, \de,s) \, \chi_0(u) \chi_1(u_1) \,du_1du_2,   \label{hatnul1}
\end{eqnarray}
where  $|u_1|\sim 1$ on the support of $\chi_1(u_1).$ 

Applying first the method of stationary phase to the integration in $u_1,$ and subsequently van der Corput's estimate of order $B$ to the integration in $u_2,$ we  easily arrive at the estimate 
\begin{equation}\label{}
\|\widehat{\nu_{l,1}^\la}\|_\infty \lesssim  (2^{-l}\la)^{-\frac 1B-\frac 13} 2^{-\frac l2 -\frac lB},
\end{equation}
which is exactly what we need to verify \eqref{analog21} (recall here also estimate \eqref{8.37n}). 

\medskip
As for the more involved estimation of $\nu_{l,1}^\la(x),$ Fourier inversion allows to write 
\begin{eqnarray}\nonumber
&&\nu_{l,1}^\la(x)=\la^3\, (2^{-l}\la)^{-\frac 1B-\frac 13}  \int e^{-i s_3 \Psi(u,s,x, \de,\la, l)}
\chi_1(s,s_3) \,\chi_1((2^{-l}\la)^\frac23 B_1(s,\,\de_1)) \\
&&\hskip6cm\times a(\si_{2^l\la^{-1}}u, \de,s) \, \chi_0(u) \chi_1(u_1) \,du_1du_2 ds   \label{nul1x}
\end{eqnarray}
where the complete phase $\Psi$ is given by 
\begin{equation}\label{Psilla}
\Psi(u,s,x, \de,\la, l):=2^l \Phi(u_1,u_2,s, \de,\la, l) +\la \Big(B_0(s,\de_1)-s_1x_1-s_2x_2-x_3\Big),
\end{equation}
with $\Phi$ given by \eqref{Philla}. Changing again from the coordinate $s_1$ to $z$ as in \eqref{ctoz}, we may also write 
\begin{eqnarray}\nonumber
&&\nu_{l,1}^\la(x)=\la^3\, (2^{-l}\la)^{-\frac 1B-1} \int e^{-i s_3 \tilde\Psi(u,z,s_2,x, \de,\la, l)}
\chi_1(s,s_3) \,\chi_1(z) \\
&&\hskip3cm\times \tilde a(\si_{2^l\la^{-1}}u, (2^{l}\la^{-1})^{\frac23}z,s_2,\de) \, \chi_0(u) \chi_1(u_1) \,du_1du_2 dz ds_2 ds_3   \label{nul1x2},
\end{eqnarray}
with phase 
\begin{eqnarray}\nonumber
&&\tilde\Psi(u,z,s_2,x, \de,\la, l):=\la (2^l\la^{-1})^\frac23 z\, \Big(x_1-s_2^\frac1{n-2}G_1(s_2,\de)\Big) \nonumber \\
&&\qquad+ \la \Big(  s_2^{\frac {n}{n-2}}G_5(s_2,\de)- x_1 s_2^{\frac {n-1}{n-2}}G_3(s_2,\de)- s_2x_2-x_3
\Big) \label{Psillat}\\
 &&+2^l\Big(u_1^3\, B_3(s_2,\de_1, (2^{-l}\la)^{-\frac 13}u_1) - z u_1  
+\phi^\sharp_{2^{-l}\la}(u_1,u_2,\, \tilde\de^{2^{-l}\la}, s_2)\Big).\nonumber
\end{eqnarray}

 We shall prove the following estimate:
 \begin{equation}\label{8.48n}
|\nu_{l,1}^\la (x)|\le C 2^{\frac l3}\la^{\frac53-\frac1{B}},
\end{equation}
with a constant $C$ which is independent of $\la, l, x$ and $\de.$

\medskip
As in \cite{IM-rest1}, this estimate is easily verified if $|x|\gg1,$ simply by means of integrations by parts in  the variables $s_2, s_3$ and $z,$ in combination with the method of stationary phase in $u_1$ and van der Corput's estimate of order $B$ in $u_2.$ Similarly, if $|x|\lesssim 1$ and $|x_1|\ll 1,$ we can arrive at the same conclusion, by first integrating by parts in $z.$ Indeed, in these situations we may gain factors $2^{-2Nl/3}\la^{-N/3} $ through integrations by parts, so that the corresponding estimates can be summed in a trivial way. 

\medskip 
Let us thus assume that $|x|\lesssim 1$ and $|x_1|\sim 1.$  Following Section 7 in \cite{IM-rest1}, we  then decompose 
 $$
\nu_{l,1}^\la=\nu_{l,I}^\la+\nu_{l,II}^\la+\nu_{l,III}^\la.
$$
where $\nu_{l,I}^\la, \nu_{l,II}^\la$  and $\nu_{l,III}^\la$ correspond to the contribution to the integral in \eqref{nul1x2} given by the regions where $\la (2^l\la^{-1})^\frac23|x_1-s_2^\frac1{n-2}G_1(s_2,\de)|\gg 2^l,$ 
$\la (2^l\la^{-1})^\frac23|x_1-s_2^\frac1{n-2}G_1(s_2,\de)|\sim 2^l$  and $\la (2^l\la^{-1})^\frac23|x_1-s_2^\frac1{n-2}G_1(s_2,\de)|\ll2^l,$  respectively. The first and the last term can easily be handled as in 
\cite{IM-rest1} by means of integrations by parts in $z,$ with subsequent exploitation of the oscillations with respect to $u_1$ and $u_2,$  which leads to the following estimate:
 \begin{equation}\label{8.49n}
|\nu_{l,I}^\la(x)|+|\nu_{l,III}^\la(x)| \le C 2^{-\frac {l}3}\la^{\frac53-\frac1{B}},
\end{equation}
which better than \eqref{8.48n} by a factor $2^{-2l/3},$ so that summation in $l$ is no problem for these terms. Nevertheless, summation in $\la$ still will require an interpolation argument if $B=3.$ 

\medskip
Let us next  concentrate on $\nu_{l,II}^\la(x),$ which is of the form
\begin{eqnarray}\nonumber
&&\nu_{l,II}^\la(x)=\la^3\, (2^{-l}\la)^{-\frac 1B-1}   \int e^{-i s_3 \tilde\Psi(u,z,s_2,x, \de,\la, l)}\tilde a(\si_{2^l\la^{-1}}u, (2^{l}\la^{-1})^{\frac23}z,s_2,\de) \,  \chi_1(s_2,s_3)\\
&&\hskip2cm\times   \,\chi_1\Big((2^l\la^{-1})^{-\frac13}(x_1-s_2^\frac1{n-2}G_1(s_2,\de))\Big)\,\chi_0(u) \chi_1(u_1) \,\chi_1(z)\,du_1 du_2\, dz  \,ds_2 ds_3   \label{nuII}.
\end{eqnarray}
Writing
\begin{eqnarray*}
&&\tilde\Psi(u,z,s_2,x, \de,\la, l)=\la \Big(  s_2^{\frac {n}{n-2}}G_5(s_2,\de)- x_1 s_2^{\frac {n-1}{n-2}}G_3(s_2,\de)- s_2x_2-x_3
\Big) \\
&&+ 2^l\Big[z\,  (2^{-l}\la)^{\frac 13})(x_1-s_2^\frac1{n-2}G_1(s_2,\de)) - zu_1
+u_1^3\, B_3(s_2,\de_1, (2^{-l}\la)^{-\frac 13}u_1)   \\
&&\hskip4cm +\phi^\sharp_{2^{-l}\la}(u_1,u_2,\, \tilde\de^{2^{-l}\la}, s_2)
\Big],
\end{eqnarray*}
and observing that here $ |(2^{-l}\la)^{\frac 13}(x_1-s_2^\frac1{n-2}G_1(s_2,\de))|\sim 1$ and $|u_1|\sim 1,$ we see that the phase $\tilde\Psi$ may have a critical point $(u_1^c,z^c)$ within the support of the amplitude, as a function of $u_1$ and $z.$  Moreover, in a similar way as in \cite{IM-rest1}, we see that this critical point will be non-degenerate. 
Of course, if  there is no critical point, we may obtain even better estimates by means of integrations by parts. Let us thus in the sequel assume that there is such a critical point. 

Applying then the method of stationary phase to the integration in $(u_1,z),$ we see that we essentially  may  write 
\begin{eqnarray*}\nonumber
&&\nu_{l,II}^\la(x)=\la^2\, (2^{-l}\la)^{-\frac 1B}   \int e^{-i s_3 \Psi_2(u_2,s_2,x, \de,\la, l)}
 a_2((2^{l}\la^{-1})^{\frac13}u_2,s_2, x,(2^{l}\la^{-1})^{\frac13},\de) \,  \chi_1(s_2,s_3)\\
&&\hskip2cm\times   \,\chi_1\Big((2^l\la^{-1})^{-\frac13}(x_1-s_2^\frac1{n-2}G_1(s_2,\de))\Big)\,\chi_0(u_2)  \, du_2 ds_2 ds_3,   
\end{eqnarray*}
where the phase $\Psi_2$ arises from $\tilde\Psi$ by replacing $(u_1,z)$ by the critical point $(u_1^c,z^c)$ (which of course also depends on the other variables).

\medskip
In order to compute $\Psi_2$  more explicitly, we go back to our original coordinates, in which our complete phase is given by (compare \eqref{7.6II})
\begin{equation}\label{phaseorig}
\la \Big(s_1y_1+s_2y_1^2\omega(\de_1y_1)+y_1^n\alpha(\de_1 y_1)+s_2\de_0 y_2+y_2^B  \,b(y_1,y_2,\de)+ r(y_1,y_2,\de) -s_1x_1 -s_2x_2 -x_3\Big).
\end{equation}

Recall also that we passed from the coordinates $(y_1,s_1)$ to the coordinates $(u_1,z)$ by means of a smooth change of coordinates (depending on the remaining variables $(y_2, s_2)$). Since the value of a function at a critical point does not depend on the chosen coordinates, arguing by means of Lemma 7.1 in \cite{IM-rest1}, we find that 
in the coordinates $(y_2,s_2)$ the phase $\Psi_2$ is given by 
\begin{equation}\label{phasenew}
\la \Big(s_2x_1^2\omega(\de_1x_1)+x_1^n\alpha(\de_1 x_1)+s_2\de_0 y_2+y_2^B  \,b(x_1,y_2,\de)+ r(x_1,y_2,\de)  -s_2x_2 -x_3\Big).
\end{equation}

And, since $y_2=(2^{-l}\la)^{-1/B}u_2,$ this means that
\begin{eqnarray*}\nonumber
&&\Psi_2(u_2,s_2,x, \de,\la, l)=\la \Big(s_2x_1^2\omega(\de_1x_1)+x_1^n\alpha(\de_1 x_1)-s_2x_2 -x_3\Big)\\
&&\hskip2cm+ 2^l\Big(u_2^B  \,b(x_1,(2^{-l}\la)^{-\frac1B}u_2,\de)+
\sum_{j=2}^{B-3}u_2^j\, \tilde\de^{2^{-l}\la}_{j+2}\, x_1^{n_j}\al_j(\de_1x_1)  \\
&&\hskip4cm+(2^{-l}\la)^{\frac {B-1}B} \, u_2[ \de_0s_2 +\de_3x_1^{n_1}\al_1(\de_1x_1)]\Big)\nonumber
\end{eqnarray*}
(compare \eqref{7.3II}).
Note  that $\pa_{s_2} (s_2^\frac1{n-2}G_1(s_2,\de))\sim 1$ because $s_2\sim1$ and $G_1(s_2,0)=1.$ Therefore, the relation $|x_1-s_2^\frac1{n-2}G_1(s_2,\de)|\sim (2^l\la^{-1})^\frac13$ can be re-written as 
$|s_2-\tilde G_1(x_1,\de)| \sim (2^l\la^{-1})^\frac13,$ where $\tilde G_1$ is again a smooth function such that $|\tilde G_1|\sim 1.$
If we write  
$$
s_2=(2^l\la^{-1})^\frac13v+\tilde G_1(x_1,\de),
$$
 then this means that $|v|\sim 1.$ We  shall therefore change variables  from $s_2$ to $v,$ which leads to 
\begin{eqnarray}\nonumber
&&\nu_{l,II}^\la(x)=\la^2\, (2^{-l}\la)^{-\frac 1B-\frac 13}   \int e^{-i s_3 \Psi_3(u_2,v,x, \de,\la, l)}
 a_3((2^{l}\la^{-1})^{\frac13}u_2,v, x,(2^{l}\la^{-1})^{\frac13},\de) \\
&&\hskip6cm\times   \chi_1(s_3) \,\chi_1(v)\,\chi_0(u_2)  \, du_2 dv ds_3   \label{nuII2},
\end{eqnarray}
with a smooth amplitude $a_3$ and the new phase function
 \begin{eqnarray}\nonumber
&&\Psi_3(u_2,v,x, \de,\la, l)=\la \Big(v\, (2^{-l}\la)^{-\frac 13}(x_1^2\omega(\de_1x_1)-x_2)+ 
 (2^{-l}\la)^{-\frac1B-\frac 13}\de_0 \, v u_2+  Q_A(x,\de) \Big)\\
&&\hskip2cm+ 2^l\Big(u_2^B  \,b(x_1,(2^{-l}\la)^{-\frac1B}u_2,\de)+
\sum_{j=2}^{B-3}u_2^j \, \tilde\de^{2^{-l}\la}_{j+2}\, x_1^{n_j}\al_j(\de_1x_1)  \label{8.52n}\\
&&\hskip4cm+ u_2\, (2^{-l}\la)^{\frac {B-1}B} [ \de_0 \tilde G_1(x_1,\de) +\de_3x_1^{n_1}\al_1(\de_1x_1)]\Big),\nonumber
\end{eqnarray}
where 
$$
Q_A(x,\de):=\tilde G_1(x_1,\de)(x_1^2\omega(\de_1x_1)-x_2)+x_1^n\alpha(\de_1 x_1)-x_3. 
$$
Applying  van der Corput's estimate of order $B$ to the integration in $u_2$ in \eqref{nuII2},  we find that 
$$
|\nu_{l,II}^\la(x)|\le C \la^2\, (2^{-l}\la)^{-\frac 1B- \frac 13}  2^{-\frac lB}=2^{\frac l3} \la^{\frac 53-\frac 1B}.
$$
This proves \eqref{8.48n},  which completes the proof of estimate \eqref{analog22}.

\bigskip
\subsubsection{Consequences of the estimates \eqref{analog11} - \eqref{analog22}}

By interpolation, these estimates imply 
\begin{eqnarray}\label{7.19II}
\|T^\la_{\de,\,Ai}\|_{p_c\to p_c'}&\lesssim &\la^{-\frac13-\frac1B+2\th_c},\\ 
\|T^\la_{\de,\,l}\|_{p_c\to p_c'}&\lesssim&2^{-\frac{l(1-3\th_c)}6} \la^{-\frac13-\frac1B+2\th_c}.\label{7.20II}
\end{eqnarray}
But, by Lemma \ref{pc-est1}, we have $\th_c\le \tilde\th_B=3/(2B+3),$ and this easily implies that the exponents of $\la$ and $2^l$ in the preceding estimates is strictly negative, if $B=4,$ and zero, if $B=3$  (where $\th_c=1/3$). We can therefore sum these estimates over all $l,$ as well as $\la\gg1,$ if $B=4,$ and the desired estimate, whereas if $B=3,$  then we only get a  uniform estimates
$$
\|T^\la_{\de,\,Ai}\|_{p_c\to p_c'}\le C, \quad \|T^\la_{\de,\,l}\|_{p_c\to p_c'}\le C, \qquad (\la \rho(\tilde\de)\lesssim 1),
$$
with a constant $C$ not depending on $\de.$ 
The case  where $B=3$  will therefore again require a complex interpolation argument in order to capture the endpoint, as in the proof of Proposition 5.2 (c) in \cite{IM-rest1}.

In particular, we have proven

\begin{proposition}\label{notbigla}
If $B=4,$ then under the assumptions in this section
$$
\sum_{\{\la\ge 1:\, \la \rho(\tilde\de)\lesssim1\}} \|T^\la_\de\|_{p_c\to p_c'}\lesssim1,
$$
where the constant in this estimate will not depend on $\de.$ 
\end{proposition}

In combination with the following proposition, which will be proved in the next two sections, this will complete the discussion of the remaining case where $B=4$ in \eqref{5.17II}, hence also the proof  of Proposition \ref{rdomain} for this case.  
\begin{proposition}\label{biglaB4}
If $B=4,$ then under the assumptions in this section
$$
\sum_{\{\la\ge 1:\, \la \rho(\tilde\de)\gg 1\}} \|T^\la_\de\|_{p_c\to p_c'}\lesssim1,
$$
where the constant in this estimate will not depend on $\de.$ 
\end{proposition}

\medskip

\setcounter{equation}{0}
\section{The case where $m=2,$ $B=4$  and $A=1$ }\label{B4A1}
According to Proposition \ref{notbigla}, we are left with  controlling the operators $T^\la_\de$ with $\la \rho\gg 1,$  where we have used the abbreviation $\rho=\rho(\tilde\de).$
 If $A=1,$ then according to \eqref{5.18II}, we have 
$h^r+1=4,$ hence $\th_c=1/4$ and $p_c'=8.$ We shall then use the first estimate in \eqref{nuest}, i.e.,
\begin{equation}\label{9.1}
\|\nu_\de^\la\|_\infty\lesssim \la^{\frac 74}.
\end{equation}
The crucial observation is that under the assumption $\la \rho\gg 1,$ we can here improve on the estimate \eqref{hatnuest} for $\widehat{\nu_\de^\la}.$ Indeed, we shall prove that 
\begin{equation}\label{9.2}
\|\widehat{\nu_\de^\la}\|_\infty\lesssim \rho^{-\frac 1{12}}\la^{-\frac 23}.
\end{equation}
Under the assumption that this estimate is valid, we obtain  from \eqref{9.1} and \eqref{9.2}  by interpolation that 
$$
\|T_\de^\la\|_{p_c\to p'_c}\lesssim (\la\rho)^{-\frac 1{16}},
$$
hence the desired remaining estimate
$$
\sum_{\la\rho\gg 1 }\|T_\de^\la\|_{p_c\to p'_c}\lesssim 1.
$$

In order to prove \eqref{9.2},  recall that 
from  \eqref{7.12II} that 
$$
\widehat{\nu_\de}(\xi)=\iint e^{-i\la s_3\Phi^\sharp(x,\de,s)} \, a(x,\de,s)\, dx_2dx_1,
$$
where  the complete phase $\Phi^\sharp$ is given by 
$$
\Phi^\sharp(x,\de,s)=x_1^3\, B_3(s_2,\de_1,x_1) -x_1 B_1(s,\de_1) +  B_0(s,\de_1)+\phi^\sharp(x,\de,\,s_2),
$$
with 
\begin{eqnarray}\nonumber
\phi^\sharp(x,\de,\,s_2)&=& x_2^4 \,b(x,\de_0^r,s_2)+\de_{4} x_2^2\, \tilde\alpha_2(x_1,\de_0^{\r},\,s_2)
+\de_{3,0} x_2 \,\tilde\al_1(x_1,\de_0^{\r},s_2)\\
&& \hskip2cm \de_{0} x_1x_2 \, \al_{1,1}(x_1,\de_0^{\r},s_2), \label{9.4}
\end{eqnarray}
and where $a$ is a smooth amplitude supported  in a small neighborhood of the origin in $x.$  Recall also from Lemma \ref{lastterm}  that  $| \al_{1,1}|\equiv 0$ and $|\tilde\al_1|\sim 1$ in Case ND,  whereas 
 $| \al_{1,1}|\sim 1$  and $\tilde\al_1$ is independent of $x_1$ in Case D. Recall also that $s_j\sim 1$ for $j=1,2,3.$ 
 
  Moreover, in Case ND we have 
  $$
\rho=\de_{3,0}^{\frac 43}+\de_4^2,
$$ 
whereas in Case D   
$$
\rho=\de_0^{\frac {12}{5}}+\de_{3,0}^{\frac 43}+\de_4^2,
$$ 
where
$\de_{3,0}\ll \de_0.$  We shall treat both cases ND and D  at the same time, assuming implicitly that $\de_0=0$ in Case ND. 

Estimate \eqref{9.2} will thus be a direct consequence of the following lemma, which can be derived from  more general results by Duistermaat (cf. Proposition 4.3.1 in  \cite{duistermaat}).
\color{black} For the  convenience of the reader, we shall give a more elementary,  direct proof for our situation,  which  requires only $C^2$- smoothness of the amplitude. Our approach will be  based on arguments similar to the ones used on  pp.196- 205 in  \cite{IKM-max}.

\begin{lemma}\label{duistuniest}
Denote by $J(\la,\de,s)$ the oscillatory
$$
J(\la,\de,s)= \chi_1(s_1,s_2)\iint e^{-i\la\Phi(x,\de,s)} \, a(x,\de,s)\, dx_1dx_2,
$$
with phase 
$$
\Phi(x,\de,s)=x_1^3\, B_3(s_2,\de_1,x_1) -x_1 B_1(s,\de_1) +\phi^\sharp(x_1,x_2,\de,\,s_2),
$$
where $\phi^\sharp$ is given by \eqref{9.4}, and where $ \chi_1(s_1,s_2)$ localizes to the region where $s_j\sim 1, j=1,2.$ Let us also put 
\begin{eqnarray}\nonumber
\tilde\rho:=\rho+|B_1(s,\de_1)|^\frac32.
\end{eqnarray}
Then the following estimate
\begin{equation}\label{duistypeg}
|J(\la,\de,s)|\le \frac{C}{\tilde\rho^\frac1{12}\la^\frac23}
\end{equation}
 holds true, provided  the amplitude  $a$ is supported in a sufficiently small
neighborhood of the origin. The constant $C$ in this estimate is independent of $\de$  and $s.$
\end{lemma}

\begin{proof} Note that $\tilde\rho\lesssim 1.$   We may indeed even assume that $\tilde\rho\ll 1.$ 

For, if $|B_1(s,\de_1)|\sim 1,$ and if we choose the support of $a$ sufficiently small in $x,$ then it is easily seen that the phase $\Phi$ has no critical point with respect to $x_1$ on the support of the amplitude, and thus an integration by parts in $x_1$ allows to estimate $|J(\la,\de,s)|\le C\la^{-1},$ which is better than what is needed for $\eqref{duistypeg}.$  And, if $|B_1(s,\de_1)|\ll 1,$ then we  also have $\tilde\rho\ll 1.$

\medskip

 We begin with the case where  $\tilde\rho\la\lesssim1.$ Here we can argue as in the proof of Lemma \ref{duist2}: Changing coordinates from $x$ to $u$ by putting 
$x=\si_{1/\la} u=(\la^{-1/3}u_1,\la^{-1/4}u_2)$ and making use of Remark \ref{duist} (with $r:=\la$),  we find that 
\begin{eqnarray*}
&&J(\la,\de,s)=\la^{-\frac 14-\frac 13}\,  \chi_1(s_1,s_2)\\
&&\hskip1cm \times\iint e^{-i s_3\Big(u_1^3\, B_3(s_2,\de_1,\la^{-\frac 13}u_1) -u_1  
\la^{\frac 23}B_1(s,\de_1)+\phi^\sharp_\la(u_1,u_2,\, \tilde\de^\la, s_2) \Big)} 
 a(\si_{\la^{-1}}u, \de,s) \,du_1du_2\, .
 \end{eqnarray*}
 Observe that  here we are  integrating over the large domain where $|u_1|\le \ve \la^{1/3}$ and $|u_2|\le \ve \la^{1/4}.$ Recall also that $\phi^\sharp_\la$ is given by \eqref{phir}, and that $\rho(\tilde\de^\la)=\la \rho(\tilde\de)\lesssim 1,$ and so  we have 
$$
|\tilde\de^\la|\lesssim 1\quad \mbox{and}\quad \la^\frac23|B_1(s,\,\de_1)|\lesssim 1.
$$
It is then easily seen by means of integrations by parts in $u_1$ respectively $u_2$ (whenever these quantities are large) that the double-integral in this expression is uniformly bounded in $\de$ and $s,$ and thus we arrive at the uniform estimate
\begin{eqnarray}\nonumber
|J(\la,\de,s)|\le \frac{C}{\la^\frac7{12}}.
\end{eqnarray}
This estimate is stronger than estimate 
\eqref{duistypeg}  when $\tilde\rho\la\lesssim1.$

\medskip
From  now on, we may thus  assume that $\La:=\tilde\rho\la\gtrsim1$. We then apply the change of coordinates 
$x=\si_{\tilde \rho}u=(\tilde\rho^{1/3} u_1,\tilde\rho^{1/4} u_2)$ and  find that 
$$J(\la,\de,s)=\tilde\rho^\frac7{12} I(\la\tilde\rho,\de,s),
$$
 where we have put
\begin{eqnarray}\label{9.5II}
I(\La,\de,s)=\chi_1(s_1,s_2)\, \iint
e^{-i\La\Phi_1(u,\de,s)}a(\sigma_{\tilde\rho}u,\de,s)\,du_1 du_2,\qquad (\La\gtrsim 1),
\end{eqnarray}
with
$$
\Phi_1(u,\de,s):=u_1^3\, B_3(s_2,\de_1, \tilde\rho^\frac13u_1)
-u_1 
B'_1(s,\de_1)+\phi(u,\tilde\rho,\de,\,s_2)
$$
and 
\begin{eqnarray}\nonumber
\phi(u,\tilde\rho,\de,\,s_2)&:=& u_2^4 \,b(\sigma_{\tilde\rho}u,\de_0^r,s_2)+\de'_{4} u_2^2\, \tilde\alpha_2(\tilde\rho^{\frac 13}u_1,\de_0^{\r},\,s_2)
+\de'_{3,0} u_2 \,\tilde\al_1(\tilde\rho^{\frac 13}u_1,\de_0^{\r},s_2)\\
&& \hskip2cm \de'_{0} u_1u_2 \, \al_{1,1}(\tilde\rho^{\frac 13}u_1,\de_0^{\r},s_2). \label{phiII}
\end{eqnarray}
Here, 
$$
B_1'(s,\de_1):=\frac{B_1(s,\de_1)}{\tilde\rho^\frac23},\
\de_0':=\frac{\de_0}{\tilde\rho^\frac5{12}}
,\ \de_{3,0}':=\frac{\de_{3,0}}{\tilde\rho^\frac34},\ 
\de_4':=\frac{\de_4}{\tilde\rho^\frac12},
$$
so  that,  in analogy with Remark \ref{duist}, we  have 
\begin{equation}\label{order1}
(\de'_0)^{\frac {12}{5}}+(\de'_{3,0})^{\frac 43}+(\de'_4)^2+|B'_1(s,\de_1)|^{\frac 32}=1.
\end{equation}
Note that in particular 
$$
\de'_0+\de'_{3,0}+\de'_4+|B'_1(s,\de_1)|\sim 1.
$$ 
In order to prove \eqref{duistypeg}, we have thus to verify the following estimate:
\begin{equation}\label{duistypeg2}
|I(\La,\de,s)|\le C\La^{-\frac 23}.
\end{equation}

Take again a  smooth cut-off function $\chi_0\in C^\infty_0(\RR^2)$ such that $\chi_0(u)=1$ for $|u|\le
L$ where $L$ is a sufficiently big fixed positive  number, and  decompose 
\begin{eqnarray}\nonumber
I(\La,\de,s)=I_0(\La,\de,s)+I_\infty(\La,\de,s),
\end{eqnarray}
with
\begin{eqnarray}\nonumber
I_0(\La,\de,s):=\chi_1(s_1,s_2)\, \iint
e^{-i\La\Phi_1(u,\de,s)}a(\sigma_{\tilde\rho}u,\de,s)\, \chi_0(u)\,du_1 du_2,
\end{eqnarray}
and
\begin{eqnarray}\nonumber
I_\infty(\La,\de,s):=\chi_1(s_1,s_2)\, \iint
e^{-i\La\Phi_1(u,\de,s)}a(\sigma_{\tilde\rho}u,\de,s)\, (1-\chi_0(u))\,du_1 du_2.
\end{eqnarray}
Note that on the support of $1-\chi_0$ we have $|u_1|\gtrsim L$ or  $|u_2|\gtrsim L.$ Thus, by  choosing $L$ sufficiently large, we see by \eqref{order1} that  the phase $\Phi_1$ has no critical point on the support of $1-\chi_0,$ and in fact  we may use integrations by parts in $u_1$ respectively $u_2$  in order to prove  that the double-integral in the  expression for $I_\infty(\La,\de,s)$  is of order $O(\La^{-1}),$  uniformly in $\de$ and $s.$  This is stronger than what is required for \eqref{duistypeg2}.

\medskip
There remains the integral $I_0(\La,\de,s)$.
Here we use arguments from \cite{IKM-max} (compare pp. 
203-205).  Recall from \eqref{order1} that $(B_1'(s,\de_1),\de'_0,
\de_{3,0}',\de_4')$ lies on the``unite sphere''
\begin{eqnarray}\nonumber
\Sigma:=\{(B_1', \de_0',\de_{3,0}',\de_4')\in \RR^4:\,|B'_1|^{\frac 32}+ (\de'_0)^{\frac {12}{5}}+(\de'_{3,0})^{\frac 43}+(\de'_4)^2=1 \}.
\end{eqnarray}
Following \cite{IKM-max} let us fix a point $((B_1')^0,\,
(\de_0')^0,(\de_{3,0}')^0, (\de_4')^0 )\in \Sigma$, a point $s^0$ on the support of $\chi_1$ and a point
$u^0=(u_1^0,u_2^0)\in \supp \chi_0$, and denote  by $\eta$ a
smooth cut-off function supported near $u^0$. By 
$I_0^\eta$ we denote the corresponding oscillatory integral containing $\eta$ as a factor in the amplitude:
\begin{eqnarray}\nonumber
I_0^\eta(\La,\de,s):=\chi_1(s_1,s_2)\, \iint
e^{-i\La\Phi_2(u,B_1',\de'_0,\de_{3,0}',\de_4',\tilde\rho,s)}a(\sigma_{\tilde\rho}u,\de,s)\, \chi_0(u)\eta(u)\,du_1 du_2,
\end{eqnarray}
where 
\begin{equation}\label{Phi2II}
\Phi_2(u,B_1',\de'_0,\de_{3,0}',\de_4',\tilde\rho,s):=u_1^3\, B_3(s_2,\de_1, \tilde\rho^\frac13u_1)
-u_1 B'_1+\phi(u,\tilde\rho,\de,\,s_2),
\end{equation}
with $\phi$ as before.

 We shall prove that $I_0^\eta$ satisfies the estimate
\begin{equation}\label{9.6II}
|I_0^\eta(\La,\de,s)|\le C\|a(\cdot, \,\de,s)\|_{C^2}\, \La^{-\frac 23}.
\end{equation}
provided $\eta$ is supported in a sufficiently small neighborhood of
$U$ of $u^0,$ $s$ lies in a sufficiently small neighborhood $S$ of $s^0$  and $(B_1',\de'_0,
\de_{3,0}',\de_4')$ in a sufficiently small neighborhood  $V$ of the point $((B_1')^0,
(\de_0')^0,(\de_{3,0}')^0, (\de_4')^0 )$ in $\Sigma.$  The constant $C$ in these estimates may depend on the ``base points'' $u^0,s^0$ and $((B_1')^0,(\de_0')^0,(\de_{3,0}')^0, (\de_4')^0 ),$ as well as on the chosen neighborhoods,  but not on $\La, \de$ and $s.$

By means of a partition of the identity argument this will imply the same type estimate for $I_0$,
hence for $I$, which will conclude the proof of Lemma
\ref{duistuniest}.
\medskip

Now, if $\nabla_u \Phi_2(u^0,\,(B_1')^0,\,
(\de_0')^0,(\de_{3,0}')^0, (\de_4')^0,\tilde\rho,s^0)\neq 0,$  then by using an integration by parts argument in a similar way as for $I_\infty$ we arrive at the same  estimate  for $I_0^\eta$ as for  $I_\infty,$ which is better than what is required.
\medskip

We may therefore assume  from now on  that $\nabla_u \Phi_2(u^0,(B_1')^0,\,(\de_0')^0,(\de_{3,0}')^0, (\de_4')^0,\tilde\rho,s^0)= 0,$ and shall distinguish two cases.

\medskip
{\bf Case 1: $u_1^0\neq0$.} In this case,  it is easy to see from  \eqref{Phi2II} and \eqref{phiII} that
$$\pa_{u_1}^2\Phi_2(u^0,(B_1')^0,\,(\de_0')^0,(\de_{3,0}')^0, (\de_4')^0,\tilde\rho,s^0)\neq 0,$$
provided $\tilde\rho$ is sufficiently small. 
Then, by the implicit function  theorem,  the phase $\Phi_2$ has a unique  critical point $u_1^c(u_2,B_1',\de'_0,
\de_{3,0}',\de_4',\tilde\rho,s)$ 
 with respect to $x_1,$ which is a smooth function of 
its variables,  provided we choose the neighborhoods $U$ etc. sufficiently small.  Indeed,  and  when   $\tilde\rho=0,$  then by \eqref{Phi2II} and  \eqref{phiII},
\begin{equation}\label{u1cII}
u_1^c(u_2,B_1',\de'_0,\de_{3,0}',\de_4',0,s )= \left(\frac{
B_1'-\de_0' u_2 \alpha_{11}(0,\de_0^{\r},s_2)}{3B_3(s_2,\de_1,0)}
\right)^\frac12. 
\end{equation}
We may thus apply the method of stationary phase to the integration with respect to the variable $u_1$ in the integral defining $I_0^\eta.$ 
Let us denote by $\Psi$  the phase function
\begin{eqnarray}\nonumber
\Psi(u_2,B_1',\de'_0,\de_{3,0}',\de_4',\tilde\rho,s ):=\Phi_2\big(u_1^c(u_2,B_1',\de'_0,
\de_{3,0}',\de_4',\tilde\rho,s),B_1',\de'_0,\de_{3,0}',\de_4',\tilde\rho,s \big)
\end{eqnarray}
which arises through this application of  the method of stationary phase. We claim that we then have 
\begin{equation}\label{bjork5}
\max_{j=4,5} |\pa_{u_2}^j\Psi(u^0_2,B_1',\de'_0,\de_{3,0}',\de_4',\tilde\rho,s )|\neq 0.
\end{equation}
Notice that it suffices to prove this for $\tilde \rho=0,$ since then the results follows also for $\tilde\rho$ sufficiently small. 

In order to prove \eqref{bjork5} when $\tilde \rho=0,$ we make use of \eqref{u1cII}. Since $|B_3|\sim 1,$ \eqref{u1cII} shows that we may assume that 
\begin{equation}\label{b1size}
|B_1'-\de_0' u_2 \alpha_{11}(0,\de_0^{\r},s_2)|\sim |u_1^c|\sim |u_1^0|.
\end{equation}
Note also that by \eqref{u1cII} we have 
$$
\Psi(u_2,B_1',\de'_0,\de_{3,0}',\de_4',0,s)=\Gamma(u_2)+u_2^4 \,b(0,\de_0^r,s_2)+\de'_{4} u_2^2\, \tilde\alpha_2(0,\de_0^{\r},\,s_2)
+\de'_{3,0} u_2 \,\tilde\al_1(0,\de_0^{\r},s_2),
$$
where we have put 
$$
\Gamma(u_2):= -2\cdot 3^{-\frac 32} B_3(s_2,\de_1,0)^{-\frac 12}\big(B_1'-\de_0' u_2 \alpha_{11}(0,\de_0^{\r},s_2)\big)^{\frac 32}.
$$
In the case ND we have $\al_{11}\equiv 0,$ and thus  $|\pa_{u_2}^4\Psi(u^0_2,B_1',\de'_0,\de_{3,0}',\de_4',\tilde\rho,s )|\neq 0.$ 

Next, if we are in  Case D, then $|\al_{11}|\sim 1,$ and  \eqref{b1size}  implies that $|\Gamma^{(j)}(u_2)|\sim |u_1^0|^{3/2}|\de'_0/u_1^0|^{j}.$ Therefore, if $\de'_0\ll |u_1^0|,$ then we find  that 
$|\pa_{u_2}^4\Psi(u^0_2,B_1',\de'_0,\de_{3,0}',\de_4',\tilde\rho,s )|\neq 0,$ and if  $\de'_0\gtrsim |u_1^0|,$ then 
$$
|\pa_{u_2}^5\Psi(u^0_2,B_1',\de'_0,\de_{3,0}',\de_4',\tilde\rho,s )|\gtrsim |u_1^0|^{\frac 32}\ne 0.
$$
This verifies \eqref{bjork5} also in this case. But, \eqref{bjork5} allows to apply van der Corput's lemma to the integration in $u_2$ (after the application of the method of stationary phase in $u_1$) and thus altogether obtain the estimate
$$
|I_0^\eta(\La,\de,s)|\le C\La^{-\frac 12-\frac15},
$$
which is again even stronger than what is required by  \eqref{9.6II}.

\medskip
{\bf Case 2: $u_1^0=0.$} Assume first, $(\de_0')^0\neq0$ and $|\alpha_{11}|\sim1$  (this situation  can occur only in the  case D). Then 
\begin{eqnarray*}
&&\pa_{u_1}\pa_{u_2}\Phi_2(u^0,(B_1')^0,
(\de_0')^0,(\de_{3,0}')^0, (\de_4')^0,0,s^0)=\de'_0\alpha_{11}(0,0,s^0_2)\neq 0,\\
&&\pa_{u_1}^2\Phi_2(u^0,(B_1')^0,
(\de_0')^0,(\de_{3,0}')^0, (\de_4')^0,0,s^0)=0.
\end{eqnarray*}
Therefore we can apply the method of  stationary phase to the integration in both variables $(u_1,u_2)$ and again
obtain  an  estimate of order $O(\La^{-1}),$ which is again  stronger  than what we need.

\medskip
From now on, we may thus assume that $(\de_0')^0=0$ (recall that in Case ND, we have  
$\alpha_{11}\equiv 0$ and are assuming that $\de_0=0,$ hence also $\de'_0=0,$ so that this assumption is automatically satisfied). 

Then necessarily also $(B'_1)^0=0,$ for otherwise, in view of \eqref{u1cII} we would have 
$|u_1^c|\sim |(B'_1)^0|\ne 0$  when $\tilde\rho=0,$ which would contradict our assumption that $u_1^0=0.$ 
Since $((B_1')^0,(\de_0')^0,(\de_{3,0}')^0, (\de_4')^0 )\in \Sigma,$ we thus see that 
$((\de'_{3,0})^0)^{\frac 43}+((\de'_4)^0)^2=1$. 

\medskip
Therefore at the ``base point'' $((B_1')^0,(\de_0')^0,(\de_{3,0}')^0, (\de_4')^0 )$ the function $\phi$ satisfies for $\tilde\rho=0$ the  inequality
$$
\sum_{j=2}^3|\pa_{u_2}^j \phi(0,u_2^0, 0,\de, s_2^0)|\neq 0,
$$
and this inequality will persist for parameters sufficiently close to this base point.

Assume first that we even have
$
\pa_{u_2}^2 \phi(0,u_2^0, 0,\de, s_2^0)\neq 0.
$
Then we can first apply the method of  stationary phase to the  $u_2$ integration, and 
subsequently van der Corput's  estimate in $u_1$ (with $N=3$), which results in  the estimate 
$
|I_0^\eta(\La,\de,s)|\le C\La^{-\frac 12-\frac13}.
$
This is again stronger than what we need. 
\medskip

There remains the case where
$
\pa_{u_2}^2 \phi(0,u_2^0, 0,\de, s_2^0)= 0
$
and 
$
\pa_{u_2}^3 \phi(0,u_2^0, 0,\de, s_2^0)\neq 0.
$
In this case the phase function $\Phi_2$ is a small smooth perturbation
of a function  $\Phi_2^0$ of the form
$$
\Phi_2^0(u_1,u_2)=c_3u_1^3 +(u_2-u_2^0)^3b_3(u_2) +c_0,
$$
where $c_3:=B_3(s_2^0,\de_1,0)\ne 0$ and where  $b_3(u_2)$ is
a smooth function such that $b_3(u_2^0)\neq0$. This means that  $\Phi_2$  has a 
so-called $D_4^+$-type singularity in the sense of \cite{agv} and the distance between  the associated Newton polyhedron and the origin is $3/2$. The estimate \eqref{9.6II} therefore follows in this situation  from  the  particular case of $D_4^+$-type singularities in  Proposition 4.3.1 of   \cite{duistermaat}.  

Alternatively, one could also first treat the integration with respect to $u_1$ by means of Lemma 6.3 in \cite{IM-rest1}, with $B=3,$ and subsequently estimate the integration in $u_2$ by means of van der Corput's lemma (we leave the details to the interested reader).

This concludes the proof of Lemma \ref{duistuniest}.
\end{proof}

\begin{remark} Notice that our phase $\Phi$ in Lemma \ref{duistuniest} is   a small perturbation of a phase of the form $c_1 x_1^3+c_2 x_2^4,$ with $c_1\ne 0\ne c_2,$ at least if we assume that $|B_1(s,\de_1)|\ll 1$ (this has been the interesting case the preceding proof).
It is, however, not true that for arbitrary perturbations of such a phase function with a small perturbation parameter $\de>0$  an estimate  analogous to \eqref{duistypeg}  of order $O\big(c(\de)\la^{-2/3}\big)$ as $\la\to\infty$  holds true. A counter example is given by the following function
\begin{eqnarray}\nonumber
\Phi(x,\de):=x_1^3+(x_2-\de)^4+4\de(x_2-\de)^3-3
\sqrt[3]{4\de^2}x_1(x_2-\de)^2+C(\de),
\end{eqnarray}
where $C(\de)$ is chosen such that $\Phi(0,\,\de)\equiv0$. Note that  $\Phi(x,0)=x_1^3+x_2^4$. 
\end{remark}

To see this,  consider an 
oscillatory integral $J(\la,\de):=\int e^{i\la\Phi(x,\de)} a(x) \, dx$ with phase function $\Phi,$  whose  amplitude is supported  in a sufficiently small neighborhood of the origin and such that $a(0)=1.$

When $\de>0$ is sufficiently small, then $\Phi$ has exactly two critical points, namely  the degenerate critical point  $x_d:=(0,\de)$ and the non-degenerate  critical point
$x_{nd}:=(6\sqrt[3]{2}\de^{4/3},-6\de)$.

Let us  consider the contribution  of  the degenerate critical point $x_d$ to
the oscillatory integral. The linear change of variables
$$
z_1=x_1-\sqrt[3]{2\de}(x_2-\de),\quad
z_2=x_1+2\sqrt[3]{2\de}(x_2-\de)
$$
transforms $x_d$ into $z_d=(0,0)$ and the  phase function into $\tilde\Phi(z)+C(\de),$ where
$$
\tilde\Phi(z) :=z_1^2z_2+\left(\frac{z_2-z_1}{3\sqrt[3]{2\de}} \right)^4.
$$
A look at the Newton polyhedron of $\tilde\Phi$ reveals  that the principal face of $\N(\tilde\Phi)$ is given by the 
compact edge $[(0,4),(2,1)]$ which lies on the line given by $\ka_1t_1+\ka_2 t_2=1,$ with associated weight  $\ka=(\ka_1,\ka_2):=(3/8,\,1/4),$ and the principal part of  $\tilde\Phi$ is given by 
$$
\tilde\Phi_{\pr}(z) =z_1^2z_2+\frac{z_2^4}{81\sqrt[3]{16\de^4}}.
$$
Moreover, the Newton distance is given by $d=8/5,$  whereas the non-trivial roots of $\tilde\Phi_{\pr}$ have multiplicity $1.$ Therefore, by Theorem 3.3 in \cite{IM-ada} the coordinates $(z_1,z_2)$ are adapted to $\tilde\Phi$
in a sufficiently small neighborhood of the origin, so that the height $h$ of $\tilde\Phi$ in the sense of Varchenko  is also given by $h=d=8/5.$ This implies that for  every sufficiently small, fixed  $\de>0$ we have that 
\begin{eqnarray}\nonumber
J(\la,\,\de)=C(\de) \la^{-\frac58}+O\left(\la^{-\frac78}\right) \qquad \mbox{as} \quad \la\to\infty,
\end{eqnarray}
with a non-trivial constant $C(\de),$ 
because the contribution of the non-degenerate critical point  $x_{nd}$ is of order $O(\la^{-1})$ (compare for instance \cite{IM-uniform}).
This  shows that an estimate of the  type
$
|J(\la,\,\de)|\le C(\de)\la^{-2/3}
$
can not hold in this example.

 \color{red}

 \color{black}

\setcounter{equation}{0}
\section{The case where $m=2,$ $B=4$  and $A=0$ }\label{B4A0}
Again, we are assuming that $\la\rho\gg 1,$  where $\rho=\rho(\tilde\de)$ is given by \eqref{rho}. Observe that if $A=0,$ then according to \eqref{5.18II}
we have
$h^r+1=4(n+3)/(n+4),$  where $n\ge 9,$  so that 
\begin{equation}\label{10.1}
\th_c\le \frac {13}{48}
\end{equation}

Here we shall again perform a frequency  decomposition  near the Airy cone,  by defining  functions $\nu_{\de,\, Ai}^\la$ and $\nu_{\de,\, l}^\la$  as follows:
$$
\widehat{\nu_{\de,\, Ai}^\la}(\xi):=\chi_0(\rho^{-\frac 23}B_1(s,\,\de_1))\widehat{\nu_\de^\la}(\xi),
$$
$$
\widehat{\nu_{\de,\, l}^\la}(\xi):=\chi_1((2^{l}\rho)^{-\frac 23}B_1(s,\,\de_1))\widehat{\nu_\de^\la}(\xi),\quad M_0\le 2^l\le \frac{\rho^{-1}}{M_1},
$$
so that
\begin{equation}\label{airydecomp}
\nu_\de^\la=\nu_{\de,\, Ai}^\la+\sum_{\{l: M_0\le 2^l\le \frac{\rho^{-1}}{M_1}\}}\, \nu_{\de,\, l}^\la.
\end{equation}

Denote by $T_{\de,\, Ai}^\la$ and $T_{\de,\, l}^\la$ the corresponding operators  of  convolution with the Fourier transforms of these functions.

\subsection{Estimation of  $T_{\de,\, Ai}^\la$} \label{Aig}
 Here we have $|\rho^{-2/3}B_1(s,\,\de_1)|\lesssim 1.$ In this case, we use the change of variables $x=:\si_\rho u:=(\rho^{1/3} u_1,\, \rho^{1/4} u_2)$ in the
integral \eqref{7.12II} defining  $\widehat{\nu_\de},$ and obtain
\begin{equation}\label{10.2}
\widehat{\nu_\de}(\xi)=\rho^\frac7{12}e^{-i\la s_3 B_0(s,\de_1)} \int_{|\si_\rho u|<\ve} e^{-i\la\rho\, s_3
\Phi_1(u,s,\de))}a(\si_\rho u,\de,s)\, du,
\end{equation}
where  the phase $\Phi_1$  has the form
\begin{eqnarray}  \label{10.3}
&&\hskip1cm\Phi_1(u,s,\de)= u_1^3\, B_3(s_2,\de_1,\rho^{\frac 13}u_1) 
-u_1  (\rho^{-\frac 23}B_1(s,\de_1))\\
&& + u_2^4 \,b(\si_\rho u,\de_0^\r,s_2)+\de'_{4} u_2^2\, \tilde\alpha_2(\rho^{\frac 13}u_1,\de_0^{\r},\,s_2)
+\de'_{3,0} u_2 \,\tilde\al_1(\de_0^{\r},s_2)
+\de'_{0} u_1u_2 \, \al_{1,1}(\rho^{\frac 13}u_1,\de_0^{\r},s_2), \nonumber
\end{eqnarray}
and where according to Remark \ref{duist} (in which we choose $r:=1/\rho$) $\rho(\widetilde{\de'})=1,$ so that 
$$
\de'_0+\de'_{3,0}+\de'_4\sim 1
$$
(recall that the coefficient $\de'_{3,0}$ does not appear in Case ND, where $\al_{1,1}=0$). We have also indicated that the amplitude $a(\si_\rho u,\de,s)$ is supported where $|\si_\rho u|<\ve,$  where we may assume that $\ve>0$ is sufficiently small, since this will become important soon.

\medskip
We shall proceed  in a somewhat  similar way as in Section 7 of \cite{IM-rest1}, by choosing a cut-off function $\chi_0\in C_0^\infty(\RR^2)$ such that $\chi_0(u)=1$  for $|u|\le R,$ where $R$ will be chosen sufficiently  large, and further decomposing
$$
\widehat{\nu_{\de}}(\xi)=\widehat{\nu_{\de,0}}(\xi)+\widehat{\nu_{\de,\infty}}(\xi),
$$
where
$$
\widehat{\nu_{\de,0}}(\xi):=\rho^\frac7{12}e^{-i\la s_3 B_0(s,\de_1)} \int e^{-i\la\rho\, s_3
\Phi_1(u,s,\de)}a(\si_\rho u,\de,s)\chi_0(u) \, du,
$$
and 
$$
\widehat{\nu_{\de,\infty}}(\xi):=\rho^\frac7{12}e^{-i\la s_3 B_0(s,\de_1)} \int e^{-i\la\rho\, s_3
\Phi_1(u,s,\de,s)}a(\si_\rho u,\de,s)(1-\chi_0(u)) \, du.
$$
 Accordingly, we decompose 
$$
\nu_{\de,\, Ai}^\la=\nu_{\de,0}^\la+\nu_{\de,\infty}^\la,
$$
where  we have put 
\begin{eqnarray*}
\widehat{\nu_{\de,0}^\la}(\xi)&:=&\chi_0(\rho^{-\frac 23}B_1(s,\,\de_1))\chi_1(s,s_3) \,\widehat{\nu_{\de,0}}(\xi),\\
\widehat{\nu_{\de,\infty}^\la}(\xi)&:=&\chi_0(\rho^{-\frac 23}B_1(s,\,\de_1))\chi_1(s,s_3) \,\widehat{\nu_{\de,\infty}}(\xi).
\end{eqnarray*}
 Recall from   \eqref{hatnuai} that  $\chi_1(s,s_3)=\chi_1(s_1,s_2,s_3)$ localizes to the region where $s_j\sim 1, j=1,2,3.$
The corresponding   operators of convolution with  $\widehat{\nu_{\de,0}^\la}$ and $\widehat{\nu_{\de,\infty}^\la}$ will be denoted by  $T_{\de,0}^\la$ and $T_{\de,\infty}^\la,$ respectively.

\medskip
Let us first  consider the operators $T_{\de,\infty}^\la:$ 
By means of integrations by parts, we easily  see that if $R$ is chosen sufficiently large, then the phase will have no critical point, and thus for  every $N\in\NN$ we have 
\begin{equation}\label{10.4}
\|\widehat{\nu_{\de,\infty}^\la}\|_\infty \lesssim \rho^{\frac 7{12}}(\la\rho)^{-N}.
\end{equation}
Moreover, by  Fourier inversion  we find that 
\begin{equation}\label{10.5}
\nu_{\de,\infty}^\la(x)=\la^3\int_{\RR^3} e^{i\la s_3(s_1x_1+s_2x_2+x_3)} \widehat{\nu_{\de,\infty}^\la}(\xi)\, ds
\end{equation}
(with $\xi=\la s_3(s_1,s_2,1)$). 
We then use the change of variables  from $s=(s_1,s_2)$ to $(z,s_2),$ where 
$$
z:= \rho^{-\frac 23}B_1(s,\de_1),
$$ 
and  find that (compare \eqref{ai2})
\begin{equation}\label{10.6}
s_1=s_2^{\frac {n-1}{n-2}}G_3(s_2,\de)-\rho^{\frac 23}z,
\end{equation}
and in particular
 \begin{equation}\label{10.7}
B_0(s,\de,\si)=-\rho^{\frac 23}z\,s_2^{\frac {1}{n-2}}G_1(s_2,\de)+s_2^{\frac {n}{n-2}}
G_5(s_2,\de).
\end{equation}
 And, if we plug the previous formula for 
$\widehat{\nu_{\de,\infty}^\la}$ into \eqref{10.5}, we see that we may  write  $\nu_{\de,\infty}^\la(x)$ as an oscillatory 

\begin{equation}\label{10.8}
\nu_{\de,\infty}^\la(x)=\rho^{\frac 7{12}+\frac 23}\la^3\int e^{-i\la s_3 \Phi_2(u,z,s_2,\de)}
\chi_0(z)(1-\chi_0(u)) a(\si_\rho u, \rho^{\frac 23} z,s,\de) \tilde\chi_1(s_2, s_3)\, du dz ds_2 ds_3
\end{equation}
with respect to the variables $u_1,u_2, z,s_2,s_3,$ where the  complete phase is given by 
\begin{eqnarray}\nonumber
&&\Phi_2(u,z,s_2,\de):= s_2^{\frac {n}{n-2}}G_5(s_2,\de)- x_1 s_2^{\frac {n-1}{n-2}}G_3(s_2,\de)\\
&&\hskip2cm - s_2x_2-x_3+ \rho^{\frac 23}z\big(x_1-s_2^{\frac {1}{n-2}}G_1(s_2,\de)\big)+\rho\,\Phi_1(u,z,s_2,\de),\label{10.9}
\end{eqnarray}
where 
according to \eqref{10.3}, the phase $\Phi_1$ is given in the new coordinates by 
\begin{eqnarray}  \label{10.13}
&&\hskip2cm\Phi_1(u,z,s_2,\de)= u_1^3\, B_3(s_2,\de_1,\rho^{\frac 13}u_1) 
-u_1 z\\
&& + u_2^4 \,b(\si_\rho u,\de_0^\r,s_2)+\de'_{4} u_2^2\, \tilde\alpha_2(\rho^{\frac 13}u_1,\de_0^{\r},\,s_2)
+\de'_{3,0} u_2 \,\tilde\al_1(\de_0^{\r},s_2)
+\de'_{0} u_1u_2 \, \al_{1,1}(\rho^{\frac 13}u_1,\de_0^{\r},s_2). \nonumber
\end{eqnarray}
\smallskip

Recall also from  \eqref{gnz} that  $|G_5|\sim 1\sim|G_3|$ (where $G_5=G_1G_3-G_2$).  The new amplitude $a$ is again a smooth function of its arguments, and $\tilde\chi_1(s_2, s_3)$ localizes to the region where $|s_2|\sim 1\sim |s_3|.$ 
Observe also  that here  $|z|\lesssim 1$ and $|\rho^{1/3}u_1|\le \ve,  |\rho^{1/4}u_2|\le \ve ,$  so that the sum of the last two terms  
in $\Phi_2$ can be viewed as a small  error term of order $O(\rho^{2/3}+\ve),$   provided $|x|\lesssim 1.$ 

\medskip
Applying  first  again  $N$ integrations by parts with respect to the variables  $u_1,u_2,$  and then van der Corput's lemma to the integration in $s_2,$  we thus find  that 
$$
|\nu_{\de,\,\infty}^{\la}(x)|\lesssim \rho^{\frac 7{12}+\frac 23}  \la^3 (\la\rho)^{-N} \la^{-\frac 13},
$$
if $|x|\lesssim 1.$  

However, if $|x|\gg 1,$ then  we may argue as in Section 7.2 of \cite{IM-rest1}:  If $|x_1|\ll |(x_2,x_3)|,$ then we easily see by means a further  integration by parts with  respect to the variables $s_2$ or  $s_3$  that 
$|\nu_{\de,\,\infty}^{\la}(x)|\lesssim \rho^{\frac 7{12}+\frac 23}  \la^3 (\la\rho)^{-N} \la^{-1},$ and if $|x_1|\gtrsim |(x_2,x_3)|,$ then an integration by parts in $z$ leads to $|\nu_{\de,\,\infty}^{\la}(x)|\lesssim \rho^{\frac 7{12}+\frac 23}  \la^3 (\la\rho)^{-N} (\la\rho^{2/3})^{-1}.$   Both these estimates are stronger than the previous one, and so altogether we have shown that 
\begin{equation}\label{10.10}
\|\nu_{\de,\,\infty}^{\la}\|_\infty \lesssim  \rho^{\frac 7{12}+\frac 23} \la^{\frac 83}(\la\rho)^{-N}.
\end{equation}

Interpolating between this estimate and \eqref{10.4}, we obtain
$$
\|T_{\de,\,\infty}^\la\|_{p_c\to p_c'}\lesssim \rho^\frac7{12}(\la\rho)^{-N}\rho^\frac{2\th_c}3\la^\frac{8\th_c}3,
$$
which implies the desired estimate
$$
\sum_{\la\rho\gg 1} \|T_{\de,\,\infty}^\la\|_{p_c\to p'_c}\lesssim
\rho^{\frac7{12}-2\th_c}\le 1,
$$
since $\th_c\le 13/48.$

\bigskip
We next turn to the main terms $\nu_{\de,0}^\la$  and the corresponding  operators $T_{\de,0}^\la.$ First we claim that 
\begin{equation}\label{10.11}
\|\widehat{\nu_{\de,0}^\la}\|_\infty \lesssim \rho^{\frac 7{12}}(\la\rho)^{-\frac 23}.
\end{equation}
This is  an immediate consequence of Lemma \ref{duistuniest}. Indeed our phase $\Phi_1$ in \eqref{10.3} is of the form required in this lemma, with $\rho$ in the lemma of size $1,$ but $\la$ in the lemma replaced by $\la\rho$ here, so that the oscillatory integral in the definition of $\widehat{\nu_{\de,0}^\la}$  can be estimated by $C/(1^{1/12}(\la\rho)^{2/3})$.


\medskip
We  finally want to estimate $\|\nu_{\de,0}^\la\|_\infty.$ In analogy with \eqref{10.8}, we may write $\nu_{\de,0}^\la(x)$ as an oscillatory integral of the form
$$
\nu_{\de,0}^\la(x)=\rho^{\frac 7{12}+\frac 23}\la^3\int e^{-i\la s_3 \Phi_2(u,z,s_2,\de)}
\chi_0(z) \chi_0(u) \tilde a(\si_\rho u, \rho^{\frac 23} z,s,\de) \tilde\chi_1(s_2, s_3)\, du dz ds_2 ds_3,
$$
with $\Phi_2$ given by \eqref{10.9}. We can reduce this to the following situation in which the amplitude is independent of $z,$ i.e., where

\begin{eqnarray}\nonumber
\nu_{\de,0}^\la(x)&=&\rho^{\frac 7{12}+\frac 23}\la^3\int e^{-i\la s_3 \Phi_2(u,z,s_2,\de)}\\
&&\times \chi_0(z) \chi_0(u) a(\si_\rho u, s,\rho^{\frac 13}, \de) \tilde\chi_1(s_2, s_3)\, du dz ds_2 ds_3.\label{nudon}
\end{eqnarray}
In fact,  we may  develop the amplitude $\tilde a$ into a convergent  series of smooth functions, each of which is a tensor product of  a smooth function of the variable $z$  with  a smooth function depending on the remaining variables only.  Thus, by considering each of the corresponding terms separately, we can reduce to the situation \eqref{nudon} (the function $\chi_0(z)$ will of course have to be different from the previous one).

\medskip

 We  claim that 
\begin{equation}\label{10.12}
\|\nu_{\de,0}^{\la}\|_\infty \lesssim  \rho^{\frac 7{12}} \la^2(\la\rho)^{-\frac 14}.
\end{equation}

Indeed, if $|x|\gg1,$ arguing in a similar way as for $\nu_{\de,\infty}^{\la}(x),$ we see that 
$|\nu_{\de,0}^{\la}(x)|\lesssim \rho^{\frac 7{12}+\frac 23}\la^3(\la\rho^{2/3})^{-N}$ for every $N\in\NN,$ which is stronger than what is needed for \eqref{10.12}.

\medskip
From now  on we shall therefore  assume that $|x|\lesssim 1.$  For such $x$ fixed,  we can argue in a  similar way as in  in Section 7.2 of \cite{IM-rest1} (compare also with Subsection \ref{tdII}):
We decompose 
\begin{equation}\label{nudedec}
\nu^\la_{\de,0}=\nu_{0,I}^\la+\nu_{0,II}^\la,
\end{equation}
where $\nu_{0,I}^\la$  and $\nu_{0,II}^\la$ denote the contributions to the integral \eqref{nudon}  by the  region $L_{I}$  given by 
$$
|x_1-s_2^{\frac {1}{n-2}}G_1(s_2,\de)|\gg \rho^\frac13, 
$$
and the region $L_{II}$  where 
$$
|x_1-s_2^{\frac {1}{n-2}}G_1(s_2,\de)|\lesssim \rho^\frac13, 
$$
respectively. Recall from \eqref{10.9} and \eqref{10.13} that 

\begin{eqnarray}\nonumber
&&\Phi_2(u,z,s_2,\de)= s_2^{\frac {n}{n-2}}G_5(s_2,\de)- x_1 s_2^{\frac {n-1}{n-2}}G_3(s_2,\de)- s_2x_2-x_3\\
&&\hskip2.5cm +z\rho\Big(\rho^{-\frac 13} \big(x_1-s_2^{\frac {1}{n-2}}G_1(s_2,\de)\big)-  u_1\Big)+\rho u_1^3\, B_3(s_2,\de_1,\rho^{\frac 13}u_1) \label{Phi2n}\\
&&\hskip-0.5cm+ \rho\Big(u_2^4 \,b(\si_\rho u,\de_0^\r,s_2)+\de'_{4} u_2^2\, \tilde\alpha_2(\rho^{\frac 13}u_1,\de_0^{\r},\,s_2)
+\de'_{3,0} u_2 \,\tilde\al_1(\de_0^{\r},s_2)
+\de'_{0} u_1u_2 \, \al_{1,1}(\rho^{\frac 13}u_1,\de_0^{\r},s_2)\Big).\nonumber
\end{eqnarray}
Let us change variables from $s_2$ first to $v:=x_1-s_2^{\frac {1}{n-2}}G_1(s_2,\de),$ and then to $w:=
\rho^{-\frac 13}v= \rho^{-\frac 13}\big(x_1-s_2^{\frac {1}{n-2}}G_1(s_2,\de)\big).$ In these new coordinates, $\Phi_2$ can be written as 
$$
\Phi_2=z\rho(w-u_1)+\Phi_3,
$$
with $\Phi_3$ of the form
\begin{eqnarray}\nonumber
&&\Phi_3(u,w,x,\de)= \Psi_3(\rho^{\frac 13} w,x,\de)+\rho u_1^3\,  B_3(\rho^{\frac 13} w,\rho^{\frac 13}u_1,x,\de)  \nonumber \\
&&\hskip3cm+ \rho\Big(u_2^4 \,b(\si_\rho u,\rho^{\frac 13} w,x,\de_0^\r)+\de'_{4} u_2^2\, \tilde\alpha_2(\rho^{\frac 13}u_1,\rho^{\frac 13} w,x,\de_0^{\r}) \label{Phi3n}\\
&&\hskip3cm+\de'_{3,0} u_2 \,\tilde\al_1(\rho^{\frac 13} w,x,\de_0^{\r})
+\de'_{0} u_1u_2 \, \al_{1,1}(\rho^{\frac 13}u_1,\rho^{\frac 13} w,x,\de_0^{\r})\Big)\, ,\nonumber
\end{eqnarray}
where $\Psi_3$ is a smooth, real-valued function. With a slight abuse of notation, we have here used the same symbols $B_3,b,\dots, \al_{1,1}$ as before, since these functions will have the same basic properties here as the corresponding  ones in \eqref{Phi2n}. A similar remark will apply to the amplitudes, which we shall always denote by the letter $a,$ even though they may change form line to line.
Moreover, we may write
\begin{eqnarray}\nonumber
\nu_{0,I}^\la(x)&=&(\rho^{\frac 7{12}+\frac 23}\la^3)\,\rho^{\frac 13}\int e^{-i\la s_3 \Phi_3(u,w,x,\de)}
 \widehat{\chi_0}\big(\la\rho s_3(w-u_1))\, (1-\chi_0(w)\big)\, \chi_0(u)\\
&&\hskip4cm\times a(\si_\rho u, \rho^{\frac 13} w,s_1,\rho^{\frac 13},x, \de) \tilde\chi_1(s_3)\, dw du  ds_3.\label{nudon2}\\
\nu_{0,II}^\la(x)&=&(\rho^{\frac 7{12}+\frac 23}\la^3)\,\rho^{\frac 13}\int e^{-i\la s_3 \Phi_3(u,w,x,\de)}
 \widehat{\chi_0}\big(\la\rho s_3(w-u_1))\, \chi_0(w)\, \chi_0(u)\nonumber \\
&&\hskip4cm\times a(\si_\rho u, \rho^{\frac 13} w,s_1,\rho^{\frac 13},x, \de) \tilde\chi_1(s_3)\, dw du ds_3.\label{nudon3}
\end{eqnarray}
Here, $\chi_0(w)$ will again denote a smooth function with compact support which is identically $1$ on a sufficiently large neighborhood of the origin.

Observe that in \eqref{nudon2} we have $|w|\gg 1\gtrsim |u_1|,$ so that $ |\widehat{\chi_0}\big(\la\rho s_3(w-u_1))|\le C_N(\la\rho|w|)^{-(N+1)}$ for every $N\in\NN,$ and we immediately obtain the estimate
\begin{equation}\label{nuIest}
\|\nu_{0,I}^\la(x)\|_\infty \le C_N (\rho^{\frac 7{12}+\frac 23}\la^3)\,\rho^{\frac 13}(\la\rho)^{-(N+1)}=
\rho^{\frac 7{12}}\la^2\,(\la\rho)^{-N},
\end{equation}
which is even stronger than \eqref{10.12}.

\medskip
In order to estimate the second term, we perform yet another change of variables from $u_1$ to $y_1$ so that $u_1=w -(\la\rho)^{-1} y_1,$ i.e.,  $y_1=\la\rho (w-u_1).$ This leads to the following expression for $\nu_{0,II}^\la(x):$

\begin{eqnarray}\nonumber
\nu_{0,II}^\la(x)&=&(\rho^{\frac 7{12}+\frac 23}\la^3)\,\rho^{\frac 13}\,(\la\rho)^{-1}  \int e^{-i\la s_3 \Phi_4(y_1,u_2,w,x,\de)}
 \widehat{\chi_0}(s_3y_1)\, \chi_0(w)\, \chi_0(w -(\la\rho)^{-1} y_1) \nonumber \\
&&\hskip1cm\times \chi_0(u_2)\, a_4\big((\la\rho^{\frac 23})^{-1} y_1,\rho^{\frac 14} u_2,w,s_1,\rho^{\frac 13},x, \de\big) \tilde\chi_1(s_3)\, dy_1du_2dw ds_3\, ,\label{nudon4}
\end{eqnarray}
with phase $\Phi_4$ of the form
\begin{eqnarray}\nonumber
&&\Phi_4(y_1,u_2,w,x,\de)= \Psi_3\big(\rho^{\frac 13}w, x,\de\big)+\rho \big(w -(\la\rho)^{-1} y_1\big)^3\, \tilde B_3\big(\rho^{\frac 13}w,(\la\rho^{\frac 23})^{-1} y_1,x,\de\big)  \nonumber \\
&&\hskip1.3cm+ \rho\Big(u_2^4 \,b\big(\rho^{\frac 13}w,(\la\rho^{\frac 23})^{-1} y_1,\rho^{\frac 14}u_2,x,\de_0^{\r}\big)+\de'_{4} u_2^2\, \tilde\alpha_2\big(\rho^{\frac 13}w,(\la\rho^{\frac 23})^{-1} y_1,\rho^{\frac 14}u_2,x,\de_0^{\r}\big) \label{Phi4n}\\
&&\hskip0.5cm+\de'_{3,0} u_2 \,\tilde\al_1\big(\rho^{\frac 13}w,x,\de_0^{\r}\big)
+\de'_{0} u_2 \,\big( w -(\la\rho)^{-1} y_1\big)\, \al_{1,1}\big(\rho^{\frac 13}w,(\la\rho^{\frac 23})^{-1} y_1,x,\de_0^{\r}\big)\Big)\, .\nonumber
\end{eqnarray}
Observe that in this integral, $|u_2|+|w|\lesssim 1$ and $|y_1|\lesssim \la\rho.$ 
Moreover, the factor $\widehat{\chi_0}(s_3y_1)$ guarantees the absolute convergence of this integral with respect to the variable $y_1.$ We can thus first apply van der Corput's estimate  of order  $M=4$ for the integration in $u_2,$ which leads to an additional factor of order $(\la\rho)^{-1/4},$ and then perform the remaining integrations in $w,y_1$ and $s_3.$ Altogether, this leads to the estimate
\begin{equation}\label{nuIIest}
\|\nu_{0,I}^\la(x)\|_\infty \le C(\rho^{\frac 7{12}+\frac 23}\la^3)\,\rho^{\frac 13}(\la\rho)^{-1)}(\la\rho)^{-\frac 14}=\rho^{\frac 7{12}}\la^2\,(\la\rho)^{-\frac 14}.
\end{equation}
In combination with \eqref{nuIest}, this proves \eqref{10.15}.

\medskip
Finally, interpolating between the estimates \eqref{10.11} and \eqref{10.12}, we obtain
$$
\|T_{\de,0}^\la\|_{p_c\to p_c'}\lesssim \rho^{\frac7{12}-2\th_c}(\la\rho)^{\frac{29}{12}\th_c-\frac 23}.
$$
But, since $\th_c\le 13/48,$ we have $\frac{29}{12}\th_c-\frac 23<0$  and $7/12-2\th_c>0,$which implies the desired estimate
$$
\sum_{\la\rho\gg 1} \|T_{\de,0}^\la\|_{p_c\to p'_c}\lesssim\rho^{\frac7{12}-2\th_c}\le 1.
$$

Altogether, we have thus proved that 
\begin{equation}\label{10.14}
\sum_{\la\rho\gg 1} \|T_{\de,Ai}^\la\|_{p_c\to p'_c}\lesssim 1.
\end{equation}

\bigskip

\subsection{Estimation of  $T_{\de,\, l}^\la$}  \label{Tdel}
Here we have $|(2^l\rho)^{-2/3}B_1(s,\,\de_1)|\sim 1.$ Recall also that $2^l\rho\le 1/M_1\ll 1.$ In this case, we use the change of variables $x=:\si_{2^l\rho} u:=((2^l\rho)^{1/3} u_1,\, (2^l\rho)^{1/4} u_2)$ in the
integral \eqref{7.12II} defining  $\widehat{\nu_\de},$ and obtain
\begin{equation} \label{10.15}
\widehat{\nu_{\de}}(\xi)=(2^l\rho)^\frac7{12}
\,e^{-i\la s_3 B_0(s,\de_1)} \int_{|\si_{2^l\rho} u|<\ve} e^{-i\la2^l\rho\, s_3
\Phi_1(u,s,\de)}a(\si_{2^l\rho} u,\de,s)\, du, 
\end{equation}
where now  the phase $\Phi_1=\Phi_{1,l}$  has the form
\begin{eqnarray}  \label{10.16}
&&\hskip0.3cm\Phi_1(u,s,\de)= u_1^3\, B_3(s_2,\de_1,(2^l\rho)^{\frac 13}u_1) 
-u_1  ((2^l\rho)^{-\frac 23}B_1(s,\de_1))  + u_2^4 \,b(\si_{2^l\rho} u,\de_0^\r,s_2)\\
&& +\de'_{4} u_2^2\, \tilde\alpha_2((2^l\rho)^{\frac 13}u_1,\de_0^{\r},\,s_2)
+\de'_{3,0} u_2 \,\tilde\al_1(\de_0^{\r},s_2)
+\de'_{0} u_1u_2 \, \al_{1,1}((2^l\rho)^{\frac 13}u_1,\de_0^{\r},s_2), \nonumber
\end{eqnarray}
and where according to Remark \ref{duist} (in which we choose $r:=1/(2^l\rho)$)  $\rho(\widetilde{\de'})=2^{-l},$ so that 
$$
(\de'_4)^2+(\de'_{3,0})^{\frac 43}+(\de'_0)^{\frac {12}5}= 2^{-l}\le \frac 1{M_0}\ll 1.
$$
(recall that the coefficient $\de'_{3,0}$ does not appear in Case ND, where $\al_{1,1}=0$). We have also indicated that the amplitude $a(\si_{2^l\rho} u,\de,s)$ is supported where $|\si_{2^l\rho} u|<\ve,$  and that  we may assume that $\ve>0$ is sufficiently small.

Observe that the second row in \eqref{10.16} is a small perturbation of the leading term, given by the first row.

\medskip
Again, we choose a cut-off function $\chi_0\in C_0^\infty(\RR^2)$ such that $\chi_0(u)=1$  for $|u|\le R,$ where $R$ will be chosen sufficiently  large, and further decompose
$$
\widehat{\nu_{\de}}(\xi)=\widehat{\nu_{\de,0}}(\xi)+\widehat{\nu_{\de,\infty}}(\xi),
$$
where now 
$$
\widehat{\nu_{\de,0}}(\xi):=(2^l\rho)^\frac7{12}e^{-i\la s_3 B_0(s,\de_1)} \int_{|\si_{2^l\rho} u|<\ve} e^{-i\la2^l\rho\, s_3
\Phi_1(u,s,\de))}a(\si_{2^l\rho} u,\de,s)\chi_0(u) \, du,
$$
and 
$$
\widehat{\nu_{\de,\infty}}(\xi):=(2^l\rho)^\frac7{12}e^{-i\la s_3 B_0(s,\de_1)} \int_{|\si_{2^l\rho} u|<\ve} e^{-i\la2^l\rho\, s_3
\Phi_1(u,s,\de))}a(\si_{2^l\rho} u,\de,s)(1-\chi_0(u)) \, du,
$$
 Accordingly, we decompose 
$$
\nu_{\de,l}^\la=\nu_{l,0}^\la+\nu_{l,\infty}^\la,
$$
where  we have put 
\begin{eqnarray*}
\widehat{\nu_{l,0}^\la}(\xi)&:=&\chi_1((2^l\rho)^{-\frac 23}B_1(s,\,\de_1))\chi_1(s,s_3) \,\widehat{\nu_{\de,0}}(\xi),\\
\widehat{\nu_{l,\infty}^\la}(\xi)&:=&\chi_1((2^l\rho)^{-\frac 23}B_1(s,\,\de_1))\chi_1(s,s_3) \,\widehat{\nu_{\de,\infty}}(\xi).
\end{eqnarray*}

The corresponding   operators of convolution with  $\widehat{\nu_{l,0}^\la}$ and $\widehat{\nu_{l,\infty}^\la}$ will  be denoted by  $T_{l,0}^\la$ and $T_{l,\infty}^\la,$ respectively. 

\medskip
Let us again first  consider the operators $T_{l,\infty}^\la:$ 
By means of integrations by parts, we easily  see that if $R$ is chosen sufficiently large, then the phase will have no critical point, and thus  for every $N\in\NN$ we have 
\begin{equation}\label{10.17}
\|\widehat{\nu_{l,\infty}^\la}\|_\infty \lesssim (2^l\rho)^{\frac 7{12}}(\la2^l\rho)^{-N}.
\end{equation}
Moreover, Fourier inversion again leads to \eqref{10.5}, and performing the change of variables  from $s=(s_1,s_2)$ to $(z,s_2),$ where here
$$
z:= (2^l\rho)^{-\frac 23}B_1(s,\de_1),
$$ 
we find that 
\begin{equation}\label{10.18}
s_1=s_2^{\frac {n-1}{n-2}}G_3(s_2,\de)-(2^l\rho)^{\frac 23}z,
\end{equation}
and in particular
 \begin{equation}\label{10.19}
B_0(s,\de,\si)=-(2^l\rho)^{\frac 23}z\,s_2^{\frac {1}{n-2}}G_1(s_2,\de)+s_2^{\frac {n}{n-2}}
G_5(s_2,\de).
\end{equation}
In a similar way as before, this leads to

\begin{eqnarray}\label{10}\label{10.20}
&&\hskip2cm\nu_{l,\infty}^\la(x)=(2^l\rho)^{\frac 7{12}+\frac 23}\la^3\\
&&\times\int e^{-i\la s_3 \Phi_2(u,z,s_2,\de)}
\chi_1(z)(1-\chi_0(u)) a(\si_{2^l\rho} u, (2^l\rho)^{\frac 23} z,s,\de) \tilde\chi_1(s_2, s_3)\, du dz ds_2 ds_3
\nonumber
\end{eqnarray}
with respect to the variables $u_1,u_2, z,s_2,s_3,$ where the  complete phase is now given by 
\begin{eqnarray}\nonumber
&&\Phi_2(u,z,s_2,\de):= s_2^{\frac {n}{n-2}} G_5(s_2,\de)- x_1 s_2^{\frac {n-1}{n-2}}G_3(s_2,\de)\\
&&\hskip2cm - s_2x_2-x_3+ (2^l\rho)^{2/3}z(x_1-s_2^{\frac {1}{n-2}}G_1(s_2,\de)+2^l\rho\,\Phi_1(u,z,s_2,\de).\label{10.21}
\end{eqnarray}
According to  \eqref{10.16}, the phase $\Phi_1$ is given in the new coordinates by 
\begin{eqnarray}  \label{10.22}
&&\hskip2cm\Phi_1(u,z,s_2,\de)= u_1^3\, B_3(s_2,\de_1,(2^l\rho)^{\frac 13}u_1) 
-u_1 z+ u_2^4 \,b(\si_{2^l\rho} u,\de_0^\r,s_2)\\
&& +\de'_{4} u_2^2\, \tilde\alpha_2((2^l\rho)^{\frac 13}u_1,\de_0^{\r},\,s_2)
+\de'_{3,0} u_2 \,\tilde\al_1(\de_0^{\r},s_2)
+\de'_{0} u_1u_2 \, \al_{1,1}((2^l\rho)^{\frac 13}u_1,\de_0^{\r},s_2). \nonumber
\end{eqnarray}
\smallskip

Arguing as in the preceding subsection, by applying  first   $N$ integrations by parts with respect to the variables  $u_1,u_2,$  and then van der Corput's lemma (of order $M=3$) to the integration in $s_2,$   we thus find  that 
$$
|\nu_{l,\,\infty}^{\la}(x)|\lesssim (2^l\rho)^{\frac 7{12}+\frac 23}  \la^3 (\la2^l\rho)^{-N} \la^{-\frac 13},
$$
first if $|x|\lesssim 1,$  and then also for   $|x|\gg 1,$ by the same kind of arguments. We thus obtain
\begin{equation}\label{10.23}
\|\nu_{l,\,\infty}^{\la}\|_\infty \lesssim  (2^l\rho)^{\frac 7{12}+\frac 23}  \la^{\frac 83} (\la 2^l\rho)^{-N}.
\end{equation}
Interpolating between this estimate and \eqref{10.17}, we find here that
$$
\|T_{l,\,\infty}^\la\|_{p_c\to p_c'}\lesssim (2^l\rho)^\frac7{12}(\la2^l\rho)^{-N}(2^l\rho)^\frac{2\th_c}3\la^\frac{8\th_c}3,
$$
which implies that 
$$
\sum_{\la\rho\gg 1} \|T_{l,\,\infty}^\la\|_{p_c\to p'_c}\lesssim
2^{(\frac 7{12} +\frac{2\th_c}3- N)l}\rho^{\frac7{12}-2\th_c}\le 1,
$$
hence the desired estimate 
$$
\sum_{\{l:M_0\le 2^l\le \frac{\rho^{-1}}{M_1}\}}\sum_{\la\rho\gg 1} \|T_{l,\,\infty}^\la\|_{p_c\to p'_c}\lesssim
\rho^{\frac7{12}-2\th_c}\le 1,
$$
since $\th_c\le 13/48.$

\bigskip
We next turn to the main terms $\nu_{l,0}^\la$  and the corresponding  operators $T_{l,0}^\la.$ First we show that 
\begin{equation}\label{10.24}
\|\widehat{\nu_{l,0}^\la}\|_\infty \lesssim (2^l\rho)^{\frac 7{12}}(\la2^l\rho)^{-\frac 34}.
\end{equation}
 Indeed,  \eqref{10.16} in combination with \eqref{bde} shows that the phase $\Phi_1$ is a small perturbation of the phase
$$
 \Phi_{1,0}(u,s):= u_1^3\, B_3(s_2,0,0) -c_1 u_1 + u_2^4 \,b_4(0) , 
$$
where   $c_1$ corresponds to a fixed value of $(2^l\rho)^{-2/3}B_1(s,\de_1),$ so that $|c_1|\sim 1.$ Now, this phase will either have a critical point with respect to the variable  $u_1$  (recall that $u\in\supp \chi_0$), and then we can apply the method of stationary phase to the $u_1$-integration, or, otherwise we may use integrations by parts in $u_1$ (for instance, if $c_1$ and $B_3$ have opposite signs).  Applying subsequently van der Corput's estimate (of order $M=4$) to the $u_1$-integration, we immediately get \eqref{10.24}.

\bigskip
We  finally want to estimate $\|\nu_{l,0}^\la\|_\infty.$ In analogy with \eqref{10.20}, we may write $\nu_{l,0}^\la(x)$
 as an oscillatory integral of the form
 \begin{eqnarray*}
&&\hskip2cm\nu_{l,0}^\la(x)=(2^l\rho)^{\frac 7{12}+\frac 23}\la^3\\
&&\times\int e^{-i\la s_3 \Phi_2(u,z,s_2,\de)}
\chi_1(z)\chi_0(u) a(\si_{2^l\rho} u, (2^l\rho)^{\frac 23} z,s,\de) \tilde\chi_1(s_2, s_3)\, du dz ds_2 ds_3
\end{eqnarray*}
with $\Phi_2$ given by \eqref{10.21}.  We can then basically argue as we did for $\nu_{\de,Ai}^\la$ in the previous subsection, only with $\rho$ replaced by $2^l\rho,$  and arrive correspondingly at the following analogue of estimate \eqref{10.12}:
\begin{equation}\label{10.25}
\|\nu_{l,0}^{\la}\|_\infty \lesssim  (2^l\rho)^{\frac 7{12}} \la^2(\la2^l\rho)^{-\frac 14}.
\end{equation}

\medskip
Finally, interpolating between the estimates \eqref{10.24} and \eqref{10.25}, we obtain
\begin{equation}\label{10.26}
\|T_{l,0}^\la\|_{p_c\to p_c'}\lesssim (2^l\rho)^{\frac 7{12}} \la^{2\th_c}(\la2^l\rho)^{\frac {\th_c}2-\frac 34}
=2^{\frac l 6(3\th_c-1)}(\la\rho)^{\frac{10\th_c-3}4}\rho^{\frac 7{12}-2\th_c}.
\end{equation}
But, since $\th_c\le 13/48,$ we have $10\th_c-3<0$  and $3\th_c-1<0,$ which implies 
 the desired estimate
 $$
\sum_{\{l:M_0\le 2^l\le \frac{\rho^{-1}}{M_1}\}}\sum_{\la\rho\gg 1} \|T_{l,0}^\la\|_{p_c\to p'_c}\lesssim 1.
$$
Altogether, we have thus proved that 
\begin{equation}\label{10.27}
\sum_{\{l:M_0\le 2^l\le \frac{\rho^{-1}}{M_1}\}}\sum_{\la\rho\gg 1} \|T_{\de,l}^\la\|_{p_c\to p'_c}\lesssim 1.
\end{equation}

This completes the proof of Proposition \ref{biglaB4} also for the case where $A=0.$

\bigskip

\setcounter{equation}{0}
\section{The case where $m=2,$ $B=3$  and $A=0:$ What still needs to be done}\label{B3A0}

What remains open in \eqref{5.17II}, hence also in the proof  of Proposition \ref{rdomain}, is the case where $B=3$ and $A=0,$ in which $\th_c=1/3$ and $p_c=6/5.$ Notice also that in  this case \eqref{rho} means that
\begin{equation}\nonumber
\rho:=\begin{cases}   \de_{3,0}^{\frac 3{2}}&   \mbox{in Case ND},\\
 \de_0^3+\de_{3,0}^{\frac 3{2}} &       \mbox{in Case D},
\end{cases}
\end{equation}
where in Case ND $\de_{3,0}\ge \de_0.$ This shows that $\rho\ge \de_0^3$ in both cases.

\medskip
Let us first observe  that estimate \eqref{4.3II} shows that we  ``trivially'' have 
$$
\sum_{\la\gtrsim\de_0^{-3}}\|T_\de^\la\|_{p_c\to p'_c}\lesssim1,
$$
since $B=3$ and $\th_c=1/3.$ 
In the sequel, we may and shall therefore always assume that $\la\ll\de_0^{-3}.$

\medskip
According to our discussion in Subsection \ref{larhos},  if  $\la\rho\lesssim1$  (where $\rho=\rho(\tilde\de)$), the 
endpoint $p=p_c$ is still left open. 
\medskip

On the other hand, if $\la\rho\gg1,$ we can basically follow our approach from the previous section, with only minor modifications. Let us describe  some more details.

\medskip
 Again, we perform  the ``Airy-cone decomposition''  \eqref{airydecomp}. In order to estimate $T_{\de,\, Ai}^\la$
 we may  here  use the scaling $x=:\si_\rho u:=(\rho^{1/3} u_1,\, \rho^{1/3} u_2),$ which leads to 
\begin{equation}\label{11.1}
\widehat{\nu_\de}(\xi)=\rho^\frac2{3}e^{-i\la s_3 B_0(s,\de_1)} \int_{|\si_\rho u|<\ve} e^{-i\la\rho\, s_3
\Phi_1(u,s,\de)}a(\si_\rho u,\de,s)\, du,
\end{equation}
where  the phase $\Phi_1$ now  has the form
\begin{eqnarray}  \label{11.2}
&&\hskip1cm\Phi_1(u,s,\de)= u_1^3\, B_3(s_2,\de_1,\rho^{\frac 13}u_1) 
-u_1  (\rho^{-\frac 23}B_1(s,\de_1))\\
&& + u_2^3 \,b(\si_\rho u,\de_0^\r,s_2)
+\de'_{3,0} u_2 \,\tilde\al_1(\de_0^{\r},s_2)
+\de'_{0} u_1u_2 \, \al_{1,1}(\rho^{\frac 13}u_1,\de_0^{\r},s_2), \nonumber
\end{eqnarray}
in place of \eqref{10.2} and \eqref{10.3}, and where
$$
\de'_0+\de'_{3,0}\sim 1.
$$
Recall that the coefficient $\de'_{0}$ does not appear in Case ND, in which  $\al_{1,1}= 0.$ If we again decompose 
$$
\nu_{\de,\, Ai}^\la=\nu_{\de,0}^\la+\nu_{\de,\infty}^\la,
$$
then one finds here that, in place of \eqref{10.4} and \eqref{10.10}, we obtain $\|\widehat{\nu_{\de,\infty}^\la}\|_\infty \lesssim \rho^{ 2/3}(\la\rho)^{-N}$ and $\|\nu_{\de,\,\infty}^{\la}\|_\infty \lesssim  \rho^{4/3} \la^{8/3}(\la\rho)^{-N}$, for every $N\in\NN.$ By interpolation, this leads to 
$$
\|T_{\de,\,\infty}^\la\|_{p_c\to p_c'}\lesssim (\la\rho)^{\frac 89-N},
$$
which implies the desired estimate
\begin{equation}\label{11.3}
\sum_{\la\rho\gg 1} \|T_{\de,\,\infty}^\la\|_{p_c\to p'_c}\lesssim 1.
\end{equation}

As for the main term $\nu_{\de,0}^\la,$ a similar type of discussion that led to \eqref{10.11} here yields the following estimate:
\begin{equation}\label{11.4}
\|\widehat{\nu_{\de,0}^\la}\|_\infty \lesssim \rho^{\frac 2{3}}(\la\rho)^{-\frac 12-\frac 13}= \rho^{\frac 2{3}}(\la\rho)^{-\frac 56}.
\end{equation}
Indeed, recall that we are assuming that $\la\rho\gg 1$ and $\la\ll \de_0^{-3},$ so that $\rho\gg \de_0^3.$ The estimate \eqref{11.4} therefore follows from the following analog of Lemma \ref{duistuniest} for the case where $B=3:$

\begin{lemma}\label{duistuniest3}
Denote by $J(\la,\de,s)$ the oscillatory
$$
J(\la,\de,s)= \chi_1(s_1,s_2)\iint e^{-i\la\Phi(x,\de,s)} \, a(x,\de,s)\, dx_1dx_2,
$$
with phase 
$$
\Phi(x,\de,s)=x_1^3\, B_3(s_2,\de_1,x_1) -x_1 B_1(s,\de_1) +\phi^\sharp(x_1,x_2,\de,\,s_2),
$$
where 
$$
\phi^\sharp(x,\de,\,s_2):= x_2^3 \,b(x,\de_0^r,s_2)
+\de_{3,0} x_2 \,\tilde\al_1(x_1,\de_0^{\r},s_2) +\de_{0} x_1x_2 \, \al_{1,1}(x_1,\de_0^{\r},s_2), 
$$
and where $ \chi_1(s_1,s_2)$ localizes to the region where $s_j\sim 1, j=1,2.$ Let us also put 
\begin{eqnarray}\nonumber
\tilde\rho:=\rho+|B_1(s,\de_1)|^\frac32,
\end{eqnarray}
and assume  that $\tilde \rho\ge M \de_0^3.$ Then the following estimate
\begin{equation}\label{duistypeg3}
|J(\la,\de,s)|\le \frac{C}{\tilde\rho^\frac1{6}\la^\frac56}
\end{equation}
 holds true, provided  the amplitude  $a$ is supported in a sufficiently small
neighborhood of the origin and $M\ge1$ is sufficiently large. The constant $C$ in this estimate is independent of $\de$  and $s.$
\end{lemma}
\begin{remark}\label{rem11.2II}
Without the assumption $\tilde \rho\gg \de_0^3,$ estimate \eqref{duistypeg3} may fail. Indeed, in the  worst case, it may happen that, after re-scaling, all of the quantities $B_1'(s,\de_1),\de_0'$ and $ \de_{3,0}'$ are of size $1,$ and  a degenerate critical point of order $4$  arises for the integration in $u_2,$ after we have applied the method of stationary phase in $u_1.$ In this case, we will only obtain an estimate $|I(\La,\de,s)|\le C \La^{-1/2-1/4},$ and  this estimate will be sharp. 
\end{remark}
\begin{proof} 
The proof is analogous to the proof of Lemma \ref{duistuniest}. In the more difficult case where $\La:=\tilde\rho\la\gtrsim 1,$ after applying the change of coordinates $x=\tilde\rho^{1/3}u,$ we see that it suffices to prove an estimate of the form $|I(\La,\de,s)|\le C \La^{-5/6},$ where the oscillatory integral $I(\La,\de,s)$ is as in \eqref{9.5II}, only with $\phi$ in the  phase $\Phi_1$ replaced by   
$$
\phi(u,\tilde\rho,\de,\,s_2):= u_2^3 \,b({\tilde\rho}^{\frac 13}u,\de_0^r,s_2)
+\de'_{3,0} u_2 \,\tilde\al_1(\tilde\rho^{\frac 13}u_1,\de_0^{\r},s_2)
+ \de'_{0} u_1u_2 \, \al_{1,1}(\tilde\rho^{\frac 13}u_1,\de_0^{\r},s_2)
$$
(compare \eqref{phiII})
Here, 
$$
B_1'(s,\de_1):=\frac{B_1(s,\de_1)}{\tilde\rho^{\frac23}},\
\de_0':=\frac{\de_0}{\tilde\rho^\frac 13},\ \de_{3,0}':=\frac{\de_{3,0}}{\tilde\rho^{\frac 23}}.
$$
Notice that our  assumption implies that  $\de_0'\ll 1,$  and then
$$
(\de'_{3,0})^{\frac 32}+|B'_1(s,\de_1)|^{\frac 32}\sim 1.
$$
This shows that the phase $\Phi_1$ will have at worst an Airy type singularity in one of the variables $u_1$ or $u_2.$ Applying thus first the method of stationary phase to the integration in one of these variables, and subsequently van der Corput's  estimates of order $3$ to the integration in the second variable, we arrive at the desired estimate for $I(\La,\de,s).$
\end{proof}

Next, in order to estimate $\nu_{\de,0}^\la(x),$ recall that 
\begin{equation}\label{nude01}
\nu_{\de,0}^\la(x)=\rho^{\frac 23+\frac 23}\la^3\int e^{-i\la s_3 \Phi_2(u,z,s_2,\de)}
\chi_0(z) \chi_0(u) a(\si_\rho u, \rho^{\frac 23} z,s,\de) \tilde\chi_1(s_2, s_3)\, du dz ds_2 ds_3,
\end{equation}
with $\Phi_2$ given by \eqref{10.9}, and in place of \eqref{10.12}  we now  get
\begin{equation}\label{11.5}
\|\nu_{\de,0}^{\la}\|_\infty \lesssim  \rho^{\frac 23} \la^2(\la\rho)^{-\frac 13}=\rho^{-\frac 43} (\la\rho)^{\frac 53}.
\end{equation}
By means of interpolation, we arrive from \eqref{11.4}  and \eqref{11.4}  at   a uniform estimate
\begin{equation}\label{}
\|T_{\de,0}^\la\|_{p_c\to p_c'}\lesssim 1.
\end{equation}
The corresponding estimate for $p<p_c$ allows for summation over all dyadic $\la\gg1,$ but in order to reach also the endpoint $p=p_c,$ similar to our discussion in \cite{IM-rest1} for the case $B=2,$ we shall have to apply a complex interpolation  argument.

\medskip 
For the estimation of the operators $T_{\de,l}^\la,$ very similar statements hold true (compare the analogous discussion in Subsection \ref{Tdel}). 

Scaling here $x=\si_{2^l\rho} u:=((2^l\rho)^{1/3} u_1,\, (2^l\rho)^{1/3} u_2)$ leads to 
\begin{equation} \label{11.7}
\widehat{\nu_{\de}}(\xi)=(2^l\rho)^{\frac 23}
\,e^{-i\la s_3 B_0(s,\de_1)} \int_{|\si_{2^l\rho} u|<\ve} e^{-i\la2^l\rho\, s_3
\Phi_1(u,s,\de)}a(\si_{2^l\rho} u,\de,s)\, du, 
\end{equation}
where now  the phase $\Phi_1=\Phi_{1,l}$  has the form
\begin{eqnarray}  \label{11.8}
&&\hskip0.3cm\Phi_1(u,s,\de)= u_1^3\, B_3(s_2,\de_1,(2^l\rho)^{\frac 13}u_1) 
-u_1  ((2^l\rho)^{-\frac 23}B_1(s,\de_1))  + u_2^3 \,b(\si_{2^l\rho} u,\de_0^\r,s_2)\\
&&\hskip2cm+\de'_{3,0} u_2 \,\tilde\al_1(\de_0^{\r},s_2)
+\de'_{0} u_1u_2 \, \al_{1,1}((2^l\rho)^{\frac 13}u_1,\de_0^{\r},s_2), \nonumber
\end{eqnarray}
and where 
$$(\de'_{3,0})^{\frac 32}+(\de'_0)^3= 2^{-l}\le \frac 1{M_0}\ll 1.$$
Moreover, in analogy with \eqref{11.3}, one easily verifies that also 
\begin{equation}\label{11.9}
\sum_{\{l:M_0\le 2^l\le \frac{\rho^{-1}}{M_1}\}}\sum_{\la\rho\gg 1} \|T_{l,\infty}^\la\|_{p_c\to p'_c}\lesssim 1.
\end{equation}

As for the main terms $\nu_{l,0}^\la$  of 
$$\nu_{\de,l}^\la=\nu_{l,0}^\la+\nu_{l,\infty}^\la,$$
which   is here given by 
$$
\widehat{\nu_{l,0}}(\xi):=(2^l\rho)^{\frac 23}e^{-i\la s_3 B_0(s,\de_1)} \int_{|\si_{2^l\rho} u|<\ve} e^{-i\la2^l\rho\, s_3
\Phi_1(u,s,\de)}a(\si_{2^l\rho} u,\de,s)\chi_0(u) \, du,
$$
we  find that in place of \eqref{10.24} we now have 
\begin{equation}\label{11.10}
\|\widehat{\nu_{l,0}^\la}\|_\infty \lesssim (2^l\rho)^{\frac 23}(\la2^l\rho)^{-\frac 56}.
\end{equation}
Moreover, after changing coordinates , we may write 
\begin{eqnarray}\label{nudel1}
&&\hskip2cm\nu_{l,0}^\la(x)=(2^l\rho)^{\frac 23+\frac 23}\la^3\\
&&\times\int e^{-i\la s_3 \Phi_2(u,z,s_2,\de)}
\chi_1(z)\chi_0(u) a(\si_{2^l\rho} u, (2^l\rho)^{\frac 23} z,s,\de) \tilde\chi_1(s_2, s_3)\, du dz ds_2 ds_3\nonumber
\end{eqnarray}
with $\Phi_2$ given by 
\begin{eqnarray}\nonumber
&&\Phi_2(u,z,s_2,\de):= s_2^{\frac {n}{n-2}}G_5(s_2,\de)- x_1 s_2^{\frac {n-1}{n-2}}G_3(s_2,\de)\\
&&\hskip2cm - s_2x_2-x_3+ (2^l\rho)^{2/3}z(x_1-s_2^{\frac {1}{n-2}}G_1(s_2,\de)+2^l\rho\,\Phi_1(u,z,s_2,\de).\label{11.11}
\end{eqnarray}
and
\begin{eqnarray}  \label{11.12}
&&\hskip2cm\Phi_1(u,z,s_2,\de)= u_1^3\, B_3(s_2,\de_1,(2^l\rho)^{\frac 13}u_1) 
-u_1 z+ u_2^3 \,b(\si_{2^l\rho} u,\de_0^\r,s_2)\\
&&\hskip 2cm +\de'_{3,0} u_2 \,\tilde\al_1(\de_0^{\r},s_2)
+\de'_{0} u_1u_2 \, \al_{1,1}((2^l\rho)^{\frac 13}u_1,\de_0^{\r},s_2). \nonumber
\end{eqnarray}
This leads to the estimate 
\begin{equation}\label{11.13}
\|\nu_{l,0}^{\la}\|_\infty \lesssim  (2^l\rho)^{\frac 23} \la^2(\la2^l\rho)^{-\frac 13}.
\end{equation}
Interpolating between this estimate and \eqref{11.10}, we obtain again only a uniform estimate
\begin{equation}\label{}
\|T_{l,0}^\la\|_{p_c\to p_c'}\lesssim 1
\end{equation}
(whereas the corresponding estimate for $p<p_c$ allows for summation over all dyadic $\la\gg1$ and all $l.$) So again, we shall have to apply a complex interpolation  argument in order to include the endpoint $p=p_c.$ 
\medskip

For the operators $T^\la_{l,\infty},$ we get a better estimate of the form $\|T_{l,\infty}^\la\|_{p_c\to p_c'}\le C_N(\la 2^l\rho)^{-N},$ which allows to sum absolutely in $l$ and $\la.$ 
\medskip

 Recall also that we have seen that we may restrict  ourselves to those $\la$ for which  $\la\ll\de_0^{-3}.$  This assumption has the additional  advantage  that we shall have to deal only with finite sums in Proposition \ref{s11.1}. Notice also that $\rho^{-1}\le \de_0^{-3},$ since we have seen that $\rho\ge\de_0^3.$ 

\medskip
Taking into account these observations, the following Proposition puts together those estimates which still need to be established in order to complete the proof of  Proposition \ref{rdomain},  and hence also that of our main result in \cite{IM-rest1}, Theorem 1.7:
\begin{proposition}\label{s11.1}
Assume that $m=2,B=3$ and $A=0$ in \eqref{5.17II}, so that $\th_c=1/3, p_c=6/5$ and $n\ge 7.$ Then the following hold true, provided $M\in\NN$ is sufficiently large and $\de$ sufficiently  small:

 \begin{itemize}
\item[(a)] If $\la\rho\lesssim 1,$ and if $\nu_{\de,Ai}^\la$ and $\nu_{\de,l}^\la$ are given by \eqref{aismall}, respectively \eqref{lsmall}, then let 
$$
\nu^{I}_{\de,Ai}:=\sum_{2^M\le\la\le 2^M\rho^{-1}}\nu_{\de,Ai}^\la \quad \mbox{and}\quad
\nu^{II}_{\de}:=\sum_{2^M\le \la\le 2^M\rho^{-1}}\,\sum_{\{l:M_0\le 2^l\le \frac{\la}{M_1}\}}\nu_{\de,l}^\la,
$$
and denote by $T^{I}_{\de,Ai}$ and $T^{II}_{\de}$ the convolution operators $\vp\mapsto \widehat{\vp*\nu^I_{\de,Ai}}$ and $\vp\mapsto \widehat{\vp*\nu^{II}_{\de}},$ respectively. Then
\begin{equation}\label{11.17}
\|T^{I}_{\de,Ai}\|_{p_c\to p'_c}\le C\quad \mbox{and}\quad \|T^{II}_{\de}\|_{p_c\to p'_c}\le C.
\end{equation}

\item[(b)] If $\la\rho\gg 1,$ and if $\nu_{\de,0}^\la$  and $\nu_{l,0}^\la$ denote the main terms of  of $\nu_{\de,Ai}^\la,$ respectively $\nu_{\de,l}^\la$   (cf. \eqref{nude01},  \eqref{nudel1}), then let
$$
\nu^{III}_{\de,Ai}:=\sum_{ 2^M\rho^{-1}< \la\le 2^{-M}\de_0^{-3}}\nu_{\de,0}^\la \quad \mbox{and}\quad 
\nu^{IV}_{\de}:=\sum_{\{l:M_0\le 2^l\le \frac{\rho^{-1}}{M_1}\}}\sum_{2^M\rho^{-1}< \la\le 2^{-M}\de_0^{-3}}\nu_{l,0}^\la,
$$and denote by $T^{III}_{\de,Ai}$ and $T^{IV}_{\de}$ the convolution operators $\vp\mapsto \widehat{\vp*\nu^{III}_{\de,Ai}}$ and $\vp\mapsto \widehat{\vp*\nu^{IV}_{\de}},$ respectively. Then
\begin{equation}\label{11.18}
\|T^{III}_{\de,Ai}\|_{p_c\to p'_c}\le C\quad \mbox{and}\quad \|T^{IV}_{\de}\|_{p_c\to p'_c}\le C.
\end{equation}
\end{itemize}
Here, the constant $C$ is independent of $\de.$
\end{proposition}

 \bigskip

\setcounter{equation}{0}
\section{Proof of Proposition \ref{s11.1} (a) : Complex interpolation}\label{propfinal}
In this section, we  assume that $B=3$ and $A=0$  in \eqref{5.17II},  and that  $\la\rho\lesssim1.$
 
 \subsection{Estimation of $T^I_{\de,Ai}$}\label{deai}
 Recall formula \eqref{hatnuai} for $\widehat{\nu_{\de,Ai}^\la}.$  Applying the Fourier inversion formula to this expression and performing the change of variables $z:=\la^\frac23B_1(s,\,\de),$ we find that we may write
 
\begin{equation}\label{12.1}
\nu_{\de,Ai}^\la(x)=\la^\frac53 \int e^{-is_3\la\Phi(z,\,s_2,\,x,\,\de)}\, 
a(z, s_2,\,\tilde\de^\la,\,\de_0^{\r},\la^{-\frac 1{9}})\chi_0(z)\tilde\chi_1(s_2)\chi_1(s_3)\, dzds_2ds_3,
\end{equation}
where $a$ is again  a smooth function of all its (bounded) variables, and where 
$$
\Phi(z,s_2,x,\de):=\phi(s_2,\,x,\de)+\la^{-\frac23}z(x_1-s_2^\frac1{n-2}G_1(s_2,\,\de)),
$$
with
$$
\phi(s_2,\,x,\de):=s_2^\frac{n}{n-2}G_5(s_2,\,\de)-s_2^\frac{n-1}{n-2}G_3(s_2,\,\de)x_1
-s_2x_2-x_3
$$
\medskip \noi
 (compare  with a similar discussion in Section 6 of \cite{IM-rest1}, in particular with  (6.22)). Recall also that $G_3(0)\ne 0$ and $G_5(0)=(G_1G_3-G_2)(0)\ne 0.$ 

In fact, a priori we have to assume that the density $a$ depends also on the variable $s_3.$ However, arguing  as in Subsection \ref{Aig}, we may develop this function into a convergent  series of smooth functions, each of which is a tensor product of  a smooth function of the variable $s_3$  with  a smooth function depending on the remaining variables only.  Thus, by considering each of the corresponding terms separately, we can reduce to the situation \eqref{12.1}, provided we choose the functions $\tilde\chi_1$ and $\chi_1,$  which localize to the regions where $|s_2|\sim 1$ and $|s_3|\sim 1,$ properly.

\medskip

We next embed  $\nu_{\de, Ai}^I$ into an analytic  family of measures 
$$
\nu_{\de,\,\zeta}^I:=\ga(\ze) \sum_{2^M\le \la\le 2^M\rho^{-1}}
\la^{\frac 23 (1-3\ze)}\nu_{\de,Ai}^\la
$$
where $\ze$ lies again in the complex strip $\Sigma$ given by   $0\le \Re \zeta\le1,$  and where $\ga(\ze):=
(1-2^{2(1-\zeta)})/(1-2^{4/3}).$

Here, summation is again  over dyadic $\la=2^j, j\in \NN.$ Observe that indeed $\nu_{\de, Ai}^I=\nu_{\de,\,\th_c}^I,$ since $\th_c=1/3.$ 

Since the supports of the $\widehat{\nu_{\de,\,Ai}^{\la}}$ are almost disjoint, \eqref{analog11} implies that 
$$
\|\widehat{\nu_{\de,\,it}^I}\|_\infty\lesssim 1\qquad \forall t\in\RR.
$$

We shall also prove that 
\begin{equation}\label{12.2}
|\nu_{\de,\,1+it}^I(x)|\le C \qquad \forall t\in\RR, x\in \RR^3.
\end{equation}
Again, Stein's interpolation  theorem will then imply that the operator $T^I_{\de,Ai}$ is bounded from $L^{p_c}$ to 
$L^{p'_c},$ which will complete the first part of Proposition \ref{s11.1} (a).
 
\medskip

In order to prove \eqref{12.2}, we first consider the case where $|x|\gg 1.$ In this case, we can use a formula similar to (6.18) in  \cite{IM-rest1}  in order to  argue  as in Subsection 6.1 of \cite{IM-rest1} and find that 
$$
|\nu_{\de,Ai}^\la(x)|\le C_N \la^{-N}, \quad N\in\NN, \mbox { if } |x|\gg 1.
$$
This estimate allows in a trivial way to sum in $\la$ and thus to  obtain \eqref{12.2}. 
\medskip 

We may therefore assume from now on  that $|x|\lesssim 1.$ We then write
\begin{equation}\label{12.3}
\nu_{\de, 1+it}^I(x)=\ga(1+it) \sum_{2^M\le \la\le 2^M\rho^{-1}}
\la^{-2it}\mu_\la(x),
\end{equation}
where 
$$
\mu_\la(x):=\la^\frac13 \int e^{-is_3\la\Phi(z,\,s_2,\,x,\,\de)}\, 
a(z, s_2,\,\tilde\de^\la,\,\de_0^{\r},\la^{-\frac 1{9}})\chi_0(z)\tilde\chi_1(s_2)\chi_1(s_3)\, dzds_2ds_3.
$$

Let us look at the contribution to this integral given by a small neighborhood of a given point $s_2^0 \sim 1.$ If 
$|\pa_{s_2}^2\phi(s_2^0,x,\de)|\sim 1,$ then van der Corput's estimate applied to the integration in $s_2$ shows that 
$|\mu_\la(x)|\lesssim \la^{1/3}\la^{-1/2}=\la^{-1/6},$ which clearly implies \eqref{12.2}. This situation arises in particular when $|x_1|\ll1,$ or when $|x_1|\sim 1$ and $G_5$ and $G_3 x_1$ have opposite signs.

\medskip
So, let us assume from now on that $|x_1|\sim 1,$ and that $G_5$ and $x_1G_3 $ have the same sign.
Notice that if $\de=0,$ then
$$
\phi(s_2,\,x,\,0)=s_2^\frac{n}{n-2}G_5(0)-s_2^\frac{n-1}{n-2}G_3(0)x_1
-s_2x_2-x_3.
$$ 
Since the exponents $n/(n-1)$ and $(n-1)/(n-2)$ are different, this   shows that there is a unique $s_2^c(0)\sim 1$ so that $\pa_{s_2}^2\phi(s_2^c(0),x,\,0)=0,$  whereas $|\pa_{s_2}^3\phi(s_2^c(0),x,\,0)|\sim 1.$ By the implicit function, we then find a smooth function  $s_2^c(x_1,\de)$ such that 
$$
\pa_{s_2}^2\phi(s_2^c(x_1,\de),x,\de)\equiv 0.
$$

Let us put here 
$$
\phi^\sharp(v,x,\de):= \phi(s_2^c(x_1,\de)+v,x,\de), \qquad \Phi^\sharp(z,v,x,\de):=\phi(z,s_2^c(x_1,\de)+v,x,\de).
$$
By means of Taylor expansion around $v=0,$ we may write
$$
\phi^\sharp(v,x,\de)=v^3 Q_3(v,x,\de)- v Q_1(x,\de)+ Q_0(x,\de),
$$
where  the $Q_j$ are smooth functions of all their variables, and where we may assume that $|Q_3(v,x,\de)|\sim 1,$ 
since we had $|\pa_{s_2}^3\phi(s_2^c(0),x,\,0)|\sim 1$. Moreover, developing 
$$
x_1-(s_2^c(x_1,\de)+v)^\frac1{n_2-2}G_1(s_2^c(x_1,\de)+v,\,\de_1)=Q_5(x,\de) +v Q_6(v,x,\de),
$$
after scaling $v\mapsto \la^{-1/3}v,$ we find that 
\begin{eqnarray*}
 \la\Phi^\sharp(z,\la^{-\frac 13}v,x,\de)&=&v^3 Q_3(\la^{-\frac 13}v,x,\de)- v \la^{\frac 23}Q_1(x,\de)+ \la Q_0(x,\de)\\
 &&\hskip1cm+ zv\,Q_6(\la^{-\frac 13}v,x,\de) +\la^{\frac 13} z\,Q_5(x,\de).
\end{eqnarray*}
This allows to re-write 
\begin{equation}\label{12.4}
\mu_\la(x)=\int\int_{ -\la^{1/3}}^{\la^{1/3}}\chi_0(z) F_\de(\la,x,z,v) \, dv dz,
\end{equation}
where $F_\de$ is of the form
\begin{eqnarray*}
F_\de(\la,x,z,v)&=&\chi_0(\la^{-\frac 13}v) \,  \widehat {\chi_1}\Big(A+ Dz -B v- v^3 Q_3(\la^{-\frac 13}v,x,\de) +zv\,Q_6(\la^{-\frac 13}v,x,\de)\Big)\\
&&\hskip1cm \tilde a(z,\la^{-\frac 13} v,\,\tilde\de^\la,\, x_1,\,\de_0^{\r},\la^{-\frac 1{9}}).
\end{eqnarray*}
Here,  $\tilde a$ is again a smooth function of all its  variables with compact support, and 
\begin{eqnarray*}
&& A=A(x,\la,\de):=\la Q_0(x,\de) , \qquad B:=B(x,\la,\de):=\la^{\frac 23} Q_1(x,\de),\\
&& D=D(x,\la,\de):=\la^{\frac 13} Q_5(x,\de).
\end{eqnarray*}
Now we can argue in a very similar way is in previous proofs, and will therefore an briefly sketch the proof.

 We first consider the contribution to the sum in \eqref{12.3} given by the $\la$'s satisfying $|A|=\la |Q_0(x,\de)|\gg 1.$  Observe that for $z$ fixed, we may estimate 
$\int_{ -\la^{1/3}}^{\la^{1/3}}| F_\de(\la,x,z,v)| \, dv$ by means of Lemma \ref{as1}, where we choose $y_2=v,$ $T:=\la^{1/3},$  $\eps=0, r_i\equiv 0$ (so that the integral in $y_1$ just yields a positive constant) and $Q(y_2)=zQ_6(y_2,x,\de).$ The condition \eqref{a1} is here also satisfied, since $(\phi^\sharp)''(0,x,\de)=0$ and $|(\phi^\sharp)'''(0,x,\de)|\sim 1.$ Thus, Lemma \ref{as1} implies that for $L$ sufficiently large, we may estimate
$$
|\mu_\la(x)|\le C\Big( \int_{I}|A+Dz|^{-\frac 16}\, |\chi_0(z)|\, dz+|J|\Big),
$$
where $I$ and $J$ denote the sets of all $z\in \supp \chi_0 $ for which  $|A+Dz|\ge L,$  respectively $|A+Dz|< L.$
The integral can easily be estimated by 
$$
|D|^{-\frac 16}\int_{I}|z+\frac A D|^{-\frac 16}\, |\chi_0(z)|\, dz\lesssim |D|^{-\frac 16}(1+|\frac A D|)^{-\frac 16}
\lesssim |A|^{-\frac 16}.
$$
Moreover, if $|D|\ll L,$ then the set $J$ is empty, if we assume that $|A|\ge 2L.$ So, let us assume that 
$|D|\gtrsim L.$  Then, on the set $J$ we have $|z+A/D|<L/|D|,$ and since $|z|\lesssim 1,$ this implies that 
$|A/D|\lesssim 1,$  and we see that $|J|\le L/|D|\lesssim  L/|A|.$  Putting all this together, we find that 
$$
|\mu_\la(x)|\le C|A|^{-\frac 16},
$$
which allows to sum over  those  $\la$ for which $\la |Q_0(x,\de)|\gg 1$ so that the estimate is independent of $x.$ 

Next, we consider the $\la$'s  for which $|A|\lesssim 1 .$ If in addition $|D|\gg 1,$ then we can argue as before and obtain an estimate of the form $|\mu_\la(x)|\le C|D|^{-1/6},$ which again allows to sum. Similarly, if $|A|\lesssim 1$  and $|D|\lesssim 1,$  but $|B|\gg 1,$ we may  apply Lemma \ref{as1} once more and obtain that $|\mu_\la(x)|\le C|B|^{-1/4}.$ This allows again to sum.

\medskip We are thus left with the  oscillatory sum \eqref{12.3} over  only those $\la$'s  for which, say, 
$\max\{|A|,|B|,|D|\}\le L.$  However, this can again easily be handled by means of Lemma \ref{simplesum}, and we arrive at \eqref{12.2}. Recall here  that by \eqref{deldil}  the components of $\tilde\de^\la$ are of the form 
$\la^{\be_i}\de_i,$  where we may assume that $|\la^{\be_i}\de_i|\le C,$ since we are assuming that $\rho(\tilde\de^\la)=\la\rho(\tilde\de)\lesssim 1.$

\subsection{Estimation of $T_\de^{II}$}\label{tdII}
  We next come to the proof   the  second estimate  in Proposition \ref{s11.1} (a), where  we still assume  that $\la\rho\lesssim1.$ 
Recall from Subsection \ref{larhos} that we have decomposed 
\begin{equation}\label{12.5}
\nu^\la_{\de,l}=\nu^\la_{l,\infty }+\nu^\la_{l,0,0 }+\nu_{l,I}^\la+\nu_{l,II}^\la+\nu_{l,III}^\la.
\end{equation}
Here, we denote by $\nu^\la_{l,0,0 }$ the contribution to $\nu_{l,0}^\la$ by the domain where $|u_1|\ll 1.$   

Again, we  embed  the measure $\nu_{\de}^{II}$ into an analytic  family of measures 
$$
\nu_{\de,\,\zeta}^{II}:=\ga(\ze) \sum_{2^M\le \la\le 2^M\rho^{-1}}\,\sum_{\{l:M_0\le 2^l\le \frac{\la}{M_1}\}}
2^{\frac l6 (1-3\ze)}\la^{\frac 23 (1-3\ze)}\, \nu_{\de,l}^\la
$$
where $\ze$ lies again in the complex strip $\Sigma.$ The analytic function $\ga(\ze)$ will be a finite product of factors $\ga_i(\ze)$ which will be specified in the course of the proof.

In view of  \eqref{analog21}, by following our standard approach it will suffice to prove the following estimate in order to  establish the second inequality in \eqref{11.17}:
\begin{equation}\label{12.6}
|\nu_{\de,\,1+it}^{II}(x)|\le C \qquad \forall t\in\RR, x\in \RR^3.
\end{equation}
By putting
$
\mu_{l,\la}:= 2^{-\frac 13 l}\la^{-\frac 43} \nu_{\de,l}^\la,
$
we may re-write 
\begin{equation}\label{12.7}
\nu_{\de,\,1+it}^{II}(x)=\ga(1+it) \sum_{2^M\le \la\le 2^M\rho^{-1}}\,\sum_{\{l:M_0\le 2^l\le \frac{\la}{M_1}\}}
2^{-\frac 12  it l}\la^{-2it}\, \mu_{l,\la}(x).
\end{equation}
\medskip

\subsubsection{Contribution by the $\nu_{l,II}^\la$ }

Let us begin with  the contribution of the main terms $\nu_{l,II}^\la$ in \eqref{12.5}, i.e.,  let us look at 
\begin{equation}\label{12.8}
\nu_{II,1+it}:=\ga(1+it) \sum_{2^M\le \la\le 2^M\rho^{-1}}\,\sum_{\{l:M_0\le 2^l\le \frac{\la}{M_1}\}}
2^{-\frac 12  it l}\la^{-2it}\, \mu_{l,\la,II}(x),
\end{equation}
where we have set $ \mu_{l,\la,II}:=2^{-\frac 13 l}\la^{-\frac 43}\, \nu_{l,II}^\la.$ 
We want to prove that

\begin{equation}\label{12.9}
|\nu_{II,\,1+it}(x)|\le C \qquad \forall t\in\RR, x\in \RR^3.
\end{equation}

Recall  from Subsection \ref{larhos} that we may restrict ourselves to those $x$ for which $|x|\lesssim1$ and $|x_1|\sim 1.$ Making use of  \eqref{nuII2} and  \eqref{8.52n}, after scaling the variable $u_2$ by the factor $2^{-l/3},$ we find that
\begin{eqnarray}\nonumber
&&\mu_{l,\la,II}(x)=  \int e^{-i s_3 \tilde\Psi_3(y_2,v,x, \de,\la, l)}
 a_3(\la^{-\frac 13}y_2,v, x,2^{-\frac l3},(2^{l}\la^{-1})^{\frac13},\de)    \label{12.10}\\
&&\hskip6cm\times   \chi_1(s_3) \,\chi_1(v)\,\chi_0(2^{-\frac l3} y_2)  \, dy_2 dv ds_3 \\
&& =\int \widehat{\chi_1}\big( \tilde\Psi_3(y_2,v,x, \de,\la, l)\big)  a_3(\la^{-\frac 13}y_2,v, x,2^{-\frac l3},(2^{l}\la^{-1})^{\frac13},\de)  \,\chi_1(v)\,\chi_0(2^{-\frac l3} y_2)  \, dy_2 dv,\nonumber
\end{eqnarray}
with
\begin{eqnarray*}\nonumber
&&\tilde\Psi_3(y_2,v,x, \de,\la, l)=v\, \la (2^{-l}\la)^{-\frac 13}(x_1^2\omega(\de_1x_1)-x_2)+   \la Q_A(x,\de) \\
&&\hskip2cm+ y_2^3  \,b(x_1,\la^{-\frac13}y_2,\de)+
 y_2\,\Big( \la^{\frac 23} [ \de_0 \tilde G_1(x_1,\de) +\de_3x_1^{n_1}\al_1(\de_1x_1)]+(2^{l}\la)^{\frac 13}\de_0 \, v \Big).
\end{eqnarray*}
We shall write this as 
\begin{equation}\label{12.11}
\tilde\Psi_3(y_2,v,x, \de,\la, l)=A+Bv + y_2^3  \,b(x_1,\la^{-\frac13}y_2,\de)+   y_2 (D+ Ev), 
\end{equation}
with
\begin{eqnarray}\nonumber
&& A=A(x,\la,\de):=\la Q_A(x,\de) , \qquad B=B(x,\la,l,\de):=2^{\frac l3} \la ^{\frac {2}3} Q_B(x,\de),\\
&& D=D(x,\la,\de):=\la^{\frac 23}Q_D(x,\de),  \qquad E=E(\la,l,\de)=2^{\frac l3} \la^{\frac 13}\de_0,\label{12.12}
\end{eqnarray}
and 
\begin{eqnarray*}
Q_A(x,\de)&:=&\tilde G_1(x_1,\de)(x_1^2\omega(\de_1x_1)-x_2)+x_1^n\alpha(\de_1 x_1)-x_3,\\ \quad Q_B(x,\de)&:=&x_1^2\omega(\de_1x_1)-x_2, \quad
Q_D(x,\de):=\de_0 \tilde G_1(x_1,\de) +\de_3x_1^{n_1}\al_1(\de_1x_1).
\end{eqnarray*}

Now, applying Lemma \ref{as2} from the appendix to the integration in $y_2,$ with $T:=2^{l/3}$ and $\de:=\la^{-1/3}$ (so that $\de T=(2^l\la^{-1})^{1/3}\ll 1$), we see that we may estimate 
\begin{eqnarray}\nonumber
&&\Big|\int \widehat{\chi_1}\big( \tilde\Psi_3(y_2,v,x, \de,\la, l)\big)  a_3(\la^{-\frac 13}y_2,v, x,2^{-\frac l3},(2^{l}\la^{-1})^{\frac13},\de) \,\chi_0(2^{-\frac l3} y_2)  \, dy_2 \Big|\\
&&\hskip 3cm\lesssim \big(1+\max\{|A+Bv|,|D+Ev|\}\big)^{-\frac 16}, \label{12.13}
\end{eqnarray}
with a constant which does not depend on $A,B,D,E$ and $v.$ 
The following simple lemma will thus  be useful:
\begin{lemma}\label{simpleintest}
Let $\ve>0,$ and consider for $A,B,D,E\in\RR$  the integral
$$
J(A,B,D,E):= \int \big(1+\max\{|A+Bv|,|D+Ev|\}\big)^{-\ve} \chi_0(v) \,dv,
$$
where as usually $\chi_0$ is a smooth, non-negative bump function with compact support. Then
$$
|J(A,B,D,E)|\le C\big(\max\{|A|, |B|,|D|,|E|\}\big)^{-\ve},
$$
where the constant $C$ is independent of $A,B,D$ and $E.$ 
\end{lemma}
\begin{proof} 
If $|A|\gg |B|,$ then $|A+Bv|\gtrsim |A|,$ and we see that $|J(A,B,D,E)|\le C|A|^{-\ve}.$ 
Next,  if $|A|\lesssim |B|$ then we apply the change of variables $v\mapsto w:=Bv$ and find that we can estimate 
\begin{eqnarray*}
|J(A,B,D,E)|\lesssim \frac 1{|B|}\int_{|w|\lesssim |B|} (1+|A+w|)^{-\ve}  \,dw
\lesssim \frac 1{|B|}\int_{|y|\lesssim |B|} (1+|y|)^{-\ve}  \,dw\lesssim |B|^{-\ve},
\end{eqnarray*}
provided $|B|\ge 1.$ Of course, if $|B|< 1,$ then we can always use the trivial estimate $|J(A,B,D,E)|\lesssim 1.$

We may now conclude the proof by interchanging the roles of $A,B$ and $D,E.$

 \end{proof}

In combination with \eqref{12.13} and  \eqref{12.10} this lemma  implies that 
\begin{equation}\label{mullaII}
|\mu_{l,\la,II}(x)|\lesssim \big(\max\{|A|, |B|,|D|,|E|\}\big)^{-\frac 16},
\end{equation}
with $A,B,D$ and $E$ given by \eqref{12.12}.

In order to estimate $\nu_{II,1+it}(x),$ we shall again distinguish various cases, in a similar way as in preceding arguments of this type, and  shall thus only briefly sketch the ideas.
\medskip

Let us first consider the contribution by those terms in \eqref{12.8} for which  $|D|\gtrsim1$ and $|E|\gtrsim1.$  Since by \eqref{mullaII} we may estimate $|\mu_{l,\la,II}(x)|\lesssim |D|^{-1/12} |E|^{-1/12},$  we can first sum  these contributions  absolutely over  all  $l$  for which  $|E|=2^{l/3} \la^{1/3}\de_0\gtrsim1,$  and subsequently over all dyadic $\la=2^j$  for which   $|D|=\la^{\frac 23}  |Q_D(x,\de)|\gtrsim1,$ and arrive at a bound which is uniform in $x$ and $\de.$ 

\medskip
In essentially the same way we can sum (absolutely) the contributions by those terms in \eqref{12.8} for which  $|A|\gtrsim 1$ and $|B|\gtrsim 1.$

\medskip
Consider next the terms for which $|E|\ll 1$ and $|D|\ll 1.$ For these terms, we re-write
\begin{equation}\label{mulaIIint}
\mu_{l,\la,II}(x)= \int e^{-i s_3 (A+Bv)} J(v,s_3)\,  \chi_1(s_3) \,\chi_1(v)\, dvds_3\,\\
\end{equation}
with
\begin{eqnarray}\nonumber
&&J(v,s_3):=  \int e^{-i s_3 \Big(y_2^3b(x_1,\la^{-\frac13}y_2,\de)+   y_2 (D+ Ev)\Big)}
 a_3(\la^{-\frac 13}y_2,v, x,2^{-\frac l3},(2^{l}\la^{-1})^{\frac13},\de)    \label{12.15}\\
&&\hskip6cm\times  \,\chi_0(2^{-\frac l3} y_2)  \, dy_2,
\end{eqnarray}
But, since $|D+Ev|\lesssim 1,$  the proof of Lemma 6.3 (a) in \cite {IM-rest1}  shows that $J(v,s_3)=g(D+Ev, s_3,x,2^{-\frac l3},(2^{l}\la^{-1})^{\frac13},\de), $ with a smooth function $g,$  and thus
$$
\mu_{l,\la,II}(x)= \int e^{-i s_3 (A+Bv)} g(D+Ev,s_3, x,2^{-\frac l3},(2^{l}\la^{-1})^{\frac13},\de)\,  \chi_1(s_3) 
\,\chi_1(v)\, dvds_3.
$$
Arguing in a similar way as in Subsection \ref{deai}, without  loss of generality we may and shall assume that 
$g(D+Ev,s_3, x,2^{-\frac l3},(2^{l}\la^{-1})^{\frac13},\de)=g(D+Ev, x,2^{-\frac l3},(2^{l}\la^{-1})^{\frac13},\de) $ is independent of $s_3.$ Then we may write 
\begin{equation}\label{12.14}
\mu_{l,\la,II}(x)= \int  g(D+Ev, x,2^{-\frac l3},(2^{l}\la^{-1})^{\frac13},\de)\, \widehat { \chi_1}(A+Bv) \,\chi_1(v)\,\, dv
\end{equation}
(alternatively, one could  also use integrations by parts in $s_3$  in the previous formula, but the other approach appears a bit clearer).

\medskip 
Recall that we  assume that either $|A|\gtrsim1$ and $|B|\ll1,$ or $|A|\ll 1$ and $|B|\gtrsim1,$ or $|A|\ll 1$ and $|B|\ll1.$

\medskip
If $|A|\gtrsim1$ and $|B|\ll1,$ then we can treat the summation in $l$ by means of Lemma \ref{simplesum}, where we choose, for $\la$ fixed,
$$
H_{\la,x}(u_1,u_2,u_3,u_4):=\int  g(D+u_1v, x,u_3,u_4,\de)\, \widehat { \chi_1}(A+u_2v) \, \chi_1(v)\,\, dv.
$$
Then clearly $\|H_{\la,x}\|_{C^1(Q)}\lesssim |A|^{-1},$ and so after summation in those $l$ for which  $|E|\ll 1$ and $|B|\ll 1 ,$ we can also sum (absolutely) in  the $\la$'s for which $|A|\gtrsim 1.$ Observe that this requires that $\ga(\ze)$  contains  a factor 
$$
\ga_1(\ze):= \frac{2^{\frac{1-\ze}2}-1}{2^{\frac 13}-1}.
$$

\medskip
Consider next the case where $|B|\gtrsim 1$ and $|A|\ll 1.$ If we write $\la=2^j,$ then $2^{l/3} \la^{2/3}=2^{k/3},$ where we put $k:=l+2j.$ We therefore pass from the summation variables $j$ and $l$ to the variables $j$ and $k,$  which allows to write 
$B=2^{ k/3}  Q_B(x).$ For $k$ fixed, we then sum first in $j$ by means of Lemma \ref{simplesum}, which  gives an estimate of order $O(|B|^{-1}),$ which then in return allows to sum (absolutely) in those $k$ for which $|B|=2^{ k/3}  |Q_B(x)|\gtrsim 1.$  Since $-l/2-2j=-k/2-j,$ the application of Lemma \ref{simplesum} requires  in this case that $\ga(\ze)$  contains  a factor 
$$
\ga_2(\ze):= \frac{2^{1-\ze}-1}{2^{\frac 23}-1}.
$$

\medskip
There remains the case where $|A|+|B|\ll 1$ and $|D|+|E|\lesssim 1.$ The summation over all $\l$'s and $\la$'s  for which these conditions are satisfied can easily be treated by means of the double summation  Lemma 8.1 in \cite{IM-rest1}, in a very similar way as this was done in the  last part of the proof of Proposition 5.2 (a)  of that article. The corresponding vector $(\al_1,\al_2)$  to be used in Lemma 8.1 will here be given by $(\al_1,\al_2)=(2,1/2),$ and  the vectors $(\be_1^k,\be_2^k)$  by $(1,0), (2/3, 1/3), (2/3, 0), (1/3, 1/3), (-1,1)$ and $(0,-1/3);$ they obviously satisfy the assumptions of  Lemma 8.1. For the application of this lemma, we need to assume that $\ga(\ze)$  contains  also a factor $\ga_3(\ze)$ given by Remark 8.2 in \cite{IM-rest1}.

\medskip 
What  remains are the contributions by those $l$ and $\la$ for which either $|D|\gtrsim1$ and $|E|\ll1,$ or $|D|\ll 1$ and $|E|\gtrsim1.$  
\medskip

We begin with the case where   $|E|\gtrsim1$ and  $|D|\ll1.$  Then we may assume in addition that $|B|\ll 1,$ for otherwise by \eqref{mullaII} we have $|\mu_{l,\la,II}(x)|\lesssim |E|^{-1/12}|B|^{-1/12},$  which allows to sum absolutely in $j$ and $l,$  as can easily seen by means of a change of the summation variables from $j$ and $l$ to $ k:=2j+l $ and  $m:=j+l$ (compare \eqref{12.12}). In a very similar way, we may also assume that $|A|\ll 1.$ 

Recall from \eqref{12.10} and \eqref{12.11} that 
\begin{eqnarray*}
&&\mu_{l,\la,II}(x)=  \int e^{-i s_3 \big(A+ y_2^3 \,b\big(x_1,\la^{-\frac 13}y_2,\de\big)+D y_2+  v(B+ Ey_2)\big)}
 a_3(\la^{-\frac 13}y_2,v, x,2^{-\frac l3},(2^{l}\la^{-1})^{\frac13},\de)   \\
&&\hskip6cm\times   \chi_1(s_3) \,\chi_1(v)\,\chi_0(2^{-\frac l3} y_2)  \, dy_2 dv ds_3\, .
\end{eqnarray*}
Again, by our usual argument, we may assume without loss of generality that $a_3$ is independent of $v.$ Then   find that 
\begin{eqnarray}\label{12.15b}
&&\mu_{l,\la,II}(x)=  \int e^{-i s_3 \big(A+ y_2^3 \,b\big(x_1,\la^{-\frac 13}y_2,\de\big))\big)}
\widehat{\chi_1}\big(s_3(B+Ey_2)\big) \nonumber  \\
&&\hskip2cm\times  a_3(\la^{-\frac 13}y_2, x,2^{-\frac l3},(2^{l}\la^{-1})^{\frac13},\de)  \chi_1(s_3) \,\chi_1(v)\,\chi_0(2^{-\frac l3} y_2)  \, dy_2  ds_3\, .
\end{eqnarray}
We then change the summation variables from $j,l$ to $j, k,$ where $k:=j+l,$ so that $E=2^{k/3} \de_0.$ Then, for $k$ fixed,  we can treat the summation in $j$ by means of Lemma \ref{simplesum}, where we choose
\begin{eqnarray*}
&&H_{k,x}(u_1,u_2,u_3,u_4,u_5):=  \int e^{-i s_3 \big(u_1+ y_2^3 \,b\big(x_1,u_2y_2,\de\big))\big)}
\widehat{\chi_1}\big(s_3(u_3+Ey_2)\big)   \\
&&\hskip4cm\times  a_3(u_2y_2, x,u_4,u_5,\de)  \chi_1(s_3) \,\chi_1(v)\,\chi_0(u_3 y_2)  \, dy_2  ds_3\, .
\end{eqnarray*}
By means of the change of variables $y_2\mapsto y_2/E$ we thus see that $|H_{k,x}(u_1,\dots,u_5)|\lesssim  |E|^{-1}$ on the natural cuboid $Q$ arising in this context, since $\widehat{\chi_1}$ is a Schwartz function.
Next, observe that  by the product rule, $\pa_{u_i}H_{k,x}(u_1,\dots, u_5)$ can be written as  finite sum  of integrals of a similar form, where the amplitude may carry  additional factors of the form $y_2^n,$ with $n=0,\dots,4.$ Again,  the change of variables $y_2\mapsto y_2/E$ shows that these can be estimated  by $C|E|^{-1}$ (or even  higher powers of $|E|^{-1}$). We thus find that $\|H_{k,x}\|_{C^1(Q)}\lesssim |E|^{-1},$ and thus after the summation over the  $\la=2^j$  this allows to subsequently also sum over  the $k$  for which $|E|=2^{k/3}\de_0\gtrsim 1.$

\medskip

There remains  the contribution by those $l$ and $\la$ for which $|D|\gtrsim1$ and $|E|\ll1.$ Observe that here we have $|D+Ev|\gtrsim 1$  in \eqref{12.15}.

Applying  the change of variables $y_2=\la^{1/3} t$  in the integral defining $J(v,s_3),$  we obtain
\begin{eqnarray}\nonumber
&&J(v,s_3)=  \la^{\frac 13}  \int e^{-i s_3 \la\big(t^3b(x_1,t,\de)+   
t(\la^{-\frac 23}(D+ Ev))\big)}\\
&&\hskip3cm\times a_3\big(t,v, x,2^{-\frac l3},(2^{l}\la^{-1})^{\frac13},\de\big)  
 \,\chi_0((\la 2^{-l})^{\frac 13}t)  \, dt.\label{Jint}
\end{eqnarray}
It is important to observe that here the phase function is independent of $l.$ Notice  also that by \eqref{12.12}
\begin{equation}\label{12.16}
\la^{-\frac 23}D=Q_D(x,\de), \quad   \la^{-\frac 23}E= (2^l\la^{-1})^{\frac 13} \de_0\ll 1,
\end{equation}
so that in particular $\la^{-\frac 23}|D+vE|\lesssim 1.$ 

\medskip
We may then argue in a similar way as in the proof of Lemma 6.3 (b) in \cite{IM-rest1} to see that for every $N\in\NN,$ 
\begin{eqnarray}
J(v,s_3)&=&|D+Ev|^{-\frac 14} a_+\big(\la^{-\frac 13} |D+Ev|^{\frac 12}, v, x,2^{-\frac l3},(2^{l}\la^{-1})^{\frac13},\de\big)  \nonumber\\
&&\hskip2cm \times \chi_0\Big({2^{-\frac l3}|D+Ev|^{\frac 12}}\Big)\, e^{-is_3|D+Ev|^{\frac 32}q_+\big(\la^{-\frac 13} |D+Ev|^{\frac 12}, x,\de\big)}  \label{12.17}\\
&+&|D+Ev|^{-\frac 14} a_-\big(\la^{-\frac 13} |D+Ev|^{\frac 12}, v, x,2^{-\frac l3},(2^{l}\la^{-1})^{\frac13},\de\big)  \nonumber \\
&&\hskip2cm \times \chi_0\Big(2^{-\frac l3}|D+Ev|^{\frac 12}\Big)\, e^{-is_3|D+Ev|^{\frac 32}q_-\big(\la^{-\frac 13} |D+Ev|^{\frac 12},x,\de\big)}   \nonumber\\
&+& (D+Ev)^{-N}\, F_N\Big( |D+vE|^{\frac 32}, \la^{-\frac 13} |D+Ev|^{\frac 12}, v, x,2^{-\frac l3},(2^{l}\la^{-1})^{\frac13},\de\Big), \nonumber
\end{eqnarray}
where  $a_\pm, q_\pm$ and $F_N$ are smooth functions of their (bounded) variables. 
Moreover, $|q_\pm(\big(0, x,(2^{l}\la^{-1})^{\frac13},\de\big)|\sim 1.$ 

Indeed, notice that the phase in \eqref{Jint} has a critical point in the support of the amplitude  only if 
$2^{-\frac l3}|D+vE|^{1/2}\lesssim 1,$ and so we obtain the first two terms in \eqref{12.17} by applying the method of stationary phase, whereas the last term arises from integrations by parts on intervals in $t$ on which there is no stationary point.

\medskip
We shall concentrate on the first term only.  The second term can be treated in the same way as the first one, and  the last term can be handled in an even easier way by a similar method, since it is of order $=O(|D|^{-N})$ and, unlike the first term, carries no oscillatory factor.

\medskip
We denote by 
\begin{eqnarray*}
\mu^1_{l,\la}(x)&:=&  \int e^{-i s_3 \big[(A+Bv) +|D+Ev|^{\frac 32} q_+\big(\la^{-\frac 13} |D+Ev|^{\frac 12}, x,\de\big)\big]}   |D+Ev|^{-\frac 14}\chi_1(s_3) \,\chi_1(v)\,\nonumber\\
&&\hskip0cm \times a_+\big(\la^{-\frac 13} |D+Ev|^{\frac 12}, v, x,2^{-\frac l3},(2^{l}\la^{-1})^{\frac13},\de\big) \,\chi_0\Big({2^{-\frac l3}|D+Ev|^{\frac 12}}\Big) \,   dvds_3 
\end{eqnarray*}
 the contribution by the first term in \eqref{12.17} to $\mu_{l,\la,II}(x),$ and by $ \nu^1_{II,1+it}(x)$ the contribution of the $\mu^1_{l,\la}(x)$ to the sum defining $ \nu_{II,1+it}(x).$

\medskip
Assuming for instance that $D>0,$  and keeping  in  mind  that, according to \eqref{12.12}, 
$\la^{-1/3} D^{1/2}=Q_D(x,\de)^{1/2}$ depends only on $x$ and $\de,$ a Taylor expansion then shows that 
\begin{eqnarray*}
|D+Ev|^{\frac 32}&=&D^{\frac 32}+\frac 32 D^{\frac 12} E v+D^{-\frac 12} E^2 r_1(D^{-1} E, v),\\
q_+\big(\la^{-\frac 13} |D+Ev|^{\frac 12}, x,\de\big)&=&q_+\big(\la^{-\frac 13} D^{\frac 12}, x,\de\big)
+\frac 12 q'_+\big(\la^{-\frac 13} D^{\frac 12}, x,\de\big)\la^{-\frac 13}D^{-\frac 12}Ev\\
&+&\la^{-\frac 13}D^{\frac 12} (D^{-1} E)^2  r_2(\la^{-\frac 13}D^{\frac 12}, D^{-1} E,v,x,\de),
\end{eqnarray*}
where $q'_+$ denotes the partial derivative of $q_+$ with respect to the first variable, and where $r_1$ and $r_2$ are smooth, real-valued functions. This implies that 
\begin{eqnarray}\label{12.22b}
&|D+Ev|^{\frac 32} q_+\big(\la^{-\frac 13} |D+Ev|^{\frac 12}, x,\de\big)=D^{\frac 32} q_+\big(\la^{-\frac 13} 
D^{\frac 12}, x,\de\big)\\
&+v\Big[\frac 12 q'_+\big(\la^{-\frac 13} D^{\frac 12}, x,\de\big)\la^{-\frac 13}DE +\frac 32 q'_+\big(\la^{-\frac 13} D^{\frac 12}, x,\de\big) D^{\frac 12} E\Big]+  r(\la^{-\frac 13}D^{\frac 12}, D^{-1} E, v,x,\de),\nonumber
\end{eqnarray}
again with a smooth, real-valued function $r.$ 
In  combination with \eqref{12.12} we then find that the complete phase in the oscillatory integral defining 
$\mu^1_{l,\la}(x)$ is  of the form
$$
s_3[A'+B'v+ r],
$$
with $A'=\la Q_{A'}(x,\de), \ B'=2^{\frac l3}\la^{\frac 23} Q_{B'}(x,\de),$ where 
\begin{eqnarray*}
Q_{A'}(x,\de)&:=&Q_A(x,\de)+Q_D(x,\de)^{\frac 32}  q_+\big(Q_D(x,\de)^{\frac 12}, x,\de\big),\\
Q_{B'}(x,\de)&:=&Q_B(x,\de)+\frac 12 q'_+\big(Q_D(x,\de)^{\frac 12}, x,\de\big)Q_D(x,\de)\de_0 \\
&&+32 q'_+\big(Q_D(x,\de)^{\frac 12}, x,\de\big)Q_D(x,\de)^{\frac 12}\de_0.
\end{eqnarray*}

Thus, if $|B'|\gg 1,$ then an integration by parts in $v$ shows that 
$$
|\mu^1_{l,\la}(x)|\lesssim |D|^{-\frac 14} |B'|^{-1}.
$$
This estimate allows to control the sum over all $l$  such that $|B'|\gg 1,$ and  subsequently the sum over all dyadic $\la$ such that $|D|\gtrsim 1,$ and we arrive at the desired uniform estimate in $x$ and $\de.$ 

\medskip
Next, if $|B'|\lesssim 1,$  then  we can argue in a similar way as  before  and apply Lemma \ref{simplesum} to the summation in $l$  by putting here 
\begin{eqnarray*}
H_{\la,x}(u_1,\dots, u_7)&:=& \int e^{-i s_3 [A'+ u_1v+r(Q_D(x)^{\frac 12}, D^{-1} u_3, v,x,\de)]}   |D+u_2v|^{-\frac 14}\chi_1(s_3) \,\chi_1(v)\,\\
&&\hskip-2.5cm \times a_+\big(|Q_D(x,\de)+u_3v|^{\frac 12}, v, u_4,u_5,\de\big) 
\,\chi_0\Big({|u_6+u_7v|^{\frac 12}}\Big) \,  dvds_3
\end{eqnarray*}
and choosing the cuboid $Q$ in the obvious way.
Then we easily see that   $\|H_{\la,x}\|_{C^1(Q)}\lesssim |D|^{-1/4},$ and so after summation in those $l$ for which $|B'|\ll 1,$ we can also sum (absolutely) in  the $\la$'s for which $|D|\gtrsim 1.$ Observe that this requires again that $\ga(\ze)$  contains  the  factor $\ga_1(\ze).$

This concludes the proof of the uniform estimate of $\nu^1_{II, 1+it}(x)$ in $x$ and $\de,$  and thus also of  estimate \eqref{12.9}.

\subsubsection{Contribution by the $\nu_{l,I}^\la$ }

Let us next consider  the contribution of the  terms $\nu_{l,I}^\la$ in \eqref{12.5}, i.e.,  let us look at 
\begin{equation}\label{12.19}
\nu_{I,1+it}:=\ga(1+it) \sum_{2^M\le \la\le 2^M\rho^{-1}}\,\sum_{\{l:M_0\le 2^l\le \frac{\la}{M_1}\}}
2^{-\frac 12  it l}\la^{-2it}\, \mu_{l,\la,I}(x),
\end{equation}
where we have set $ \mu_{l,\la,I}:=2^{-\frac 13 l}\la^{-\frac 43}\, \nu_{l,I}^\la.$ 
We want to prove that

\begin{equation}\label{12.20}
|\nu_{I,\,1+it}(x)|\le C \qquad \forall t\in\RR, x\in \RR^3.
\end{equation}

Our discussion in Section \ref{m2b34}  shows that 
\begin{eqnarray*}
&&\hskip1cm\mu_{l,\la,I}(x)=\la^{\frac 13}\, 2^l   \int e^{-i s_3 \tilde\Psi(u,z,s_2,x, \de,\la, l)}\tilde a\big(\si_{2^l\la^{-1}}u, (2^{l}\la^{-1})^{\frac23}z,s_2,\de\big) \,  \chi_1(s_2,s_3)\\
&&\hskip0cm\times   \,(1-\chi_0)\Big(\ve(2^l\la^{-1})^{-\frac13}(x_1-s_2^\frac1{n-2}G_1(s_2,\de))\Big)\,\chi_0(u) \chi_1(u_1) \,\chi_1(z)\,du_1 du_2\, dz  \,ds_2 ds_3   ,
\end{eqnarray*}
where $\ve>0$ is small.
Moreover,
\begin{eqnarray*}
\tilde\Psi(u,z,s_2,x, \de,\la, l)&=&\la^{\frac 13} 2^{\frac{2l}3}\big(x_1-s_2^\frac1{n-2}G_1(s_2,\de)
 -  (2^l\la^{-1})^{\frac 13} u_1\big) \, z \\
&+&\la \Big(  s_2^{\frac {n}{n-2}}G_5(s_2,\de)- x_1 s_2^{\frac {n-1}{n-2}}G_3(s_2,\de)- s_2x_2-x_3
\Big)\\
&+& 2^l\big(u_1^3\, B_3(s_2,\de_1, (2^{-l}\la)^{-\frac 13}u_1)   +\phi^\sharp_{2^{-l}\la}(u_1,u_2,\, \tilde\de^{2^{-l}\la}, s_2)\big),
\end{eqnarray*}
with $\phi^\sharp$ given by \eqref{7.13II}. Observe that the first term is here much bigger than $2^l,$ so that it dominates the third term.  

We change variables from $s_2$ to $v:=x_1-s_2^\frac1{n-2}G_1(s_2,\de).$  Then $s_2$ is a smooth function $s_2\big(v,x_1, \de\big),$ and we may re-write 
\begin{eqnarray*}
&&\hskip1cm\mu_{l,\la,I}(x)=\la^{\frac 13}\, 2^l   \int e^{-i s_3 \tilde\Psi_1(u,z,v,x, \de,\la, l)}\, a_1\big(\si_{2^l\la^{-1}}u, (2^{l}\la^{-1})^{\frac13},z,v,x,\de\big) \,  \chi_0(v)\chi_1(s_3)\\
&&\hskip4cm\times   \,(1-\chi_0)\big(\ve(2^l\la^{-1})^{-\frac13}v\big)\,\chi_0(u) \chi_1(u_1) \,\chi_1(z)\,du_1 du_2\, dz  \,v\, ds_3,
\end{eqnarray*}
where $\tilde\Psi_1$ if of the form
\begin{eqnarray*}
\tilde\Psi_1(u,z,v,x, \de,\la, l)&=&\big(\la^{\frac 13} 2^{\frac{2l}3}\, v-2^lu_1\big)z +2^l g_1\big(u,v, x, (2^l\la^{-1})^{1/3}, \tilde\de^{2^{-l}\la},\de\big)\\
&+&\la  g_2\big(u,v, x, \de\big),
\end{eqnarray*}
with smooth, real-valued  functions $g_1$ and $ g_2.$ Integrating $N$ times by parts in $z,$ and subsequently  changing coordinates from $v$ to $w:=(2^l\la^{-1})^{-\frac13}v,$ we arrive at the following expression for $\mu_{l,\la,I}(x):$

\begin{eqnarray}\nonumber
\mu_{l,\la,I}(x)&=& 2^{(\frac {4}3-N)l}   \int e^{-i s_3 \Psi(y,z,w,x, \de,\la, l)}\, a\big(\si_{2^l\la^{-1}}y, (2^{l}\la^{-1})^{\frac13},z,(2^l\la^{-1})^{\frac13}w,\de\big)\, \chi_1(s_3)\\
&&\hskip-0.5cm\times  \chi_0((2^l\la^{-1})^{\frac13}w) \,(1-\chi_0)(\ve w)\,\chi_0(y) \chi_1(y_1) \,\chi_1(z)\frac {1}{(w-y_1)^N}\,dy_1 dy_2\, dz\, ds_3 \, dw \label{nuI},
\end{eqnarray}
with phase $\Psi$ of the form
\begin{eqnarray}\nonumber
\Psi(y,z,w,x, \de,\la, l)&=&2^l (w-y_1)z +2^l g_1\big(y,(2^l\la^{-1})^{\frac13}w, x, (2^l\la^{-1})^{1/3}, \tilde\de^{2^{-l}\la},\de\big)\\
&+&\la g_2\big(y,(2^l\la^{-1})^{\frac13}w,x, \de\big)\label{12.22}.
\end{eqnarray}
Notice that we have also changed the names of variables $u$ to $y,$ in order to avoid possible confusion in the later application of Lemma \ref{simplesum}. Recall also  that in this integral, $|w|\gg 1\sim u_1.$ 

A Taylor expansion of $g_2$ in the ``variable'' $ (2^l\la^{-1})^{\frac13}w$ then shows that we may re-write the phase in the form
\begin{eqnarray}\nonumber
\Psi(y,z,w,x, \de,\la, l)&=&2^l \Big[(w-y_1)z +g_1\big(y,(2^l\la^{-1})^{\frac13}w, x, (2^l\la^{-1})^{1/3}, \tilde\de^{2^{-l}\la},\de\big)  \label{12.23}.\\
&&\hskip2cmh_0\big(y,(2^l\la^{-1})^{\frac13}w,x, \de\big) w^3\Big]\\
&+&\la h_3\big(y,x, \de\big)+\la^{\frac 23} 2^{\frac l3}\, h_2\big(y,x, \de\big)w+\la^{\frac 13} \,2^{\frac {2l}3} h_1\big(y,x, \de\big)w^2, \nonumber
\end{eqnarray}
where $h_0,\dots, h_3$ are again smooth, real-valued functions of their (bounded) variables.
\medskip

The following lemma will be useful.

\begin{lemma}\label{s12.1} Let $\be_1,\dots,\be_n\in ]0,\infty[$ be given, pairwise  distinct  positive numbers. For any complex numbers 
$\al_1,\dots,\al_n\in\bC,$  denote by $\La$ the set of all dyadic numbers $\la=2^j$ such that 
$\max_{k=1,\dots,n}\la^{\be_k}|\al_k|\ge 1.$ Then there exists an exceptional set $\La_e\subset \La$ depending on the $\al_k$ and $\be_k$ whose cardinality is bounded by a constant $C_1(\be_1,\dots,\be_n)$  depending only on  the $\be_1,\dots,\be_n$ such that for all $\la\in\La\setminus \La_e$  we have that $|\sum_{k=1}^n \la^{\be_k}\al_k|\ge 2/3,$ and moreover
\begin{equation}\label{12.24}
\sum_{\la\in\La\setminus \La_e}\big|\sum_{k=1}^n \la^{\be_k}\al_k\big|^{-1}\le C_2(\be_1,\dots,\be_n),
\end{equation}
where the constant $C_2(\be_1,\dots,\be_n)$ depends  only on the numbers $\be_k.$
\end{lemma}

\begin{proof} We may assume without loss of generality that $\al_k\ne 0$ for every $k.$ 
Observe that if $k\ne l$ and, say, $\be_k>\be_l,$ then we have that $1/4\le (2^{\be_kj}|\al_k|)/(2^{\be_lj}|\al_l|)\le 4$ if and only if 
$$
\Big| j+\frac{\log_2|\frac{\al_k}{\al_l}|}{\be_k-\be_l}\Big|\le \frac 2{\be_k-\be_l}.
$$
We therefore define the set $\La_e$ to be the set of  all dyadic numbers  $\la=2^j\in \La$  satisfying this conditions for at least on pair $k\ne l.$ The cardinality of $\La_e$ is then clearly bounded by 
$\binom n2\, 4\max_{k\ne l}|\be_k-\be_l|^{-1}.$ Moreover, if $\la\in\La\setminus \La_e,$ and if we choose a permutation $(k(1),\dots,k(n))$ of $(1,\dots,n)$ so that
$$
\la^{\be_{k(1)}}|\al_{k(1)}|>\la^{\be_{k(2)}}|\al_{k(2)}|>\cdots >\la^{\be_{k(n)}}|\al_{k(n)}|,
$$
then we have indeed even 
$$
\la^{\be_{k(1)}}|\al_{k(1)}|> 4\la^{\be_{k(2)}}|\al_{k(2)}|>\cdots >4^{n-1}\la^{\be_{k(n)}}|\al_{k(n)}|.
$$
This implies that 
$$
\big|\sum_{k=1}^n \la^{\be_k}\al_k\big|\ge \la^{\be_{k(1)}}|\al_{k(1)}|\big( 1-\sum_{l=1}^{n-1} 4^{-l}\big)\ge
\frac 23  \la^{\be_{k(1)}}|\al_{k(1)}|\ge \frac 23.
$$
And, since $\sum_{j\in\bZ:2^{\be_lj}|\al_l|\ge 1}(2^{\be_lj}|\al_l|)^{-1}\le (1-2^{-\be_l})^{-1},$ we obtain \eqref{12.14}, with 
$$
C_2(\be_1,\dots,\be_n):=\frac 32 n!\, \max_k\frac 1{1-2^{-\be_k}}.
$$
\end{proof}

In order to prove \eqref{12.20}, let us consider
\begin{eqnarray*}
F(t,y,z,w,x, \de, l):= \sum_{2^M\le \la\le 2^M\rho^{-1}}\
\la^{-2it}\,\int e^{-i s_3 \Psi(y,z,w,x, \de,\la, l)} \,\chi_1(s_3)\, ds_3
\end{eqnarray*}
We shall prove that 
\begin{equation}\label{12.25}
|F(t,y,z,w,x, \de,\la, l)|\le C \frac{2^l (1+|w|^3)}{|2^{-i2t}-1|},
\end{equation}
with a constant $C$ not depending on $t,y,z,,x, \de$ and  $l.$ By choosing $N$ in \eqref{nuI} sufficiently big, we see that this estimate will imply \eqref{12.20}, provided $\ga(\ze)$ contains a factor 
$$
\ga_4(\ze):= \frac{2^{2(1-\ze)}-1}{3}.
$$

Let us  put $\be_3:=1,\be_2:=2/3,\be_1:=1/3$ and, given $y,w,x,\de$  and $l,$  also
 $\al_3:= h_3\big(y,x, \de\big), \al_2:= 2^{\frac l3}\, h_2\big(y,x, \de\big)w, \al_1:=2^{\frac {2l}3} h_1\big(y,x, \de\big)w^2.$ Accordingly, we set
\begin{eqnarray*}
\La_1=\La_1(y,w,x,\de,l)&:=&\big\{\la=2^j: 2^M\le \la\le 2^M\rho^{-1} \ \mbox{and} \\
 &&\hskip-2cm\max\{\la |h_3\big(y,x, \de\big)|,\, \la^{\frac 23} |2^{\frac l3}\, h_2\big(y,x, \de\big)w|,\, \la^{\frac 13} \,2^{\frac {2l}3} |h_1\big(y,x, \de\big)w^2|\}\ge 1.\big\}\\
 &=&\big\{\la=2^j: 2^M\le \la\le 2^M\rho^{-1} \ \mbox{and} \ \max_{k=1,\dots,3}\la^{\be_k}|\al_k|\ge 1\big\}\, ,
\end{eqnarray*}
and 
\begin{eqnarray*}
\La_2=\La_2(y,w,x,\de,l)&:=&\big\{\la=2^j: 2^M\le \la\le 2^M\rho^{-1} \ \mbox{and} \\
 &&\hskip-2cm\max\{\la |h_3\big(y,x, \de\big)|,\, \la^{\frac 23} |2^{\frac l3}\, h_2\big(y,x, \de\big)w|,\, \la^{\frac 13} \,2^{\frac {2l}3} |h_1\big(y,x, \de\big)w^2|\}< 1\big\}\, .
\end{eqnarray*}
We also denote by $\La_e\subset \La$ the set of exceptional $\la$'s given by Lemma \ref{s12.1} for this choice of $\be_k$ and $\al_k.$ 
Correspondingly, we decompose $F=F_e+F_1+F_2,$ where $F_e,F_1$ and $F_2$ are defined as $F,$ only with summation over the dyadic $\la$'s restricted to the subsets $\La_e,$ $\La_1\setminus\La_e$ and $\La_2,$ respectively. 

\medskip
For $F_e,$ we then trivially get the estimate $|F(t,y,z,w,x, \de,\la, l)|\le C,$ since the cardinality of $\La_e$ is bounded by a constant not depending on the arguments of $F_e.$ 
\medskip

Next, in order to estimate $F_1,$ let us re-write 
$$
\int e^{-i s_3 \Psi(y,z,w,x, \de,\la, l)} \,\chi_1(s_3)\, ds_3=
\int e^{-i s_3(\la^{\be_1}\al_1+\la^{\be_2}\al_2+\la^{\be_3}\al_3)}\, \Big( e^{-i s_3 \Psi_0(y,z,w,x, \de,\la, l)} \,\chi_1(s_3)\Big)\, ds_3\,, 
$$ 
where $\Psi_0$ denotes the phase 
\begin{eqnarray*}
\Psi_0(y,z,w,x, \de,\la, l)&:=&2^l \Big[(w-y_1)z +g_1\big(y,(2^l\la^{-1})^{\frac13}w, x, (2^l\la^{-1})^{1/3}, \tilde\de^{2^{-l}\la},\de\big)  \\
&&\hskip2cmh_0\big(y,(2^l\la^{-1})^{\frac13}w,x, \de\big) w^3\Big]\, .
\end{eqnarray*}
Observe that $|\Psi_0(y,z,w,x, \de,\la, l)|\le C2^l(1+|w|).$ An integrating by parts in $s_3$  therefore shows that 
$$
\big| \int e^{-i s_3 \Psi(y,z,w,x, \de,\la, l)} \,\chi_1(s_3)\, ds_3\big|\le 
 C\,  \frac {2^l(1+|w|)}{|\la^{\be_1}\al_1+\la^{\be_2}\al_2+\la^{\be_3}\al_3|}.
$$
We may thus control the sum over all $\la\in \La_1$ by means of Lemma \ref{s12.1} and obtain the estimate
$$
|F_1(t,y,z,w,x, \de,\la, l)|\le C 2^l (1+|w|).
$$
\smallskip

Finally, $F_2$ can again be estimated by means of Lemma \ref{simplesum}. Indeed,  observe that in the sum defining $F_2(t,y,z,w,x, \de,\la, l),$  the expressions 
\begin{eqnarray*}
&& (2^l\la^{-1})^{\frac13}w,\ (2^l\la^{-1})^{1/3}, \tilde\de^{2^{-l}\la},
 \la h_3\big(y,x, \de\big), \ \la^{\frac 23} 2^{\frac l3}\, h_2\big(y,x, \de\big)w, \ \la^{\frac 13} \,2^{\frac {2l}3} h_1\big(y,x, \de\big)w^2
\end{eqnarray*}
are all uniformly bounded. Therefore, we may here put
\begin{eqnarray*}
&&H(u_1,\dots, u_6):=\int e^{-is_3\big(2^l [(w-y_1)z +g_1(y,u_1, x, u_2,u_3,\de)+h_0 (y,u_1,x, \de) w^3]+u_4+u_5+u_6\big)}
\, \chi_1(s_3)\, ds_3,
\end{eqnarray*}
with the $a_k$ in Lemma \ref{simplesum} given by 
$$
a_1:=2^{l/3}2^w,\, a_2:=2^{l/3}, \dots, a_4:=h_3(y,x, \de), \, a_5:=2^{\frac l3}\, h_2(y,x, \de)w, \, a_6:=2^{\frac {2l}3} h_1(y,x, \de)w^2,
$$
and the obvious corresponding cuboid $Q.$ Then clearly $\|H\|_{C^1(Q)}\le C 2^l(1+|w|^3),$ and thus Lemma \ref{simplesum} yields the estimate
$$
|F_2(t,y,z,w,x, \de,\la, l)|\le C \frac{2^l (1+|w|^3)}{|2^{-i2t}-1|}.
$$
This concludes the proof of estimate \eqref{12.25}, and thus also of \eqref{12.20}.

\subsubsection{Contribution by the $\nu_{l,III}^\la$ }

The contribution of the  terms $\nu_{l,III}^\la$ in \eqref{12.5} can be treated in a very similar way as the one by the terms $\nu_{l,I}^\la.$ Indeed, arguing  as  before, we here arrive at the following expression for  $ \mu_{l,\la,III}:=2^{-\frac 13 l}\la^{-\frac 43}\, \nu_{l,III}^\la:$ 
\begin{eqnarray}\nonumber
\mu_{l,\la,III}(x)&=& 2^{(\frac {4}3-N)l}   \int e^{-i s_3 \Psi(y,z,w,x, \de,\la, l)}\, a\big(\si_{2^l\la^{-1}}y, (2^{l}\la^{-1})^{\frac13},z,(2^l\la^{-1})^{\frac13}w,\de\big)\, \chi_1(s_3)\\
&&\hskip-0.5cm\times  \chi_0\big((2^l\la^{-1})^{\frac13}w\big) \,\chi_0\big(\frac w \ve\big)\,\chi_0(y) \chi_1(y_1) \,\chi_1(z)\frac {1}{(w-y_1)^N}\,dy_1 dy_2\, dz\, ds_3 \, dw \label{nuIII}.
\end{eqnarray}
The phase $\Psi$ is still given by \eqref{12.22}. Notice that now $|w|\ll 1\sim u_1.$ The arguments used in the preceding subsection therefore carry over to this case, with minor modifications (even simplifications). 

\subsubsection{Contribution by the $\nu_{l,\infty}^\la$ }
Let us next  look at 
\begin{equation}\label{12.27}
\nu_{\infty,1+it}:=\ga(1+it) \sum_{2^M\le \la\le 2^M\rho^{-1}}\,\sum_{\{l:M_0\le 2^l\le \frac{\la}{M_1}\}}
2^{-\frac 12  it l}\la^{-2it}\, \mu_{l,\la,\infty}(x),
\end{equation}
where we have set $ \mu_{l,\la,\infty}:=2^{-\frac 13 l}\la^{-\frac 43}\, \nu_{l,\infty}^\la.$ 
We want to prove that

\begin{equation}\label{12.28}
|\nu_{\infty,\,1+it}(x)|\le C \qquad \forall t\in\RR, x\in \RR^3.
\end{equation}

To this end, recall  formulas \eqref{8.38n} and \eqref{8.39n} for $\nu_{l,\infty}^\la(x).$  From these  formulas, it is easy to see that $|\nu_{l,\infty}^\la(x)|\lesssim 2^{-lN}\la^{-N}$ if $|x|\gg 1,$ and thus summation in $l$ and $\la$ is no problem in this case. So, assume that $|x|\lesssim 1.$ Then the second term in the phase $\Psi(z,s_2,\de)$ in \eqref{8.39n} can be absorbed into the amplitude $a_{N,l},$ and we arrive at an expression of the following form for $ \mu_{l,\la,\infty}(x):$
\begin{eqnarray*}  
 \mu_{l,\la,\infty}(x)&=& 2^{-lN}\, \la^{\frac 13}\int e^{-is_3\la\big(s_2^{\frac {n}{n-2}}G_5(s_2,\de)- x_1 s_2^{\frac {n-1}{n-2}}G_3(s_2,\de)- s_2x_2-x_3\big)} \\
&&\hskip0.5cm \times a_{N,l} \Big(z,  s_2,s_3,\,\de_0^{\r},\tilde\de^{2^{-l}\la},(2^{-l}\la)^{-\frac13}, \la^{-\frac 1{9}}\Big) \chi_1(z)  \,\chi_1(s_2)\chi_1(s_3)\, dz ds_2 ds_3,
\end{eqnarray*}
where $a_{N,l}$ is  a smooth function of all its (bounded) variables such that $\|a_{N,l}\|_{C^k}$ is uniformly bounded in $l.$ Denote by $\Psi(s_2)=\Psi(s_2,x,\de)=s_2^{\frac {n}{n-2}}G_5(s_2,\de)- x_1 s_2^{\frac {n-1}{n-2}}G_3(s_2,\de)- s_2x_2-x_3$ the phase appearing in this integral.  We can now argue in a similar way as in Subsection \ref{deai}:

Since $s_2\sim 1$  in the integral, we see that if $|x_1|\ll1,$ then $|\pa^2_{s_2}\Psi(s_2)|\sim 1,$  and van der Corputs' estimate implies that $|\mu_{l,\la,\infty}(x)|\lesssim 2^{-lN}
\la^{1/3-1/2}.$ We can then sum the series \eqref{12.27} absolutely to arrive at \eqref{12.28}. Let us therefore assume from now on  that $|x_1|\sim 1,$ and that the sign of $x_1$ is such that there is a point 
$s_2^c(x,\de)\sim 1$ such that $\pa^2_{s_2}\Psi(s_2^c(x,\de),x,\de)=0.$ This point is  then unique, by the implicit function theorem, since  $|\pa^3_{s_2}\Psi(s_2^c(x,\de),x,\de)|\sim 1.$ Changing coordinates from $s_2$ to $v:=s_2-s_2^c(x,\de),$  and applying a Taylor expansion in $v,$ we see that the phase can be written in the form
$$
Q_3(v,x,\de)\,v^3 - Q_1(x,\de)\, v+ Q_0(x,\de),
$$
with smooth functions $Q_j.$ Scaling in $v$ by a factor $\la^{-1/3}$ then leads to an expression of the following form  for $ \mu_{l,\la,\infty}(x):$
\begin{eqnarray}  \nonumber
 \mu_{l,\la,\infty}(x)&=& 2^{-lN}\,\int e^{-is_3\big( Q_3(\la^{-\frac 13}w,x,\de)\,w^3 + \la^{\frac 23} Q_1(x,\de)\, w+ \la Q_0(x,\de)\big)} \\
&&\hskip-0.5cm \times a_{N,l} \Big(z,  \la^{-\frac 13}w,s_3,\,\de_0^{\r},\tilde\de^{2^{-l}\la},(2^{-l}\la)^{-\frac13}, \la^{-\frac 1{9}}\Big) \chi_1(z)  \,\chi_0(\la^{\frac 13} w)\chi_1(s_3)\, ds_3 dw dz.\label{12.29}
\end{eqnarray}
Performing first the integration in $s_3,$ this easily implies the following estimate:
\begin{eqnarray*} 
|\mu_{l,\la,\infty}(x)|&&\lesssim  2^{-lN}\,\int \int_{-\la^{\frac 13}}^{\la^{\frac 13}}\Big(1+\big|Q_3(\la^{-\frac 13}w,x,\de)\,w^3 + \la^{\frac 23} Q_1(x,\de)\, w+ \la Q_0(x,\de)\big|\Big)^{-N} \\
&&\hskip5cm \times  \chi_1(z)  \,\chi_0(\la^{\frac 13} w) dw dz.
\end{eqnarray*}
Putting $A:=\la Q_0(x,\de), \, B:=\la^{ 2/3} Q_1(x,\de)$ and $T:=\la^{1/3}$ in Lemma \ref{as1}, we then find that 
$$
|\mu_{l,\la,\infty}(x)|\lesssim 2^{-lN} \max\{|A|^{\frac 13}, |B|^{\frac 12}\}^{\epsilon-\frac 12}.
$$
This estimate allows to sum over all $\la$ such that $\max\{\la |Q_0(x,\de)|,\la^{ 2/3} |Q_1(x,\de)|\}> 1,$ 
even absolutely. 

There remains the summation over all $\la$ such that $\la |Q_0(x,\de)|\le 1$ and 
$\la^{ 2/3} |Q_1(x,\de)|\le 1.$ However, in view of \eqref{12.29}, this sum can easily be controlled by means of Lemma \ref{simplesum}, as we have done in many similar cases  before, and and we shall therefore skip the details. Altogether, we arrive at \eqref{12.28}.

\subsubsection{Contribution by the $\nu_{l,00}^\la$ }

We finally come to the  contribution of the  terms $\nu_{l,00}^\la$ in \eqref{12.5}.  Recall from Subsection \ref{nudell} that 
\begin{eqnarray*}
&&\widehat{\nu_{l,00}^\la}(\xi):=(2^{-l}\la)^{-\frac 23} \chi_1(s,s_3) \,\chi_1((2^{-l}\la)^\frac23 B_1(s,\,\de_1))  \,e^{-is_3 \la B_0(s,\de_1)}\\
&&\hskip1cm\times\iint e^{-i s_3 2^l\Phi(u_1,u_2,s, \de,\la, l)} 
a(\si_{2^l\la^{-1}}u, \de,s) \, \chi_0(u)\chi_0\big( \frac {u_1}\ve\big)\,du_1du_2,
\end{eqnarray*}
where we assume $\ve>0$ to be sufficiently small. Moreover, the phase $\Phi$ is given by \eqref{Philla} and \eqref{philla}, with $B=3.$ Since 
$$
|\tilde\de^{2^{-l}\la}|\ll 1\quad \mbox{and}\quad (2^{-l}\la)^\frac23|B_1(s,\,\de_1)| \sim 1,
$$
we see that we can again integrate by parts in $u_1,$ in order to gain factors $2^{-lN},$ and then the same type of argument that led to the expression \eqref{hatnulinf} for $\widehat{\nu_{l,\infty}^\la}(\xi)$ can be applied in order to see that an analogous expression
 \begin{eqnarray*}
  &&\widehat{\nu_{l,00}^\la}(\xi)=2^{-lN}\, \la^{-\frac 23}  \chi_1(s,s_3) 
\,\chi_1((2^{-l}\la)^\frac23 B_1(s,\,\de_1))\,e^{-is_3 \la B_0(s,\de_1)} \\
 &&\hskip2cm \times\,  \tilde a_{N,l}\Big((2^{-l}\la)^\frac23 B_1(s,\,\de_1), s,s_3,\,\tilde\de^{2^{-l}\la}, \de_0^{\r},(2^{-l}\la)^{-\frac13},\la^{-\frac 1{3B}}\big), 
 \end{eqnarray*}
 can be obtained for $\widehat{\nu_{l,00}^\la}(\xi)$ too, 
 where $\tilde a_{N,l}$ is  again a  smooth function of all its (bounded) variables such that $\|a_{N,l}\|_{C^k}$ is uniformly bounded in $l.$ From here on, we can argue exactly as  for the $\nu_{l,\infty}^\la.$

\medskip

This concludes the proof of the second estimate in \eqref{11.17}, and thus also the proof of part (a) of Proposition \ref{s11.1}.

\setcounter{equation}{0}
\section{Proof of Proposition \ref{s11.1} (b) : Complex interpolation}\label{propfinalb}
In this section, we  assume that $B=3$ and $A=0$ in \eqref{5.17II},   and that  $\la\rho\gg 1.$
 
 \subsection{Estimation of $T^{III}_{\de,Ai}$}\label{deaib}

As usually, we embed   $\nu_{\de, Ai}^{III}$ into an analytic  family of measures 
$$
\nu_{\de,\,\zeta}^{III}:=\ga(\ze) \sum_{2^M\rho^{-1}< \la\le 2^{-M}\de_0^{-3}}
\big(\rho^{-\frac 45}(\la\rho)\big)^{\frac 56 (1-3\ze)}\nu_{\de,0}^\la\, ,
$$
where $\ze$ lies  in the complex strip $\Sigma$ given by   $0\le \Re \zeta\le1.$  
 Since the supports of the $\widehat{\nu_{\de,0}^{\la}}$ are almost disjoint, and since, according to \eqref{11.4}, $\|\nu_{\de,0}^{\la}\|_\infty \lesssim \rho^{-4/3} (\la\rho)^{ 5/3},$ we see that 
$$
\|\widehat{\nu_{\de,\,it}^{III}}\|_\infty\lesssim 1\qquad \forall t\in\RR.
$$
Again, by  Stein's interpolation  theorem, it will therefore suffice to prove the following estimate:
\begin{equation}\label{13.1}
|\nu_{\de,\,1+it}^{III}(x)|\le C \qquad \forall t\in\RR, x\in \RR^3.
\end{equation}

Now, if $|x|\gg1,$ arguing in a similar way as for the case $B=4$ in Subsection \ref{Aig}, we see that 
$|\nu_{\de,0}^{\la}(x)|\lesssim \rho^{\frac 23+\frac 23}\la^3(\la\rho^{2/3})^{-N}$ for every $N\in\NN,$ which allows to sum absolutely in $\la$ and obtain  \eqref{13.1}.

\medskip
From now on, we shall therefore assume that $|x|\lesssim 1.$ We then re-write
\begin{equation}\label{13.2}
\nu_{\de,\,1+it}^{III}(x)=\ga(1+it)\sum_{2^M\rho^{-1}< \la\le 2^{-M}\de_0^{-3}}
\big(\rho^{-\frac 45}(\la\rho)\big)^{-\frac 52 it}\mu_\la(x)\, ,
\end{equation}
where $\mu_\la:=\rho^{4/3}(\la\rho)^{-5/3}\nu_{\de,0}^\la.$  Recall also from \eqref{11.5} that $\|\nu_{\de,0}^{\la}\|_\infty \lesssim  \rho^{-\frac 43} (\la\rho)^{\frac 53},$  which barely fails to be sufficient to obtain  \eqref{13.1}.

We therefore need again a more refined reasoning.  Following our discussion for the case $B=4$ in Subsection \ref{Aig}, we decompose 
$
\nu^\la_{\de,0}=\nu_{0,I}^\la+\nu_{0,II}^\la
$
as in \eqref{nudedec}. In the same way in which we had derived  \eqref{10.21}, we find here  that
\begin{equation}\label{nuIIIest}
\|\nu_{0,I}^\la(x)\|_\infty \le C_N \rho^{\frac 23}\la^2\,(\la\rho)^{-N}.
\end{equation}
If we denote by $\mu_{\la,I}:=\rho^{4/3}(\la\rho)^{-5/3}\nu_{0,I}^\la,$ then this estimate shows that we can sum the corresponding series in \eqref{13.2}, with $\mu_\la$ replaced by $\mu_{\la,I},$ absolutely, and obtain the desired uniform estimate in $x$ and $\de.$

\medskip
What remains are the contributions by the $\nu_{0,II}^\la.$ In order to keep the notation simple, we therefore shall assume from now on that $\mu_\la=\rho^{4/3}(\la\rho)^{-5/3}\nu_{0,II}^\la.$  

In analogy with our formulas \eqref{nudon4} and \eqref{Phi4n}, we  then find  the following expressions for 
$\mu_\la:$
\begin{eqnarray*}\nonumber
\mu_\la(x)&=&(\la\rho)^{\frac 13} \int e^{-i\la s_3 \Phi_4(y_1,u_2,w,x,\de)}
 \widehat{\chi_0}(s_3y_1)\, \chi_0(w)\, \chi_0(w -(\la\rho)^{-1} y_1) \nonumber \\
&&\hskip1cm\times \chi_0(u_2)\, a_4\big((\la\rho^{\frac 23})^{-1} y_1,\rho^{\frac 13} u_2,w,s_1,\rho^{\frac 13},x, \de\big) \tilde\chi_1(s_3)\, dy_1du_2dw ds_3\, 
\end{eqnarray*}
with phase $\Phi_4$ of the form
\begin{eqnarray*}\nonumber
&&\Phi_4(y_1,u_2,w,x,\de)= \Psi_3\big(\rho^{\frac 13}w, x,\de\big)+\rho \big(w -(\la\rho)^{-1} y_1\big)^3\, \tilde B_3\big(\rho^{\frac 13}w,(\la\rho^{\frac 23})^{-1} y_1,x,\de\big)  \nonumber \\
&&\hskip2cm+ \rho\Big(u_2^3 \,b\big(\rho^{\frac 13}w,(\la\rho^{\frac 23})^{-1} y_1,\rho^{\frac 13}u_2,x,\de_0^{\r}\big) +\de'_{3,0} u_2 \,\tilde\al_1\big(\rho^{\frac 13}w,x,\de_0^{\r}\big)\\
&&\hskip2cm
+\de'_{0} u_2 \,\big( w -(\la\rho)^{-1} y_1\big)\, \al_{1,1}\big(\rho^{\frac 13}w,(\la\rho^{\frac 23})^{-1} y_1,x,\de_0^{\r}\big)\Big)\, .\nonumber
\end{eqnarray*}
Recall  that in this integral, $|u_2|+|w|\lesssim 1$ and $|y_1|\lesssim \la\rho.$ 
Moreover, the factor $\widehat{\chi_0}(s_3y_1)$ guarantees the absolute convergence of this integral with respect to the variable $y_1.$ We  also recall that
$
\de'_0+\de'_{3,0}\sim 1,
$
and that the coefficient $\de'_{3,0}$ does not appear in Case ND, where $\al_{1,1}=0$.  Finally, we perform the change of variables $u_2=(\la\rho)^{-1/3} y_2$ and obtain 
\begin{eqnarray}\nonumber
\mu_\la(x)&=& \int e^{-is_3 \Phi_5(y,w,x,\de;\la)}
 \widehat{\chi_0}(s_3y_1)\, \chi_0(w)\, \chi_0(w -(\la\rho)^{-1} y_1) \nonumber \\
&&\hskip1cm\times \chi_0((\la\rho)^{-1/3} y_2)\, a_5\big((\la\rho^{\frac 23})^{-1} y_1,\la^{-\frac 13} y_2,w,s_1,\rho^{\frac 13},x, \de\big) \tilde\chi_1(s_3)\, ds_3dy_1dy_2 dw \, ,\label{nudon43}
\end{eqnarray}
with phase $\Phi_5$ of the form
\begin{eqnarray}\nonumber
&&\Phi_5(y,w,x,\de;\la)=\la \Psi_3\big(\rho^{\frac 13}w, x,\de\big)+\la\rho \big(w -(\la\rho)^{-1} y_1\big)^3\, \tilde B_3\big(\rho^{\frac 13}w,(\la\rho^{\frac 23})^{-1} y_1,x,\de\big)  \nonumber \\
&&\hskip3cm+ y_2^3 \,\tilde b\big(\rho^{\frac 13}w,(\la\rho^{\frac 23})^{-1} y_1,\la^{-\frac 13}y_2,x,\de_0^{\r}\big) 
\label{Phi4n3}\\
&&\hskip0cm
+y_2\,  (\la\rho)^{\frac 23} \Big( \de'_{3,0}  \,\tilde\al_1\big(\rho^{\frac 13}w,x,\de_0^{\r}\big)
+\de'_{0} \,\big( w -(\la\rho)^{-1} y_1\big)\, \tilde\al_{1,1}\big(\rho^{\frac 13}w,(\la\rho^{\frac 23})^{-1} y_1,x,\de_0^{\r}\big)\Big)\, .\nonumber
\end{eqnarray}
Performing a Taylor expansion with respect to the  bounded quantities  $(\la\rho^{\frac 23})^{-1} y_1$ and 
$(\la\rho)^{-1} y_1,$ we see that we may re-write the phase as
\begin{eqnarray}\nonumber
&&\Phi_5(y,w,x,\de;\la)=A+ By_2+ \tilde b\big(\rho^{\frac 13}w,(\la\rho^{\frac 23})^{-1} y_1,\la^{-\frac 13}y_2,x,\de_0^{\r}\big)\, y_2^3\\
&& \hskip2cm+r_1(y_1) +(\la\rho)^{-\frac 13}y_2r_2(y_1)\, , \label{Phi5n}
\end{eqnarray}
where 
\begin{eqnarray*}\nonumber
&&A:= \la\Big[\Psi_3\big(\rho^{\frac 13}w, x,\de\big)+\rho w^3 \tilde B_3\big(\rho^{\frac 13}w,0,x,\de\big) \Big]
=:\la\, Q_A(\rho^{\frac 13}w,x,\de);\nonumber\\
&&B:=(\la\rho)^{\frac 23} \Big( \de'_{3,0}  \,\tilde\al_1\big(\rho^{\frac 13}w,x,\de_0^{\r}\big)+
\de'_{0}  w  \tilde\al_{1,1}\big(\rho^{\frac 13}w,0,x,\de_0^{\r}\big)\Big)=:\la^{\frac 23}\, Q_B(\rho^{\frac 13}w,x,\de)\, ,
\end{eqnarray*}
and where $r_1$ and $r_2$ are of the form
\begin{eqnarray*}
&&r_1(y_1)=R_1\big(w,(\la\rho)^{-1} y_1, {\rho^\frac 13}, x,\de\big)\, y_1\, ,\\
&&r_2(y_1)=R_2\big(w,(\la\rho)^{-1} y_1, {\rho^\frac 13}, x,\de\big)\, y_1\, ,
\end{eqnarray*}
with smooth functions $R_1,R_2$ of their bounded entries $w,(\la\rho)^{-1} y_1, {\rho^\frac 13}, x$ and $\de.$

Let us put, for $w$ fixed such that $|w|\lesssim 1,$ 
\begin{eqnarray*}\nonumber
\mu_\la(w,x)&:=& \int e^{-is_3 \Phi_5(y,w,x,\de;\la)}
 \widehat{\chi_0}(s_3y_1)\, \chi_0(w -(\la\rho)^{-1} y_1) \nonumber \\
&&\hskip0cm\times \chi_0((\la\rho)^{-1/3} y_2)\, a_5\big((\la\rho^{\frac 23})^{-1} y_1,\la^{-\frac 13} y_2,w,s_1,\rho^{\frac 13},x, \de\big) \tilde\chi_1(s_3)\, ds_3dy_1dy_2 \, .
\end{eqnarray*}
By means of integrations by parts in $s_3$ and exploiting the rapid decay of $\widehat{\chi_0}(s_3y_1)$ in $y_1,$  we may estimate
\begin{eqnarray*}
&&|\mu_\la(w,x)|\le C_N \int\int\limits_{|y_2|\lesssim(\la\rho)^{\frac 13}}\int \Big( 1+\big|A+ By_2+ \tilde b\big(\rho^{\frac 13}w,(\la\rho^{\frac 23})^{-1} y_1,\la^{-\frac 13}y_2,x,\de_0^{\r}\big)\, y_2^3 \\
 &&\hskip3cm +r_1(y_1)+(\la\rho)^{-\frac 13}y_2r_2(y_1)\big|\Big)^{-N}(1+|y_1|)^{-N} \, dy_1dy_2\, .
\end{eqnarray*}

Observe first that  $|B|\lesssim (\la\rho)^{\frac 23}.$ Thus, if $|A|\gg \la\rho,$ then the term $A$ becomes dominant, and  we can clearly estimate $|\mu_\la(x)|\le C |A|^{-N}$ for every $N\in\NN.$  Otherwise, if  $|A|\lesssim \la\rho,$
then by choosing $T:=c(\la\rho)^{1/3}$ in Lemma \ref{as1}, with  a suitable constant $c>0,$  we see that all assumptions of this lemma are satisfied, and we obtain the estimate
\begin{equation}\label{13.6}
|\mu_\la(w,x)|\le C \max\{|A|^{\frac 13}, |B|^{\frac 12}\}^{-\frac 12}.
\end{equation}

This estimate thus holds no matter how large $|A|$ is. 

Consider the function
$$
F(t,w,x,\de):=\sum_{2^M\rho^{-1}< \la\le 2^{-M}\de_0^{-3}}
\big(\rho^{-\frac 45}(\la\rho)\big)^{-\frac 52 it}\mu_\la(w,x)\, ,
$$
for $|w|\lesssim 1.$ We shall prove that 
\begin{equation}\label{13.7}
|F(t,w,x,\de)|\le C \frac 1{|2^{-i\frac 52 t}-1|},
\end{equation}
with a constant $C$ not depending on $t,w, x$ and $\de.$ This estimate will immediately yield the desired estimate for the contributions of the $\nu_{0,II}^\la$ and thus complete the proof of \eqref{13.1}, provided we choose 
$$
\ga(\ze):= \frac{2^{\frac 52 (1-\ze)}-1}{2^{\frac 53}-1}.
$$
Given $w,x,\de,$ denote by $\La(w,x,\de)$ the set of all dyadic $\la$ from our summation range 
$\La:=\{\la=2^j:2^M\rho^{-1}< \la\le 2^{-M}\de_0^{-3}\},$
 for which either  $|A|=\la\, |Q_A(\rho^{\frac 13}w,x,\de)|>1$ or $|B|=\la^{\frac 23}\, |Q_B(\rho^{\frac 13}w,x,\de)|>1.$ 
We the decompose $F(t,w,x,\de)=F_1(t,w,x,\de)+F_2(t,w,x,\de),$ where $F_1(t,w,x,\de)$ and  $F_2(t,w,x,\de)$ are defined as $F,$ only with summation restricted to the subsets $\La(w,x,\de)$ and $\La\setminus \La(w,x,\de),$ respectively.  Then, by \eqref{13.6}, we clearly have that 
$$
|F_1(t,w,x,\de)|\le \sum_{\la\in \La(w,x,\de)}
|\mu_\la(w,x)|\le C,
$$
and we are thus left with $F_2(t,w,x,\de).$ 

In the corresponding sum, we have $\la\, |Q_A(\rho^{\frac 13}w,x,\de)|\le1$ and  $\la^{\frac 23}\, |Q_B(\rho^{\frac 13}w,x,\de)|\le 1,$  and therefore  $F_2$ can again be estimated by means of Lemma \ref{simplesum}. Indeed,  we may here put
\begin{eqnarray*}
&&H(u_1,\dots, u_6)\\
&&\hskip0.5cm:=\int e^{-is_3 \big(u_1+u_2y_2+\tilde b (\rho^{\frac 13}w,u_3 y_1,u_4y_2,x,\de_0^{\r})\, y_2^3 +
R_1(w,u_5 y_1, {\rho^\frac 13}, x,\de)\, y_1 +u_6 y_2 R_2(w,u_5 y_1, {\rho^\frac 13}, x,\de)\big)}\\
 &&\hskip0.5cm\widehat{\chi_0}(s_3y_1)\, \chi_0(w -u_5 y_1) \chi_0(u_6 y_2)\, a_5\big(u_3 y_1,u_4 y_2,w,s_1,\rho^{\frac 13},x, \de\big) \tilde\chi_1(s_3)\, ds_3dy_1dy_2 \, ,
\end{eqnarray*}
where the variables $u_1,\dots, u_6$ correspond to the bounded expressions $\la Q_A(\rho^{\frac 13}w,x,\de),$  
$\la^{\frac 23}Q_B(\rho^{\frac 13}w,x,\de), (\la\rho^{2/3})^{-1}, \la^{-1/3}, (\la\rho)^{-1}$ and $(\la\rho)^{-1/3},$ respectively. By means of integrations by parts in the variable  $y_2$ for $|y_2|\gg 1$ (or, alternatively, in $s_3$), it is then easily verified that 
$\|H\|_{C^1(Q)}\le C,$  where $Q$ denotes the obvious  cuboid $Q$  appearing in this situation.  Thus, estimate \eqref{13.7} follows from Lemma \ref{simplesum}.

\subsection{Estimation of $T^{IV}_{\de}$}\label{deIV}
The estimation of the operator $T^{IV}_{\de}$ will follow similar ideas as the one for $T^{II}_{\de}.$ Nevertheless, for the convenience of the reader, we will give  some details.

As usually, we embed   $\nu_{\de}^{IV}$ into an analytic  family of measures 
$$
\nu^{IV}_{\de,\zeta}:=\ga(\zeta)\sum_{\{l:M_0\le 2^l\le \frac{\rho^{-1}}{M_1}\}}\sum_{2^M\rho^{-1}< \la\le 2^{-M}\de_0^{-3}}  \big((2^l\rho)^{-\frac 45}(\la2^l\rho)\big)^{\frac 56 (1-3\ze)}\nu_{l,0}^\la,
$$
where $\ze$ lies  in the complex strip $\Sigma$ given by   $0\le \Re \zeta\le1.$  
 Since the supports of the $\widehat{\nu_{\de,0}^{\la}}$ are almost disjoint,  estimate  \eqref{11.10} shows that
$$
\|\widehat{\nu_{\de,\,it}^{IV}}\|_\infty\lesssim 1\qquad \forall t\in\RR.
$$
Again, by  Stein's interpolation  theorem, it will therefore suffice to prove the following estimate:
\begin{equation}\label{13.9}
|\nu_{\de,\,1+it}^{IV}(x)|\le C \qquad \forall t\in\RR, x\in \RR^3,
\end{equation}
where we write 
\begin{equation}\label{13.10}
\nu_{\de,\,1+it}^{IV}(x)=\ga(1+it) \sum_{\{l:M_0\le 2^l\le \frac{\rho^{-1}}{M_1}\}}\sum_{2^M\rho^{-1}< \la\le 2^{-M}\de_0^{-3}}  (\la (2^l\rho)^{\frac 15})^{-\frac 52 it}\mu_{l,\la}(x),
\end{equation}
with $\mu_{l,\la}:=\la^{-5/3} (2^l\rho)^{-1/3}\nu_{l,0}^\la.$

Regretfully, it seems that the approach in the previous subsection cannot be applied in the present situation, and a  more refined analysis is needed, similar to our discussion in Subsection \ref{tdII}.  From \eqref{nudel1} to 
\eqref{11.12} we get  that
\begin{eqnarray*}
\mu_{l,\la}=(2^l\rho) \la^{\frac 43}\int e^{-i\la s_3 \Phi_2(u,z,s_2,x,\de)}
\chi_1(z)\chi_0(u) a(\si_{2^l\rho} u, (2^l\rho)^{\frac 23} z,s,\de) \tilde\chi_1(s_2, s_3)\, du dz ds_2 ds_3,
\end{eqnarray*}
where
\begin{eqnarray}\nonumber
&&\Phi_2(u,z,s_2,x,\de)= s_2^{\frac {n}{n-2}}G_5(s_2,\de)- x_1 s_2^{\frac {n-1}{n-2}}G_3(s_2,\de)- s_2x_2-x_3\\
&&\hskip1.5cm +z(2^l\rho)\Big((2^l\rho)^{-\frac 13} \big(x_1-s_2^{\frac {1}{n-2}}G_1(s_2,\de)\big)-  u_1\Big)+(2^l\rho) u_1^3\, B_3\big(s_2,\de_1,(2^l\rho)^{\frac 13}u_1\big) \label{Phi2nl}\\
&&\hskip1.5cm+ (2^l\rho)\Big(u_2^3 \,b(\si_{2^l\rho} u,\de_0^\r,s_2)
+\de'_{3,0} u_2 \,\tilde\al_1(\de_0^{\r},s_2)\nonumber
+\de'_{0} u_1u_2 \, \al_{1,1}\big((2^l\rho)^{\frac 13}u_1,\de_0^{\r},s_2\big)\Big).\nonumber
\end{eqnarray}
Now,  if $|x|\gg1,$ we see that 
$|\mu_{l,\la}(x)|\lesssim \,2^l\rho\la^{4/3}(\la (2^l\rho)^{2/3})^{-N}$ for every $N\in\NN,$ which is stronger than what is needed for \eqref{13.9}.

\medskip
From now  on we shall therefore  assume that $|x|\lesssim 1.$  For such $x$ fixed,  we again decompose 
\begin{equation}\label{nudedecl}
\nu^\la_{l,0}=\nu_{l,I}^\la+\nu_{l,II}^\la,
\end{equation}
where $\nu_{l,I}^\la$  and $\nu_{l,II}^\la$ denote the contributions to the integral above  by the  region $L_{I}$  where
$
|x_1-s_2^{\frac {1}{n-2}}G_1(s_2,\de)|\gg (2^l\rho)^\frac13, 
$
and the region $L_{II}$  where 
$
|x_1-s_2^{\frac {1}{n-2}}G_1(s_2,\de)|\lesssim (2^l\rho)^\frac13, 
$
respectively.  Then, in analogy with \eqref{nuIIIest}, by means of integrations by parts in $z$ we obtain 
\begin{equation}\label{nuIIIestl}
\|\nu_{l,I}^\la(x)\|_\infty \le C_N =(2^l\rho)^{\frac 23}\la^2\,(\la2^l\rho)^{-N}.
\end{equation}
If we denote by $\mu_{l,\la,I}:=\la^{-5/3} (2^l\rho)^{-1/3}\nu_{l,I}^\la,$ then this estimate shows that we can sum the corresponding series in \eqref{13.10}, with $\mu_{l,\la}$ replaced by $\mu_{l,\la,I},$ absolutely, and obtain the desired uniform estimate in $x$ and $\de.$

\medskip
What remains are the contributions by the $\nu_{l,II}^\la.$ In order to keep the notation simple, we therefore shall assume from now on that $\mu_{l,\la}=\la^{-5/3} (2^l\rho)^{-1/3}\nu_{l,I}^\la,$ i.e., that
\begin{eqnarray}\label{13.14}
&&\mu_{l,\la}(x)=(2^l\rho) \la^{\frac 43}\int e^{-i\la s_3 \Phi_2(u,z,s_2,x,\de)}
 a\big((2^l\rho)^{\frac 13} u, (2^l\rho)^{\frac 23} z,s,\de\big)\, \chi_1(z)\chi_0(u)\nonumber\\
&&\hskip3cm \chi_0\big((2^l\rho)^{-\frac13}(x_1-s_2^{\frac {1}{n-2}}G_1(s_2,\de)\big) \, \tilde\chi_1(s_2, s_3)\, du dz ds_2 ds_3,
\end{eqnarray}
Given a point $u^0,s_2^0,z^0$ such that the amplitude in this integral does not vanish, we want to understand the contribution of a small neighborhood of this point to the integral.  Assume first that $\pa_{u_1}\Phi_2(u^0,z^0,s^0_2,x,\de)\ne 0.$ Then, integrations by parts in $u_1$ allow to gain factors $(\la2^l\rho)^{-N},$  and so we can again sum the corresponding contributions to  $\nu_{\de,\,1+it}^{IV}(x)$ absolutely. 

Let us next assume that $\pa_{u_1}\Phi_2(u^0,z^0,s^0_2,x,\de)= 0.$ For a short while, it  will then be helpful to change coordinates from  $s_2$ first to $v:=x_1-s_2^{\frac {1}{n-2}}G_1(s_2,\de),$ and then to $w:=
(2^l\rho)^{-\frac 13}v= (2^l\rho)^{-\frac 13}\big(x_1-s_2^{\frac {1}{n-2}}G_1(s_2,\de)\big)$ in a similar way is in Subsection \ref{Aig}, and re-write
\begin{eqnarray*}
&&\mu_{l,\la}(x)=(\la 2^l\rho)^{\frac 43}\int e^{-i\la s_3 \tilde\Phi_2(u,z,w,s_2,x,\de)}
 a\big((2^l\rho)^{\frac 13} u, (2^l\rho)^{\frac 23} z,(2^l\rho)^{\frac 13}w,x,\de\big)\, \chi_1(z)\chi_0(u)\nonumber\\
&&\hskip7cm \chi_0(w)) \, \chi_1(s_3)\, du dz ds_2 ds_3,
\end{eqnarray*}
where
\begin{eqnarray*}\nonumber
&&\tilde\Phi_2= z(2^l\rho)(w-u_1)+\Psi_3(\rho^{\frac 13} w,x,\de)+(2^l\rho) u_1^3\,  B_3\big((2^l\rho)^{\frac 13} w,(2^l\rho)^{\frac 13}u_1,x,\de\big) \\
&&\hskip2cm+ (2^l\rho)\Big(u_2^3 \,b\big((2^l\rho)^{\frac 13} u,(2^l\rho)^{\frac 13} w,x,\de_0^\r\big)
+\de'_{3,0} u_2 \,\tilde\al_1\big(2^l\rho)^{\frac 13} w,x,\de_0^{\r}\big)\\
&&\hskip2cm+\de'_{0} u_1u_2 \, \al_{1,1}\big(2^l\rho)^{\frac 13}u_1,2^l\rho)^{\frac 13} w,x,\de_0^{\r}\big)\Big)\, ,\nonumber
\end{eqnarray*}
By $w^0$ we denote the value of $w$ corresponding to $s_2^0.$ We now can see that there is also a unique critical point of the phase with respect to the variable $z,$ at $z^0,$  provided $w=u_1^0.$ Thus, if $w^0=u_1^0,$ then the phase has a critical point with respect to the variable $(u_1,s_2)$ respectively $(u_1,w);$  otherwise, we can again integrate by parts  in $z,$ which allows to gain  factors $(\la2^l\rho)^{-N}$ as before, and we are done. 
So, assume that $w^0=u_1^0.$ Since the phase is linear in $z,$ and since $|\pa_z\pa_{u_1}\tilde \Phi_2|\sim 2^l\rho$ at the critical point, we see that we may apply the method of stationary phase to  the double integration with respect to the variables $(u_1,z)$  and gain in  particular  a factor $(\la2^l\rho)^{-1}.$ Having realized this, we may come back to our previous formula \eqref{13.14}, and knowing  that we may apply the method of stationary phase to the integration with  respect to the variables 
$(u_1,z)$ as well, we see that we may essentially write 
\begin{eqnarray}\label{13.15}
&&\mu_{l,\la}(x)=\la^{\frac 13}\int e^{-i\la s_3 \Psi_2(u_2,s_2,x,\de,l)}
 a_2\big((2^l\rho)^{\frac 13} u_2, (2^l\rho)^{\frac 13},s,\de\big)\,\chi_0(u_2)\nonumber\\
&&\hskip3cm \chi_0\big((2^l\rho)^{-\frac13}(x_1-s_2^{\frac {1}{n-2}}G_1(s_2,\de)\big) \, \tilde\chi_1(s_2, s_3)\, du_2 ds_2 ds_3,
\end{eqnarray}
where the phase $\Psi_2$ arises from  $\Phi_2$ by replacing $(u_1,z)$ by the critical point $(u^0_1,z^0).$

Now, arguing exactly as in Subsection \ref{Aig}, by means of Lemma  7.1 in \cite{IM-rest1}  we find that the phase $\Psi_2$ is given by  the expression in \eqref{phasenew}, with $B=3,$ i.e., 
$$
\Psi_2=s_2x_1^2\omega(\de_1x_1)+x_1^n\alpha(\de_1 x_1)+s_2\de_0 y_2+y_2^3  \,b(x_1,y_2,\de)+ r(x_1,y_2,\de)  -s_2x_2 -x_3.
$$
Moreover, since we here have changed coordinates from $y_2$ to $u_2$ so that $y_2=(2^l\rho)^{1/3}u_2,$ 
this means that
\begin{eqnarray}\nonumber
&&\Psi_2(u_2,s_2,x, \de, l)=s_2x_1^2\omega(\de_1x_1)+x_1^n\alpha(\de_1 x_1)-s_2x_2 -x_3\\
&&\hskip3cm+ (2^l\rho) \,u_2^3 \,b\big(x_1,(2^l\rho)^{1/3}u_2,\de\big)+
(2^l\rho)^{1/3} u_2 [\de_0 s_2+\de_3 x_1^{n_1}\al_1(\de_1x_1)]  \label{13.16}
\end{eqnarray}
(compare \eqref{7.3II}).
Note  that $\pa_{s_2} (s_2^\frac1{n-2}G_1(s_2,\de))\sim 1$ because $s_2\sim1$ and $G_1(s_2,0)=1.$ Therefore, the relation $|x_1-s_2^\frac1{n-2}G_1(s_2,\de)|\lesssim (2^l\rho)^{1/3}$ can be re-written as 
$|s_2-\tilde G_1(x_1,\de)|\lesssim (2^l\rho)^{1/3},$ where $\tilde G_1$ is again a smooth function such that $|\tilde G_1|\sim 1.$
If we write  
$$
s_2=(2^l\rho)^{\frac 13}v+\tilde G_1(x_1,\de),
$$
 then this means that $|v|\lesssim 1.$ We  shall therefore change variables  from $s_2$ to $v,$ which leads to the following expression for $\mu_{l,\la}(x):$
 \begin{eqnarray*}
&&\mu_{l,\la}(x)=(\la2^l\rho)^{\frac 13}\int e^{-i\la s_3 \Psi_3(u_2,v,x,\de,l)}
 a_3\big((2^l\rho)^{\frac 13} u_2, (2^l\rho)^{\frac 13}v,x,\de\big)\nonumber\\
&&\hskip5cm \times \chi_0(u_2)\chi_0(v) \, \chi_1(s_3)\, du_2 dv ds_3,
\end{eqnarray*}
with a smooth amplitude $a_3$ and the new phase function
\begin{eqnarray*}\nonumber
&&\Psi_3(u_2,v,x, \de, l)=v\, (2^l\rho)^{\frac 13}\big( x_1^2\omega(\de_1x_1)-x_2\big)+  (2^l\rho)^{\frac 23}  \de_0\, vu_2
+Q_A(x,\de) \\
&&\hskip3cm+ (2^l\rho) \,u_2^3 \,b\big(x_1,(2^l\rho)^{\frac 13}u_2,\de\big)+
(2^l\rho)^{\frac 13} u_2 Q_D(x,\de) 
\end{eqnarray*}
(compare with  the corresponding expressions in \eqref{12.10} to \eqref{12.12}). Finally, putting $y_2:=(\la2^l\rho)^{1/3} u_2,$ we find that 
\begin{eqnarray}\label{13.17}
&&\mu_{l,\la}(x)=\int e^{-i s_3 \Psi_4(y_2,v,x,\de,\la,l)}
 a_4\big(\la^{-\frac 13} y_2, (2^l\rho)^{\frac 13}v,x,\de\big)\nonumber\\
&&\hskip5cm \times \chi_0\big((\la 2^l\rho)^{-\frac 13}y_2\big)\, \chi_1(s_3)\,\chi_0(v)\,  dv ds_3 dy_2\\
&&=\int\widehat{\chi_1}\big(\Psi_4(y_2,v,x,\de,\la,l)\big)\, a_4\big(\la^{-\frac 13} y_2, (2^l\rho)^{\frac 13}v,x,\de\big)\,\chi_0(v)\,\chi_0\big((\la 2^l\rho)^{-\frac 13}y_2\big)\, dy_2 dv \nonumber
\end{eqnarray}
with a smooth amplitude $a_4$ and  phase function
\begin{eqnarray*}\nonumber
&&\Psi_4(y_2,v,x, \de,\la,l)=v\, \la(2^l\rho)^{\frac 13}\big( x_1^2\omega(\de_1x_1)-x_2\big)+  \la^{\frac 23}(2^l\rho)^{\frac 13}  \de_0\, vy_2+\la Q_A(x,\de) \\
&&\hskip3cm+ \,y_2^3 \,b\big(x_1,\la^{-\frac 13}y_2,\de\big)+
\la^{\frac 23} y_2 Q_D(x,\de) 
\end{eqnarray*}
 We shall write this as 
\begin{equation}\label{13.18}
\Psi_4(y_2,v,x, \de,\la,l)=A+Bv + y_2^3 \,b\big(x_1,\la^{-\frac 13}y_2,\de\big)+   y_2 (D+ Ev), 
\end{equation}
with
\begin{eqnarray}\nonumber
&& A:=\la Q_A(x,\de) , \qquad B:=\la2^{\frac l3}\, \rho^{\frac 13} Q_B(x,\de),\\
&& D:=\la^{\frac 23}Q_D(x,\de),  \qquad E:= \la^{\frac 23}2^{\frac l3} \, (\rho^{\frac 13}\de_0)\, . \label{13.19}
\end{eqnarray}
Here, $Q_A(x,\de),Q_B(x,\de)$ and $Q_D(x,\de)$ are as in \eqref{12.12}. 

\medskip

Applying Lemma \ref{as2} in the appendix with  $T:=(\la2^l\rho)^{1/3} =(\la\rho)^{1/3} 2^{l/3}\gg 1$  and $\de:=\la^{-1/3}$ so that $\de T=(2^l\rho)^{1/3}\ll1,$ and subsequently Lemma \ref{simpleintest} in a similar way as in Subsection \ref{tdII}, we find that in analogy with \eqref{mullaII} we have that
\begin{equation}\label{mullaIV}
|\mu_{l,\la}(x)|\lesssim \big(\max\{|A|, |B|,|D|,|E|\}\big)^{-\frac 16},
\end{equation}
now with $A,B,D$ and $E$ given by \eqref{13.19}.
\medskip

In order to estimate $\nu_{\de,1+it}^{IV}(x),$ we shall again distinguish various cases.
\medskip

As in the discussion of $\nu_{\de,1+it}^{II}(x)$ in Subsection \ref{tdII}, the contributions by those terms in \eqref{13.10} for which either   $|D|\gtrsim1$ and $|E|\gtrsim1,$   or   $|A|\gtrsim 1$ and $|B|\gtrsim 1,$ can easily handled by means of estimate \eqref{mullaIV} (compare  with \eqref{13.19}). 

\medskip
Consider next the terms for which  $|D|\ll 1$ and $|E|\ll 1.$ For these terms, it will   be useful to re-write $\mu_{l,\la}(x)$ in analogy with \eqref{mulaIIint} as  
 \begin{equation}\label{13.21}
\mu_{l,\la}(x)=  \int e^{-i s_3 (A+Bv)} J(v,s_3)\,\chi_1(s_3)\,\chi_0(v)\,\chi_0(D+Ev)\, dvds_3,
\end{equation}
with
\begin{eqnarray}\nonumber
&&J(v,s_3):=  \int e^{-i s_3 \Big(y_2^3b\big(x_1,\la^{-\frac 13}y_2,\de\big)+   y_2 (D+ Ev)\Big)}
 a_4\big(\la^{-\frac 13} y_2, (2^l\rho)^{\frac 13}v,x,\de\big)\nonumber\\
&&\hskip5cm \times \chi_0\big((\la 2^l\rho)^{-\frac 13}y_2\big) \, dy_2\,.\label{13.22}
\end{eqnarray}
\medskip
From here we arrive without loss of generality  at the following analogue of \eqref{12.14}:
\begin{equation}\label{13.21b}
\mu_{l,\la}(x)= \int  g(D+Ev, x,\la^{-\frac13},(2^l\rho)^{\frac 13},\de)\, \widehat { \chi_1}(A+Bv) \,\chi_0(v)\,\, dv,
\end{equation}
where $g$ is a smooth function of its bounded arguments. 

 \medskip 
Recall that we  assume that either $|A|\gtrsim1$ and $|B|\ll1,$ or $|A|\ll 1$ and $|B|\gtrsim1,$ or $|A|\ll 1$ and $|B|\ll1.$

\medskip
If $|A|\gtrsim1$ and $|B|\ll1,$ then we can treat the summation in $l$ by means of Lemma \ref{simplesum}, where we choose, for $\la$ fixed,
$$
H_{\la,x}(u_1,u_2,u_3,u_4):=\int  g(D+u_1v, x,u_3,u_4,\de)\, \widehat { \chi_1}(A+u_2v) \, \chi_0(v)\,\, dv.
$$
Then clearly $\|H_{\la,x}\|_{C^1(Q)}\lesssim |A|^{-1},$ and so after summation in those $l$ for which  $|E|\ll 1$ and $|B|\ll 1 ,$ we can also sum (absolutely) in  the $\la$'s for which $|A|\gtrsim 1.$ Observe that this requires that $\ga(\ze)$  contains  a factor 
$$
\ga_1(\ze):= \frac{2^{\frac{1-\ze}2}-1}{2^{\frac 13}-1}.
$$

\medskip
Consider next the case where $|B|\gtrsim 1$ and $|A|\ll 1.$ If we write $\la=2^j,$ then $\la 2^{l/3}=2^{k/3},$  if we put $k:=l+3j.$ We therefore pass from the summation variables $j$ and $l$ to the variables $j$ and $k,$  which allows to write 
$B=2^{ k/3}  \rho^{1/3}Q_B(x).$ For $k$ fixed, we then sum first in $j$ by means of Lemma \ref{simplesum}, which  gives an estimate of order $O(|B|^{-1}),$ which then in return allows to sum (absolutely) in those $k$ for which $|B|\gtrsim 1.$ The application of Lemma \ref{simplesum} requires  in this case that $\ga(\ze)$  contains  a factor 
$$
\ga_2(\ze):= \frac{2^{1-\ze}-1}{2^{\frac 23}-1}.
$$

\medskip
There remains the sub-case where $|A|+|B|\ll 1$ and $|D|+|E|\lesssim 1.$ The summation over all $\l$'s and $\la$'s  for which these conditions are satisfied can easily be treated by means of the double summation  Lemma 8.1 in \cite{IM-rest1}, in a very similar way as in Subsection \ref{tdII}.

\medskip 
What  remains are the contributions by those $l$ and $\la$ for which either $|D|\gtrsim1$ and $|E|\ll1,$ or $|D|\ll 1$ and $|E|\gtrsim1.$  
\medskip

We begin with the case where   $|E|\gtrsim1$ and  $|D|\ll1.$  Then we may assume in addition that $|B|\ll 1,$ for otherwise by \eqref{mullaIV} we have $|\mu_{l,\la}(x)|\lesssim |E|^{-1/12}|B|^{-1/12},$  which allows to sum absolutely in $j$ and $l.$   In a very similar way, we may also assume that $|A|\ll 1.$ 

Recall next from \eqref{13.17} and \eqref{13.18} that 
\begin{eqnarray*}
&&\mu_{l,\la}(x)=  \int e^{-i s_3 \big(A+ y_2^3 \,b\big(x_1,\la^{-\frac 13}y_2,\de\big)+D y_2+  v(B+ Ey_2)\big)}
 a_4\big(\la^{-\frac 13} y_2, (2^l\rho)^{\frac 13}v,x,\de\big)   \\
&&\hskip6cm\times \chi_1(s_3) \chi_0(v)\,\chi_0\big((\la 2^l\rho)^{-\frac 13}y_2\big)\, dy_2 dv ds_3\, .
\end{eqnarray*}
Again, by our usual argument, we may assume without loss of generality that $a_4$ is independent of $v.$ Then we  find that, in analogy with \eqref{12.15b}, 
\begin{eqnarray*}
&&\mu_{l,\la}(x)=  \int e^{-i s_3 \big(A+ y_2^3 \,b\big(x_1,\la^{-\frac 13}y_2,\de\big))\big)}
\widehat{\chi_0}\big(s_3(B+Ey_2)\big) \nonumber  \\
&&\hskip2cm\times  a_4\big(\la^{-\frac 13} y_2, (2^l\rho)^{\frac 13},x,\de\big)  \chi_1(s_3) \,\chi_0\big((\la 2^l\rho)^{-\frac 13}y_2\big)  \, dy_2  ds_3\, .
\end{eqnarray*}
We then change the summation variables from $j,l$ to $j, k,$ where $k:=2j+l,$ so that $E=2^{k/3} \rho^{1/3}\de_0.$ Then, for $k$ fixed,  we can treat the summation in $j$ by means of Lemma \ref{simplesum}, where we choose
\begin{eqnarray*}
&&H_{k,x}(u_1,u_2,u_3,u_4,u_5):=  \int e^{-i s_3 \big(u_1+ y_2^3 \,b\big(x_1,u_2y_2,\de\big))\big)}
\widehat{\chi_0}\big(s_3(u_3+Ey_2)\big)   \\
&&\hskip4cm\times  a_4(u_2y_2,u_4, x,\de)  \chi_1(s_3) \,\chi_1(v)\,\chi_0(u_5 y_2)  \, dy_2  ds_3\, .
\end{eqnarray*}
Arguing in the same way as in the corresponding case of Subsection \ref{tdII}, we find by means of the change of variables $y_2\mapsto y_2/E$  that  $\|H_{\la,x}\|_{C^1(Q)}\lesssim |E|^{-1},$ and thus after the summation over the  $\la=2^j$  this allows to subsequently also sum over  the $k$  for which $|E|\gtrsim 1.$

\medskip

There remains  the contribution by those $l$ and $\la$ for which $|D|\gtrsim1$ and $|E|\ll1.$ Observe that here we have $|D+Ev|\gtrsim 1$  in \eqref{13.22}.

Applying  the change of variables $y_2=\la^{1/3} t$  in the integral defining $J(v,s_3),$  we obtain
\begin{eqnarray*}
J(v,s_3)=  \la^{\frac 13}  \int e^{-i s_3 \la\big(t^3b(x_1,t,\de)+   
t(\la^{-\frac 23}(D+ Ev))\big)}
a_4\big(t, (2^l\rho)^{\frac 13}v,x,\de\big) 
 \chi_0\big((2^l\rho)^{-\frac 13}t\big)\, \chi_0(v) \, \chi_1(s_3)\, dt.
\end{eqnarray*}
Observe also that by \eqref{13.19}
\begin{equation}\label{13.23}
\la^{-\frac 23}D=Q_D(x,\de), \quad   \la^{-\frac 23}E= \de_0\ll 1,
\end{equation}
so that in particular $\la^{-\frac 23}|D+vE|\lesssim 1.$

Arguing in a similar way as in Subsection \ref{tdII}, we find that for every $N\in \NN,$ 
\begin{eqnarray}
J(v,s_3)&=&|D+Ev|^{-\frac 14} a_+\big(\la^{-\frac 13} |D+Ev|^{\frac 12}, v, x,(2^l\rho)^{\frac 13},\de\big)   \nonumber\\
&&\hskip1.5cm \times \chi_0\Big({(\la 2^l\rho)^{-\frac 13}|D+Ev|^{\frac 12}}\Big)\, e^{-is_3|D+Ev|^{\frac 32}q_+\big(\la^{-\frac 13} |D+Ev|^{\frac 12}, x,\de\big)}  \label{13.24}\\
&+&|D+Ev|^{-\frac 14} a_-\big(\la^{-\frac 13} |D+Ev|^{\frac 12}, v, x,(2^l\rho)^{\frac 13},\de\big)   \nonumber \\
&&\hskip1.5cm \times \chi_0\Big({(\la 2^l\rho)^{-\frac 13}|D+Ev|^{\frac 12}}\Big)\, e^{-is_3|D+Ev|^{\frac 32}q_-\big(\la^{-\frac 13} |D+Ev|^{\frac 12},x,\de\big)}   \nonumber\\
&+&  (D+Ev)^{-N}\, F_N\Big( |D+vE|^{\frac 32}, \la^{-\frac 13} |D+Ev|^{\frac 12}, v, x,2^{-\frac l3},(2^{l}\la^{-1})^{\frac13},\de\Big), \nonumber
\end{eqnarray}
where  $a_\pm, q_\pm$ and $F_N$ are smooth functions of their (bounded) variables. 
Moreover, $|q_\pm(\big(0, x,(2^{l}\la^{-1})^{\frac13},\de\big)|\sim 1.$

\medskip
We shall concentrate on the first term only.  The second term can be treated in the same way as the first one, and  the last term can be handled in an even easier way by a similar method, since it is of order $=O(|D|^{-N})$ and, unlike the first term, carries no oscillatory factor.

\medskip
We denote by 
\begin{eqnarray*}
\mu^1_{l,\la}(x)&:=&  \int e^{-i s_3 \big[(A+Bv) +|D+Ev|^{\frac 32} q_+\big(\la^{-\frac 13} |D+Ev|^{\frac 12}, x,\de\big)\big]}   |D+Ev|^{-\frac 14}\chi_1(s_3) \,\chi_0(v)\,\nonumber\\
&&\hskip0cm \times a_+\big(\la^{-\frac 13} |D+Ev|^{\frac 12}, v, x,(2^l\rho)^{\frac 13},\de\big)  \, \chi_0\Big({(\la 2^l\rho)^{-\frac 13}|D+Ev|^{\frac 12}}\Big) \,   dvds_3 
\end{eqnarray*}
 the contribution by the first term in \eqref{13.24} to $\mu_{l,\la}(x),$ and by $ \nu^1_{\de,1+it}(x)$ the contribution of the $\mu^1_{l,\la}(x)$ to the sum defining $ \nu_{\de,1+it}^{IV}(x).$

\medskip
Assuming for instance that $D>0,$  and making use of  \eqref{12.22b}, we here find that the complete phase in the oscillatory integral defining 
$\mu^1_{l,\la}(x)$ is  of the form
$$
s_3[A'+B'v+ r],
$$
with $A'=\la Q_{A'}(x,\de), \ B'=\la 2^{\frac l3} Q_{B'}(x,\de),$ where 
\begin{eqnarray*}
Q_{A'}(x,\de)&:=&Q_A(x,\de)+Q_D(x,\de)^{\frac 32}  q_+\big(Q_D(x,\de)^{\frac 12}, x,\de\big),\\
Q_{B'}(x,\de)&:=&\rho^{\frac13}Q_B(x,\de)+\frac 12 q'_+\big(Q_D(x,\de)^{\frac 12}, x,\de\big) \rho^{\frac 13}Q_D(x,\de)\de_0 \\
&&+\frac 32 q'_+\big(Q_D(x,\de)^{\frac 12}, x,\de\big)\rho^{\frac 13}Q_D(x,\de)^{\frac 12}\de_0,
\end{eqnarray*}
and where $r$ is again a bounded error  term.

Thus, if $|B'|\gg 1,$ then an integration by parts in $v$ shows that 
$$
|\mu^1_{l,\la}(x)|\lesssim |D|^{-\frac 14} |B'|^{-1}.
$$
This estimate allows to control the sum over all $l$  such that $|B'|\gg 1,$ and  subsequently the sum over all dyadic $\la$ such that $|D|\gtrsim 1,$ and we arrive at the desired uniform estimate in $x$ and $\de.$ 

\medskip
Next, if $|B'|\lesssim 1,$  then  we can argue in a similar way as  before  and apply Lemma \ref{simplesum} to the summation in $l$  by putting here 
\begin{eqnarray*}
H_{\la,x}(u_1,\dots, u_7)&:=& \int e^{-i s_3 [A'+ u_1v+r(Q_D(x)^{\frac 12}, D^{-1} u_3, v,x,\de)]}   |D+u_2v|^{-\frac 14}\chi_1(s_3) \,\chi_0(v)\,\\
&&\hskip-2.5cm \times a_+\big(|Q_D(x,\de)+u_3v|^{\frac 12}, v, x,u_4,\de\big) 
\,\chi_0\Big({|u_6+u_7v|^{\frac 12}}\Big) \,  dvds_3
\end{eqnarray*}
and choosing the cuboid $Q$ in the obvious way.
Then we easily see that   $\|H_{\la,x}\|_{C^1(Q)}\lesssim |D|^{-1/4},$ and so after summation in those $l$ for which $|B'|\ll 1,$ we can also sum (absolutely) in  the $\la$'s for which $|D|\gtrsim 1.$ Observe that this requires again that $\ga(\ze)$  contains  the  factor $\ga_1(\ze).$

This concludes the proof of the uniform estimate of $\nu^1_{\de, 1+it}(x)$ in $x$ and $\de,$  and thus also of  estimate \eqref{13.9}.

\medskip 

The proof of Proposition \ref{s11.1} is now complete.

 \bigskip

\setcounter{equation}{0}
\section{Appendix: Integral estimates of van der Corput type}\label{appendix}

\begin{lemma}\label{as1}
Let $b=b(y_1,y_2)$ be a $C^2$- function on $\RR\times [-1,1]$ such that $b(0,0)\ne 0,$ $\|b(y_1,\cdot)\|_{C^2([-1,1])}\le c_1$ for every $y_1\in\RR$  and 
\begin{equation}\label{a1}
|b(y_1,y_2)-b(0,0)|\le \ve,\quad\mbox{and}\quad  c_2|y_2|^{3-j}\le \Big|\pa_{y_2}^j\Big(b\y y_2^3\Big)\Big|,\ j=1,2
\end{equation}
for every $\y\in\RR\times [-1,1],$  where $0<c_1\le c_2.$ Furthermore, let $Q=Q(y_2)$ be a smooth function on 
$[-1,1]$ such that $\|Q\|_{C^2([-1,1])}\le c_1,$and let $A,B,T$ be real numbers so that  $\max\{|A|,|B|\}\ge L,$ 
$T\ge L$ and 
\begin{equation}\label{a2}
|A|\le T^3,\quad |B|\le T^2,
\end{equation}
and let $r_i=r_i(y_1),\  i=1,2,$ be a measurable functions on $\RR$  such that $|r_i(y_1)|\le c(1+|y_1|).$ For $\epsilon\ge 0$ and  $N\ge 2,$ we put 
\begin{eqnarray*}
I_\epsilon(A,B,T)&:=&
 \hskip-0.5cm\int\limits_{-T}^T \int\limits_\infty^\infty \Big(1+\Big|A-\Big(B+Q(\frac {y_2}T)\Big)y_2- b(y_1,\frac {y_2}T)\, y_2^3 +r_1(y_1)+\frac {y_2}T r_2(y_1)\Big|\Big)^{-N}\\
&& \hskip1cm \times \Big(1+|y_1|\Big)^{-N}  |y_2|^\epsilon \, dy_1 dy_2.
\end{eqnarray*}
Then, for $N$ sufficiently large,  there are constants $C>0$ and $\ve_0>0,$ which only depend on  the constants $c, c_1$ and $c_2,$ such that for all  functions $b$ and $q$ and all $A,B,T$  with the assumed properties, and all  $\ve\le \ve_0$  and $L\ge C,$ we have 
$$
|I_\epsilon(A,B,T)|\le C \max\{|A|^{\frac 13}, |B|^{\frac 12}\}^{\epsilon-\frac 12}.
$$
\end{lemma}

\begin{proof} It will be convenient for the proof to call a constant $C$ admissible, if it depends only on the constants $\ve, \epsilon, b(0,0)$ and $c,c_1,c_2$ from the statement of the lemma. All constants $C$ appearing within the proof will be admissible, but may change from line to line.

\medskip We begin by the observation that the second assumption in \eqref{a1}  implies that also 
\begin{equation}\label{a1n}
 c_2|y_2|^{3-j}\le \Big|\pa_{y_2}^j\Big(b(y_1,\tau y_2) y_2^3\Big)\Big|,\ j=1,2, \, |y_2|<\tau^{-1},
\end{equation}
for every $\tau>0.$ Indeed, if we fix $y_1$ and put $\psi_\tau(y_2):= b(y_1,\tau y_2) y_2^3,$ then 
$\psi_\tau(y_2)=\psi_1(\tau y_2)/\tau^3,$  so that $\psi''_\tau(y_2)=\psi''_1(\tau y_2)/\tau.$ This immediately leads to \eqref{a1n}.

\medskip
Let us first  assume that $|A|\ge L\ge 1.$ Without loss of generality, we may assume  that  $A,B\ge 0$ (if necessary, we may change the signs of $r$ or $b,q$ as well as  of $y_2$). We may then choose $\al,\be\ge 0$ so that $A=\al^3, B=\be ^2.$ 
\medskip

Next, by  convolving $(1+|\cdot|)^{-N}$ with a suitable smooth bump function, we may choose a smooth, non-negative function $\rho$ on $\RR$ which is integrable and such that also its Fourier transform is integrable and so that $(1+|x|)^{-N}\le \rho(x)\le 2(1+|x|)^{-N},$  and put 
\begin{eqnarray*}
J_\epsilon(\al,\be,T)&:=&\int\limits_{-T}^T \int\limits_\infty^\infty\rho\Big(\al^3-(\be^2+Q(\frac {y_2} T))\,y_2- b(y_1,\frac {y_2}T)\, y_2^3 +r_1(y_1)+\frac {y_2}T r_2(y_1)\Big)\\
&& \hskip 1cm \times \Big(1+|y_1|\Big)^{-N} |y_2|^\epsilon \, dy_1 dy_2.
\end{eqnarray*}
It then suffices to prove that
\begin{equation}\label{a3}
|J_\epsilon(\al,\be,T)|\le C \max\{\al,\be\}^{\epsilon-\frac 12}
\end{equation}
whenever  $L^{1/3}\le \al \le T, \ 0\le \be\le T.$

To this end, performing  the change of variables $y_2=\al s,$ we  re-write 
\begin{eqnarray}\nonumber
J_\epsilon(\al,\be,T)&=& \al^{1+\epsilon}  \int\limits_\infty^\infty \int\limits_{- \frac{T}{\al}}^{\frac {T}{\al}}\rho\Big(\al^3 [1-(\ga +\frac 1{\al^2} Q(\frac \al T s))\, s - b(y_1,\frac {\al}T s)\, s^3]\\
&& \hskip 2cm+r_1(y_1)+s\frac \al T  r_2(y_1)\Big)\Big(1+|y_1|\Big)^{-N} |s|^\epsilon\, ds dy_1, 
\label{a4}
\end{eqnarray}
where 
$$
\ga:=\Big(\frac \be \al\Big)^2.
$$

{\bf 1. Case: $ \ga\ge 1,$ i.e., $\be\ge \al.$} Changing variables $s=\ga^{1/2} t,$ we re-write
\begin{eqnarray*}
J_\epsilon(\al,\be,T)&=& \beta^{1+\epsilon}\int\limits_\infty^\infty \int\limits_{-\frac {T}{\be}}^{\frac { T}{\be}}
\rho\Big(\al^3 -\beta^3 \Big((1+\frac 1{\be^2}Q(\frac \be T t))\, t - b(y_1,\frac {\be}T t)\, t^3\Big)\\
&&\hskip2cm+r_1(y_1)+t\frac \beta T  r_2(y_1)\Big) \Big(1+|y_1|\Big)^{-N}|t|^\epsilon\, dt dy_1.
\end{eqnarray*}

Now, there are admissible constants $C_1,C_2\ge 1 ,$  so that 
if $|t|\ge C_1,$ then 
$$
|\al^3 -\beta^3 (\Big((1+\frac 1{\be^2}Q(\frac \be T t))\, t  - b(y_1,\frac {\be}T t)\, t^3\Big)|\ge  C_2\beta^3 |t|^3.
$$ 
 Notice also that $|r_1(y_1)+t\frac \beta T  r_2(y_1)|\le 2c (1+|y_1|).$ 
Integrating separately in $y_1$ over the sets where $1+|y_1|\le C_2\beta^3 |t|^3/(4c)$ and where 
$1+|y_1|> C_2\beta^3 |t|^3/(4c),$ one then easily  finds that the contribution $J_I(\al,\be,T)$ by the region where $|t|\ge C_1$ to the integral $J_\epsilon(\al,\be,T)$ can be estimated by
$$
|J_I(\al,\be,T)|\le C \beta^{1+\epsilon} \Big( ( \beta^3)^{-N}+ (\beta^3)^{1-N}\Big)\le 2C \be^{4+\epsilon-3N}.
$$
Assume next that  $|t|< C_1.$ We then choose $\chi\in C_0^\infty(\RR)$ such that $\chi(t)=1$ when $|t|\le C_1$ and $\chi(t)=0$ when $|t|\ge 2 C_1,$ and with corresponding control of the derivatives of $\chi.$ The contribution  by the region where $|t|< C_1$ to the integral $J_\epsilon(\al,\be,T)$ can then be estimated by
$$
J_{II}(\al,\be,T):= \beta^{1+\epsilon} \iint
\rho\Big(\al^3 -\beta^3  \phi_{y_1}(t) +r_1(y_1)\Big) \Big(1+|y_1|\Big)^{-N} \chi(t)|t|^{\epsilon}\, dt dy_1,
$$
where we have set 
$$
\phi_{y_1}(t):= \Big(1+\frac 1{\be^2}Q(\frac \be T t)-\frac{r_2(y_1)}{\beta^2 T}\Big)t - b(y_1,\frac {\be}T t)\, t^3, \quad |t|\le  C_1.
$$
Recall here that $T/\be\ge 1$ (for $1\le |t|\le 2C_1,$ we may extend the function $b$ in a suitable way, if necessary).

By Fourier inversion, this can be estimated by 
\begin{eqnarray*}
|J_{II}(\al,\be,T)|&\le& C \beta^{1+\epsilon}
 \Big|\iint
 \int e^{-i \xi\beta^3 \phi_{y_1}(t)} \chi(t)|t|^\epsilon\, dt\, e^{i\xi(\al^3 +ir(y_1))} \Big(1+|y_1|\Big)^{-N}
  \hat \rho(\xi)  d\xi dy_1\Big|\\
 &\le& C \beta^{1+\epsilon}
 \iint \Big| \int e^{-i \xi\beta^3 \phi_{y_1}(t)} \chi(t)|t|^\epsilon\, dt\Big| \Big(1+|y_1|\Big)^{-N}
  |\hat \rho(\xi)|  d\xi dy_1.
 \end{eqnarray*}

Now, if  $ |r_2(y_1)/(\beta^2 T )|\ge \frac 12,$ then $c(1+|y_1|)\ge \beta^2T/2\ge1,$ so trivially the integration in $y_1$ yields that 
 $$
 |J_{II}(\al,\be,T)\le C \beta^{1+\epsilon}(\beta^2 T)^{1-N}\le C \beta^{3+\epsilon-2N}
 $$
 for every $N\in \NN.$ 
 
 So, assume that $|r_2(y_1)/(\beta^2 T)|< \frac 12.$ Then the phase $\phi_{y_1}(t)$ has no degenerate critical point on the support of $\chi(t),$  if we assume $\ve$ to be sufficiently small, then $b$ is a small perturbation of the constant function $b(0,0),$  in the sense of \eqref{a1}. It is then easily verified that our assumptions (and in particular \eqref{a1n}) imply that $\phi_{y_1}(t)$ satisfies the hypotheses 
 of the van der Corput type Lemma 2.2 in \cite{IM-rest1}, with $M=2$ and constants $c_1,c_2>0$   which are admissible, provided we choose $L$ sufficiently large. Therefore, the lemma shows that the inner integral with respect to $t$ is bounded by $C(1+|\xi|\beta^3)^{-1/2},$ which implies that 
 $$
 |J_{II}(\al,\be,T)\le C \beta^{1+\epsilon} (\beta^3)^{-\frac 12}=C \be^{\epsilon-\frac 12}.
 $$

{\bf 2. Case: $ \ga< 1,$ i.e., $\beta<\al.$} Then there are admissible  constants $C_3,C_4\ge 1$ so that  if $|s|\ge C_3,$ then 
$\al^3|1-(\ga +\frac 1{\al^2} Q(\frac \al T s))\, s - b(y_1,\frac {\al}T s)\, s^3|\ge C_4\al^3|s|^3$ in \eqref{a4}. Arguing in a similar way as in the first case, this implies that  the contribution $J_{III}(\al,\be,T)$ by the region where $|s|\ge C_3$ to the integral $J_\epsilon(\al,\be,T)$ in \eqref{a4}  can be estimated by
$$
|J_{III}(\al,\be,T)|\le C \al^{4+\epsilon-3N}.
$$
Similarly, there are admissible  constants $C_5,C_6>0$ so that  if $|s|\le  C_5,$ then 
$|1-(\ga +\frac 1{\al^2} Q(\frac \al T s))\, s - b(y_1,\frac {\al}T s)\, s^3|\ge C_6$ in \eqref{a4}, and this implies that  the contribution $J_{IV}(\al,\be,T)$ by the region where $|s|\le C_5$ to the integral $J_\epsilon(\al,\be,T)$ can be estimated by
$$
|J_{IV}(\al,\be,T)|\le C \al^{4+\epsilon-3N}.
$$
Finally, on the set where   $C_5<|s|< C_3,$ the phase $\phi_{y_1}(s):= (\frac{r_2(y_1)}{\al ^2 T}-(\ga +\frac 1{\al^2} Q(\frac \al T s)))s - b(y_1,\frac {\al}T s)\, s^3$ has again no degenerate critical point, and way conclude in a similar way as in the first case that the contribution $J_V(\al,\be,T)$ by this region can be estimated by
$$
 |J_V(\al,\be,T)\le C \al^{1+\epsilon}  (\al^3)^{-\frac 12}=C \al^{\epsilon-\frac 12}.
 $$
Combining all these estimates, we arrive at the conclusion of the lemma when $|A|\ge 1.$ 
\smallskip

Finally, when $|A|\le L$ and $|B|\ge L,$ then we may  indeed assume without loss of generality that $\al=0.$ Arguments very similar to those that we have applied before then show that $|J_\epsilon(\al,\be,T)|\le C \be^{\epsilon-\frac 12},$ which concludes the proof of the lemma.
\end{proof}

 \begin{lemma}\label{as2}
Let $b=b(y)$ be a $C^2$- function on $[-1,1]$ such that $b(0)\ne 0$ and $\|b\|_{C^2([-1,1])}\le c_1.$  Furthermore, let  $A$ and $B$ be real numbers,  and let  $\de_0\in ]0,1[.$ For $T\ge L$   and $\de>0$   such that $\de<1$ and $\de T\le \de_0,$  we put 
\begin{eqnarray*}
I(A,B)&:=&
 \hskip-0.5cm\int_\RR \big|\rho\big(A+By+ b(\de y) y^3 \big)\big| \,\chi_0\big(\frac y T\big)\, dy\, ,
 \end{eqnarray*}
 where $\rho\in \S(\RR)$ denotes a fixed Schwartz function and  $\chi_0\in C^\infty_0(\RR)$  a  non-negative bump function supported in the interval $[-1,1].$  Then,  for $\de_0$ sufficiently small and $L$ sufficiently large, we have 
\begin{equation}\label{a6}
|I(A,B)|\le C \big(1+\max\{|A|^{\frac 13}, |B|^{\frac 12}\}\big)^{-\frac 12},
\end{equation}
where the constant $C$ depends only on $c_1, \de_0, \rho$ and $\chi_0,$ but not on  $A,B,T$ and $\de.$ 
\end{lemma}

\begin{proof} 
In a similar way is in the preceding proof, we may dominate the function $|\rho|$ by a non-negative Schwartz function. Let us therefore assume without loss of generality that $\rho\ge 0.$  

Let us first assume that  $\max\{|A|,|B|\}\ge L.$ 
If $|A|\le T^3/\de_0^3$ and $|B|\le T^2/\de_0^2,$  estimate \eqref{a6} follows easily from the previous Lemma \ref{as1} (just replace $\chi_0(\cdot/T)$ by the characteristic function of the interval $[-1/\de,1/\de]$ ). 

Next, if $|B|> T^2/\de_0^2,$ using Fourier inversion, we may estimate
\begin{equation}\label{a7}
|I(A,B)|\le C\int \big|\int e^{is \phi(y)}\chi_0\big (\frac yT\big)\, dy \big| \,| \hat\rho(s)| \, ds\,,
\end{equation}
where $\phi(y):=By+ b(\de y) y^3.$ And, be means of integration by parts, we  obtain that 
$$
\big|\int e^{is \phi(y)}\chi_0\big (\frac yT\big)\, dy \big|\le  C|s|^{-1}
\int_{-T}^T\Big( \Big|\frac{\phi''(y)}{\phi'(y)^2}\Big|+\Big|\frac {1}{T\phi'(y)}\Big|\Big) \, dy.
$$
Since  for $|y|\le T$ we have 
$$
|\phi'(y)|=|B-y^2\big(3b(\de y) +\frac {\de y}3  b'(\de y) \big)|\ge \frac {|B|}2\quad \mbox{and}\quad 
|\phi''(y)|\le C T,
$$
if we choose $\de_0$ sufficiently small, we see that $\big|\int e^{is \phi(y)}\chi_0\big (\frac yT\big)\, dy \big|\le C /|sB|$ for $\de_0$ sufficiently small. On the other hand, trivially we have  $\big|\int e^{is \phi(y)}\chi_0\big (\frac yT\big)\, dy \big|\le C T,$  and taking the geometric mean of these estimates and using  that $T<\de_0|B|^{1/2}$  we find that
 $$
 \big|\int e^{is \phi(y)}\chi_0\big (\frac yT\big)\, dy \big|\le C |s|^{-\frac 12} |B|^{-\frac 14}.
 $$ 
 In combination with \eqref{a7} this  confirms estimate \eqref{a6} also in this case.
 
 There remains the case where $|A|>T^3/\de_0^3$ and $|B|\le T^2/\de_0^2.$ Here we may  estimate
 $$
 |A+By+ b(\de y) y^3|\ge |A|-|B|T-\|b\|_\infty T^3\ge \frac{|A|}2,
 $$
 provided $\de_0$ is sufficiently small. This implies that $ |I(A,B)|\le C|A|^{-N}T\le C |A|^{-N+1/3}$ for every $N\in \NN,$ which is stronger than what we need for \eqref{a6}.
 \medskip
 
Finally, if  $\max\{|A|,|B|\}< L,$ then  the the rapid decay of $\rho$ easily implies that $ |I(A,B)|\le C.$
 \end{proof}

\end{document}